\documentclass{amsart}
\usepackage{graphicx}

\usepackage{bm,caption,amsbsy,enumitem,amsmath,amsthm,
amssymb,mathtools,amsfonts,multirow,verbatim,tikz,tikz-cd,thmtools,thm-restate,xr,xcolor}
\usepackage[numeric]{amsrefs}
\usetikzlibrary{matrix,arrows,decorations.pathmorphing}
\usepackage[top=1.25in, bottom=1in, left=1.25in, right=1.25in,marginpar=1in]{geometry}
\usepackage{url}
\DefineSimpleKey{bib}{myurl}
\usepackage[hidelinks]{hyperref}
\usepackage{chngcntr}
\usepackage[all]{xy}
\usepackage{leftidx}
\counterwithin{figure}{section}
\numberwithin{equation}{section}

\newcommand\myurl[1]{\url{#1}}

\BibSpec{webpage}{%
  +{}{\PrintAuthors} {author}
  +{,}{ \textit} {title}
  +{}{ \parenthesize} {date}
  +{,}{ \myurl} {myurl}
  +{,}{ } {note}
  +{.}{ } {transition}
}

\newtheorem{innercustomthm}{Theorem}
\newenvironment{customthm}[1]
  {\renewcommand\theinnercustomthm{#1}\innercustomthm}
  {\endinnercustomthm}

\newtheorem{thm}{Theorem}[section]
\newtheorem{prop}[thm]{Proposition}
\newtheorem{conj}[thm]{Conjecture}
\newtheorem{cor}[thm]{Corollary}
\newtheorem{lem}[thm]{Lemma}

\theoremstyle{definition}
\newtheorem{define}[thm]{Definition}

\theoremstyle{remark}
\newtheorem{rem}[thm]{Remark}
\newtheorem{example}[thm]{Example}
\newtheorem{question}[thm]{Question}

\newcommand{\ve}[1]{\boldsymbol{\mathbf{#1}}}
\newcommand{\R}{\mathbb{R}}
\newcommand{\F}{\mathbb{F}}
\newcommand{\Z}{\mathbb{Z}}
\newcommand{\N}{\mathbb{N}}

\newcommand{\T}{\mathbb{T}}

\renewcommand{\d}{\partial}
\renewcommand{\subset}{\subseteq}

\renewcommand{\tilde}{\widetilde}
\renewcommand{\hat}{\widehat}
\newcommand{\iso}{\cong}

\DeclareMathOperator{\Char}{{Char}}

\DeclareMathOperator{\cotr}{{cotr}}

\DeclareMathOperator{\Cr}{{Cr}}

\DeclareMathOperator{\SO}{{SO}}
\DeclareMathOperator{\gr}{{gr}}

\DeclareMathOperator{\Hom}{{Hom}}
\DeclareMathOperator{\Id}{{Id}}
\DeclareMathOperator{\id}{{id}}

\DeclareMathOperator{\Int}{{int}}

\DeclareMathOperator{\MCG}{{MCG}}

\DeclareMathOperator{\Spin}{{Spin}}

\DeclareMathOperator{\Tor}{{Tor}}

\DeclareMathOperator{\Sing}{{Sing}}

\DeclareMathOperator{\Tors}{{Tors}}

\DeclareMathOperator{\tr}{{tr}}
\DeclareMathOperator{\tw}{tw}

\renewcommand{\wr}{\operatorname{wr}}

\DeclareMathOperator{\Imm}{{Imm}}

\DeclareMathOperator{\Cob}{Cob}

\DeclareMathOperator{\Top}{Top}

\newcommand{\bF}{\mathbb{F}}

\newcommand{\bH}{\mathbb{H}}

\newcommand{\bK}{\mathbb{K}}
\newcommand{\bL}{\mathbb{L}}

\newcommand{\bS}{\mathbb{S}}
\newcommand{\bT}{\mathbb{T}}
\newcommand{\bU}{\mathbb{U}}

\newcommand{\CP}{\mathbb{CP}}

\newcommand{\cA}{\mathcal{A}}
\newcommand{\cB}{\mathcal{B}}
\newcommand{\cC}{\mathcal{C}}
\newcommand{\cD}{\mathcal{D}}

\newcommand{\cF}{\mathcal{F}}
\newcommand{\cG}{\mathcal{G}}
\newcommand{\cH}{\mathcal{H}}
\newcommand{\cI}{\mathcal{I}}

\newcommand{\cM}{\mathcal{M}}

\newcommand{\cR}{\mathcal{R}}
\newcommand{\cS}{\mathcal{S}}
\newcommand{\cT}{\mathcal{T}}
\newcommand{\cU}{\mathcal{U}}

\newcommand{\frs}{\mathfrak{s}}
\newcommand{\frt}{\mathfrak{t}}

\newcommand{\as}{\ve{\alpha}}
\newcommand{\bs}{\ve{\beta}}
\newcommand{\gs}{\ve{\gamma}}

\newcommand{\Ds}{\ve{\Delta}}
\newcommand{\xs}{\ve{x}}
\newcommand{\ys}{\ve{y}}
\newcommand{\ws}{\ve{w}}
\newcommand{\zs}{\ve{z}}
\newcommand{\ts}{\ve{t}}
\newcommand{\sigmas}{\ve{\sigma}}
\newcommand{\taus}{\ve{\tau}}
\newcommand{\etas}{\ve{\eta}}
\newcommand{\xis}{\ve{\xi}}
\newcommand{\zetas}{\ve{\zeta}}

\newcommand{\from}{\leftarrow}

\newcommand{\cCFL}{\mathcal{C\!F\!L}}
\newcommand{\cHFL}{\mathcal{H\!F\! L}}

\newcommand{\CFK}{\mathit{CFK}}

\newcommand{\CF}{\mathit{CF}}
\newcommand{\HF}{\mathit{HF}}
\newcommand{\HFh}{\widehat{\mathit{HF}}}

\newcommand{\HFL}{\mathit{HFL}}

\newcommand{\HFK}{\mathit{HFK}}

\renewcommand{\a}{\alpha}
\renewcommand{\b}{\beta}
\newcommand{\g}{\gamma}
\renewcommand{\S}{\Sigma}

\newcommand{\x}{\mathbf{x}}
\newcommand{\y}{\mathbf{y}}

\newcommand{\PD}{\mathit{PD}}

\DeclareMathOperator{\st}{{st}}
\DeclareMathOperator{\Sw}{{Sw}}

\DeclareMathOperator{\dest}{{dest}}

\newcommand{\I}[1]{\leftidx{_{#1}}{I}}
\renewcommand{\L}[2]{\leftidx{_{#2}}{L}{_{#1}}}

\newcommand{\fil}{\mathit{fil}}
\newcommand{\hCFKf}{\hat{\CFK}{}^{\fil}}
\newcommand{\hCFKfw}{\hat{\CFK}{}^{\fil,w}}
\newcommand{\hCFKfz}{\hat{\CFK}{}^{\fil,z}}

\newcommand{\Tw}{\mathit{Tw}}

\newcommand{\Sphere}{\mathbb{S}}

\DeclareMathOperator{\Surf}{{Surf}}

\newcommand{\tHFK}{\mathit{tHFK}}
\newcommand{\tCFK}{\mathit{tCFK}}
\newcommand{\tF}{\mathit{tF}}
\newcommand{\tvt}{\mathit{t}\ve{t}}

\title{Stabilization distance bounds from link Floer homology}

\author{Andr\'as Juh\'asz}%
\address{Mathematical Institute, University of Oxford, Andrew Wiles Building,
Radcliffe Observatory Quarter, Woodstock Road, Oxford, OX2 6GG, UK}%
\email{juhasza@maths.ox.ac.uk}%

\author{Ian Zemke}
\address{Department of Mathematics\\Princeton University\\  Princeton, NJ 08544, USA}
\email{izemke@math.princeton.edu}

\begin{document}

\begin{abstract}
  We consider the set of connected surfaces in the 4-ball with boundary a fixed knot in the 3-sphere.
  We define the stabilization distance between two surfaces as the minimal $g$ such that
  we can get from one to the other using stabilizations and destabilizations through surfaces of genus at most $g$.
  Similarly, we consider a double point distance between two surfaces of the same genus,
  which is the minimum over all regular homotopies connecting the two surfaces
  of the maximal number of  double points appearing in the homotopy. 

  To many of the concordance invariants defined using Heegaard Floer homology,
  we construct an analogous invariant for a pair of surfaces. We show
  that these give lower bounds on the stabilization distance and the double point distance.
  We compute our invariants for some pairs of deform-spun slice disks
  by proving a trace formula on the full infinity knot Floer complex, and by
  determining the action on knot Floer homology of an automorphism of the connected sum of
  a knot with itself that swaps the two summands.
   We use our invariants to find pairs of slice disks with arbitrarily large distance with respect to many of the metrics we consider in this paper. 
   We also answer a slice disk analogue of Problem~1.105 (B) from Kirby's problem list by showing the existence of non-0-cobordant slice disks.
\end{abstract}

\maketitle

\tableofcontents

\section{Introduction}

There is a natural stabilization operation on smooth, oriented surfaces in 4-manifolds
where one attaches an embedded 1-handle to the surface. This operation was recently considered by Baykur and Sunukjian~\cite{BaySunStabilizations}.
When the 1-handle is unknotted, this does not change the fundamental group of the surface complement.
They asked the following question: If two surfaces are topologically isotopic,
then do they become smoothly isotopic after a single unknotted 1-handle stabilization?
They verified this
for most known constructions of pairs of exotic surfaces, such as rim surgery. A related question is how many stabilizations are required to make a given pair of surfaces isotopic.

There is a parallel notion of stabilization for 4-manifolds. A classical result of Wall~\cite{Wall} states that if two smooth, simply-connected 4-manifolds
are homeomorphic, then they become diffeomorphic after taking connected sums with
some number of copies of $\Sphere^2 \times \Sphere^2$. It is an open conjecture
whether a single copy of $\Sphere^2 \times \Sphere^2$ always suffices. Recently, Lin and Mukherjee~\cite{Lin-Mukherjee} constructed a pair of surfaces-with-boundary in a punctured K3 that are topologically isotopic but not smoothly isotopic, and remain so after stabilizing their complements once with $\Sphere^2 \times \Sphere^2$. They showed this using family Bauer--Furuta invariants.

In this paper, we construct invariants that provide lower bounds on the number of
1-handle stabilizations required to make two surfaces isotopic.
Using a generalization of 1-handle stabilization, we endow the set of smooth, connected, oriented,
and properly embedded surfaces in $B^4$ with boundary a knot $K$ with a type of metric
that we call the \emph{stabilization distance}. The stabilization distance between two surfaces
bounds from below the number of 1-handle stabilizations required to make them isotopic.
Using the link Floer TQFT, we define several invariants of pairs of surfaces bounding $K$ that
give lower bounds on the stabilization distance. We compute these invariants for certain pairs
of slice disks arising from deform-spinning, and observe they
often give non-trivial lower bounds. Furthermore, we give examples in Section~\ref{sec:examples} where these lower bounds can be arbitrarily large.

Throughout this paper, we work in the smooth category, and all manifolds
are assumed to be oriented, unless otherwise stated.

\subsection{Metric filtrations on the set of surfaces bounding a knot}

In Definition~\ref{def:stab}, we introduce a very general type of stabilization operation for surfaces in 4-manifolds
that extends the 1-handle stabilization considered by Baykur and Sunukjian.
Let $S$ be a properly embedded surface in a 4-manifold~$W$.
To obtain the stabilization of $S$, we choose a 4-ball $B \subset \Int(W)$ such that $B \cap S$
is a collection of disks that can be isotoped into $\d B^4$ relative to their boundaries.
In particular, $\d B \cap S$ is an unlink. We then replace $S \cap B$
with a properly embedded surface $S_0 \subset B$ such that $\d S_0 = \d B \cap S$.
Any two surfaces in the same relative homology class can be related by a finite sequence of
such stabilizations and destabilizations (in fact, 1-handle stabilizations suffice).

Let $K$ be a knot in $\Sphere^3$. We denote by $\Surf(K)$ the set of isotopy classes of connected, properly embedded surfaces
in $B^4$ with boundary $K$. In Definition~\ref{def:stabgenus}, we introduce the \emph{stabilization distance}
$\mu_{\st}(S, S')$ of a pair of surfaces $S$, $S' \in \Surf(K)$
to be the minimum of
\[
\max\{g(S_1), \dots, g(S_k) \}
\]
over sequences of connected, properly embedded surfaces $S_1$, \dots, $S_k$ in $\Surf(K)$
connecting $S$ and~$S'$ such that consecutive surfaces are related by a stabilization or a destabilization.
Furthermore, if $K$ is slice and $S \in \Surf(K)$, then
we define the \emph{destabilizing genus} $g_{\dest}(S)$ to be the stabilization distance of $S$ from
the subset of slice disks.

We also define the \emph{M-distance} function $M_{(S, S')} \colon [0,2] \to \R^{\ge 0}$
of a pair of surfaces $S$, $S' \in \Surf(K)$, which is similar to the stabilization distance,
but where the stabilization operation is allowed to change
the ambient 4-manifold. Instead of changing the surface in a 4-ball,
one can glue in a pair $(X_0, S_0)$, where $\d X_0 = \Sphere^3$ and $\d S_0$ is an unlink,
and $b_2^+(X_0) = b_1(X_0) = 0$.
The \emph{$M$-degree} of a pair $(W, S)$, where $W$ is a compact 4-manifold and $S$ is a
properly embedded surface, is defined in Section~\ref{sec:M-distance}. It is a function
on $[0,2]$ that measures not only the genus of $S$,
but also a homological quantity depending on $[S] \in H_2(W)$
and the intersection form $Q_W$ of $W$.
The $M$-distance of a pair of surfaces $S$, $S' \in \Surf(K)$
minimizes the maximal $M$-degree along sequences $(W_1, S_1) \dots, (W_n, S_n)$
connecting $(B^4, S)$ and $(B^4, S')$ such that $(W_i, S_i)$
and $(W_{i+1}, S_{i+1})$ are related by the above stabilization operation.

Another notion of distance we consider in this paper is the \emph{cobordism distance}, which we denote $\mu_{\Cob}(S,S')$.  If $S$, $S'\in \Surf(K)$, we set $\mu_{\Cob}(S,S')$ to be the minimal $g$ such that there is a 5-dimensional cobordism $(I \times B^4, Y)$, where $Y$ is a smoothly and properly embedded, oriented 3-manifold-with-corners such that 
\[
\d Y= -(\{0\}\times S)\cup (\{1\}\times S')\cup (I\times K),
\]
projection of $Y$ to $I$ is Morse, and such that each regular level set of $Y$ is a surface such that the sum of the genera of its components is at most $g$. We say that $S$ and $S'$ are \emph{strictly $g$-cobordant} if $\mu_{\Cob}(S, S') = g$. Compare this to the notion of \emph{$g$-cobordism} of 2-knots defined by Melvin~\cite{MelvinPHD}, where one requires every component of each level set to have genus at most $g$, and \emph{$g$-concordance}, where, in addition, $Y \approx I \times S^2$. Two surfaces are strictly $0$-cobordant if and only if they are $0$-cobordant. 
In particular, $\mu_{\Cob}$ defines an ultrametric on the set of $0$-cobordism classes of surfaces.
It is straightforward to see that
\[
\mu_{\Cob}(S,S')\le \mu_{\st}(S,S').
\]

Note that Sunukjian showed that there are infinitely many distinct $0$-concordance classes of 2-knots in $S^4$ \cite{SunukjianConcordance}. Dai and Miller \cite{DaiMillerConcordance} improved this result to show that the 0-concordance monoid of 2-knots was infinitely generated.

For $g \in \N$, let $\Surf_g(K)$ denote the subset of $\Surf(K)$ consisting
of genus $g$ surfaces. If $S$, $S' \in \Surf_g(K)$, then they are regularly homotopic relative to $K$.
We define the \emph{double point distance} $\mu_{\Sing}(S, S')$ as $\tilde{\mu}_{\Sing}(S,S') + g$, where
$\tilde{\mu}_{\Sing}(S,S')$ is obtained by minimizing half the maximal number of double points
that appear during a regular homotopy from $S$ to $S'$; see Definition~\ref{def:sing}.
When $g(S) \neq g(S')$, we set $\mu_{\Sing}(S, S') = \infty$.
Motivated by an earlier version of this paper, Singh~\cite{Singh} showed that 
\[
\mu_{\st}(S, S') \le \mu_{\Sing}(S, S')  + 1
\]
using techniques from Gabai's proof~\cite{light-bulb} of the 4-dimensional light bulb theorem. It is an open problem whether the $+1$ is necessary in the above formula. 

Both $\mu_{\st}$ and $\mu_{\Sing}$ are \emph{metric filtrations} on $\Surf(K)$;
i.e., they are nonnegative, symmetric, and satisfy the ultrametric inequality
\[
\mu(S, S'') \le \max\{\, \mu(S, S'), \mu(S', S'') \,\}
\]
for any $S$, $S'$, $S'' \in \Surf(K)$; see Section~\ref{sec:metric-filtration}.
Furthermore, $M_{(S,S')}(t)$ is also a metric filtration for every $t \in [0,2]$.
If $S$, $S' \in \Surf_0(K)$, then $\mu_{\st}(S, S') = 0$ if and only if $S$ and $S'$ are related
by genus zero stabilizations, hence $\Surf_0(K)/\{\text{2-knots}\}$ is an ultrametric space.
However, in general, $\mu_{\st}(S, S) = g(S)$, and if $\mu_{\st}(S, S') = g(S) = g(S')$,
it is possible that $S$ and $S'$ are not related by genus zero stabilizations.

\subsection{Lower bounds from Heegaard Floer homology}
Our aim is to provide computable lower bounds on the stabilization distance, the double point distance, the destabilizing genus, and the cobordism distance using the link Floer TQFT of the second author~\cite{ZemCFLTQFT}, which extends the TQFT of the
first author~\cite{JCob} from $\widehat{\HFL}$ to the full infinity complex $\cCFL^\infty$.
If $\bK = (K, w, z)$ is a doubly-based knot in $\Sphere^3$, then $\cCFL^\infty(\bK)$ is a $\Z\oplus \Z$
filtered chain complex over the two-variable Laurent polynomial ring
\[
\cR^\infty:=\bF_2[U,V,U^{-1},V^{-1}].
\]
We give an overview of  $\cCFL^\infty(\bK)$ in Section~\ref{sec:background}.
There is a variation, denoted $\cCFL^-(\bK)$, corresponding to the subspace
of $\cCFL^\infty(\bK)$ in nonnegative bi-filtration, which is a module over the ring
\[
\cR^-:=\bF_2[U,V].
\]

Knot Floer homology has been used by many authors to provide numerical concordance invariants that provide deep geometric information. Important examples are the knot invariants $\tau$, $\nu$, $V_k$ for $k \in \N$,
and $\Upsilon$, introduced by Ozsv\'ath--Szab\'o~\cite{OS4ballgenus, OSRationalSurgeries}, Hom--Wu~\cite{HomWuNu+}, Rasmussen~\cite{RasmussenKnots},
and Ozsv\'ath--Stipsicz--Szab\'o~\cite{OSSUpsilon}, respectively.
Since  $\tau$, $\nu$, $V_k$, and $\Upsilon$ are concordance invariants, they
all vanish when the knot is slice.

It is well known that the knot invariants $\tau$, $\nu$, and $V_k$ give lower bounds on the 4-ball genus $g_4(K)$:
\[
\tau(K) \le \nu(K) \le g_4(K) \text{ and } V_k(K) \le \left\lceil \frac{g_4(K)-k}{2} \right\rceil
\]
whenever $k \le g_4(K)$; see \cite{OS4ballgenus}*{Theorem~1.1},
\cite{HomWuNu+}*{Section~2}, and \cite{RasmussenGodaTeragaito}*{Theorem~2.3}.
The invariant $\Upsilon$ also gives lower bounds on the 4-ball genus; see
\cite{OSSUpsilon}*{Theorem~1.11}. The invariants $\tau$ and $\Upsilon$
satisfy more general versions of the genus bound involving surfaces
in negative definite 4-manifolds $W$ with $b_1(W) = b_2^+(W) = 0$; see
\cite{OS4ballgenus}*{Theorem~1.1} and
\cite{ZemAbsoluteGradings}*{Theorem~1.1}.

In Section~\ref{sec:invariants}, we show that, by mirroring their constructions,
we can define secondary versions of all of the above knot invariants for a pair of surfaces
$S$, $S' \in \Surf(K)$ in $B^4$ with boundary a knot $K$ in $\Sphere^3$.
We show that these give formally analogous lower bounds on the metric filtrations
$\mu_{\st}$, $\mu_{\Sing}$, and $M$:

\begin{thm}\label{thm:maintheorem}
  Let $S$, $S' \in \Surf(K)$. Then
  \[
  \tau(S, S') \le \nu(S, S') \le \min\{\, \mu_{\st}(S, S') \text{, } \mu_{\Sing}(S, S') \,\}.
  \]
  Furthermore,
  \[
  \Upsilon_{(S,S')}(t) \le M_{(S,S')}(t)
  \]
  for every $t \in [0,2]$.
  Finally, for every $k \in \N$,
  \begin{equation}
  V_k(S, S') \le \left\lceil \frac{\mu_{\Sing}(S, S') - k}{2} \right \rceil. \label{eq:Vk-mu-Sing-intro}
  \end{equation}
  If $S$ and $S'$ are disks, then Equation~\eqref{eq:Vk-mu-Sing-intro} also holds with $\mu_{\st}$ in place of $\mu_{\Sing}$.
\end{thm}

The bounds stated in Theorem~\ref{thm:maintheorem} are proven separately throughout the paper;
see Theorems~\ref{thm:tauD1D2bound}, \ref{thm:Vkbound}, \ref{thm:doublepointbound},
\ref{thm:hinvdoublepointhighergenus}, and~\ref{thm:Upsilonbound},
as well as Proposition~\ref{prop:tau+}. In Section~\ref{sec:kappa}, we introduce a novel
invariant $\kappa(S, S')$ that does not have an analogue for knots,
and which only gives a lower bound on $\mu_{\Sing}(S, S')$; see Theorem~\ref{thm:kappaandmuSing}.

Using the link Floer TQFT, we also introduce another integer invariant
$\cI(S)$ of a surface $S$ which has $g(S) > 0$ and has boundary a slice knot $K$;
see Definition~\ref{def:depth}. In Theorem~\ref{thm:I}, we prove that $\cI(S)$ bounds the stabilization distance
between $S$ and the subset of slice disks:
\[
\cI(S) \le g_{\dest}(S).
\]
However, $\cI(S)$ is in general difficult to compute,
as it involves determining the link cobordism maps for infinitely many
decorations on the surface $S$. Giving a lower bound on $\cI(S)$ is theoretically feasible,
since it only involves finding two decorations on $S$ satisfying a
simple condition, whose induced cobordism maps disagree.
We define an analogue of $\kappa(S, S')$ for a single surface, $\kappa_0(S)$, which gives a computable
lower bound on $g_{\dest}(S)$; see Theorem~\ref{thm:kappa0gdest}.\footnote{In subsequent work \cite{JZGenusBounds}, we show that $\cI(S),\kappa_0(S)\in \{g(S),g(S)+1\}$, and hence $\cI(S)$ and $\kappa_0(S)$ give potential obstructions to surfaces being stabilized (cf.  Proposition~5.5 therein). By analyzing the proof of \cite{JZGenusBounds}*{Theorem~1.7} and by considering cases where the bound is sharp, one may find surfaces where $\cI(S)=g(S)+1$. For example, any genus 2 slice surface for $T_{2,3}\# T_{2,3}$ has $\cI(S)=3$.}

The invariants $\tau$, $\kappa$, and $\kappa_0$ can all be derived from $\Upsilon$.
As opposed to the case of knots, $\Upsilon_{(S, S')}(t)$ is not a symmetric function.
We will see in Theorem~\ref{thm:upsilon-slope} that, for all $t\in [0,2]$ sufficiently close to~0, we have
\[
\Upsilon_{(S,S')}(t)= \tau(S,S') \cdot t.
\]
However, for $t$ sufficiently close to $2$, we have
\[
\Upsilon_{(S,S')}(t)=\begin{cases} (\kappa_0(S)-g(S))\cdot (2-t)+ g(S)\cdot t& \text{ if } g(S)>g(S'),\\
(\kappa(S,S')-g(S))\cdot (2-t)+g(S)\cdot t& \text{ if } g(S)=g(S').
\end{cases}
\]

We now review the construction of the invariants.
If $S \in \Surf(K)$, then it can be viewed as a
link cobordism from $\emptyset$ to $(\Sphere^3, K)$.
If we decorate it such that the type-$\ws$ region is a bigon, then it induces a filtered chain map
\[
\ve{t}_{S, \zs}^\infty \colon \cR^\infty \to \cCFL^\infty(\bK),
\]
well-defined up to filtered chain homotopy; see Section~\ref{sec:invariant}.
If instead the type-$\zs$ region is a bigon, we obtain a map $\ve{t}_{S, \ws}^\infty$.
We call $\ve{t}_{S, \zs}^\infty$ and $\ve{t}_{S, \ws}^\infty$ the \emph{extremal principal invariants}
of the surface $S$.

Given surfaces $S$, $S' \in \Surf(K)$, the invariants
\[
\tau(S, S') \text{, } \nu(S, S') \text{, } V_k(S, S') \text{, and } \Upsilon_{S, S'}(t)
\]
are all extracted from the pair of maps $(\ve{t}^\infty_{S, \zs}, \ve{t}^\infty_{S', \zs})$
by algebraically mirroring the construction of their knot invariant counterparts.
Hence, we think of our invariants as \emph{secondary} versions of the knot invariants.
The invariant $\kappa_0(S)$ is derived from $\ve{t}_{S, \ws}^\infty$, and we obtain
$\kappa(S, S')$ from the pair $(\ve{t}^\infty_{S, \ws}, \ve{t}^\infty_{S', \ws})$.

For example, to obtain $\tau$, we set $U = 0$ and $V = 1$ in the $\cR$-module
$\cCFL^-(\bK)$, and obtain the $\Z$-filtered complex $\hCFKfz(\bK)$ whose homology is $\HFh(\Sphere^3) \cong \bF_2$.
Given a pair of surfaces $S$, $S' \in \Surf(K)$,
the elements $\ve{t}_{S, \zs}^-(1)$ and $\ve{t}_{S', \zs}^-(1)$ of $\cCFL^-(\bK)$ give rise to
elements $\hat{t}_{S, \zs}(1)$ and $\hat{t}_{S', \zs}(1)$ of $\hCFKfz(\bK)$. We define the invariant
\[
\tau(S, S') := \min\{ n \ge \max\{g, g'\} : \left[\hat{t}_{S, \zs}(1) \right] =
\left[\hat{t}_{S', \zs}(1)\right] \text{ as elements of }  H_*(\hCFKfz_n(\bK))\},
\]
where $\hCFKfz_n(\bK)$ is the part of $\hCFKfz(\bK)$ in filtration at most~$n$.

\subsection{Computing the invariants for deform-spun disks using the trace formula}

In Section~\ref{sec:examples}, we compute the secondary invariants for
several pairs of deform-spun slice disks, using the computer algebra software
SageMath~\cite{sagemath}. We exhibit several
examples where the lower bound on the stabilization distance is 2 or 3, and a family with arbitrarily large distance. 
We note that, in our examples, the pairs of slice disks are \emph{not}
topologically isotopic relative to their boundaries.

Let $K$ be a knot in $\Sphere^3$, and suppose that the 3-ball $B$ intersects $K$ in an unknotted arc.
Then $(\Sphere^3 \setminus \Int(B), K \setminus \Int(B))$ is a ball-arc pair $(B^3, a)$.
Suppose that we are given an isotopy $\phi \colon I \times \Sphere^3 \to \Sphere^3$ of $\Sphere^3$ that is the identity on $B$, and such
that $\phi_0 = \Id_{\Sphere^3}$ and $\phi_1(K) = K$, where $\phi_t(x) := \phi(t, x)$ for every $t \in I$ and $x \in \Sphere^3$. 
Then the \emph{deform-spun} slice disk $D_{K, \phi} \subset B^4$ is defined by taking
\[
\bigcup_{t \in I} \{t\} \times \phi_t(a) \subset I \times B^3,
\]
and rounding the corners along $\{0, 1\} \times \d B^3$.
This is a slice disk of $-K \# K$. The isotopy class of $D_{K, \phi}$, relative to $-K \# K$,
only depends on the diffeomorphism $\phi_1$, so we will write $D_{K, \phi_1}$ instead of $D_{K, \phi}$.

In this paper, we consider three automorphisms of a ball-arc pair.
The first is the \emph{roll-spinning} automorphism $r$, which
corresponds to a positive Dehn twist in the longitudinal direction of the knot $K$.
The automorphism $r$ is supported in a neighborhood of $K$.
The second is the \emph{twist-spinning} automorphism $t$, which is similar to roll-spinning,
but instead twists in the meridional direction.
The third automorphism $\ve{R}^\pi$ that we consider is specific to
knots of the form $K \# K$. The \emph{summand-swapping} automorphism $\ve{R}^\pi$ of $K \# K$ is the
composition of an isometry of $\R^3$ that swaps the two copies of $K$, and a
half Dehn twist that ensures $\ve{R}^\pi$ fixes $K \# K$ pointwise.

We show that, if $K$ is a knot in $\Sphere^3$, then the roll-spun and twist-spun slice disks $D_{K, r}$,
$D_{K, t} \in \Surf_0(-K \# K)$ satisfy
\[
\mu_{st}(D_{K, r}, D_{K, t}) \le 2;
\]
see Proposition~\ref{prop:rollspinningcomputation}. We conjecture that the upper bound can be improved to~$1$.

We obtain lower bounds on the stabilization distance and the double point distance using our secondary invariants.
For this end, we first compute the extremal principal invariant
$\ve{t}_{D_{K,{\phi_1}}}^\infty(1) \in \cCFL^\infty( -K \# K)$ of a deform-spun slice disk. 
We define the canonical cotrace map
\[
\cotr\colon \cR^\infty\to \cCFL^\infty(Y,\bL,\frs)\otimes_{\cR^\infty} \cCFL^\infty(-Y,-\bL,\frs)
\]
as the dual of the trace pairing.

\begin{thm}\label{thm:spinning}
  Let $D_{K, \varphi}$ be a slice disk of the doubly-based knot $-\bK \# \bK = (-K \# K, w, z)$, obtained by deform-spinning
  a doubly-based knot $\bK = (K, w, z)$ in $\Sphere^3$ using an automorphism $\varphi$ of $(\Sphere^3, K)$.
  If we write $C := \cCFL^\infty(\bK)$, then
  \[
  E \circ \ve{t}_{D_{K, \varphi}}^\infty \simeq (\id \otimes \varphi_*) \circ \cotr \in
  \Hom_{\cR^\infty}(\cR^\infty, C^\vee \otimes C),
  \]
  where  $E \colon \cCFL^\infty(-\bK \# \bK) \to C^\vee \otimes C$ is the chain homotopy equivalence
  induced by a pair-of-pants link cobordism.
\end{thm}

Our proof of Theorem~\ref{thm:spinning} uses a much more general
trace formula for the full link Floer TQFT, Theorem~\ref{thm:heegaardtriplemap},
extending a result from our earlier work \cite[Theorem~1.1]{JZContactHandles} for sutured Floer
homology, as well as a result of the second author for the graph TQFT
\cite{ZemDualityMappingTori}*{Theorem~1.6}.

Hence, to compute the secondary invariants of deform-spun slice disks, it
remains to determine the induced diffeomorphism map $\varphi_*$ in specific examples.
The map $r_*$ induced by the basepoint moving diffeomorphism is well-known,
and is given by the formula
\[
r_*\simeq \id+\Phi\circ \Psi,
\]
which was proven on $\HFK^-(\bK)$ by Sarkar~\cite{SarkarMovingBasepoints}, and
extended to $\cCFL^\infty(\bK)$ by the second author~\cite{ZemQuasi}. The maps
$\Phi$ and $\Psi$ are two special endomorphisms of knot Floer homology, which
we describe in Section~\ref{sec:basepointactions}. It turns out that $\tau(D_{K,\id}, D_{K,r}) \le 1$
for the canonical deform-spun slice disks $D_{K,\id}$ and the roll-spun slice disks $D_{K,r}$;
see Proposition~\ref{prop:roll} .

To provide more interesting examples beyond roll-spinning, we compute a
formula for the map induced by the summand-swapping automorphism $\ve{R}^\pi$ on
$\cCFL^\infty(\bK \# \bK)$ in Theorem~\ref{thm:rigidmotionformula}. We prove
that there is a chain homotopy equivalence identifying
$\cCFL^\infty(\bK \# \bK)$ and $\cCFL^\infty(\bK) \otimes \cCFL^\infty(\bK)$ under which
\[
\ve{R}^\pi_* \simeq  \Sw \circ \left(\id \otimes \id + (\Phi \circ \Psi) \otimes \id +
\Phi \otimes \Psi \right),
\]
where $\Sw$ is the map that swaps the two tensor factors.
To our knowledge, after Sarkar's formula, this is the only other known formula
for a mapping class group action on the knot Floer complexes of a family of knots in $\Sphere^3$.

\subsection{Families with large distance}

Using the trace formula and our invariants, we prove the following:

\begin{thm}\label{thm:arbitrarily-large-torsion-order-intro}
 Given $n\ge 0$, there is a knot $K_n$ and a pair of slice disks $D_1$ and $D_2$ for $K_n$, such that $\tau(D_1,D_2)\ge n$. In particular $\omega(D_1,D_2)\ge n$ for $\omega \in \{\mu_{\st}, \mu_{\Sing}, \mu_{\Cob}\}$. 
\end{thm}

The slice disks appearing in our proof of Theorem~\ref{thm:arbitrarily-large-torsion-order-intro} are deform spun slice disks of 
\[
T_{p,q}\# T_{p,q}\# \bar T_{p,q} \# \bar T_{p,q} 
\]
for various $p$ and $q$.

We note that the above results appeared after the work of Miller and Powell \cite{MillerPowell}, who constructed for each $n$ a pair of slice disks such that any stabilization sequence between $D_1$ and $D_2$ required at least $n$ stabilizations. Their result does not imply ours, since it focuses on the total number of stabilizations, as opposed to the maximal genus. Hence it does not give a lower bound on the cobordism distance. See also work of Miyazaki \cite{MiyazakiStabilizations}.

Theorem~\ref{thm:arbitrarily-large-torsion-order-intro} gives an answer to a slice disk analogue of Problem 1.105 (B) of Kirby's Problem List~\cite{KirbyList} that asks whether every 2-knot is $0$-null-cobordant:

\begin{cor}
For every $g$, there is a knot $K_g$ and a pair of slice disks that are not strictly $g$-cobordant. In particular, there are slice disks that are not 0-cobordant in the sense of Melvin~\cite{MelvinPHD}.
\end{cor}

\begin{rem}
 We note that Dai, Mallick and Stoffregen have independently found examples of slice disks with large stabilization distance \cite{DMSEquivariant}. Additionally, they use some of the techniques of this paper in their work to study equivariant knots.
\end{rem}

\subsection*{Acknowledgements}
We would like to thank David Gabai and Maggie Miller for helpful conversations.
The first author was supported by a Royal Society Research Fellowship,
and the second author by an NSF Postdoctoral Research Fellowship (DMS-1703685).
This project has received funding from the European Research Council (ERC)
under the European Union's Horizon 2020 research and innovation programme
(grant agreement No 674978). We also thank an anonymous referee for feedback on an earlier draft.

\section{The stabilization distance of a pair of surfaces}
\label{sec:stabilizationgenus}

In this section, we first review deform-spun slice disks.
We then introduce the notion of a metric filtration,
which is a generalization of an ultrametric,
but where the distance of a point from itself can be nonzero.
We proceed to define the stabilization distance of a pair of surfaces with boundary a given knot, which
is an instance of a metric filtration. Finally, we show that the stabilization distance
of a 1-roll-spun and a 1-twist-spun slice disk is always at most two.

\subsection{Deform-spun slice disks}\label{sec:spinning}

We review the definitions of deform-spun slice disks \cite[Section~3]{SliceDisks}.
Originally, Litherland~\cite{Litherland} introduced deform-spinning to construct
2-knots in $\R^4$, generalizing twist-spinning, due to Zeeman~\cite{twisting},
and roll-spinning, due to Fox~\cite{rolling}. An analogous construction can be used to obtain
slice disks in $B^4$. The following is \cite[Definition~3.1]{SliceDisks}:

\begin{define}\label{def:spinning}
  Let $K$ be a knot in $\Sphere^3 = -B^3 \cup B^3$ such that $K$ intersects $-B^3$ in an unknotted arc,
  and write $a = K \cap B^3$.
  Furthermore, let $\phi \colon I \times \Sphere^3 \to \Sphere^3$ be an isotopy of $\Sphere^3$ such that
  $\phi_0 = \Id_{\Sphere^3}$, $\phi_t|_{-B^3} = \Id_{-B^3}$ for every $t \in I$, and $\phi_1(K) = K$.
  The \emph{deform-spun} slice disk $D_{K, \phi} \subset B^4$ is defined by taking
  \[
  \bigcup_{t \in I} \{t\} \times \phi_t(a) \subset I \times B^3,
  \]
  and rounding the corners along $\{0, 1\} \times \d B^3$.
  The surface $D_{K, \phi}$ is a slice disk for $-K \# K$, where $-K$ stands for $(-\Sphere^3, -K)$.
\end{define}

Note that $\d \R^4_+ = \R^3 = \R^3_+ \cup \R^3_-$.
Intuitively, we consider the arc $a$ in $\R^3_-$, which we
rotate about the plane $\R^2 = \R^3_+ \cap \R^3_-$ in $\R^4_+$,
while applying the isotopy $\phi$, until we reach $\R^3_+$.
The following result is \cite[Lemma~3.3]{SliceDisks}, which immediately follows from
the work of Hatcher~\cite{Hatcher}.

\begin{lem}
  Let $\varphi$ be an automorphism of $(\Sphere^3, K)$ such that $\varphi|_{-B^3} = \Id_{-B^3}$.
  Then there is an isotopy $\phi \colon I \times \Sphere^3 \to \Sphere^3$
  as in Definition~\ref{def:spinning}, such that $\phi_1 = \varphi$. Furthermore,
  the isotopy class of the deform-spun disk $D_{K, \phi}$ only depends on $\varphi$,
  which we denote by $D_{K, \varphi}$.
\end{lem}

We now recall the definition of roll-spinning, based on the description of
Litherland~\cite{Litherland}*{Example~2.2}. The following is \cite[Definition~3.5]{SliceDisks}:

\begin{define}\label{def:rolling}
  Let $K$ be a knot in $\Sphere^3$. Choose a tubular neighborhood $N(K)$ of $K$ as well as an identification
  $N(K) \approx K \times B^2$ which induces the Seifert framing of $K$.
  Let $X = \Sphere^3 \setminus \Int(N(K))$ be the knot exterior.
  Furthermore, let $\d X \times I$ be a collar of $\d X$ in $X$, with $\d X\times \{0\}$ be identified with $\d X\subset X$.
  We identify $K$ with $\R/\Z$. Choose a smooth monotonic function $\varphi \colon \R \to I$
  such that $\varphi(s) = 0$ for $s \le 0$ and $\varphi(s) = 1$ for $s \ge 1$.
  We define the \emph{rolling diffeomorphism} $r \colon (\Sphere^3, K) \to (\Sphere^3, K)$ by the formula
  \[
  r(\bar{x}, \bar{\theta}, s) = \left(\overline{x + \varphi(s)}, \bar{\theta}, s \right)
  \text{ for } (\bar{x}, \bar{\theta}, s) \in K \times \d B^2 \times I \approx \d X \times I,
  \]
  and let $r(p) = p$ for $p \in \Sphere^3 \setminus (\d X \times I)$.

  Similarly, we define the \emph{twisting diffeomorphism} $t\colon (\Sphere^3,K)\to (\Sphere^3,K)$ by the formula
 \[
  t(\bar{x}, \bar{\theta}, s) = \left(\overline{x}, \overline{\theta + \varphi(s)}, s \right)
  \text{ for } (\bar{x}, \bar{\theta}, s) \in K \times \d B^2 \times I \approx \d X \times I,
  \]
  and let $t(p) = p$ for $p \in \Sphere^3 \setminus (\d X \times I)$.

  Let $B \subset N(K)$ be an open ball that intersects $K$ in an arc.
  Then $(\Sphere^3 \setminus B, K \setminus B)$ is diffeomorphic to a ball-arc pair
  $(B^3, a)$, and $r|_B = \id_B$. We define the \emph{$(k,l)$-twist-roll-spin} of $K$ to be $D_{K, t^kr^l}$.
\end{define}

Note that $D_{K, \id}$ is simply the spun slice disk of $K$, obtained using the identity deformation.
We will call this the \emph{canonical} slice disk of $-K \# K$.

\subsection{Metric filtrations}\label{sec:metric-filtration}

\begin{define}\label{def:metric-filtration}
  Let $X$ be a set. We say that a function $\mu \colon X \times X \to \R^{\ge 0}$
  is a \emph{metric filtration} on $X$ if it is symmetric, and satisfies the \emph{ultrametric inequality}
    \[
    \mu(x,x'') \le \max\{\mu(x,x'), \mu(x',x'')\}
    \]
    for every $x$, $x'$, $x'' \in X$.
\end{define}

The metric filtrations appearing in the paper will all be instances of the following construction.

\begin{example}\label{ex:metric-filtration}
  Let $X$ be a path-connected topological space, and $f \colon X \to \R^{\ge 0}$ a continuous function.
  Given points $x$, $x' \in X$, we define
  \[
  \mu(x,x') := \inf_{\substack{\gamma \colon I \to X \\ \gamma(0) = x \text{, } \gamma(1) = x'}} \max_{t \in I} (f\circ\g)(t),
  \]
  where the infimum is taken over continuous paths $\g$. This is clearly a metric filtration.
  Note that $\mu(x,x) = f(x)$.

  Similarly, let $G = (V, E)$ be a connected graph, and $f \colon V \to \R^{\ge 0}$ a function.
  For $x$, $x' \in V$, we define $\mu(x,x')$ to be the infimum of
  \[
  \max\{f(x_1), \dots, f(x_n)\}
  \]
  over paths $x_1, \dots, x_n$ in $V$ such that $x_1 = x$ and $x_n = x'$.
  This is a special case of the above construction: Set $X$ to be the 1-complex
  associated to $G$, and extend $f$ over the 1-cells of $X$ linearly. Typically we will be interested in graphs where $f$ is integrally valued on $V$. For these graphs, we can replace the infimum with a minimum.
\end{example}

\begin{define}\label{def:normalization}
  Let $\mu$ be a metric filtration on the set $X$. Then we define its \emph{normalization} as
  \[
  \tilde{\mu}(x,x') := \mu(x,x') - \min\{\mu(x,x), \mu(x',x')\}.
  \]
\end{define}

Recall that $g \colon X \times X \to \R^{\ge 0}$ is a pseudometric on $X$ if it is symmetric, satisfies the triangle
inequality, and $g(x,x) = 0$ for every $x \in X$ (but $g(x,y) = 0$ does not necessarily
imply that $x = y$).

\begin{lem}\label{lem:normalization}
  Let $\mu$ be a metric filtration on the set $X$. Then its normalization $\tilde{\mu}$
  is a pseudometric.
\end{lem}

\begin{proof}
  Choose points $x$, $x'$, $x'' \in X$, and
  write $a = \mu(x,x)$, $a' = \mu(x',x')$, and $a'' = \mu(x'',x'')$.
  It is clear that $\tilde{\mu}(x,x) = 0$.
  If we apply the ultrametric inequality to the triple $x$, $x'$, $x$
  we obtain that $\mu(x,x') \ge a$. Similarly, by considering the
  triple $x'$, $x$, $x'$, we get that $\mu(x,x') \ge a'$, and hence
  $\mu(x,x') \ge \max\{a,a'\}$. In particular, $\tilde{\mu}(x,x') \ge 0$.

  It remains to prove the triangle inequality
  \[
  \tilde{\mu}(x, x'') \le \tilde{\mu}(x, x') + \tilde{\mu}(x', x'').
  \]
  By definition,
  \[
  \tilde{\mu}(x, x') = \mu(x, x') - \min\{a, a'\} \quad {and} \quad
  \tilde{\mu}(x', x'') = \mu(x', x'') - \min\{a', a''\}.
  \]
  Without loss of generality, we can assume that
  \[
  \max\{\mu(x,x'), \mu(x',x'')\} = \mu(x',x'').
  \]
  Hence, by the ultrametric inequality, $\mu(x,x'') \le \mu(x',x'')$.
  So it suffices to show that
  \[
  -\min\{a, a''\} \le \mu(x,x') - \min\{a, a'\} - \min\{a', a''\}.
  \]
  We saw that $\max\{a, a'\} \le \mu(x,x')$, hence we only need to prove that
  \[
  \min\{a, a'\} + \min\{a', a''\} \le \max\{a, a'\} + \min\{a, a''\}.
  \]
  This holds because $\max\{a, a'\}$ bounds both terms on the left-hand side from above,
  while $\min\{a, a''\}$ bounds at least one of them from above.
\end{proof}

\subsection{The stabilization distance}
In this section, we describe a collection of topological numerical invariants associated
to pairs of surfaces bounding a knot.
Given a properly embedded surface $S$ in a 4-manifold $W$, we describe
several ways of increasing the genus of $S$ within $W$.  Our description
is inspired by Baykur and Sunukjian~\cite{BaySunStabilizations}.
The most general notion, and the one we will focus on in this paper, is the following:

\begin{define}\label{sec:stabilizationgenusdef}
Let $S$ be a connected and properly embedded surface in a 4-manifold $W$.
Suppose that $B^4 \subset \Int(W)$ is an embedded 4-ball such that $\d B^4 \cap S$
is an unlink of $m$ components, written as $U_1\cup \cdots \cup U_m$. Furthermore, suppose
that $S \cap B^4$ is a collection of $m$ pairwise disjoint and properly embedded disks
$D_1\cup \cdots \cup D_m$ that can simultaneously be smoothly isotoped into $\d B^4$ relative to their boundaries.
Let $S_0$ be an oriented, connected, properly embedded genus~$n$ surface in $B^4$ such
that $\d S_0 = U_1 \cup \dots \cup U_m$. The surface
\[
S' := (S \setminus \Int(B^4)) \cup S_0
\]
is the \emph{$(m,n)$-stabilization} of $S$ along $(B^4, S_0)$.
We say the \emph{genus} of an $(m,n)$-stabilization is $m + n - 1$,
which we note is $g(S') - g(S)$.

If $S'$ is the $(m,n)$-stabilization of $S$, then we say that $S$ is the \emph{$(m,n)$-destabilization} of $S'$. In this paper, a \emph{stabilization} refers to an $(m,n)$-stabilization and a \emph{destabilization} to an $(m,n)$-destabilization for some $m$ and $n$.
\end{define}

A schematic of an $(m,n)$-stabilization is shown in Figure~\ref{fig::1}.

\begin{figure}[ht!]
	\centering
\begingroup%
  \makeatletter%
  \providecommand\color[2][]{%
    \errmessage{(Inkscape) Color is used for the text in Inkscape, but the package 'color.sty' is not loaded}%
    \renewcommand\color[2][]{}%
  }%
  \providecommand\transparent[1]{%
    \errmessage{(Inkscape) Transparency is used (non-zero) for the text in Inkscape, but the package 'transparent.sty' is not loaded}%
    \renewcommand\transparent[1]{}%
  }%
  \providecommand\rotatebox[2]{#2}%
  \newcommand*\fsize{\dimexpr\f@size pt\relax}%
  \newcommand*\lineheight[1]{\fontsize{\fsize}{#1\fsize}\selectfont}%
  \ifx\svgwidth\undefined%
    \setlength{\unitlength}{193.34541526bp}%
    \ifx\svgscale\undefined%
      \relax%
    \else%
      \setlength{\unitlength}{\unitlength * \real{\svgscale}}%
    \fi%
  \else%
    \setlength{\unitlength}{\svgwidth}%
  \fi%
  \global\let\svgwidth\undefined%
  \global\let\svgscale\undefined%
  \makeatother%
  \begin{picture}(1,0.87966255)%
    \lineheight{1}%
    \setlength\tabcolsep{0pt}%
    \put(0,0){\includegraphics[width=\unitlength,page=1]{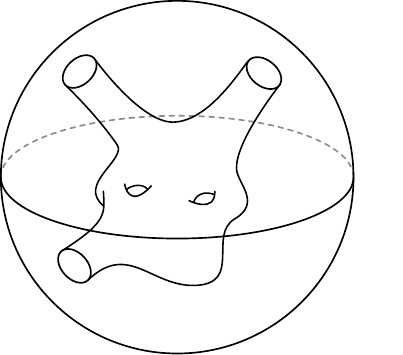}}%
    \put(0.7785861,0.21333928){\color[rgb]{0,0,0}\makebox(0,0)[rt]{\lineheight{1.25}\smash{\begin{tabular}[t]{r}$B^4$\end{tabular}}}}%
    \put(0.7531135,0.07885458){\color[rgb]{0,0,0}\makebox(0,0)[lt]{\lineheight{1.25}\smash{\begin{tabular}[t]{l}$\d B^4$\end{tabular}}}}%
    \put(0.6275688,0.44192739){\color[rgb]{0,0,0}\makebox(0,0)[lt]{\lineheight{1.25}\smash{\begin{tabular}[t]{l}$S_0$\end{tabular}}}}%
    \put(0.2528231,0.74948203){\color[rgb]{0,0,0}\makebox(0,0)[lt]{\lineheight{1.25}\smash{\begin{tabular}[t]{l}$U_1$\end{tabular}}}}%
    \put(0.72177467,0.63588075){\color[rgb]{0,0,0}\makebox(0,0)[lt]{\lineheight{1.25}\smash{\begin{tabular}[t]{l}$U_2$\end{tabular}}}}%
    \put(0.23204236,0.1329872){\color[rgb]{0,0,0}\makebox(0,0)[lt]{\lineheight{1.25}\smash{\begin{tabular}[t]{l}$U_3$\end{tabular}}}}%
  \end{picture}%
\endgroup%

	\caption{A $(3,2)$-stabilization.}\label{fig::1}
\end{figure}
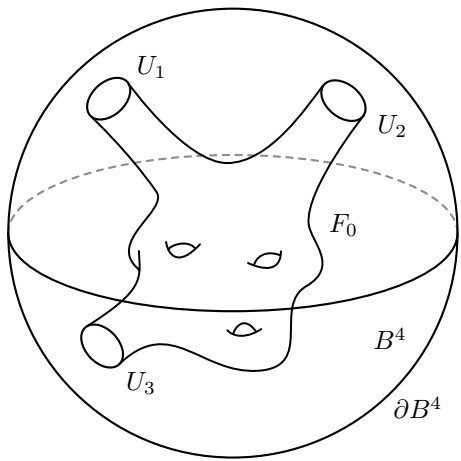

Note that performing a $(1,0)$-stabilization is the same as taking the connected sum of a surface $S$ with a 2-knot contained in a small 4-ball, disjoint from $S$. We additionally define the following special cases of stabilization:

\begin{define}\label{def:stab}
Let $S$ be a properly embedded surface in the 4-manifold $W$,
and suppose that $S'$ is obtained from $S$ by an $(m,n)$-stabilization along $(B^4, S_0)$.
\begin{enumerate}
  \item We say $S'$ is an \emph{unknotted surface stabilization} of $S$
    if $m = 1$ and $S_0$ is smoothly isotopic into $\d B^4$ relative to $\d S_0$. If, furthermore, $n = 1$, then we call this a \emph{trivial stabilization}.
  \item We say $S'$ is a \emph{1-handle stabilization} of $S$ if $(m,n) = (2,0)$
    and $S_0 \cup D_1 \cup D_2$ is the boundary of a 3-dimensional 1-handle that is embedded in $W$.
\end{enumerate}
\end{define}

Unknotted surface stabilization and 1-handle stabilization
have been introduced by Boyle~\cite{Boyle},
and studied further by Baykur and Sunukjian~\cite{BaySunStabilizations}.

\begin{lem}
  The surface $S'$ can be obtained from $S$ by a genus $g$ unknotted surface stabilization
  if and only if it can be obtained by $g$ disjoint trivial stabilizations.
\end{lem}

\begin{proof}
  Suppose that $S'$ is obtained from $S$ by an unknotted surface stabilization along $(B^4, S_0)$.
  After isotoping $S_0$ into $\Sphere^3$, it becomes a Seifert surface of the unknot $\d S_0$.
  If $g > 0$, the map
  \[
  \pi_1(S_0 \setminus N(\d S_0)) \to \pi_1(\Sphere^3 \setminus N(\d S_0)) \cong \Z
  \]
  is not injective, hence $S_0$ is compressible in $\Sphere^3$ by the loop theorem.
  Compressing corresponds to reversing a 1-handle stabilization in $\Sphere^3$.
  If we push the interior of the 1-handle into $B^4$, it becomes unknotted.
  By induction on the genus of $S_0$, we see that $S'$ can be obtained from $S$ by $g$ consecutive
  trivial stabilizations. However, in dimension $4$, we can always isotope consecutive trivial stabilizations
  to be disjoint from each other. The opposite implication is straightforward.
\end{proof}

As in \cite{SliceDisks}*{Definition~3.8}, we can define the \emph{peripheral map}
\[
h_S \colon \pi_1(\d W \setminus \d S) \to \pi_1(W \setminus S)
\]
for a properly embedded surface $S$ in $W$. Given surfaces $S$ and $S'$ in $W$
with $\d S = \d S'$,
we say that their peripheral maps are \emph{equivalent} if there is an
isomorphism
\[
g \colon \pi_1(W \setminus S) \to \pi_1(W \setminus S')
\]
such that $h_{S'} = g \circ h_{S}$. The equivalence class of the peripheral map is clearly
an invariant of $S$ up to ambient isotopy in $W$ fixing $\d W$ pointwise.

\begin{lem}\label{lem:trivstab}
  Suppose that $S'$ is obtained from $S$ by a trivial stabilization.
  Then their peripheral maps are equivalent.
\end{lem}

\begin{proof}
  Boyle~\cite[Lemma~11]{Boyle} showed that $\pi_1(W \setminus S) \cong \pi_1(W \setminus S')$.
  Indeed, a trivial stabilization corresponds to taking the connected sum of $(W, S)$ and $(\Sphere^4, T^2)$,
  where $T^2$ is an unknotted torus. Since $\Sphere^4 \setminus T^2$ is homotopy equivalent to
  the suspension of $\Sphere^3 \setminus T^2$,
  we have $\pi_1(\Sphere^4 \setminus T^2) \cong \Z$, generated by the meridian of $T^2$.
  Hence, the claim follows from the Seifert--van Kampen theorem.
  As the connected sum is taken in the interior of $W$, it follows that the peripheral maps are equivalent.
\end{proof}

As shown by Boyle~\cite[Lemma~9]{Boyle}, a nontrivial 1-handle stabilization might change
the fundamental group of the surface complement, though it is always a quotient of
the original group. Based on his work,
Baykur and Sunukjian~\cite[Lemma~3]{BaySunStabilizations} determined when two
1-handle stabilizations give isotopic surfaces. If the 1-handle $h$ is attached along the points $a$, $b \in S$,
then we can act on the homotopy class of the core of $h$ by either adding the class of the
meridian of $S$, or pre- or post-composing with the push-off of a loop in $\pi_1(S, a)$
or $\pi_1(S, b)$. The equivalence class of the homotopy class of the core of $h$
determines the resulting surface up to isotopy, also in the case when $\S$ has boundary.

\begin{cor}\label{cor:trivial}
  Let $D_\varphi$ and $D_{\varphi'}$ be deform-spun slice disks of a knot $-K \# K$,
  where $\varphi$ and $\varphi'$ are non-isotopic automorphisms of $(\Sphere^3, K)$ that are
  fixed in a neighborhood of a point of $K$. Then one cannot obtain $D_{\varphi'}$
  from $D_\varphi$ by a sequence of trivial stabilizations and destabilizations (or, equivalently, by
  unknotted surface stabilizations and destabilizations).
\end{cor}

\begin{proof}
  According to the proof of \cite[Proposition~3.9]{SliceDisks}, the peripheral maps
  of $D_\varphi$ and $D_{\varphi'}$ are inequivalent. Hence, the result follows from Lemma~\ref{lem:trivstab}.
\end{proof}

\begin{prop}\label{prop:1-handle}
  Let $S$ and $S'$ be compact, properly embedded surfaces in the compact 4-manifold $W$
  such that $\d S = \d S'$ and $[S] = [S'] \in H_2(W, \d S)$. Then $\S$ and $\S'$ become ambient isotopic
  relative to $\d W$ after finitely many 1-handle stabilizations.
\end{prop}

\begin{proof}
  This is a relative version of \cite[Theorem~5]{BaySunStabilizations},
  and the proof is analogous. The idea is that one removes a neighborhood
  of $S \cap S'$, and chooses a relative Seifert manifold $M$ for
  \[
  (S \cup {-}S') \setminus N(S \cap S').
  \]
  Then a self-indexing Morse function on $M$ with only index 1 and 2 critical points
  that is minimal along $S$ and maximal along $S'$ gives the required handles.
\end{proof}

\begin{define}\label{def:stabgenus}
Suppose that $S$ and $S'$ are connected, properly embedded surfaces in the
compact 4-manifold $W$ such that $\d S = \d S'$ is a knot $K$ in $\d W$.
We define the \emph{stabilization distance} of the pair $(S, S')$, for which we write
$\mu_{\st}(S, S')$, to be the minimum of
\[
\max\{ g(S_1), \dots, g(S_k) \}
\]
over sequences of connected, properly embedded surfaces $S_1, \dots, S_k$ in $W$
such that
\begin{enumerate}
  \item $S_1 = S$ and $S_k = S'$,
  \item $\d S_i = K$ for $i \in \{1, \dots, k\}$, and
  \item\label{it:stab} $S_i$ and $S_{i+1}$ are related by a stabilization or destabilization, up to proper isotopy,
  for $i \in \{1, \dots, k-1\}$.
\end{enumerate}
We define $\mu_{\st}(S,S')$ to be $\infty$ if no such sequence exists. Analogously, we define $\overline{\mu}_{\st}(S, S')$ by minimizing
\[
\max\{ g(S_1), \dots, g(S_k) \} - \min\{ g(S_1), \dots, g(S_k) \}
\]
over the same set of sequences. Finally, we let
\[
\tilde{\mu}_{\st}(S, S') = \mu_{\st}(S, S') - \min\{g(S), g(S')\},
\]
which we call the \emph{normalized stabilization distance}.

The \emph{(trivial) 1-handle distance} of $S$ and $S'$ is defined similarly to $\mu_{\st}$,
except that $S_i$ and $S_{i+1}$ are related by adding or removing a (trivial) 1-handle.
We denote the 1-handle distance by $\mu_1(S, S')$, and the trivial 1-handle distance by $\mu_1^0(S, S')$.
\end{define}

 We observe that
\[
\tilde{\mu}_{\st}(S, S') \le \overline{\mu}_{\st}(S, S') \le \mu_{\st}(S, S') \le \mu_1(S, S') \le \mu_1^0(S, S').
\]
Furthermore, if $[S] = [S']$ in $H_2(W, K)$,
then $\mu_1(S, S')$ is finite by Proposition~\ref{prop:1-handle},
and hence so are $\mu_{\st}(S, S')$ and $\overline{\mu}_{\st}(S, S')$.
On the other hand, $\mu_1^0(S, S')$ might be infinite by Corollary~\ref{cor:trivial}.
Note that $\overline{\mu}_{\st}(S, S') = 0$ if and only if $S$ and $S'$ become
isotopic after taking connected sums with 2-knots. Since $\mu_{\st}(S, S) = g(S)$,
the normalized distance satisfies $\tilde{\mu}_{\st}(S, S) = 0$.

Consider the graph whose vertices are isotopy classes (rel.~boundary) of connected surfaces in $W$ with boundary the knot $K$
in a fixed relative homology class in $H_2(W, K)$, and whose edges correspond to $(m,n)$-stabilization
for some $m$ and $n$. If we apply the procedure outlined in Example~\ref{ex:metric-filtration}
to the genus function, we obtain $\mu_{\st}$, which is hence a metric filtration in the sense
of Definition~\ref{def:metric-filtration}.
Its normalization in the sense of Definition~\ref{def:normalization} is $\tilde{\mu}_{\st}$.
So, as a special case of Lemma~\ref{lem:normalization}, we obtain the following:

\begin{lem}
  Let $W$ be a compact 4-manifold, and $K$ a knot in $\d W$.
  The normalized stabilization distance $\tilde{\mu}_{\st}$ is a pseudometric
  when restricted to surfaces in a given class in $H_2(W, K)$.
\end{lem}

For a knot $K$ in $\Sphere^3$, let us write $\Surf(K)$ for the set of isotopy classes of
connected, oriented, properly embedded surfaces in $B^4$ with boundary $K$,
and $\Surf(K)/\{\text{2-knots}\}$ for $\Surf(K)$ modulo genus~0 stabilizations.
We denote by $\Surf_g(K)$ the subset of genus $g$ surfaces in $\Surf(K)$.
If $K$ is slice, we will write
\[
\cD(K) := \Surf_0(K)
\]
for the set of isotopy classes of slice disks of $K$ in $B^4$,
and $\cD(K)/\{\text{2-knots}\}$ for $\cD(K)$ modulo genus~0 stabilizations.

\begin{rem}
Note that $(\Surf(K)/\{\text{2-knots}\}, \tilde{\mu}_{\st})$ is a \emph{pseudometric space}
(i.e., $\tilde{\mu}_{\st}(S, S') = 0$ can hold for $S \neq S'$ in $\Surf(K)/\{\text{2-knots}\}$),
while $(\Surf(K)/\{\text{2-knots}\}, \overline{\mu}_{\st})$ is a \emph{metric space}.
Furthermore, $\mu_{\st}$ is a metric filtration, so it satisfies the ultrametric inequality
\[
\mu_{\st}(S, S'') \le \max\{ \mu_{\st}(S, S'), \mu_{\st}(S', S'') \}
\]
for any $S$, $S'$, $S'' \in \Surf(K)$, and so does $\tilde{\mu}_{\st}$ when restricted
to $\Surf_g(K)$.

If one of $S$, $S' \in \Surf(K)$ is a disk,
and $S_1, \dots, S_k$ is a sequence of surfaces connecting $S$ and $S'$,
as in Definition~\ref{def:stabgenus}, then
\[
\min\{ g(S_1), \dots, g(S_k) \} = 0,
\]
so $\mu_{\st}(S, S')$ and $\overline{\mu}_{\st}(S, S')$
are both obtained by minimizing $\max\{g(S_1), \dots, g(S_k)\}$.
Hence $\mu_{\st}$ and $\overline{\mu}_{\st}$ agree on $\cD(K)/\{\text{2-knots}\}$, and
$(\cD(K)/\{\text{2-knots}\}, \mu_{\st})$ is an \emph{ultrametric space};
i.e., a metric space that satisfies the ultrametric inequality.
Our invariants from Heegaard Floer homology naturally give bounds on $\mu_{\st}$,
hence we will not study $\overline{\mu}_{\st}$ in the rest of this paper.
\end{rem}

\begin{define}\label{def:dest}
If $K$ is a slice knot and $S \in \Surf(K)$,
we define the \emph{destabilizing genus} $g_{\dest}(S)$ of $S$
to be the minimum of
\[
\max\{ g(S_1), \dots, g(S_n) \}
\]
over sequences of properly embedded surfaces $S_1, \dots, S_n$ in $B^4$ such that
\begin{enumerate}
\item $S_{i+1}$ is obtained from $S_i$ via stabilization or destabilization for $i \in \{1, \dots, n-1\}$,
\item $S_1 = S$, and $S_n$ is a slice disk of $K$.
\end{enumerate}
\end{define}

By definition, $g_{\dest}(S) \ge g(S)$. Furthermore, if $D$ is a slice disk of $K$,
then $g_{\dest}(S) \le \tilde{\mu}_{\st}(S, D)$, and hence $g_{\dest}(S)$ is finite.
In fact, $g_{\dest}(S)$ is the distance of $S$ from $\cD(K)/\{\text{2-knots}\}$ in the pseudometric space
$(\Surf(K)/\{\text{2-knots}\}, \tilde{\mu}_{\st})$.

\begin{prop}\label{prop:mu-upper}
  Let $S_1$, $S_2 \in \Surf(K)$. Then
  \[
  \mu_{\st}(S_1, S_2) \le 2g(K) + \max\{g(S_1), g(S_2)\},
  \]
  where $g(K)$ is the Seifert genus of $K$.
\end{prop}

\begin{proof}
  Let $S$ be a minimal genus Seifert surface for $K$, and choose an open ball $B \subset \Int(S)$.
  Consider the product $S \times I \subset \Sphere^3$, where we identify
  $S$ with $S \times \{0\}$. For $i \in \{1, 2\}$, isotope $S_i$ near $\d S_i$ such that
  it becomes a surface $S_i'$ with boundary $K \times \{1\}$. We let
  \[
  \S_i := ((S \setminus B) \times \{0\}) \cup (\d B \times I) \cup ((S \setminus B) \times \{1\}) \cup S_i'.
  \]
  Then $\S_i$ is a surface of genus $2g(K) + g(S_i)$ that can be obtained from $S_i$
  by $2g(K)$ 1-handle stabilizations in $\Sphere^3$. Indeed, let $a_1, \dots, a_{2g}$ be pairwise disjoint arcs in $S \setminus B$ with boundary on $\d B$ that span $H_1(S, B)$. If we compress $\Sigma_i$ along the curves
  \[
  (a_i \times \{0, 1\}) \cup (\d a_i \times I)
  \] 
  using the compressing disks $a_i \times I \subset S \times I$, we obtain $S_i$, up to proper isotopy.
  
  If we push $\S_i \setminus (S \times \{0\})$ into $\Int(B^4)$ relative to $\d B \times \{0\}$,
  we obtain a surface $\S_i'$ that is a stabilization of $S$ with $m = 1$.
  The sequence of surfaces $S_1$, $\S_1'$, $S$, $\S_2'$, $S_2$ satisfies the requirements of
  Definition~\ref{def:stabgenus}, and has maximal genus $2g(K) + \max\{ g(S_1), g(S_2) \}$.
\end{proof}

\begin{cor}
  If $K$ is a slice knot and $S \in \Surf_g(K)$, then
  \[
  g_{\dest}(S) \le 2g(K) + g.
  \]
\end{cor}

\begin{prop}\label{prop:st-dest}
  Let $K$ be a slice knot, and $S$, $S' \in \Surf(K)$. Then
  \[
  \tilde{\mu}_{\st}(S, S') \ge  |g_{\dest}(S) - g_{\dest}(S')|.
  \]
\end{prop}

\begin{proof}
  Since the claim is symmetric in $S$ and $S'$, we can assume that $g_{\dest}(S) \le g_{\dest}(S')$.
  Let $D$ be a slice disk for $K$ such that $\tilde{\mu}_{\st}(D, S) = g_{\dest}(S)$.
  Then
  \[
  g_{\dest}(S) + \tilde{\mu}_{\st}(S, S') =
  \tilde{\mu}_{\st}(D, S) + \tilde{\mu}_{\st}(S, S') \ge \tilde{\mu}_{\st}(D, S') \ge g_{\dest}(S')
  \]
  by the triangle inequality, and the result follows.
\end{proof}

\subsection{An upper bound on the distance between 1-roll-spun and 1-twist-spun slice disks}

Let $t^n r^m$ denote the $n$-twist-$m$-roll-spinning diffeomorphism of $K$.
We will show the following:

\begin{prop}\label{prop:rollspinningcomputation}
If $K$ is a knot in $\Sphere^3$, then the deform-spun slice disks $D_{K, r}$,
$D_{K, t}\in \Surf_0(-K \# K)$ satisfy
\[
\mu_{st}(D_{K, r}, D_{K, t}) \le 2.
\]
\end{prop}

\begin{proof}
Let $B_0 \subset \Sphere^3$ denote a 3-ball that intersects $K$ in an unknotted arc.
We consider the knotted ball-arc pair $(B,a)$, where $B := \Sphere^3 \setminus \Int(B_0)$ and $a = K \cap B$.
We present both slice disks $D_{K, r}$ and $D_{K, t}$ as movies
of ball-arc pairs which start and end at $(B, a)$.

\begin{figure}[ht!]
	\centering
\begingroup%
  \makeatletter%
  \providecommand\color[2][]{%
    \errmessage{(Inkscape) Color is used for the text in Inkscape, but the package 'color.sty' is not loaded}%
    \renewcommand\color[2][]{}%
  }%
  \providecommand\transparent[1]{%
    \errmessage{(Inkscape) Transparency is used (non-zero) for the text in Inkscape, but the package 'transparent.sty' is not loaded}%
    \renewcommand\transparent[1]{}%
  }%
  \providecommand\rotatebox[2]{#2}%
  \newcommand*\fsize{\dimexpr\f@size pt\relax}%
  \newcommand*\lineheight[1]{\fontsize{\fsize}{#1\fsize}\selectfont}%
  \ifx\svgwidth\undefined%
    \setlength{\unitlength}{408.98576228bp}%
    \ifx\svgscale\undefined%
      \relax%
    \else%
      \setlength{\unitlength}{\unitlength * \real{\svgscale}}%
    \fi%
  \else%
    \setlength{\unitlength}{\svgwidth}%
  \fi%
  \global\let\svgwidth\undefined%
  \global\let\svgscale\undefined%
  \makeatother%
  \begin{picture}(1,0.51063013)%
    \lineheight{1}%
    \setlength\tabcolsep{0pt}%
    \put(0,0){\includegraphics[width=\unitlength,page=1]{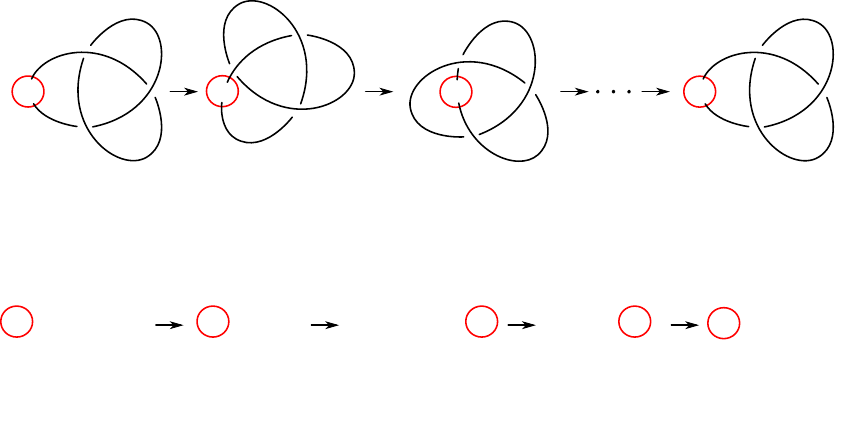}}%
    \put(0.46136019,0.32637058){\color[rgb]{0,0,0}\makebox(0,0)[t]{\lineheight{1.25}\smash{\begin{tabular}[t]{c}$D_{K,t^{\wr(\cD)}r}$\end{tabular}}}}%
    \put(0.46136057,0.00187201){\color[rgb]{0,0,0}\makebox(0,0)[t]{\lineheight{1.25}\smash{\begin{tabular}[t]{c}$D_{K,t}$\end{tabular}}}}%
    \put(0,0){\includegraphics[width=\unitlength,page=2]{fig36.pdf}}%
    \put(0.02891755,0.20251902){\color[rgb]{0.54901961,0.54901961,0.54901961}\makebox(0,0)[lt]{\lineheight{1.25}\smash{\begin{tabular}[t]{l}$\ell$\end{tabular}}}}%
    \put(0,0){\includegraphics[width=\unitlength,page=3]{fig36.pdf}}%
  \end{picture}%
\endgroup%

	\caption{The slice disks $D_{K, t^{\wr(\cD)}r}$ and $D_{K, t}$ of $-K \# K$. In the top row, we rotate the diagram counterclockwise a full turn in the plane, and consecutive frames differ by a small rotation.}
\label{fig::36}
\end{figure}

We begin with describing a movie for $D_{K, t^n r}$ for any $n \in \N$.
We pick a diagram $\cD$ for $K$ with $\wr(\cD) = n$. We view the diagram $\cD$ as nearly being
embedded in a plane $P$. The movie for $D_{K, t^n r}$ consists of
rotating the diagram $\cD$ about an axis perpendicular to $P$ (and shifting
along the axis perpendicular to $P$ slightly), while translating $\cD$
in the plane so that the image of $K$ intersects $B$ in an unknotted arc.
This is a movie for $D_{K,t^nr}$ for some $n$. The exponent $n$
is equal to the difference between the blackboard framing and the Seifert
framing. Since the difference between the blackboard framing and the Seifert
framing is $\wr(\cD)$, it follows that this movie represents $D_{K, t^{\wr(\cD)} r}$.

Next, we describe a movie for the slice disk $D_{K, t}$. We pick a line
$\ell$ in $\R^3$ which coincides with $K$ inside $B_0$, and is disjoint
from $K$ outside a small neighborhood of $B_0$. The movie for
$D_{K, t}$ is obtained by rotating $a$ in a full twist about $\ell$.
Schematics of the movies for $D_{K, t^{\wr(\cD)} r}$ and $D_{K, t}$
are shown in  Figure~\ref{fig::36}.

We now present a stabilization sequence from $D_{K, t^{\wr(\cD)}r}$ to
$D_{K,t}$ that has maximal genus two. Let us write $\{a_s : s \in I\}$ for the movie
of arcs corresponding to $D_{K, t^{\wr(\cD)}r}$.
Suppose that $a_s = \phi_s(K) \cap B$ for a 1-parameter family of rigid motions
$\phi_s \colon \Sphere^3 \to \Sphere^3$ for $s \in I$ that nearly preserve the plane $P$.
We give $K$ a parametrization $\gamma(s)$ such that $\phi_s(\gamma(s))$ is the
center point of the ball $B_0 = \Sphere^3 \setminus \Int(B)$ (note that $B_0$ is the
region inside the red ball in Figure~\ref{fig::36}).

\begin{figure}[ht!]
	\centering
	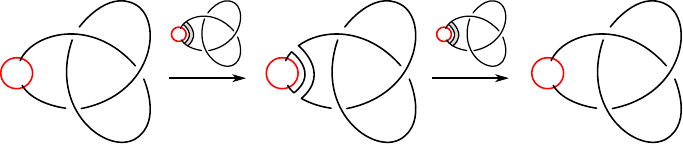
	\caption{Attaching a 1-handle to the disk $D_{K, t}$ by adding a pair
	of bands to the beginning of the movie for $D_{K,t}$.}
\label{fig::38}
\end{figure}

Let $0 < s_1 < \dots < s_n < 1$ be the times such that $\gamma(s_i)$
is the lower point of a crossing of $\cD$, which we denote $c_i$.
Let $\Cr(\cD)$ be the set $\{c_1, \dots, c_n\}$ of crossings of $\cD$.
As $s$ passes $s_i$ for some $i \in \{1, \dots, n\}$,
the overstrand of the crossing $c_i$ passes over $\d B$ in the movie $a_s$.

As a first step, we attach a 1-handle to $D_{K, t^{\wr(\cD)}r}$, by adding
two bands to the beginning of the movie $a_s$,  as in Figure~\ref{fig::38}.
We can move the second band to the end of the movie. This breaks each arc $a_s$ into a knot $K_s$
disjoint from $\d B$, which we can view as a copy of our original knot $K$,
as well as a small, boundary-parallel arc attached to $\d B$.
We now wish to continuously pull the family of knots $K_s$ downward,
such that they do not pass over $\d B$ for any $s$ (or, phrased another
way, such that there is a path from $\d B$ to $\infty \in \Sphere^3$
disjoint from $K_s$ for all $s \in I$).

Given $\ve{c} \subset \Cr(\cD)$, we let $S_{\ve{c}} \in \Surf_1(-K \# K)$
denote the genus one surface obtained by modifying the movie $\{a_s : s \in I\}$
such that the upper strand of $c_i$ passes over $\d B$ if $c_i \in \ve{c}$, and the upper
strand of $c_i$ passes under $\d B$ if $c_i \not\in \ve{c}$; see Figure~\ref{fig::39}.

 \begin{figure}[ht!]
	\centering
\begingroup%
  \makeatletter%
  \providecommand\color[2][]{%
    \errmessage{(Inkscape) Color is used for the text in Inkscape, but the package 'color.sty' is not loaded}%
    \renewcommand\color[2][]{}%
  }%
  \providecommand\transparent[1]{%
    \errmessage{(Inkscape) Transparency is used (non-zero) for the text in Inkscape, but the package 'transparent.sty' is not loaded}%
    \renewcommand\transparent[1]{}%
  }%
  \providecommand\rotatebox[2]{#2}%
  \newcommand*\fsize{\dimexpr\f@size pt\relax}%
  \newcommand*\lineheight[1]{\fontsize{\fsize}{#1\fsize}\selectfont}%
  \ifx\svgwidth\undefined%
    \setlength{\unitlength}{307.73171522bp}%
    \ifx\svgscale\undefined%
      \relax%
    \else%
      \setlength{\unitlength}{\unitlength * \real{\svgscale}}%
    \fi%
  \else%
    \setlength{\unitlength}{\svgwidth}%
  \fi%
  \global\let\svgwidth\undefined%
  \global\let\svgscale\undefined%
  \makeatother%
  \begin{picture}(1,0.69605428)%
    \lineheight{1}%
    \setlength\tabcolsep{0pt}%
    \put(0,0){\includegraphics[width=\unitlength,page=1]{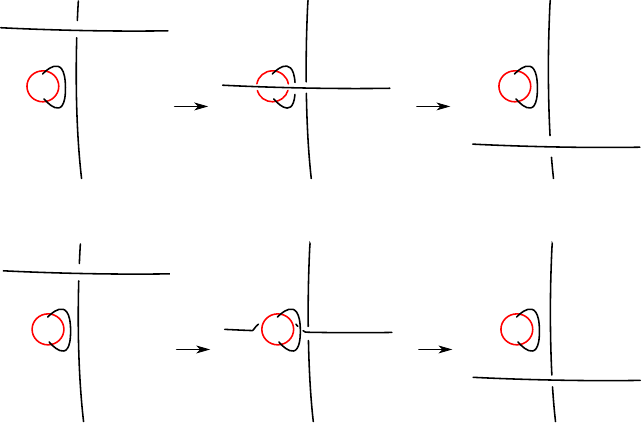}}%
    \put(0.12920163,0.65866246){\color[rgb]{0,0,0}\makebox(0,0)[lt]{\lineheight{1.25}\smash{\begin{tabular}[t]{l}$c_i$\end{tabular}}}}%
    \put(0.13622014,0.28504903){\color[rgb]{0,0,0}\makebox(0,0)[lt]{\lineheight{1.25}\smash{\begin{tabular}[t]{l}$c_i$\end{tabular}}}}%
    \put(0.47897812,0.37984857){\color[rgb]{0,0,0}\makebox(0,0)[t]{\lineheight{1.25}\smash{\begin{tabular}[t]{c}$c_i\in \ve{c}$\end{tabular}}}}%
    \put(0.47897717,0.00306553){\color[rgb]{0,0,0}\makebox(0,0)[t]{\lineheight{1.25}\smash{\begin{tabular}[t]{c}$c_i\not \in \ve{c}$\end{tabular}}}}%
  \end{picture}%
\endgroup%

	\caption{On the top row, we show a portion of the surface $S_{\ve{c}}$
	 when $c_i \in \ve{c}$. The upper strand of a crossing $c_i$
	 passes over $\d B$. On the bottom row, we show a portion of the
	 movie for $S_{\ve{c}}$ when $c_i \not\in \ve{c}$. The upper strand of
	 the crossing $c_i$ passes underneath $\d B$.}
\label{fig::39}
\end{figure}
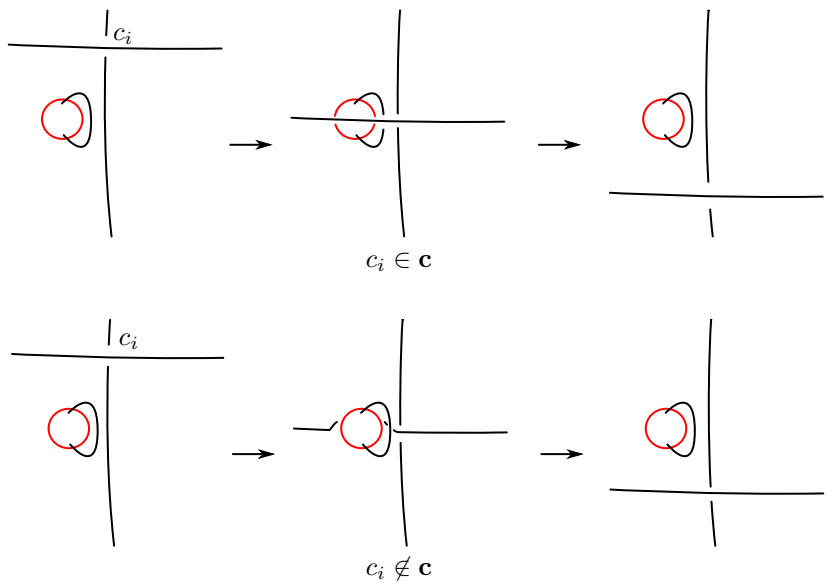

Let $c_i$ and $c_j$ be consecutive crossings in $\ve{c}$ of opposite sign.
We claim that $S_{\ve{c}}$ and $S_{\ve{c} \setminus \{c_i, c_j\}}$
become isotopic after a single stabilization. The stabilization is
obtained by attaching a band connecting the upper strands of the crossings $c_i$ and $c_j$,
followed by attaching the dual band; see Figure~\ref{fig::40}.

 \begin{figure}[ht!]
	\centering
\begingroup%
  \makeatletter%
  \providecommand\color[2][]{%
    \errmessage{(Inkscape) Color is used for the text in Inkscape, but the package 'color.sty' is not loaded}%
    \renewcommand\color[2][]{}%
  }%
  \providecommand\transparent[1]{%
    \errmessage{(Inkscape) Transparency is used (non-zero) for the text in Inkscape, but the package 'transparent.sty' is not loaded}%
    \renewcommand\transparent[1]{}%
  }%
  \providecommand\rotatebox[2]{#2}%
  \newcommand*\fsize{\dimexpr\f@size pt\relax}%
  \newcommand*\lineheight[1]{\fontsize{\fsize}{#1\fsize}\selectfont}%
  \ifx\svgwidth\undefined%
    \setlength{\unitlength}{407.99734537bp}%
    \ifx\svgscale\undefined%
      \relax%
    \else%
      \setlength{\unitlength}{\unitlength * \real{\svgscale}}%
    \fi%
  \else%
    \setlength{\unitlength}{\svgwidth}%
  \fi%
  \global\let\svgwidth\undefined%
  \global\let\svgscale\undefined%
  \makeatother%
  \begin{picture}(1,0.26712722)%
    \lineheight{1}%
    \setlength\tabcolsep{0pt}%
    \put(0,0){\includegraphics[width=\unitlength,page=1]{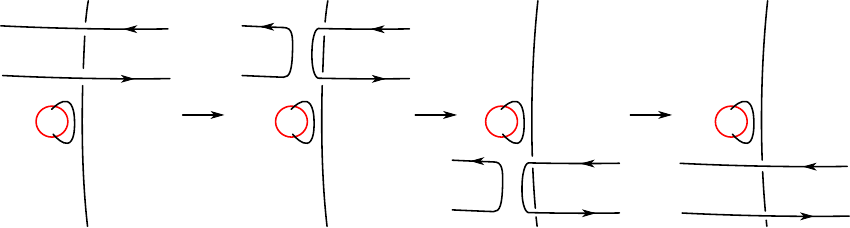}}%
    \put(0.10320771,0.18273669){\color[rgb]{0,0,0}\makebox(0,0)[lt]{\lineheight{1.25}\smash{\begin{tabular}[t]{l}$c_i$\end{tabular}}}}%
    \put(0.10835624,0.24220925){\color[rgb]{0,0,0}\makebox(0,0)[lt]{\lineheight{1.25}\smash{\begin{tabular}[t]{l}$c_j$\end{tabular}}}}%
  \end{picture}%
\endgroup%

	\caption{A movie for a common stabilization of $S_{\ve{c}}$ and
	$S_{\ve{c} \setminus \{c_i, c_j\}}$, when $c_i$ and $c_j$ are
	consecutive crossings in $\ve{c}$ of opposite sign.
    In the movie, a band is added between the
	first and the second frames. An isotopy connects the second and
	third frames. The third and fourth frames are related by attaching the dual band.}
\label{fig::40}
\end{figure}

In the case that $\wr(\cD)=0$, the number of positive crossings is
equal to the number of negative crossings, so we can eliminate all
crossings from $\ve{c}$ pairwise via the stabilization sequence described above. Hence
\[
\mu_{st}(S_{\Cr(\cD)}, S_{\emptyset}) \le 2.
\]
We note that $S_{\Cr(\cD)}$ is a genus one stabilization of $D_{K, r}$,
while the surface $S_{\emptyset}$ is described by a movie that
starts at $(B, a)$, then has a copy of $K$ break off and move away
from $\d B$. This copy of $K$ rotates $\tw(\cD)$-many times near a plane far
away from $\d B$, where $\tw(\cD)$ denotes
the \emph{twisting number} of $\cD$, and then $K$ is reattached
to the arc on $\d B$ via a band. 

It is a general fact that $\tw(\cD) + \wr(\cD)$ is always odd for
any diagram of a knot (as can be verified by noting that Reidemeister
moves and crossing changes do not change the quantity $\tw(\cD) + \wr(\cD)$
modulo 2, and that $\tw + \wr = 1$ for a trivial diagram of an unknot).
Since we picked $\cD$ to satisfy $\wr(\cD) = 0$, we conclude that
\[
\tw(\cD) \equiv 1 \pmod 2.
\]
Since a path of rotations about a line in $\R^3$ induces the
generator of $\pi_1(\SO(3)) \iso \Z_2$, we can assume that the
copy of $K$ which breaks off rotates in a plane exactly once.

Hence, our movie for $S_{\emptyset}$ becomes a copy of $K$ which
breaks off of $(B, a)$ as in Figure~\ref{fig::38}, then makes one
full rotation in a plane (away from $\d B$), and is finally
reconnected to the arc attached to $\d B$.
A continuous deformation transforms this final movie into a
stabilization of the movie for $D_{K, t}$ shown in Figure~\ref{fig::36}.
\end{proof}

We conclude this section with the following conjecture (compare Proposition~\ref{prop:roll}, below).

\begin{conj}\label{conj:rt}
If $K$ is a knot in $\Sphere^3$, then the deform-spun slice disks $D_{K, r}$,
$D_{K, t} \in \Surf_0(-K \# K)$ satisfy
\[
\mu_{st}(D_{K, r}, D_{K, t}) \le 1.
\]
\end{conj}

\section{Background on the link Floer TQFT}
\label{sec:technicalbackground}

In this section, we recall some previous results about the link Floer TQFT which we will need
to compute the effect of stabilization on the link cobordism maps in Section~\ref{sec:stab},
to determine the map induced by the summand-swapping diffeomorphism in Section~\ref{sec:rigidmotiondeformation},
and to prove the trace formula in Section~\ref{sec:trace}.

\subsection{The full link Floer TQFT}\label{sec:background}

We first recall the category whose objects are multi-based links, and
whose morphisms are decorated link cobordisms, following the notation
of the second author~\cite{ZemCFLTQFT}; see also the equivalent construction of the first author~\cite{JCob}.

\begin{define}
A \emph{multi-based link} $\bL = (L,\ws,\zs)$ in a closed, oriented (not necessarily connected) 3-manifold $Y$ is an oriented link
$L \subset Y$, together with two disjoint collections of basepoints $\ws$, $\zs \subset L$ such that
\begin{enumerate}
\item each component of $L$ has at least two basepoints;
\item the basepoints along a link component of $L$ alternate between $\ws$ and $\zs$, as one traverses the link;
\item each component of $Y$ has at least one component of $L$, and each component of $L$ has at least two basepoints. 
\end{enumerate}
\end{define}

\begin{define}
Let $Y_1$ and $Y_2$ be 3-manifolds containing multi-based links $\bL_1 = (L_1, \ws_1, \zs_1)$
and $\bL_2 = (L_2, \ws_2, \zs_2)$, respectively. A \emph{decorated link cobordism} from
$(Y_1, \bL_1)$ to $(Y_2, \bL_2)$ is a pair $(W,\cF)=(W,(S,\cA))$, where
\begin{enumerate}
\item $W$ is an oriented cobordism from $Y_1$ to $Y_2$,
\item $S$ is a properly embedded, oriented surface in $W$ with $\d S = -L_1 \cup L_2$, and
\item $\cA$ is a properly embedded 1-manifold in $S$
that divides $S$ into two subsurfaces, $S_{\ws}$ and $S_{\zs}$,
that meet along $\cA$, such that $\ws_1$, $\ws_2 \subset S_{\ws}$
and $\zs_1$, $\zs_2 \subset S_{\zs}$.
\end{enumerate}
\end{define}

Multi-based links and equivalence classes of decorated link cobordisms form a category.

The first author~\cite{JCob} showed that decorated link cobordisms induce functorial maps
on the hat version of link Floer homology.
The second author \cite{ZemCFLTQFT} extended this to the \emph{full infinity
complex}, denoted $\cCFL^\infty$, which is a minor variation of the
infinity complexes of Ozsv\'{a}th--Szab\'{o}~\cite{OSKnots} and
Rasmussen~\cite{RasmussenKnots}. We now review the construction of $\cCFL^\infty$.

Let $\cR^-$ denote the ring $\bF_2[U,V]$, and let $\cR^\infty$ denote the
ring $\bF_2[U,V,U^{-1},V^{-1}]$, obtained by inverting $U$ and $V$ in $\cR^-$.
Suppose that $\bL=(L,\ws,\zs)$ is a multi-based link in $Y$. Given a
multi-pointed diagram $(\Sigma,\as,\bs,\ws,\zs)$ for $(Y,\bL)$ (see
\cite{OSLinks}*{Definition~3.1}), the complex $\cCFL^\infty(Y,\bL,\frs)$ is
the free module over $\cR^{\infty}$ generated by intersection points $\xs\in
\bT_{\a}\cap \bT_{\b}$ with $\frs_{\ws}(\xs)=\frs$. Over $\bF_2$, the generators are the monomials
\[
U^i V^j \cdot \xs,
\]
where $i$, $j \in \Z$.

The module $\cCFL^\infty(Y,\bL,\frs)$ has a filtration $\cG$ over
$\Z\oplus \Z$, where the subset $\cG_{(n,m)}\subset \cCFL^\infty(Y,\bL,\frs)$
is generated over $\bF_2$ by monomials $U^i V^j \cdot \xs$ with $i\ge n$ and $j\ge m$.
We denote the $\cR^-$-submodule $\cG_{(0,0)}$ by $\cCFL^-(Y,\bL,\frs)$, and
call it the \emph{full minus complex}. It is generated over $\bF_2$ by monomials $U^i
V^j  \cdot \xs$ with $i$, $j \ge 0$.

There is a filtered, $\cR^\infty$-equivariant endomorphism $\d$ of $\cCFL^\infty(Y,\bL,\frs)$,
defined by the formula
\[
\d \xs = \sum_{\ys\in \bT_{\a}\cap \bT_{\b}}\sum_{\substack{\phi\in \pi_2(\xs,\ys)\\
\mu(\phi) = 1}} \# \hat{\cM}(\phi) \cdot U^{n_{\ws}(\phi)} V^{n_{\zs}(\phi)} \cdot \ys,
\]
which satisfies $\d \circ \d = 0$.

When $[L] = 0$ in $H_1(Y)$ and $\frs$ is torsion,
the chain complex $(\cCFL^\infty(Y,\bL,\frs),\d)$ has several gradings.
Ozsv\'{a}th and Szab\'{o} defined a homological grading and an Alexander
grading. It is convenient for our purposes to repackage their two gradings
into three gradings that satisfy a linear dependency. Namely, there are two
homological gradings, $\gr_{\ws}$ and $\gr_{\zs}$, and an Alexander grading
$A$, which together satisfy
\[
A = \frac{1}{2}(\gr_{\ws}-\gr_{\zs}).
\]
Note that $V$ is $+1$ graded with respect to $A$, and $U$ is $-1$ graded.

When $\bK = (K,w,z)$ is a doubly-based knot in $\Sphere^3$,
we will write $\cCFL^\infty(\bK)$ for $\cCFL^\infty(\Sphere^3, \bK,\frs_0)$, where $\frs_0$ is the unique $\Spin^c$ structure on $\Sphere^3$.

The second author \cite{ZemCFLTQFT} constructed cobordism maps for the full knot
and link Floer complexes. Given a decorated link cobordism
$(W, \cF)$ from $(Y_1, \bL_1)$ to $(Y_2, \bL_2)$ and a $\Spin^c$ structure $\frs\in \Spin^c(W)$,
there is an induced $\cR^\infty$-equivariant, filtered chain map
\[
F_{W,\cF,\frs} \colon \cCFL^\infty(Y_1,\bL_1,\frs|_{Y_1}) \to \cCFL^\infty(Y_2,\bL_2,\frs|_{Y_2}),
\]
well-defined up to filtered, $\cR^\infty$-equivariant chain homotopy.

\subsection{Basepoint actions on link Floer homology}
\label{sec:basepointactions}

Let $\bL = (L, \ws, \zs)$ be a multi-based link in the 3-manifold $Y$, and fix $\frs \in \Spin^c(Y)$.
We recall that, for each $w\in \ws$ and $z\in \zs$, there are distinguished endomorphisms $\Phi_w$ and $\Psi_z$ of $\cCFL^\infty(Y,\bL,\frs)$. If $\cH = (\Sigma,\as,\bs,\ws,\zs)$ is a
diagram for  $(Y, \bL)$, then $\Phi_w$ and $\Psi_z$ can be defined via the formulas
\begin{equation}
\Phi_w(\xs) = \sum_{\ys \in \bT_\a \cap \bT_\b} \sum_{\substack{\phi \in \pi_2(\xs,\ys)\\
\mu(\phi) = 1}} n_w(\phi) \# \hat{\cM}(\phi) \cdot U^{n_{\ws}(\phi)-1} V^{n_{\zs}(\phi)} \cdot \ys,
\label{eq:defPhi}
\end{equation}
and
\begin{equation}
\Psi_z(\xs) = \sum_{\ys\in \bT_\a \cap \bT_\b} \sum_{\substack{\phi \in \pi_2(\xs,\ys)\\
\mu(\phi) = 1}} n_z(\phi) \# \hat{\cM}(\phi) \cdot U^{n_{\ws}(\phi)} V^{n_{\zs}(\phi)-1} \cdot \ys
\label{eq:defPsi}
\end{equation}
for $\x \in \T_{\a} \cap \T_{\b}$ with $\frs_{\ws}(\x) = \frs$.
According to \cite{ZemConnectedSums}*{Lemma~4.1}, 
the endomorphisms $\Phi_w$
and $\Psi_z$ are the link cobordism maps induced by the two decorations of $(I\times Y,I\times L)$ shown in Figure~\ref{fig::6}.
When we are working with doubly-based knots, we will often write $\Phi$ and $\Psi$ for the maps~$\Phi_w$
and~$\Psi_z$, respectively.

\begin{figure}[ht!]
	\centering
\begingroup%
  \makeatletter%
  \providecommand\color[2][]{%
    \errmessage{(Inkscape) Color is used for the text in Inkscape, but the package 'color.sty' is not loaded}%
    \renewcommand\color[2][]{}%
  }%
  \providecommand\transparent[1]{%
    \errmessage{(Inkscape) Transparency is used (non-zero) for the text in Inkscape, but the package 'transparent.sty' is not loaded}%
    \renewcommand\transparent[1]{}%
  }%
  \providecommand\rotatebox[2]{#2}%
  \newcommand*\fsize{\dimexpr\f@size pt\relax}%
  \newcommand*\lineheight[1]{\fontsize{\fsize}{#1\fsize}\selectfont}%
  \ifx\svgwidth\undefined%
    \setlength{\unitlength}{276.87460579bp}%
    \ifx\svgscale\undefined%
      \relax%
    \else%
      \setlength{\unitlength}{\unitlength * \real{\svgscale}}%
    \fi%
  \else%
    \setlength{\unitlength}{\svgwidth}%
  \fi%
  \global\let\svgwidth\undefined%
  \global\let\svgscale\undefined%
  \makeatother%
  \begin{picture}(1,0.21120918)%
    \lineheight{1}%
    \setlength\tabcolsep{0pt}%
    \put(0,0){\includegraphics[width=\unitlength,page=1]{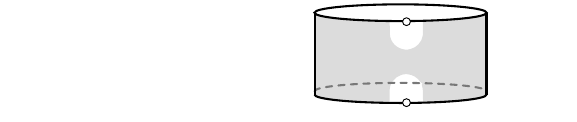}}%
    \put(0.87759499,0.11094733){\color[rgb]{0,0,0}\makebox(0,0)[lt]{\lineheight{0}\smash{\begin{tabular}[t]{l}$\Psi_z$\end{tabular}}}}%
    \put(0.67713367,0.00542042){\color[rgb]{0,0,0}\makebox(0,0)[lt]{\lineheight{0}\smash{\begin{tabular}[t]{l}$z$\end{tabular}}}}%
    \put(0,0){\includegraphics[width=\unitlength,page=2]{fig6.pdf}}%
    \put(0.33208817,0.10961862){\color[rgb]{0,0,0}\makebox(0,0)[lt]{\lineheight{0}\smash{\begin{tabular}[t]{l}$\Phi_w$\end{tabular}}}}%
    \put(0.12400369,0.00511354){\color[rgb]{0,0,0}\makebox(0,0)[lt]{\lineheight{0}\smash{\begin{tabular}[t]{l}$w$\end{tabular}}}}%
    \put(0,0){\includegraphics[width=\unitlength,page=3]{fig6.pdf}}%
  \end{picture}%
\endgroup%

	\caption{The two decorated link cobordisms for $\Phi_w$ and $\Psi_z$. The diagrams indicate the decorations on the surface $I\times L$. Here we denote $\ws$ by solid basepoints, and $\zs$ by open basepoints. The shaded regions denote $\Sigma_{\ws}$ and the unshaded regions denote $\Sigma_{\zs}$. }\label{fig::6}
\end{figure}

According to \cite{ZemCFLTQFT}*{Lemma~4.9}, the basepoint actions satisfy
\begin{equation}
\Phi_w^2 \simeq \Psi_{z}^2 \simeq 0.
\label{eqn:R7}
\end{equation}
Note that the dividing sets on the decorated link cobordisms corresponding to $\Phi_w^2$ and $\Psi_z^2$
both contain a closed curve that bounds a disk in either $S_{\ws}$ or $S_{\zs}$.

\subsection{Quasi-stabilizations and basepoint moving maps}
\label{sec:backgroundquasibasepoint}

We now review the quasi-stabilization maps. Suppose $\bL = (L,\ws,\zs)$ is a
multi-based link in $Y$, and suppose that $w$ and $z$ are two new
basepoints contained in a single component of $L \setminus (\ws \cup \zs)$. Let us assume
that~$w$ immediately follows~$z$ with respect to the orientation of $L$,
and write
\[
\bL_{w,z}^+ := (L, \ws \cup \{w\}, \zs \cup \{z\}).
\]
There are two \emph{quasi-stabilization} maps
\[
S_{w,z}^+ \text{, } T_{w,z}^+ \colon \cCFL^\infty(Y, \bL, \frs) \to \cCFL^\infty(Y, \bL^+_{w,z}, \frs),
\]
as well as two \emph{quasi-destabilization} maps $S_{w,z}^-$ and $T_{w,z}^-$, defined in the opposite direction.
If instead $z$ follows $w$ with respect to the orientation of $L$, then we
obtain similar maps $S_{z, w}^\pm$ and $T_{z,w}^\pm$.

We briefly summarize the construction of the quasi-stabilization maps. See \cite{MoIntSurg}*{Section~6} and  \cite{ZemCFLTQFT}*{Section~4} for further details. Suppose
$\cH=(\Sigma,\as,\bs,\ws,\zs)$ is a diagram for $(Y,\bL)$. Let $w'$ and $z'$
denote the basepoints adjacent to $w$ and $z$ on $\bL$. There is a component
$A$ of $\Sigma\setminus \as$ which contains~$w'$ and~$z'$. We pick a simple
closed curve $\alpha_s\subset A$ that cuts $A$ into two components, one of
which contains $w'$ and the other $z'$. We add another curve, $\beta_0$,
that bounds a small disk on $\Sigma$, which is cut into two bigons by
$\alpha_s$, and is disjoint from the other $\as$ curves. Inside one bigon, we
place $w$. In the other bigon, we place $z$. See Figure~\ref{fig::43}.

\begin{figure}[ht!]
	\centering
	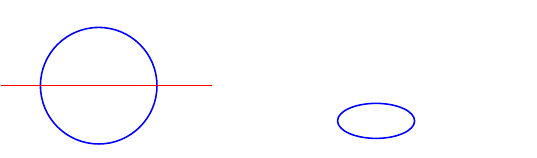
	\caption{A quasi-stabilization of Heegaard diagrams. On the left, a local picture of the quasi-stabilization near $w$ and $z$ is shown. On the right, we have a complete example of a quasi-stabilization. The shaded region denotes $A$. Also shown are the points $w'$ and $z'$. }
    \label{fig::43}
\end{figure}

We write $\theta^{\ws}$ and $\theta^{\zs}$ for the higher $\gr_{\ws}$- and
$\gr_{\zs}$-graded intersection points of $\alpha_s \cap \beta_0$,
respectively. The maps $S_{w,z}^+$ and $T_{w,z}^+$ are defined via the
formulas
\begin{equation}
S_{w,z}^+(\xs) := \xs \times \theta^{\ws} \quad \text{and}
\quad T_{w,z}^+(\xs) := \xs \times \theta^{\zs}\label{eq:quasi-stabilization}
\end{equation}
for $\x \in \T_{\a} \cap \T_{\b}$, extended $\cR^{\infty}$-equivariantly.
The quasi-destabilization maps are defined dually, via the equations
\begin{equation}
S_{w,z}^-(\xs\times \theta^{\ws})=0, \quad S_{w,z}^-(\xs\times \theta^{\zs}) =
\xs, \quad T_{w,z}^-(\xs\times \theta^{\ws})=\xs,\quad \text{and} \quad T_{w,z}^-(\xs\times \theta^{\zs}) = 0,
\label{eq:quasi-destabilization}
\end{equation}
extended $\cR^{\infty}$-equivariantly.
The decorations on $(I \times Y, I \times L)$ inducing the
quasi-stabilization maps $S_{w,z}^+$, $S_{w,z}^-$, $T_{w,z}^+$, and
$T_{w,z}^-$ are shown in Figure~\ref{fig::15}.

\begin{figure}[ht!]
	\centering
	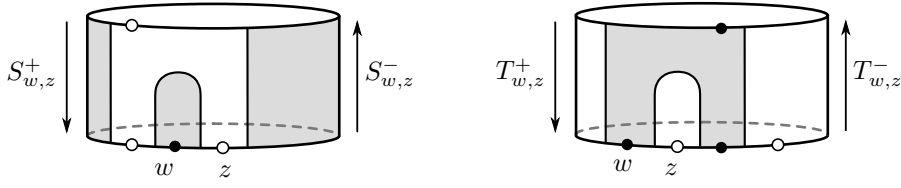
	\caption{The decorated link cobordisms for the quasi-stabilization maps.}
    \label{fig::15}
\end{figure}

We will need the following relations between the quasi-stabilization maps and the basepoint actions:
\begin{align}
T_{w,z}^+ &\simeq \Psi_z S_{w,z}^+, \label{eq:R1} \\
T_{w,z}^- &\simeq S_{w,z}^-\Psi_z, \label{eq:R2} \\
\Phi_w    &\simeq S_{w,z}^+S_{w,z}^-, \label{eq:R3} \\
\Psi_z    &\simeq T_{w,z}^+T_{w,z}^-. \label{eq:R4}
\end{align}
Proofs of equations~\eqref{eq:R1}--\eqref{eq:R4} can be found in \cite{ZemCFLTQFT}*{Lemmas~4.10 and 4.11}.
Examples of the corresponding dividing sets for the relations in equations~\eqref{eq:R1}--\eqref{eq:R4}
appear in Figure~\ref{fig::16}.

\begin{figure}[ht!]
	\centering
	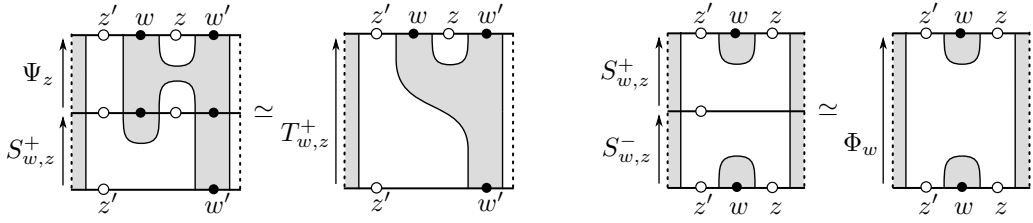
	\caption{Dividing set manipulations corresponding to the relations
    $T_{w,z}^+ \simeq \Psi_z S_{w,z}^+$ and $\Phi_w \simeq S_{w,z}^+ S_{w,z}^-$.}
    \label{fig::16}
\end{figure}

Next, we review the connection between the basepoint moving maps
and the quasi-stabilization maps.
We first focus on using the quasi-stabilization maps to move a single
basepoint, while fixing all other basepoints. Suppose that $(L, \ws_0 \cup \{w'\}, \zs)$
is a multi-based link in $Y$. Suppose that $(w,z)$ are a new
pair of basepoints in a single component of $L \setminus (\ws_0 \cup \{w'\} \cup \zs)$,
such that $z$ is adjacent to~$w'$. Suppose further that, according to
the orientation of $L$, the three basepoints appear in the order $w'$, $z$, $w$.
We can construct a diffeomorphism
\[
\tau^{w\from w'} \colon (Y,L,\ws_0\cup \{w'\},\zs) \to (Y,L,\ws_0\cup \{w\}, \zs),
\]
by moving $w'$ to $w$ along the arc connecting them, but fixing all of $Y$ outside a neighborhood of this arc.
According to \cite{ZemCFLTQFT}*{Lemma~4.24},
\begin{equation}
T^+_{w,z}\simeq T_{z,w'}^+ \tau^{w \from w'}_*.\label{eq:movebasepointandquasistabilization}
\end{equation}
Equation~\eqref{eq:movebasepointandquasistabilization} has a simple description in terms of dividing sets and link cobordisms,
shown in Figure~\ref{fig::25}.

\begin{figure}[ht!]
	\centering
	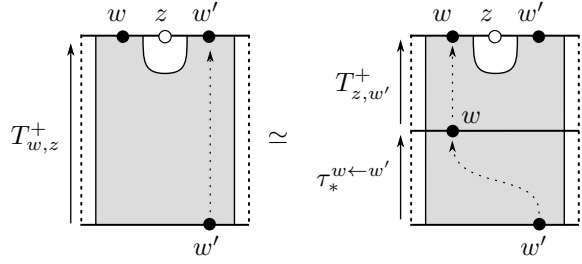
	\caption{The interpretation of equation~\eqref{eq:movebasepointandquasistabilization}
    in terms of decorated link cobordisms.}
\label{fig::25}
\end{figure}

Using the quasi-stabilization maps, we can also move a pair of adjacent
basepoints at the same time. Suppose that $\bL = (L,\ws_0,\zs_0)$ is a
multi-based link (though we allow the case where $\ws_0$ and $\zs_0$
intersect a single  component of $L$ trivially). Suppose that $(w',z',w,z)$
are four new consecutive basepoints on a single
component of $L \setminus (\ws_0 \cup \zs_0)$. We assume that these four basepoints appear in the order $w'$, $z'$, $w$ and $z$ on the link, with $z$ being the first and $w'$ being the last with respect to the orientation of $L$. There is a diffeomorphism
\[
\rho^{(w',z') \from (w,z)} \colon (Y ,L, \ws_0 \cup \{w\}, \zs_0 \cup \{z\}) \to
(Y, L, \ws_0 \cup \{w'\}, \zs_0 \cup \{z'\}),
\]
obtained by moving the pair $(w,z)$ to $(w',z')$, but fixing everything outside a neighborhood
of an interval containing the four basepoints $(w',z',w,z)$.
According to \cite{ZemCFLTQFT}*{Lemma~4.27}, there is a chain homotopy
\begin{equation}
\rho^{(w',z') \from (w,z)}_* \simeq S_{w,z}^- T_{w',z'}^+.
\label{eq:R5}
\end{equation}
Equation~\eqref{eq:R5} can be interpreted in terms of the manipulation of
dividing sets shown in Figure~\ref{fig::17}.

\begin{figure}[ht!]
	\centering
	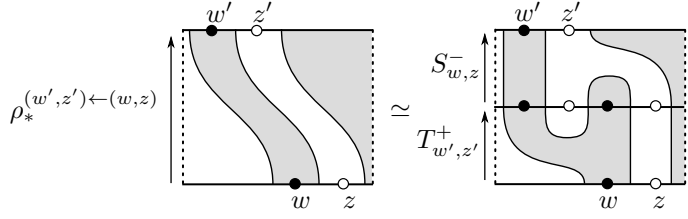
	\caption{The interpretation of the relation $\rho^{(w',z') \from (w,z)}_*\simeq S_{w,z}^- T_{w',z'}^+$
    in terms of dividing sets.}
    \label{fig::17}
\end{figure}

Another useful relation  is the following:
\begin{equation}
T_{w,z}^+ S_{w,z}^- + S_{w,z}^+ T_{w,z}^- + \id \simeq 0
\label{eq:R6}.
\end{equation}
See \cite{ZemCFLTQFT}*{Lemma~4.13}. Note that Equation~\eqref{eq:R6} follows immediately from the formulas in Equations~\eqref{eq:quasi-stabilization} and~\eqref{eq:quasi-destabilization}.
For a pictorial description, see Figure~\ref{fig::18}.
Equation~\eqref{eq:R6} is an example of the \emph{bypass relation},
which is often helpful when doing computations in the link Floer TQFT.

\begin{figure}[ht!]
	\centering
	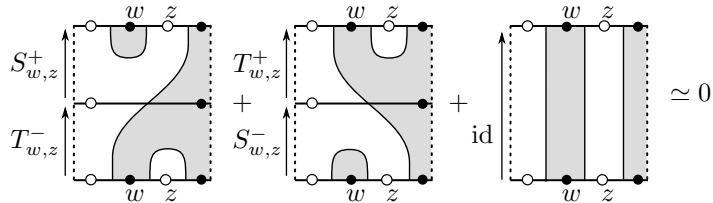
	\caption{A pictorial description of the bypass relation, equation~\eqref{eq:R6}.}
    \label{fig::18}
\end{figure}

\subsection{Cobordism maps for saddles}

Next, we discuss the maps for saddle cobordisms.
Suppose that $\bL = (L,\ws,\zs)$ is a multi-based link in $Y$,
and that $B$ is an embedded band for $L$ that has both ends in subarcs of $L \setminus (\ws\cup \zs)$
that run from $\ws$ to $\zs$, or has both ends in subarcs that run from $\zs$ to $\ws$.
Assuming that $\bL(B)$, the result of band surgery, is also a valid multi-based link, there is a map
\[
F_B^{\zs} \colon \cCFL^\infty(Y,\bL,\frs) \to \cCFL^\infty(Y,\bL(B),\frs),
\]
described in \cite{ZemCFLTQFT}*{Section~6}.
The map $F_{B}^{\zs}$ corresponds to a saddle link cobordism,
with an index~1 critical point occurring in the type-$\zs$ subregion.
There is another map $F_{B}^{\ws}$, with the same domain and codomain as $F_{B}^{\zs}$,
that corresponds to adding a band to the type-$\ws$ subregion.

The relation between the band maps and the basepoint maps
is studied in \cite{ZemCFLTQFT}*{Section~9.1}. According to \cite{ZemCFLTQFT}*{Lemma~9.1},
\begin{equation}
F_{B}^{\zs} \Phi_w + \Phi_w F_{B}^{\zs} \simeq 0.
\label{eq:R7}
\end{equation}
In contrast, the map $F_{B}^{\zs}$ does not always commute with $\Psi_z$.
Instead, if $z$ is either of the two $\zs$-basepoints adjacent to the ends of $B$, then
\begin{equation}
F_B^{\zs} \Psi_z + \Psi_zF_B^{\zs} \simeq F_{B}^{\ws},
\label{eq:commutatorFBzandPsi}
\end{equation}
according to \cite{ZemCFLTQFT}*{Proposition~9.3}.
In fact, the three dividing sets corresponding to the maps in equation~\eqref{eq:commutatorFBzandPsi}
can be interpreted as a bypass relation on a saddle cobordism; see \cite{ZemCFLTQFT}*{Figure~9.1}.

Note that, if $z$ and $z'$ are the two $\zs$-basepoints adjacent to the ends of $B$,
then equation~\eqref{eq:commutatorFBzandPsi} holds for both $z$ and $z'$.
As a consequence, if we sum both relations, we obtain
\begin{equation}
F_B^{\zs}(\Psi_z + \Psi_{z'}) + (\Psi_z + \Psi_{z'}) F_B^{\zs} \simeq 0.
\label{eq:R8}
\end{equation}

\subsection{Birth cobordisms and quasi-stabilizations}
\label{sec:birthsandquasis}

We recall from \cite{ZemCFLTQFT}*{Section~7.1} the birth cobordism map.
Suppose $\bU=(U,w,z)$ is a doubly-based unknot, which is unlinked from the
multi-based link in $Y$, and is given a distinguished Seifert disk $D$. In
this situation, there is a well defined birth map
\[
\cB_{\bU,D}^+\colon \cCFL^\infty(Y,\bL,\frs)\to \cCFL^\infty(Y,\bL\cup \bU,\frs).
\]
The map $\cB_{\bU,D}^+$ corresponds to a birth cobordism in $I\times Y$,
where the disk portion of the link cobordism surface is decorated with a
single dividing arc.

The map $\cB_{\bU,D}^+$ can be computed as follows. We pick a Heegaard
diagram $(\Sigma,\as,\bs,\ws,\zs)$ for $(Y,\bL)$ such that $w,z\in \Sigma$,
and $D\cap \Sigma$ consists of an embedded arc in $\Sigma\setminus (\as\cup
\bs)$ that connects $w$ and $z$. We add two new curves, $\alpha_0$ and
$\beta_0$, that bound a small disk containing $D\cap \Sigma$, and intersect
in a pair of points. Let $\theta_{\alpha_0,\beta_0}^+\in \alpha_0\cap
\beta_0$ denote the higher  Maslov graded intersection point (the designation
is the same for both $\gr_{\ws}$ and $\gr_{\zs}$). The map $\cB_{\bU,D}^+$ is
defined via the formula
\[
\cB_{\bU,D}^+(\xs)=\xs\times \theta^+_{\alpha_0,\beta_0}
\]
for $\x \in \T_{\a} \cap \T_{\b}$, and extended $\cR^{\infty}$-equivariantly.
There is also a death map $\cD_{\bU,D}^-$, defined in the opposite direction,
though it will not make an appearance in this paper.

Suppose $L$ is a link in $Y$, and $U$ is an unknot that bounds a Seifert
disk $D$,  disjoint from $L$. Suppose further that $B$ is a band connecting
$U$ and $L$, which is disjoint from the interior of $D$. Let
\[
\phi\colon  (Y,L)\to (Y,(L\cup U)(B))
\]
denote a diffeomorphism which is the identity outside a neighborhood of $B\cup D$.

Since a birth cobordism adds two basepoints, the composition of a birth
cobordism map and a band map is not simply the diffeomorphism map $\phi_*$.
Instead the composition is a quasi-stabilization. More precisely, if
$\bU=(U,w,z)$, and $B$ is an $\alpha$-band (i.e., has both ends on strands of
$L$ that lie in the $\alpha$-handlebody), then
\begin{equation}
F_{B}^{\ws}\cB^+_{\bU,D}\simeq T_{w,z}^+ \phi_*.
\label{eq:band-birth=quasi}
\end{equation}
A proof of equation~\eqref{eq:band-birth=quasi} can be found in
\cite{ZemCFLTQFT}*{Proposition~8.5}. Note that there are other variations of
equation~\eqref{eq:band-birth=quasi}. For example, if $B$ is instead a $\beta$-band, then
\begin{equation}
F_{B}^{\ws} \cB^+_{\bU,D} \simeq T_{z,w}^+ \phi_*.
\label{eq:band-birth=quasi2}
\end{equation}
A schematic illustrating equation~\eqref{eq:band-birth=quasi} is shown in Figure~\ref{fig::26}.

\begin{figure}[ht!]
	\centering
\begingroup%
  \makeatletter%
  \providecommand\color[2][]{%
    \errmessage{(Inkscape) Color is used for the text in Inkscape, but the package 'color.sty' is not loaded}%
    \renewcommand\color[2][]{}%
  }%
  \providecommand\transparent[1]{%
    \errmessage{(Inkscape) Transparency is used (non-zero) for the text in Inkscape, but the package 'transparent.sty' is not loaded}%
    \renewcommand\transparent[1]{}%
  }%
  \providecommand\rotatebox[2]{#2}%
  \newcommand*\fsize{\dimexpr\f@size pt\relax}%
  \newcommand*\lineheight[1]{\fontsize{\fsize}{#1\fsize}\selectfont}%
  \ifx\svgwidth\undefined%
    \setlength{\unitlength}{270.71379391bp}%
    \ifx\svgscale\undefined%
      \relax%
    \else%
      \setlength{\unitlength}{\unitlength * \real{\svgscale}}%
    \fi%
  \else%
    \setlength{\unitlength}{\svgwidth}%
  \fi%
  \global\let\svgwidth\undefined%
  \global\let\svgscale\undefined%
  \makeatother%
  \begin{picture}(1,0.3819263)%
    \lineheight{1}%
    \setlength\tabcolsep{0pt}%
    \put(0,0){\includegraphics[width=\unitlength,page=1]{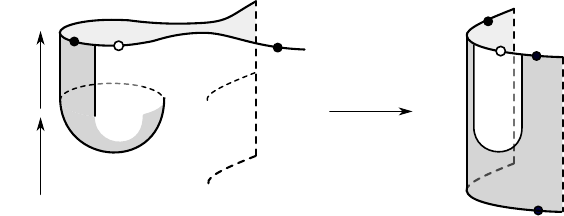}}%
    \put(0.05852456,0.09652627){\color[rgb]{0,0,0}\makebox(0,0)[rt]{\lineheight{1.25}\smash{\begin{tabular}[t]{r}$\cB^+_{\bU,D}$\end{tabular}}}}%
    \put(0.05852456,0.23885825){\color[rgb]{0,0,0}\makebox(0,0)[rt]{\lineheight{1.25}\smash{\begin{tabular}[t]{r}$F_B^{\ws}$\end{tabular}}}}%
    \put(0,0){\includegraphics[width=\unitlength,page=2]{fig26.pdf}}%
    \put(0.78910031,0.11903598){\color[rgb]{0,0,0}\makebox(0,0)[rt]{\lineheight{1.25}\smash{\begin{tabular}[t]{r}$T_{w,z}^+$\end{tabular}}}}%
    \put(0,0){\includegraphics[width=\unitlength,page=3]{fig26.pdf}}%
  \end{picture}%
\endgroup%

	\caption{A manipulation demonstrating equation~\eqref{eq:band-birth=quasi}.
    The composition of a birth cobordism followed by a band is (up to diffeomorphism) a quasi-stabilization.}
\label{fig::26}
\end{figure}

\subsection{4-dimensional connected sums of link cobordisms}

Suppose $(W_1,\cF_1)$ and $(W_2,\cF_2)$ are two link cobordisms, and pick
two embedded 4-balls $D_1\subset W_1$ and $D_2\subset W_2$ such that $D_i\cap \cF_i$ consists of a 2-dimensional disk which intersects the dividing set of $\cF_i$ in a single arc. We glue $W_1\setminus \Int(D_1)$ to $W_2\setminus \Int(D_2)$ using an orientation-reversing diffeomorphism which restricts to an orientation-reversing diffeomorphism of $\cF_1\cap \d D_1$ and $\cF_2\cap \d D_2$, and is compatible with the dividing sets.  We write $(W_1\# W_2,\cF_1\# \cF_2)$ for the resulting link cobordism.  Using a handle cancellation argument
in the connected sum region, one can prove that
\begin{equation}
F_{W_1\# W_2, \cF_1\# \cF_2,\frs_1\# \frs_2}\simeq F_{W_1,\cF_1,\frs_1}\otimes F_{W_2, \cF_2, \frs_2}.
\label{eq:connectedsum=tensorprod}
\end{equation}
See \cite{ZemConnectedSums}*{Proposition~5.2} for a detailed proof. To make
use of equation~\eqref{eq:connectedsum=tensorprod}, it will be convenient to
have a more explicit description of the cobordism map for $(W_1\# W_2,\cF_1\#
\cF_2)$. To this end, suppose  the following:
\begin{enumerate}
\item $(Y_1,\bL_1)$ and $(Y_2,\bL_2)$ are two 3-manifolds with multi-based links.
\item $\bS_0\subset Y_1\sqcup Y_2$ is a framed 0-sphere with one foot in $Y_1$ and the other in $Y_2$.
\item  $\bS_2\subset Y_1\# Y_2$ is the dual framed 2-sphere.
\item $B$ is a band in $Y_1\# Y_2$ that connects $\bL_1$ and $\bL_2$, and intersects $\bS_2$ in a single arc.
\item $B$ is adjacent to the basepoints $w_1$ and $z_1$ on $\bL_1$, as well as $w_2$ and $z_2$ on $\bL_2$.
\item $B'$ is the band in $Y_1\# Y_2$ attached to $\bL_1 \# \bL_2$ dual to $B$.
\end{enumerate}
When $(W_i, \cF_i)$ is the identity cobordism of $(Y_i, \bL_i)$ for $i \in \{1, 2\}$,
equation~\eqref{eq:connectedsum=tensorprod} can be rewritten as
\begin{equation}
F_{\bS_2} F_{B'}^{\ws} \Phi_w F_{B}^{\ws} F_{\bS_0}\simeq \id,
\label{eq:connectedsum=tensorprod2}
\end{equation}
where $w \in \{w_1,w_2\}$. Figure~\ref{fig::27} shows equation~\eqref{eq:connectedsum=tensorprod2}
in terms of dividing sets.

\begin{figure}[ht!]
\centering
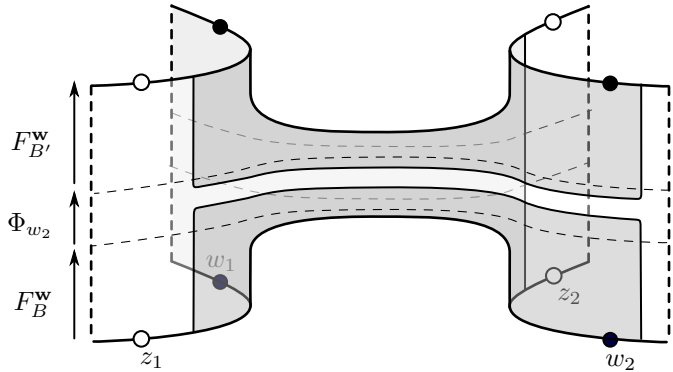
\caption{A schematic of equation~\eqref{eq:connectedsum=tensorprod2}, the
effect of taking the connected sum of two link cobordisms. There is
additionally a 4-dimensional 1-handle and 3-handle, which are not shown.}
\label{fig::27}
\end{figure}

\section{Heegaard Floer invariants of surfaces}\label{sec:invariants}

\subsection{Variations on the knot Floer complex}
\label{sec:variations}
In this section, we describe several variations on the full infinity complex $\cCFL^\infty(Y,\bL,\frs)$, which we will use to define our invariants. We focus on the case that $Y=\Sphere^3$ and $\bL=\bK = (K, w, z)$ is a doubly-based knot.

The first variation we consider is the  \emph{standard infinity complex}, denoted $\CFK^\infty(\bK)$. It is defined as the homogeneous subset of $\cCFL^\infty(\bK)$ in Alexander grading zero; i.e., the $\bF_2$-vector space generated by monomials
\[
U^i V^j\cdot \xs
\]
with $A(\xs)+j-i=0$. Since the actions of $U$ and $V$ are $-1$ and $+1$ graded with respect to the Alexander grading, the chain complex $\CFK^\infty(\bK)$ is not an $\bF_2[U,V]$-module. However the product $\hat{U} := UV$ is 0-graded with respect to $A$, and we view $\CFK^\infty(\bK)$ as an $\bF_2[\hat{U},\hat{U}^{-1}]$-module. The complex $\CFK^\infty(\bK)$ contains essentially the same information as $\cCFL^\infty(\bK)$. We will use $\CFK^\infty(\bK)$ to define our invariants $V_k(S,S')$ and $\cI(S)$.

There are two \emph{small minus complexes},
\[
\CFK^-_{U=0}(\bK):=\cCFL^-(\bK)\otimes_{\cR^-} \cR^-/(U)\,\text{ and }\, \CFK^-_{V=0}(\bK):=\cCFL^-(\bK)\otimes_{\cR^-} \cR^-/(V),
\]
which are modules over $\bF_2[V]$ and $\bF_2[U]$, respectively.
By inverting $V$ or $U$, respectively, we obtain the \emph{small infinity complexes}
\[
\CFK^\infty_{U=0}(\bK) \quad \text{and} \quad \CFK^\infty_{V=0}(\bK),
\]
whose homologies are canonically isomorphic to $\bF_2[V,V^{-1}]$ and $\bF_2[U,U^{-1}]$, respectively.

By setting $V=1$ in the complex $\CFK^-_{U=0}(\bK)$, we obtain the \emph{Alexander filtered complex}
\[
\hCFKfz(\bK),
\]
described by Ozsv\'{a}th and Szab\'{o}~\cite{OS4ballgenus}.
It has an increasing filtration over $\Z$; i.e., we have an increasing sequence of subcomplexes
\[
\hCFKfz_i(\bK) \subset \hCFKfz_{i+1}(\bK)
\]
for $i \in \Z$. The subspace $\hCFKfz_i(\bK)$ is generated by the set of
intersection points $\xs \in \bT_{\a} \cap \bT_{\b}$ with $A(\xs) \le i$.
The differential counts holomorphic disks which are allowed to pass over $z$, but not $w$.
Note that
\[
\hCFKfz_i(\bK) = \hCFKfz(\bK) \iso \hat{\CF}(\Sphere^3)
\]
for sufficiently large $i$, while $\hCFKfz_i(\bK) = \{0\}$ for sufficiently negative $i$. In particular, the total homology of $\hCFKfz(\bK)$ is isomorphic to $\hat{\HF}(\Sphere^3)\iso \bF_2$.

Symmetrically, one can filter using the $w$ basepoint to get a $\Z$-filtered chain complex $\hCFKfw(\bK)$. The filtration $\hCFKfw(\bK)$ is \emph{decreasing}; i.e.,
\[
\hCFKfw_{i+1}(\bK)\subset \hCFKfw_{i}(\bK),
\]
where $\hCFKfw_{i}(\bK)$ is generated over $\bF_2$ by intersection points $\xs$ with $A(\xs)\ge i$.

\begin{rem}\label{rem:equivalence-filtered-versus-polynomial}
The  chain complex $\CFK^-_{U=0}(\bK)$ and the filtered chain complex $\hCFKfz(\bK)$ contain equivalent information. To see this, note that $\hCFKfz(\bK)$ is obtained from $\CFK^-_{U=0}(\bK)$ by setting $V=1$, and using the filtration induced by the Alexander grading. In the other direction, $\CFK^-_{U=0}(\bK)$ is obtained by taking a basis of intersection points of $\hCFKfz(\bK)$, and weighting an intersection point $\ys$ appearing in $\d \xs$ by $V^{A(\xs)-A(\ys)}$.
\end{rem}

Furthermore, there is a conjugation symmetry of knot Floer homology that allows one to recover $\CFK^-_{V=0}(\bK)$ from $\CFK^-_{U=0}(\bK)$, and vice versa, and similarly recover $\hCFKfw(\bK)$ from $\hCFKfz(\bK)$.
Our invariant $\tau(S,S')$ will be defined in Section~\ref{subsec:taudef} using $\hCFKfz(\bK)$,
while $\kappa(S,S')$ and $\kappa_0(S)$ in Section~\ref{sec:kappa}
in terms of $\CFK^-_{U=0}(\bK)$.

Another variation we use to construct our invariants is the \emph{$t$-modified complex}, denoted $\tCFK^-(\bK)$, due to Ozsv\'{a}th--Stipsicz--Szab\'{o} \cite{OSSUpsilon}. If $t=\frac{m}{n}\in [0,2]$ is a rational number with $m$ and $n$ relatively prime integers and $n\neq 0$, then we consider the ring $\bF_2[v^{1/n}]$, where $v$ has grading $-1$. The ring $\bF_2[v^{1/n}]$ can be given an action of $\bF_2[U,V]$ by having $U$ act by  $v^{2-t}$ and $V$ act by $v^t$. The chain complex $\tCFK^-(\bK)$ is defined as the tensor product
\[
\tCFK^-(\bK):=\cCFL^-(\bK)\otimes_{\bF_2[U,V]} \bF_2[v^{1/n}].
\]
The invariant $\Upsilon_{(S,S')}(t)$ will be defined in Section~\ref{sec:upsilon} using $\tCFK^-(\bK)$.

A final variation is the \emph{hat complex}
\[
\hat{\CFK}(\bK):=\cCFL^-(\bK)\otimes_{\cR^-} \cR^-/(U,V).
\]
We will not discuss $\hat{\CFK}(\bK)$ extensively in this paper,
since it does not contain enough information to compute most of our invariants.

\subsection{The principal invariants of a surface bounding a knot}\label{sec:invariant}
We now describe two generalizations of the slice disk invariant
$t_{D} \in \hat{\HFK}(\bK)$, defined by Marengon and the first author~\cite{JMConcordance},
to higher genus surfaces in the full infinity complex.

\begin{define}
Let $\bK=(K,w,z)$ be a multi-based knot in $\Sphere^3$, and let $S \in \Surf_g(K)$ be a surface in $B^4$ bounding $K$.
If $\cA$ is a decoration on $S$ consisting of a single arc which divides $S$ into two connected subsurfaces, then we say that the map $F_{B^4,(S,\cA)}$ is a \emph{principal invariant} of the surface $S$.

Let $\cA_{\zs}$ denote the decoration on $S$ consisting of a single arc
such that $g(S_{\zs})=g(S)$ and $g(S_{\ws})=0$; see Figure~\ref{fig::30}. We  define
\[
\ve{t}^\infty_{S, \zs} := F_{B^4, (S, \cA_{\zs})} \colon \cR^\infty \to \cCFL^\infty(\bK).
\]
Similarly, if $\cA_{\ws}$ denotes the decoration on $S$ with $g(S_{\ws})=g(S)$ and $g(S_{\zs})=0$,  we define
\[
\ve{t}^\infty_{S, \ws} := F_{B^4, (S, \cA_{\ws})} \colon \cR^\infty \to \cCFL^\infty(\bK).
\]
We call $\ve{t}^\infty_{S, \ws}$ and $\ve{t}^\infty_{S, \zs}$ the \emph{extremal principal invariants}
of the surface~$S$.
\end{define}

Both $\ve{t}^\infty_{S, \ws}$ and $\ve{t}^\infty_{S, \zs}$ are filtered, $\cR^\infty$-equivariant
chain map that are well-defined up to filtered, $\cR^\infty$-equivariant chain homotopy.
Furthermore, both of them induce isomorphisms on homology by \cite[Theorem~9.9]{ZemAbsoluteGradings}.
By \cite[Theorem~1.4]{ZemAbsoluteGradings}, the map $\ve{t}_{S, \ws}^\infty$
decreases the Alexander grading by $g(S)$, while $\ve{t}_{S, \zs}^\infty$ increases it by $g(S)$.
When $D \in \cD(K)$ is a slice disk for $K$, then $\cA_{\zs} = \cA_{\ws}$,
and we denote $\ve{t}^\infty_{D, \zs} = \ve{t}^\infty_{D, \ws}$ by $\ve{t}^\infty_D$.

\begin{figure}
	\centering
\begingroup%
  \makeatletter%
  \providecommand\color[2][]{%
    \errmessage{(Inkscape) Color is used for the text in Inkscape, but the package 'color.sty' is not loaded}%
    \renewcommand\color[2][]{}%
  }%
  \providecommand\transparent[1]{%
    \errmessage{(Inkscape) Transparency is used (non-zero) for the text in Inkscape, but the package 'transparent.sty' is not loaded}%
    \renewcommand\transparent[1]{}%
  }%
  \providecommand\rotatebox[2]{#2}%
  \newcommand*\fsize{\dimexpr\f@size pt\relax}%
  \newcommand*\lineheight[1]{\fontsize{\fsize}{#1\fsize}\selectfont}%
  \ifx\svgwidth\undefined%
    \setlength{\unitlength}{119.93686037bp}%
    \ifx\svgscale\undefined%
      \relax%
    \else%
      \setlength{\unitlength}{\unitlength * \real{\svgscale}}%
    \fi%
  \else%
    \setlength{\unitlength}{\svgwidth}%
  \fi%
  \global\let\svgwidth\undefined%
  \global\let\svgscale\undefined%
  \makeatother%
  \begin{picture}(1,0.67895642)%
    \lineheight{1}%
    \setlength\tabcolsep{0pt}%
    \put(0,0){\includegraphics[width=\unitlength,page=1]{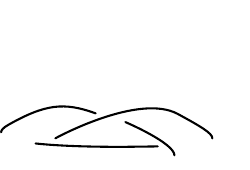}}%
    \put(0.89452391,0.05727341){\color[rgb]{0,0,0}\makebox(0,0)[lt]{\lineheight{1.25}\smash{\begin{tabular}[t]{l}$K$\end{tabular}}}}%
    \put(0.84449826,0.4056707){\color[rgb]{0,0,0}\makebox(0,0)[lt]{\lineheight{1.25}\smash{\begin{tabular}[t]{l}$\Sigma$\end{tabular}}}}%
    \put(0,0){\includegraphics[width=\unitlength,page=2]{fig30.pdf}}%
  \end{picture}%
\endgroup%

	\caption{The dividing set $\cA_{\zs}$ used to define the map $\ve{t}_{\Sigma,\zs}^\infty$.}
    \label{fig::30}
\end{figure}

Although the chain complexes $\cCFL^\infty(\bK)$ and $\cCFL^-(\bK)$ contain equivalent information,
it is convenient to also define maps
\[
\ve{t}^-_{S, \ws} \text{, } \ve{t}^-_{S, \zs} \colon \cR^-\to \cCFL^-(\bK)
\]
as the cobordism maps on the full minus complexes. For a slice disk $D$, we again have
$\ve{t}^-_{D, \ws} = \ve{t}^-_{D, \zs}$, which we denote by $\ve{t}^-_D$.
By \cite{JZContactHandles}*{Theorem~1.4}, when we set $U = V = 0$, the map $\ve{t}_D^-$,
defined using the maps from \cite{ZemCFLTQFT}, becomes $t_D$, defined in \cite{JMConcordance}.

We note that the elements $[\ve{t}_{S,\ws}^-(1)]$ and $[\ve{t}_{S,\zs}^-(1)]$
in the homology group $\cHFL^-(\bK):=H_*(\cCFL^-(\bK))$ contain exactly the
same information as the $\cR^-$-equivariant, $\Z\oplus \Z$-filtered chain
homotopy types of the maps
\[
\ve{t}_{S,\ws}^-,\, \ve{t}_{S,\zs}^-\colon \cR^-\to \cCFL^-(\bK),
\]
since two filtered, equivariant maps from $\cR^-$ to $\cCFL^-(\bK)$ are
filtered, equivariantly chain homotopic if and only if their values on $1\in
\cR^-$ differ by $\d \eta$ for some  $\eta \in \cCFL^-(\bK)$. Nonetheless,
we will usually not view $\ve{t}_{S,\ws}^-$ and $\ve{t}_{S,\zs}^-$ as elements of
$\cHFL^-(\bK)$ because, on its own, the group $\cHFL^-(\bK)$ is not sufficient to define our invariants.

We now describe some basic properties of the invariants $\ve{t}_{S,\ws}^\infty$ and $\ve{t}_{S,\zs}^\infty$.
Let $r$ denote the rolling automorphism of $(\Sphere^3, \bK)$ that consists of a Dehn twist about $K$;
see Definition~\ref{def:rolling}. Then
\begin{equation}
r_* \circ \ve{t}_{S, \zs}^\infty \simeq \ve{t}_{S, \zs}^\infty,
\label{eq:SarkarmaptSz}
\end{equation}
and similarly for $\ve{t}_{S,\ws}^\infty$,
as the corresponding dividing sets differ by a Dehn twist about $\d S$, and are hence isotopic.

\begin{define}
Let $C$ and $C'$ be free, $\Z \oplus \Z$-filtered chain complexes over $\cR^\infty$.
We say that a map $f \colon C \to C'$ is \emph{skew-equivariant} and \emph{skew-filtered} if
\[
f \circ U = V \circ f, \quad f \circ V = U \circ f, \quad \text{and} \quad f(\cG_{i,j}(C)) \subset \cG_{j,i}(C').
\]
Given skew-equivariant and skew-filtered chain maps $f \text{, } g \colon C \to C'$,
we say that they are \emph{skew-equivariant and skew-filtered chain homotopic}, and write $f \eqsim g$,
if they are chain homotopic through a skew-equivariant and skew-filtered chain homotopy.
\end{define}

There is a skew-equivariant and skew-filtered homotopy automorphism
\[
\iota_K \colon \cCFL^\infty(\bK) \to \cCFL^\infty(\bK),
\]
defined as the composition of a tautological conjugation automorphism of
$\cCFL^\infty(\bK)$, and the map induced by a half Dehn twist about $K$
that switches $w$ and $z$; see \cite{HMInvolutive}*{Section~6.1}.
The map $\iota_K$ satisfies $\iota_K^2 \simeq r_*$.

Using the conjugation formula for the link cobordism maps \cite{ZemConnectedSums}*{Theorem~1.3},
as well as the same manipulation of dividing sets as in equation~\eqref{eq:SarkarmaptSz}, one obtains
\begin{equation}
\ve{t}_{S ,\ws}^\infty \circ \iota_{\cR^\infty} \eqsim \iota_K \circ \ve{t}_{S,\zs}^\infty,
\label{eq:conjugationandtSigmaw}
\end{equation}
where $\iota_{\cR^\infty} \colon \cR^\infty \to \cR^\infty$ denotes the unique
involution that switches $U$ and $V$.
Using equation~\eqref{eq:SarkarmaptSz}, we can rewrite equation~\eqref{eq:conjugationandtSigmaw} as
\begin{equation}
\iota_K \circ \ve{t}_{S, \ws}^{\infty} \circ \iota_{\cR^\infty} \simeq \ve{t}_{S, \zs}^{\infty}.
\label{eq:equationiotaupsilon}
\end{equation}
Hence, together with the conjugation automorphism, $\ve{t}_{S,\ws}^\infty$ and $\ve{t}_{S,\zs}^{\infty}$
provide essentially equivalent information.

In the following subsections, we introduce several invariants of a pair of surfaces $S$, $S' \in \Surf(K)$,
in order to give lower bounds on their stabilization distance. These are all derived from the principal invariants
for $S$ and $S'$. Furthermore, our invariants $\tau$, $\nu$, $V_k$, and $\Upsilon$
are constructed as algebraic analogues of the homonymous knot invariants from knot Floer homology.
Hence, we shall sometimes call these \emph{secondary versions} of their knot invariant counterparts.

\subsection{The tau invariant}\label{subsec:taudef}

Let $\bK = (K, w, z)$ be a doubly-based knot in $\Sphere^3$.
We now describe a map
\[
\tau \colon \Surf(K) \times \Surf(K) \to \Z^{\ge 0}.
\]
It is a secondary version of the concordance invariant defined by Ozsv\'{a}th and Szab\'{o} \cite{OS4ballgenus}.
Let $(\S,\as,\bs,w,z)$ be a doubly-pointed diagram representing $\bK$. We will define
the invariant $\tau(S, S')$ for $S$, $S' \in \Surf(K)$ in terms of the Alexander filtered
complex $\hCFKfz(\bK)$ described in Section~\ref{sec:variations}.

Following Section~\ref{sec:invariant}, a surface $S \in \Surf_g(K)$ induces a chain map
$\hat{t}_{S, \zs} \colon \bF_2 \to \hCFKfz(\bK)$ whose  image is contained in $\hCFKfz_g(\bK)$. We recall that for sufficiently large $n$, we have
\begin{equation}
H_*(\hCFKfz_n(\bK))\iso \HFh(S^3)=\bF.\label{eq:large-n-iso-F}
\end{equation} Additionally, it follows from the reduction theorem \cite{ZemCFLTQFT}*{Theorem~C} that with respect to the isomorphism in equation~\eqref{eq:large-n-iso-F}, the element $\hat{t}_{S,\ve{z}}(1)$ represents $1\in \bF$. We make the following definition: 

\begin{define}
Let $\bK = (K, w, z)$ be a doubly-based knot in $\Sphere^3$.
Given surfaces $S \in \Surf_g(K)$ and $S' \in \Surf_{g'}(K)$, we define the invariant
\[
\tau(S, S') = \min \left\{ n \ge \max\{g, g'\} : \left[\hat{t}_{S, \zs}(1) \right] =
\left[\hat{t}_{S', \zs}(1)\right] \text{ in } H_*\left(\hCFKfz_n(\bK)\right) \right\}.
\]
\end{define}

It is straightforward to see that $\tau(S, S')$ is independent of the choice of basepoints on $K$, and is furthermore a finite integer.

\begin{lem}
  Let $S_1$, $S_1'$, $S_2$, $S_2' \in \Surf(K)$ be surfaces such that
  $[S_1] = [S_2]$ and $[S_1'] = [S_2']$ in $\Surf(K)/\{\text{2-knots}\}$.
  Then
  \[
  \tau(S_1, S_1') = \tau(S_2, S_2').
  \]
\end{lem}

\begin{proof}
This follows from the observation that $\hat{t}_{S, \zs}$
is unchanged, up to filtered chain homotopy,
if we take the connected sum of $S$ with a 2-knot,
which can be shown by adapting the proof of \cite[Lemma~4.2]{SliceDisks}.
\end{proof}

We recall that the concordance invariant $\tau(K)$ may be
computed from the $\bF_2[V]$-module
\[
\CFK^-_{U=0}(\bK) := \cCFL^-(\bK) \otimes_{\cR^-} \cR^-/(U),
\]
obtained from $\cCFL^-(\bK)$ by setting $U = 0$; see Ozsv\'ath--Szab\'o--Thurston \cite{OSTLegendrian}*{Lemma~A.2}. (Note that Ozsv\'{a}th, Szab\'{o}, and Thurston use the $V=0$ version of knot Floer homology. Using the conjugation symmetry, these are equivalent perspectives).

 Analogously, we can reformulate $\tau(S, S')$ in terms of $\CFK^-_{U=0}(\bK)$. Let us write $t_{S, \zs}^-$ and $t_{S', \zs}^-$ for the maps from $\bF_2[V]$ to $\CFK_{U=0}^-(\bK)$
 induced by $\ve{t}_{S, \zs}^-$ and $\ve{t}_{S', \zs}^-$, respectively.

\begin{lem}\label{lem:slightreformulationoftau}
Let $\bK = (K, w, z)$ be a doubly-based knot in $\Sphere^3$.
If $S \in \Surf_g(K)$ and $S' \in \Surf_{g'}(K)$, then
\[
\tau(S, S') = \min\{\, n \ge \max\{g, g'\} : V^{n - g} \cdot [t_{S, \zs}^-(1)] =
V^{n- g'} \cdot [t_{S', \zs}^-(1)] \text{ in } \HFK^-_{U=0}(\bK) \,\}.
\]
\end{lem}

\begin{proof}
Let us write $\zeta(S, S')$ for the right-hand side. If $n = \zeta(S, S')$, then
\[
V^{n - g} \cdot t_{S, \zs}^-(1) + V^{n - g'} \cdot t_{S', \zs}^-(1) = \d x
\]
for some $x \in \CFK_{U=0}^-(\bK)$. Note that $\d$ preserves the Alexander
grading, $V$ increases it by one, $A(t_{S, \zs}^-(1)) = g$,
and $A(t_{S', \zs}^-(1)) = g'$. Hence, we can assume that $A(x) = n$.
Consequently, using the identification
\[
\hCFKfz(\bK) \iso \CFK_{U=0}^-(\bK) \otimes_{\bF_2[V]} \bF_2[V]/(V-1),
\]
we have $x \otimes 1 \in \hCFKfz_n(\bK)$, and
\[
\hat{t}_{S, \zs}(1) + \hat{t}_{S', \zs}(1) = \d (x \otimes 1) \in \hCFKfz_n(\bK).
\]
It follows that
\[
\tau(S, S') \le \zeta(S, S').
\]

Conversely, suppose $n = \tau(S, S')$. Then
$\hat{t}_{S, \zs}(1) + \hat{t}_{S', \zs}(1) = \d x$ for some $x \in \hCFKfz_n(\bK)$. We write $x$ as a sum of intersection points $x=\sum_{i=1}^n \xs_i$. We define an element $\tilde x\in \CFK_{U=0}^-(\bK)$ in Alexander grading $n$ via the formula $\tilde x=\sum_{i=1}^n V^{n-A(\xs_i)} \xs_i$. The element $\tilde x$ satisfies 
\[
V^{n-g} \cdot t_{S, \zs}^-(1) + V^{n-g'} \cdot t_{S', \zs}^-(1) = \d \tilde{x}.
\]
Compare Remark~\ref{rem:equivalence-filtered-versus-polynomial}.
 It follows that
\[
\zeta(S, S') \le \tau(S, S'),
\]
which concludes the proof.
\end{proof}

Let $K_1$ and $K_2$ be knots in $\Sphere^3$.
Given surfaces $S_1 \in \Surf(K_1)$ and $S_2 \in \Surf(K_2)$, their boundary
connected sum $S_1 \natural S_2$ is an element of $\Surf(K_1 \# K_2)$.

\begin{prop}
Let $K_1$ and $K_2$ be knots in $\Sphere^3$.
If $S_1$, $S_1' \in \Surf(K_1)$ and $S_2$, $S_2' \in \Surf(K_2)$
are surfaces of genera $g_1$, $g_1'$, $g_2$, and $g_2'$, respectively, then
\[
\tau(S_1 \natural S_2, S_1' \natural S_2') \le
\max \left\{\tau(S_1, S_1') + \max\{g_2, g_2'\},
\tau(S_2, S_2') + \max\{g_1, g_1'\} \right\}.
\]
Furthermore, when $g_1 = g_1' = g_2 = g_2' = 0$, then equality holds.
\end{prop}

\begin{proof}
Using the connected sum formula \cite[Proposition~5.1]{ZemConnectedSums}, there is chain homotopy equivalence
\[
F \colon \CFK_{U=0}^-(\bK_1) \otimes_{\bF_2[V]} \CFK_{U=0}^-(\bK_2) \to \CFK_{U=0}^-(\bK_1 \# \bK_2),
\]
where the decoration on $\bK_1 \# \bK_2$ is the decoration of $\bK_1$.
Furthermore, the map $F$ can be taken to be the
link cobordism map for a 1-handle cobordism containing a band. By the functoriality
of the link cobordism maps,
\[
F \circ(t_{S_1, \zs}^- \otimes t_{S_2, \zs}^-) \simeq t_{S_1 \natural S_2, \zs}^-,
\]
and similarly for $S_1'$ and $S_2'$.

By the K\"{u}nneth theorem for tensor products over a PID, there is a short exact sequence
\[
\begin{split}
0 \to \HFK^-_{U=0}(\bK_1) \otimes_{\bF_2[V]} \HFK^-_{U=0}(\bK_2) \xrightarrow{G} \HFK^-_{U=0}(\bK_1 \# \bK_2) \to \\
\Tor_{\bF_2[V]}^1 (\HFK^-(\bK_1), \HFK^-_{U=0}(\bK_2)) \to 0,
\end{split}
\]
where $G$ is the composition of the natural map
\[
\HFK^-_{U=0}(\bK_1) \otimes_{\bF_2[V]} \HFK^-_{U=0}(\bK_2) \to H_*(\CFK_{U=0}^-(\bK_1) \otimes_{\bF_2[V]} \CFK_{U=0}^-(\bK_2))
\]
and $F_*$. The map $G$ sends $[t_{S_1, \zs}^-(1)] \otimes [t_{S_2, \zs}^-(1)]$ to $[t_{S_1 \natural S_2, \zs}^-(1)]$.

For $i \in \{1, 2\}$, the $\bF_2[V]$-module $\HFK^-_{U=0}(\bK_i)$
splits (non-canonically) as $\bF_2[V] \oplus T_i$, where $T_i$ is
a torsion $\bF_2[V]$-module, and
$t_{S_i, \zs}^-(1) = V^{g_i} \oplus s_i$ and $t_{S_i', \zs}^-(1) = V^{g_i'} \oplus s_i'$ for some $s_i$, $s_i' \in T_i$.
Let
\begin{equation}\label{eqn:def-n}
n = \max \left\{\tau(S_1, S_1') + \max\{g_2, g_2'\},
\tau(S_2, S_2') + \max\{g_1, g_1'\} \right\}.
\end{equation}
Then we claim that
\begin{equation}\label{eqn:sumineq}
V^{n - g_1 - g_2} \cdot ([t_{S_1, \zs}^-(1)] \otimes [t_{S_2, \zs}^-(1)]) =
V^{n - g_1' - g_2'} \cdot ([t_{S_1', \zs}^-(1)] \otimes [t_{S_2', \zs}^-(1)])
\end{equation}
as elements of $\HFK^-(\bK_1) \otimes \HFK^-(\bK_2)$. This is equivalent to
\[
\begin{split}
&V^{n - g_1 - g_2}((V^{g_1} \otimes V^{g_2}) \oplus (s_1 \otimes V^{g_2}) \oplus (V^{g_1} \otimes s_2) \oplus (s_1 \otimes s_2)) = \\
&V^{n - g_1' - g_2'}((V^{g_1'} \otimes V^{g_2'}) \oplus (s_1' \otimes V^{g_2'}) \oplus (V^{g_1'} \otimes s_2')
\oplus (s_1' \otimes s_2')).
\end{split}
\]
For $i \in \{1, 2\}$ and $k \ge \tau(S_i, S_i')$,
by Lemma~\ref{lem:slightreformulationoftau}, we have
\[
V^k \oplus (V^{k-g_i} s_i) = V^{k - g_i} \cdot [t_{S_i, \zs}^-(1)] =
V^{k - g_i'} \cdot [t_{S_i', \zs}^-(1)] = V^k \oplus (V^{k-g_i'} s_i').
\]
Together with equation~\eqref{eqn:def-n},
this implies that $V^{n-g_i} s_i = V^{n-g_i'} s_i'$ and
\[
V^{n - g_1 - g_2} (s_1 \otimes s_2) = V^{n - g_1' - g_2'} (s_1' \otimes s_2'),
\]
so equation~\eqref{eqn:sumineq} holds.
The result follows by applying $G$ to equation~\eqref{eqn:sumineq},
and invoking Lemma~\ref{lem:slightreformulationoftau}.

When $g_1 = g_1' = g_2 = g_2' = 0$, then we choose
\[
n = \tau(S_1 \natural S_2, S_1' \natural S_2').
\]
By Lemma~\ref{lem:slightreformulationoftau},
$n$ satisfies equation~\eqref{eqn:sumineq}, which implies
that $V^n s_i = V^n s_i'$ for $i \in \{1, 2\}$, and hence
\[
V^n \cdot [t^-_{S_i, \zs}] = V^n \cdot [t^-_{S_i', \zs}].
\]
So $n \ge \max\{\tau(S_1, S_1'), \tau(S_2, S_2')\}$, and equality holds, as claimed.
\end{proof}

\subsection{An infinitesimal refinement of tau}\label{sec:tau+}
In this section, we describe a refinement of $\tau(S, S')$,
inspired by work of Ozsv\'{a}th--Szab\'{o}~\cite{OSRationalSurgeries},
Hom--Wu~\cite{HomWuNu+}, and Hom~\cite{HomEpsilon},

Let $\bar{\Z}$ denote $\Z \cup \{-\infty,\infty\}$,
and write $\bar{\Z}^{\le 0} = [-\infty, 0] \cap \bar{\Z}$.
Given a knot $K$ in $\Sphere^3$, we will define a symmetric map
\[
\tau^+ \colon \Surf(K) \times \Surf(K) \to \N \times \bar{\Z}^{\le 0}.
\]
The invariant $\tau^+(S, S')$ takes the form
\[
\tau^+(S, S') := (\tau(S, S'), \tau'(S, S')),
\]
where $\tau(S, S')$ is the integer defined in Section~\ref{subsec:taudef},
and $\tau'(S, S')$ is an element of $\bar{\Z}^{\le 0}$ that we define shortly.
We will think of $\tau'$ as a second-order version of $\tau$, or an infinitesimal refinement.

To define $\tau'$, we introduce some notation. If $(i,j) \in \Z \oplus \Z$, let
\[
R_{i,j} := \{\, (m,n) \in \Z \oplus \Z : m \ge i \text{, } n \ge j \,\}.
\]
If $\cS \subset \Z \oplus \Z$, let $H(\cS)$ denote
the \emph{filtered hull} of $\cS$; i.e.,
\[
H(\cS) = \bigcup_{(i,j) \in \cS} R_{i,j}.
\]
We say that $\cS$ is a \emph{filtered} shape if
\[
\cS = H(\cS).
\]

Let $w$ and $z$ be basepoints on $K$, and write $\bK = (K, w, z)$.
If $\cS \subset \Z \oplus \Z$, let $C(\bK,\cS)$ denote the subspace of $\CFK^\infty(\bK)$
generated over $\bF_2$ by monomials $U^i V^j \cdot \xs$ with $A(\xs) + j - i = 0$
and $(i,j) \in \cS$. If $\cS$ is a filtered shape, then $C(\bK,\cS)$ is a subcomplex of $\CFK^\infty(\bK)$.

More generally, we say $\cS \subset \Z \times \Z$ is a \emph{sub-quotient} shape if, whenever
$(i,j)$, $(m,n) \in \cS$ with $i \le m$ and $j \le n$,
then the entire rectangle spanned by the points $(i,j)$ and $(m,n)$ is
contained in~$\cS$.
If $\cS$ is a sub-quotient shape, then $C(\bK,\cS)$ is in fact a sub-quotient
complex of $\CFK^\infty(\bK)$; i.e., there are subcomplexes
$C_{\text{in}}(\bK, \cS) \subset C_{\text{out}}(\bK, \cS) \subset \CFK^\infty(\bK)$ such that
$C_{\text{out}}(\bK, \cS)/C_{\text{in}}(\bK, \cS)$
is chain isomorphic to $C(\bK, \cS)$. Indeed, the two sub-complexes of $\CFK^\infty(\bK)$ are
\[
\begin{split}
C_{\text{out}}(\bK, \cS) &:= C\left(\bK, H(\cS)\right) \text{, and}\\
C_{\text{in}}(\bK, \cS) &:= C\left(\bK, H(\cS) \setminus \cS\right).
\end{split}
\]
We note that, if $\cS \subset \Z \oplus \Z$ is an arbitrary subset, then its filtered hull
$H(\cS)$ is automatically a filtered shape. It is an easy exercise to show that,
if $\cS$ is a sub-quotient shape, then $H(\cS) \setminus \cS$ is filtered. Hence
$C_{\text{in}}(\bK, \cS)$ and $C_{\text{out}}(\bK, \cS)$ are both subcomplexes of $\CFK^\infty(\bK)$.

We note that the map $\ve{t}^-_{S,\zs}$ naturally has image in the $g(S)$ Alexander graded subspace of $\cCFL^-(\bK)$.
Hence, there is a well-defined map
\[
V^{-g(S)}\cdot \ve{t}_{S,\zs}^-\colon \bF_2[\hat{U}]\to C(\bK, R_{0,-g(S)}).
\]
Furthermore, if $\cS$ is a sub-quotient shape of $\Z\oplus \Z$
such that $R_{0,-g(S)}\subset H(\cS)$, then there is a natural map
\[
q\colon C(\bK,R_{0,-g(S)})\to C(\bK,\cS),
\]
which is the composition of the inclusion map $C(\bK,R_{0,-g(S)}) \to C_{\text{out}}(\bK,\cS)$,
followed by the quotient map $C_{\text{out}}(\bK,\cS) \to C_{\text{out}}(\bK,\cS)/C_{\text{in}}(\bK,\cS)\iso C(\bK,\cS)$.
In particular,  $[V^{-g(S)}\cdot \ve{t}_{S,\zs}^-(1)]$ determines a well-defined element of $H_*(C(\bK,\cS))$.

Define
\[
\I{n} := \{0\} \times ([-n, \infty) \cap \Z),
\]
 which is a sub-quotient shape.
Noting that $C(\bK,\I{n})$ is chain isomorphic to $\hCFKfz_n(\bK)$, we obtain the following:

\begin{lem}\label{lem:tausubquotientshapes}
Given a doubly-based knot $\bK = (K, w, z)$ and surfaces $S$, $S' \in \Surf(K)$, we have
\[
\tau(S,S') = \min\left\{ n \ge \max\{g(S),g(S')\} :
[V^{-g(S)}\cdot \ve{t}^-_{S, \zs}(1)] = [V^{-g(S')} \cdot \ve{t}^{-}_{S', \zs}(1)]\text{ in } H_*(C(\bK, \I{n}))\right\}.
\]
\end{lem}

We are now ready to define the refinement $\tau'(S, S')$.

\begin{define}
For $n$, $m \in \Z$, let
\[
\L{m}{n} := \I{n} \cup ([0,m] \cap \Z) \times \{-n\}.
\]
This is an L-shaped subset of $\Z \times \Z$, and hence a sub-quotient shape.

Let $\bK = (K, w, z)$ be a doubly-pointed knot in $\Sphere^3$,
let $S$, $S' \in \Surf(K)$, and write $\tau = \tau(S, S')$.
Then we define
\[
\tau'(S, S') = -\sup \{\, m \in \Z : [V^{-g(S)}\cdot\ve{t}^{-}_{S, \zs}(1) - V^{-g(S')}\cdot \ve{t}^-_{S', \zs}(1)] = 0 \in
H_* \left(C(\bK, \L{m}{\tau} \right) \,\}.
\]
\end{define}

Note that $\L{0}{\tau} = \I{\tau}$, and $[V^{-g(S)}\cdot \ve{t}^{-}_{S, \zs}(1) - V^{-g(S')}\cdot \ve{t}^-_{S', \zs}(1)] = 0$
in $H_*(C(\bK, \I{\tau}))$ by the definition of $\tau$. It follows that $\tau'(S, S') \le 0$.
However, if $x \in C(\bK, \I{\tau})$ satisfies
\[
\d x = V^{-g(S)}\cdot \ve{t}^{-}_{S, \zs}(1) - V^{-g(S')} \cdot \ve{t}^-_{S', \zs}(1)
\]
in $C(\bK, \I{\tau})$, then $\d x$ might have some nonzero terms in $([0,m] \cap \Z) \times \{-\tau\}$ for $m > 0$.
Hence $[V^{-g(S)}\cdot \ve{t}^{-}_{S, \zs}(1) - V^{-g(S')}\cdot \ve{t}^-_{S', \zs}(1)]$ might not be zero
in $H_* \left(C(\bK, \L{m}{\tau})\right)$ for $m > 0$,
and this is what $\tau'(S, S')$ measures.

The invariant $\tau^+(S,S')$ was inspired by the concordance invariant $\nu(K)$,
due to Ozsv\'{a}th and Szab\'{o} \cite{OSRationalSurgeries}*{Definition~9.1},
which gives an improved 4-ball genus bound over $\tau(K)$ by at most 1. We now extract
an analogue of $\nu$ from $\tau^+$ for pairs of surfaces in $\Surf(K)$,
though there is some information lost when doing this as
$\tau'$ can take any value in $\bar{\Z}^{\le 0}$.

\begin{define}
Let $K$ be a knot in $\Sphere^3$ and $S$, $S' \in \Surf(K)$. Then let
\[
\nu(S, S') =
\begin{cases}
\tau(S, S') &\text{if } \tau'(S, S') = -\infty, \\
\tau(S, S') + 1 &\text{otherwise.}
\end{cases}
\]
\end{define}

We will see that $\nu$ gives a lower bound on the stabilization and double point
distances in Proposition~\ref{prop:tau+}.

\subsection{A sequence of local $h$-invariants}\label{sec:V}

Let $K$ be a knot in $\Sphere^3$ and $S$, $S' \in \Surf_g(K)$.
Modeled on the invariants $V_k(K)$ of large surgeries from knot Floer
homology, also referred to as Rasmussen's \emph{local $h$-invariants}
\cite{RasmussenKnots},  we describe a sequence of integer invariant $V_k(S,S')$
for $k\ge g$, such that
\[
V_{g}(S,S')\ge V_{g+1}(S,S')\ge \cdots\ge 0,
\]
and such that $V_k(S, S') = 0$ for $k$ sufficiently large.

\begin{define}
Let $\bK = (K, w, z)$ be a doubly-based knot in $\Sphere^3$, and let $S$, $S' \in \Surf_g(K)$.
To define $V_k(S,S')$, we consider the subcomplex
\[
A_k^-(\bK) := C(\bK,R_{0, -k})
\]
of $\CFK^\infty(\bK)$.
We think of the complexes $A_k^-(\bK)$ as modules over the ring $\bF_2[\hat{U}]$, where $\hat{U} = UV$.

The map $\ve{t}_{S,\zs}^-$ increases the Alexander grading by $g(S)$. If $k\ge g(S)$,
then $V^{-g(S)}\cdot \ve{t}_{S,\zs}^-(1)$ has Alexander grading $0$, and determines a
well-defined element of $H_*(A_k^-(\bK))$. We define the invariant
\[
V_k(S,S') := \min \left\{\, n \in \N : \hat{U}^n \cdot [V^{-g}\cdot \ve{t}^-_{S,\zs}(1)] =
\hat{U}^n \cdot [V^{-g}\cdot \ve{t}^-_{S',\zs}(1)] \text{ in } H_*(A_k^-(\bK)) \,\right\}.
\]
\end{define}

\begin{rem}
The above definition of $V_k$ can be easily adapted to surfaces of different genera;
however, we specialize to the case when $g(S) = g(S')$ since our topological applications for $V_k$
only hold when this is the case.
\end{rem}

We now show that the invariants $V_k$ of pairs of surfaces in $\Surf_g(K)$ satisfy many of
the same properties as Rasmussen's local $h$-invariants. The reader should
compare the following to \cite{RasmussenKnots}*{Proposition~7.6}:

\begin{lem}\label{lem:monotonicity}
If $k \ge g$, then
\[
V_{k}(S,S') \ge V_{k+1}(S,S') \ge V_{k}(S,S')-1.
\]
\end{lem}

\begin{proof}
There is a natural, grading-preserving inclusion of chain complexes
\[
i_{k} \colon A_{k}^-(\bK) \hookrightarrow A_{k+1}^-(\bK),
\]
which becomes an isomorphism on homology after we invert $\hat{U}$, and satisfies
\[
(i_{k})_*([ V^{-g} \cdot \ve{t}_{S,\zs}^-(1)])= [V^{-g}\cdot  \ve{t}_{S,\zs}^-(1)].
\]
Hence
\[
V_{k+1}(S,S') \le V_{k}(S,S').
\]

Multiplication by $\hat{U}$ induces a $-2$-graded inclusion $A_{k+1}^-(\bK) \hookrightarrow A_k^-(\bK)$ of chain complexes,
which becomes an isomorphism on homology after we invert $\hat{U}$. The map sends
$[V^{-g}\cdot \ve{t}_{S,\zs}(1)]\in H_*(A_{k+1}^-(\bK))$
 to $\hat{U} \cdot [V^{-g}\cdot \ve{t}_{S,\zs}^-(1)]\in H_*(A_k^-(\bK))$, and similarly for $S'$. Hence, if
\[
\begin{split}
\hat{U}^n \cdot [V^{-g}\cdot \ve{t}_{S,\zs}^-(1)] &= \hat{U}^n \cdot [V^{-g}\cdot \ve{t}_{S',\zs}^-(1)]\text{ in } H_*(A_{k+1}^-(\bK)) \text{, then} \\
\hat{U}^{n+1} \cdot [V^{-g}\cdot \ve{t}_{S,\zs}^-(1)] &= \hat{U}^{n+1} \cdot [V^{-g}\cdot \ve{t}_{S',\zs}^-(1)]\text{ in }  H_*(A_{k}^-(\bK)).
\end{split}
\]
We conclude that $V_k(S,S') \le V_{k+1}(S,S')+1$.
\end{proof}

The reader should compare the following to \cite{RasmussenKnots}*{Proposition~7.7}:

\begin{lem}
If $S$, $S'\in \Surf_g(K)$ and $g \le k < \tau(S,S')$, then $0 < V_{k}(S,S')$.
\end{lem}

\begin{proof}
By Lemma~\ref{lem:slightreformulationoftau},
we can describe $\tau(S, S')$ as the minimal $n$
such that $V^{n-g} \cdot [t_{S,\zs}^-(1)] = V^{n-g}\cdot  [t_{S',\zs}^-(1)]$ in
\[
\HFK^-_{U=0}(\bK) = H_*\left( \cCFL^-(\bK) \otimes \bF_2[U,V]/(U) \right).
\]
Note that multiplication by $V^k$ determines an inclusion of chain complexes
\[
A_k^-(\bK) \hookrightarrow \cCFL^-(\bK),
\]
which we compose with the natural map $\cCFL^-(\bK) \to \CFK_{U=0}^-(\bK)$ given by
$\xs \mapsto \xs \otimes 1$, corresponding to setting $U = 0$.
The induced map on homology sends $[V^{-g}\cdot \ve{t}_{S,\zs}^-(1)]$, $[V^{-g}\cdot \ve{t}_{S',\zs}^-(1)] \in H_*(A_k^-(\bK))$ to
$V^{k-g} \cdot [t_{S,\zs}^-(1)]$, $V^{k-g} \cdot [t_{S',\zs}^-(1)] \in \HFK^-_{U=0}(\bK)$, respectively.
If $g\le k < \tau(S, S')$, then
\[
V^{k-g} \cdot [t_{S,\zs}^-(1)] \neq V^{k-g} \cdot [t_{S',\zs}^-(1)] \text{ in } \HFK^-_{U=0}(\bK),
\]
and consequently,
\[
[V^{-g}\cdot \ve{t}_{S,\zs}^-(1)] \neq [V^{-g}\cdot \ve{t}_{S',\zs}^-(1)] \text{ in } H_*(A_k^-(\bK)).
\]
Hence $V_k(S, S') > 0$, as claimed.
\end{proof}

\subsection{The upsilon invariant}\label{sec:upsilon}
Let $K$ be a knot in $\Sphere^3$, and let $S$, $S' \in \Surf(K)$.
We now describe our invariant
\[
\Upsilon_{(S, S')} \colon [0,2] \to \R^{\ge 0}.
\]
It is a secondary version of Ozsv\'{a}th, Stipsicz, and Szab\'{o}'s \cite{OSSUpsilon} invariant~$\Upsilon_K(t)$.

We recall the \emph{$t$-modified} version of knot Floer homology, described
by Ozsv\'{a}th, Stipsicz, and Szab\'{o}. Suppose that $t = \frac{m}{n}\in [0,2]$ is a
rational number with $m \in \N$ and $n \in \Z_+$ relatively prime.
We define $\tCFK^-(\bK)$ to be the free $\bF_2[v^{1/n}]$-module generated by $\bT_{\a} \cap \bT_{\b}$, where $v$ is a formal variable.
Similarly, let $\tCFK^\infty(\bK)$ be the free $\bF_2[v^{1/n}, v^{-1/n}]$-module generated
by $\bT_{\a} \cap \bT_{\b}$. The modules $\tCFK^-(\bK)$ and $\tCFK^\infty(\bK)$ are equipped with a
differential $\d$ that satisfies
\[
\d \xs := \sum_{\ys \in \bT_{\a} \cap \bT_{\b}} \sum_{\substack{\phi \in \pi_2(\xs,\ys)\\
\mu(\phi) = 1}} \# \hat{\cM}(\phi) \cdot v^{t n_{z}(\phi) + (2-t) n_{w}(\phi)} \cdot \ys
\]
for $\x \in \T_\a \cap \T_\b$.

We note that $\tCFK^-(\bK)$ and $\tCFK^\infty(\bK)$ can easily be expressed in
terms of $\cCFL^-(\bK)$ and $\cCFL^\infty(\bK)$, respectively, as we now describe.
We give $\bF_2[v^{1/n}]$ the structure of an $\bF_2[U,V]$-module, where
$U$ acts by $v^{2-t}$ and $V$ acts by $v^{t}$. With this action, we have
a canonical isomorphism
\[
\tCFK^-(\bK) \iso \cCFL^-(\bK) \otimes_{\bF_2[U,V]} \bF_2[v^{1/n}],
\]
as well as a similar isomorphism involving $\tCFK^\infty(\bK)$ and
$\cCFL^\infty(\bK)$. Note that, in particular, if $(W,\cF) \colon
(\Sphere^3, \bK_1) \to (\Sphere^3, \bK_2)$ is a decorated link cobordism, then
the map $F_{W,\cF,\frs}$ determines a $t$-modified version
\[
\tF_{W,\cF,\frs} := F_{W,\cF,\frs} \otimes \id_{\bF_2[v^{1/n}]} \colon \tCFK^-(\bK_1) \to \tCFK^-(\bK_2).
\]
Finally, we note that there is a $t$-grading on $\cCFL^-(\bK)$, defined via the formula
\[
\gr_t(\xs) := \frac{t}{2} \cdot \gr_{\zs}(\xs) + \left(1 - \frac{t}{2} \right) \cdot \gr_{\ws}(\xs).
\]
This induces a well-defined grading on $\tCFK^-(\bK)$,
for which we also write $\gr_t$. With respect to $\gr_t$, the variable $v$ is $-1$ graded.

If $S \in \Surf_{g}(K)$ and $S' \in \Surf_{g'}(K)$, the invariants
$\ve{t}_{S,\zs}^-$ and $\ve{t}_{S',\zs}^-$ admit $t$-modified versions
$\tvt_{S,\zs}^-$ and $\tvt_{S',\zs}^-$, respectively. Furthermore,
the elements $\tvt_{S,\zs}^-(1)$ and $\tvt_{S',\zs}^-(1)$ for $1 \in \bF_2[v^{1/n}]$
have $\gr_t$-grading $-t \cdot g$ and $-t \cdot g'$, respectively.

\begin{define}
For $t = \frac{m}{n} \in [0,2]$, we define
\[
\Upsilon_{(S,S')}(t) := \min \{\,s = k/n \ge \max \{t \cdot g, t \cdot g'\} :
v^{s - t \cdot g} \cdot [\tvt^-_{S,\zs}(1)] = v^{s - t \cdot g'} \cdot
[\tvt^-_{S',\zs}(1)] \in \tHFK^-(\bK)\,\}.
\]
\end{define}

There is an alternate definition of the invariant $\Upsilon_{(S,S')}(t)$, which is more amenable to computations, and is based on Livingston's description of the corresponding knot invariant \cite{LivingstonUpsilon}.
If $t\in [0,2]$, there is a filtration $\cG_{s}^t(\bK)$ of $\CFK^\infty(\bK)$ which is indexed by a parameter $s\in \R$. The set $\cG_s^{t}(\bK)$ is the $\bF_2$-module generated by monomials $U^i V^j \cdot \xs$ with $A(\xs)+j-i=0$ and
\[
t\cdot j+(2-t)\cdot i\ge -s.
\]
If $s\ge t\cdot g(S)$, then it is straightforward to see that $[V^{-g(S)}\cdot \ve{t}_{S,\zs}^-(1)]$ is a well-defined element of $H_*(\cG_s^t(\bK))$. It is not hard to adapt \cite{LivingstonUpsilon}*{Section~14.1} to establish the following:

\begin{lem}\label{lem:livingstonreformupsilon}
If $S \in \Surf_g(K)$ and $S' \in \Surf_{g'}(K)$, then
\[
\Upsilon_{(S,S')}(t) =  \min\left\{s\ge t\cdot \max\{g,g'\} :
[V^{-g} \cdot \ve{t}_{S,\zs}^-(1)] = [V^{-g'} \cdot \ve{t}_{S',\zs}^-(1)] \text{ in } H_*(\cG_s^t(\bK)) \right\}.
\]
\end{lem}

\subsection{The kappa and kappa-nought invariants}\label{sec:kappa}
If $\bK=(K,w,z)$ is a doubly-based knot in $\bS^3$, let $\CFK^-_{U=0}(\bK)$
and $\CFK^\infty_{U=0}(\bK)$ denote the small minus and infinity knot Floer complexes
described in Section~\ref{sec:variations}.

\begin{lem}
If $g(S)>0$, then
\[
[t_{S,\ws}^\infty(1)] = 0 \in \HFK^\infty_{U=0}(\bK).
\]
\end{lem}

\begin{proof}
For $n \in \Z$, let $\CFK^\infty_{U=0}(\bK)_n$ denote the subspace of $\CFK^\infty_{U=0}(\bK)$ in Alexander grading $n$.
Explicitly, the subspace $\CFK^\infty_{U=0}(\bK)_n$ is generated by monomials of the form $V^i\cdot \xs$, where
\[
A(\xs) + i = n.
\]
We define the reduction map
\[
R_w^n\colon \CFK^\infty_{U=0}(\bK)_n\to \hat{\CF}( \Sphere^3 , w ),
\]
by the formula $R_w^n(V^i \cdot \xs)=\xs$. It is straightforward to see that
$R_w^n$ is a chain map. Furthermore, since the differential on
$\CFK^\infty_{U=0}(\bK)$ preserves the Alexander grading, the map $R_w^n$ is
a chain isomorphism. Consider the chain isomorphism
\[
R_{w} := \bigoplus_{n \in \Z} R_w^n \colon \CFK^\infty_{U=0}(\bK)\to \bigoplus_{n\in \Z} \hat{\CF}(\Sphere^3,w).
\]
In particular, since $\hat{\HF}( \Sphere^3 , w )$ is supported in $\gr_{\ws}$-grading 0, it follow that $\HFK^\infty_{U=0}(\bK)$ is as well. The map $\ve{t}_{S,\ws}^\infty(1)$ has $\gr_{\ws}$-grading $-2g(S)$ by the grading formula in \cite{ZemAbsoluteGradings}*{Theorem~1.4}. It follows that
\[
[\ve{t}_{S,\ws}^\infty(1)]=0\in \HFK^\infty_{U=0}(\bK),
\]
completing the proof.
\end{proof}

\begin{define}
Let $\bK = (K, w, z)$ be a doubly-based knot in $\Sphere^3$, and let $S \in \Surf_g(K)$ for $g > 0$.
Then we let
\[
\kappa_0(S) := \min\left\{\, n \ge g : V^{n-g} \cdot [t_{S,\ws}^-(1)] = 0 \text{ in } \HFK^-_{U=0}(\bK) \,\right\}.
\]
If $S \in \Surf_0(K)$, we set $\kappa_0(S) = 0$.
\end{define}

We note that the element $\ve{t}_{S,\ws}^-(1)$ lives in Alexander grading $-g$. Hence
\[
\kappa_0(S) = g + \min\left\{\, k \in \N : [\hat{t}_{S,\ws}(1)] = 0 \text{ in } H_*\left(\hCFKfz_{k-g}(\bK)\right) \,\right\}.
\]

\begin{define}
If $g>0$ and $S$, $S'\in \Surf_g(K)$, we define the invariant
\[
\kappa(S,S'):=\min\left\{\, n \ge g : V^{n-g} \cdot [t^-_{S,\ws}(1)] =
V^{n-g} \cdot [t^-_{S',\ws}(1)] \text{ in } \HFK^-_{U=0}(\bK) \,\right\}.
\]
\end{define}

Note that
\[
\kappa(S,S') \le \max\{\kappa_0(S), \kappa_0(S')\}.
\]
We emphasize that the invariant $\tau(S,S')$ is defined
in terms of the maps $t_{S,\zs}^-$ and $t_{S',\zs}^-$,
while $\kappa(S,S')$ is defined in terms of the maps $t_{S,\ws}^-$ and $t_{S',\ws}^-$.
Also, unlike the invariant $\tau(S,S')$, the
definition of $\kappa(S,S')$ only makes sense if $g(S) = g(S') > 0$.

\subsection{Upsilon near $0$ and $2$}

Ozsv\'{a}th, Stipsicz, and Szab\'{o} \cite{OSSUpsilon}*{Proposition~1.6} proved that the knot invariant
$\Upsilon_K(t) = -\tau(K)\cdot t$ near $t=0$. In this section, we prove a
similar result for $\Upsilon_{(S,S')}(t)$.

\begin{thm}\label{thm:upsilon-slope}
Suppose that $S$, $S'\in \Surf(K)$. For all $t\in [0,2]$ sufficiently close to 0, we have
\[
\Upsilon_{(S,S')}(t)= \tau(S,S')\cdot t.
\]
For $t$ sufficiently close to $2$, we have
\[
\Upsilon_{(S,S')}(t)=\begin{cases} (\kappa_0(S)-g(S))\cdot (2-t)+ g(S)\cdot t& \text{ if } g(S)>g(S'),\\
(\kappa(S,S')-g(S))\cdot (2-t)+g(S)\cdot t& \text{ if } g(S)=g(S').
\end{cases}
\]
\end{thm}

\begin{proof}
The argument we present is an adaptation of Livingston's proof of the
analogous fact \cite{LivingstonUpsilon}*{Theorem~13.1} for the knot invariants $\Upsilon_K(t)$,
and we use the reformulation of
$\Upsilon_{(S,S')}(t)$ in terms of filtrations on $\CFK^\infty$ from
Lemma~\ref{lem:livingstonreformupsilon}. We focus on $\Upsilon_{(S,S')}(t)$ near $t=2$ when
$g(S)>g(S')$, since the other cases are straightforward adaptations of this.
Let us write $g=g(S)$ and $g'=g(S')$.

Define the following sub-quotient shapes of $\Z\oplus \Z:$
\begin{align*}
H_{-g+1}&:=\{\,(i,j): j\ge -g+1\,\},\\
T_{-g,k}&:=\{\,(i,j): i\ge k \text{, } j=-g\,\}\\
Z_{-g, k}&:= H_{-g+1} \cup T_{-g,k}.
\end{align*}

\begin{figure}[ht!]
	\centering
	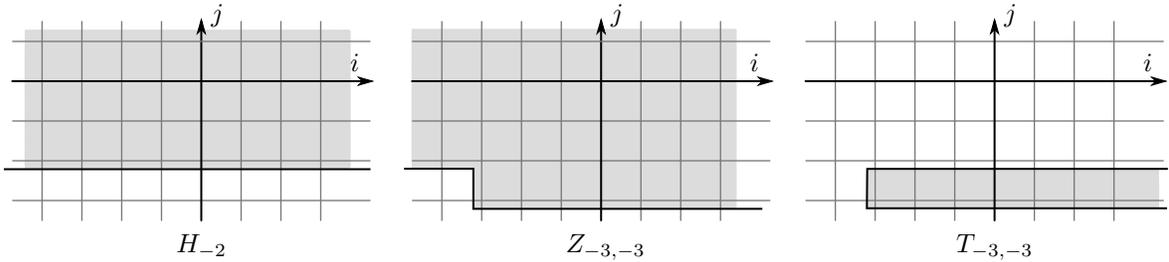
	\caption{Examples of the sub-quotient shapes $H_{-g+1}$, $Z_{-g,k}$, and $T_{-g,k}$ of $\Z \oplus \Z$.}
\label{fig::34}
\end{figure}

It is straightforward to see that, for $t$ sufficient close to $2$, the
complex $\cG_s^t(\bK)$ is always equal to $C(\bK,Z_{i,j})$ for some $i$, $j$.
Hence $[V^{-g}\cdot \ve{t}_{S,\zs}^-(1)]=[V^{-g'}\cdot \ve{t}_{S',\zs}^-]\in
H_*(\cG_{s}^t(\bK))$ if and only if
\[
C(\bK, Z_{-g,m})\subset \cG_{s}^t(\bK),
\]
where
\[
m := \max\left\{\, k\le 0: [V^{-g}\cdot \ve{t}_{S,\zs}^-(1)]=
[V^{-g'} \cdot \ve{t}^-_{S',\zs}(1)]\text{ in } H_*(C(\bK,Z_{-g,k}))\,\right\}.
\]
Consequently, an easy computation shows that
\begin{equation}
\Upsilon_{(S,S')}(t)=-m(2-t)+g t,\label{eq:Upsilonzfortnear2}
\end{equation}
for $t$ near $2$.

There is a short exact sequence of chain complexes
\begin{equation}
0\to C(\bK, H_{-g+1})\xrightarrow{i}  C(\bK, Z_{-g,k})\xrightarrow{q} C(\bK, T_{-g,k})\to 0.\label{eq:SES}
\end{equation}
Since $H_*(C(\bK, T_{-g,k}))$ is a torsion $\bF_2[\hat{U}]$-module (in fact, $\hat U$ has vanishing action), while
$H_*(C(\bK,H_{-g+1}))$ is torsion-free (in fact, isomorphic to $\bF_2[\hat U]$),
the connecting homomorphism of the
long exact sequence associated to equation~\eqref{eq:SES} vanishes.
Consequently, there is a short exact sequence
\begin{equation}
0\to H_*(C(\bK, H_{-g+1}))\xrightarrow{i} H_*(C(\bK, Z_{-g,k}))\xrightarrow{q} H_*(C(\bK, T_{-g,k}))\to 0.
\label{eq:LESsplits}
\end{equation}

Furthermore, since $H_*(C(\bK,H_{-g+1})) \iso \bF_2[\hat{U}]$, it follows that
$q$ is injective on the torsion submodule of $H_*(C(\bK, Z_{-g, k}))$.
Since $[V^{-g} \cdot \ve{t}_{S,\zs}^-(1)] + [V^{-g'} \cdot
\ve{t}_{S',\zs}^-(1)]$ is a torsion element of $H_*(C(\bK,Z_{-g,k}))$, it
follows that $[V^{-g} \cdot \ve{t}_{S,\zs}^-(1) + V^{-g'}\cdot  \ve{t}_{S',\zs}^-(1)]$
is zero in $H_*(C(\bK, Z_{-g,k}))$ if and only if its image under $q$ is zero
in $H_*(C(\bK, T_{-g,k}))$. Consequently,
\begin{equation}\label{eq:shiftbyUV^g}
m=\max\left\{k\le 0:  [V^{-g}\cdot \ve{t}^-_{S,\zs}(1)]=
[V^{-g'} \cdot \ve{t}^-_{S',\zs}(1)] \text{ in } H_*(C(\bK, T_{-g,k}))\right\}.
\end{equation}
Note that if $g>g'$, then  $[V^{-g'}\cdot \ve{t}^-_{S',\zs}]= 0$ as an element of $ H_*(C(\bK, T_{-g,k}))$,
so equation~\eqref{eq:shiftbyUV^g} implies that
\begin{equation}\label{eq:invariantforS'vanishes}
m=\max\left\{k\le 0:  [V^{-g}\cdot \ve{t}^-_{S,\zs}(1)]=0 \text{ in } H_*(C(\bK, T_{-g,k}))\right\}.
\end{equation}

Next, we note that multiplication by $\hat{U}^g$ induces a chain isomorphism
between $C(\bK, T_{-g,k})$ and $C(\bK, T_{0, k+g})$. The group
$C(\bK, T_{0,k+g})$ is the $\bF_2$-module generated by intersection points
$\xs\in \bT_{\a}\cap \bT_{\b}$ with $A(\xs)\ge k+g$. The differential on
$C(\bK, T_{0,k+g})$ counts holomorphic disks which are allowed to go over $w$,
but not $z$. This is simply the subcomplex $\hCFKfw_{k+g}(\bK)\subset \hCFKfw(\bK)$.
Applying the conjugation symmetry of knot Floer homology to
equation~\eqref{eq:invariantforS'vanishes} implies that
\begin{equation}
\begin{split}
m &= \max\left\{k\le 0 : [U^g \cdot \ve{t}^-_{S,\zs}(1)] = 0 \text{ in } H_*(C(\bK, T_{0,k+g}))\right\}\\
&= -\min\left\{k \ge 0 : [V^g \cdot \ve{t}_{S,\ws}^-(1)] = 0 \text{ in } H_*(\hCFKfz_{k-g}(\bK))\right\}\\
&= g - \kappa_0(S).
\end{split}
\label{eq:computem}
\end{equation}
Equations~\eqref{eq:Upsilonzfortnear2} and \eqref{eq:computem} together imply that, for $t$ near 2, we have
\[
\Upsilon_{(S,S')}(t)=(\kappa_0(S) - g) \cdot (2-t) + g \cdot t.
\]
Similar arguments apply when $g = g'$, and for $t$ close to 0.
\end{proof}

\subsection{Further properties of the secondary invariants}
Our secondary invariants satisfy a monotonicity condition with respect
to stacking link cobordisms.

\begin{prop}\label{prop:concordancemonotonicity}
Suppose that $(I \times \Sphere^3, S)$ is a link cobordism from $(\Sphere^3, K_0)$
to $(\Sphere^3, K_1)$, and $S_0$, $S_0' \in \Surf(K_0)$.
Let $S_1$, $S_1' \in \Surf(K_1)$ be the surfaces obtained by
stacking $S_0$ or $S_0'$ and $S$, respectively. Then
\[
\tau(S_0, S_0') + g(S) \ge \tau(S_1, S_1').
\]
When $g(S_0) = g(S_0')$, the invariant $\kappa$ satisfies an analogous inequality; furthermore,
\[
V_{k}(S_0, S_0') \ge V_{k+g(S)}(S_1, S_1')
\]
for $k \ge g(S_0)$.
Finally,
\[
\Upsilon_{(S_0, S_0')}(t) + (1 - |1-t|)\cdot g(S) \ge \Upsilon_{(S_1, S_1')}(t),
\]
for any $S_0$, $S_0' \in \Surf(K)$ and $t \in [0,2]$.
\end{prop}

\begin{proof}
Choose basepoints $w_i$ and $z_i$ on $K_i$ for $i \in \{0,1\}$,
and a decoration $\cA$ on $S$ such that $S_{\ws}$ is a strip
containing $w_0$ and $w_1$.
Using the functoriality of the link cobordism maps,
it sends $\ve{t}_{S_0, \zs}^\infty(1)$ to $\ve{t}_{S_1, \zs}^\infty(1)$
and $\ve{t}_{S_0', \zs}^\infty(1)$ to $\ve{t}_{S_1', \zs}^\infty(1)$. Furthermore, by \cite[Theorem~1.4]{ZemAbsoluteGradings}, the map $F_{I \times \Sphere^3, (S, \cA)}$
increases the Alexander grading by $g(S)$. From these two facts, all claims can be proven quickly.
\end{proof}

A concordance $\cC$ from $K_0$ to $K_1$ is called \emph{invertible}
if there is a concordance $\cC'$ from $K_1$ to $K_0$ such that $\cC' \circ \cC$
is the identity cobordism from $K_0$ to itself; see Sumners~\cite{Sumners}.

\begin{cor}
  Suppose that $K_0$ and $K_1$ are knots in $\Sphere^3$ and $S_0$, $S_0' \in \Surf(K_0)$.
  If $\cC$ is an invertible concordance from $K_0$ to $K_1$,
  let $S_1$ and $S_1'$ denote the surfaces in  $\Surf(K_1)$
  obtained by stacking $S_0$ or $S_0'$ and $\cC$, respectively. Then
  \[
  \omega(S_0, S_0') = \omega(S_1, S_1')
  \]
  for $\omega \in \{\tau, \Upsilon\}$.
  If $g(S_0) = g(S_0')$, then the same equality holds for $\omega \in  \{V_k,\kappa\}$, provided $k\ge g$.
\end{cor}

\begin{proof}
Let $\cC'$ be a left inverse of $\cC$.
We first apply Proposition~\ref{prop:concordancemonotonicity}
to the surfaces $S_0$, $S_0'$, and to the concordance $\cC$,
to obtain that $\omega(S_0, S_0') \ge \omega(S_1, S_1')$.
Note that we recover $S_0$ if we stack $S_1$ and $\cC'$,
and $S_0'$ if we stack $S_1'$ and $\cC'$.
Hence, if we apply Proposition~\ref{prop:concordancemonotonicity}
to the surfaces $S_1$, $S_1'$, and to the concordance $\cC'$,
we obtain that $\omega(S_1, S_1') \ge \omega(S_0, S_0')$.
\end{proof}

Like $\mu_{\st}$ and $\mu_{\Sing}$, our secondary invariants satisfy
the following ultrametric inequality:

\begin{prop}\label{prop:ultra}
If $K$ is a knot in $\Sphere^3$ and $S_1$, $S_2$,  $S_3 \in \Surf(K)$, then
\[
\omega(S_1, S_3) \le \max\{\, \omega(S_1, S_2), \omega(S_2, S_3) \,\}
\]
for $\omega\in \{\tau,\Upsilon\}$, where the inequality is to be
taken pointwise when $\omega = \Upsilon$.
If $g(S_1) = g(S_2) = g(S_3)$, then the
inequality holds with $\omega\in \{V_k,\kappa\}$, as well.
\end{prop}

\begin{proof}
All of the invariants are described in terms of when two distinguished
elements become equal in homology, after multiplying by some power of $V$, $\hat{U}$, or $v$.
\end{proof}

\begin{lem}
Let $K$ be a knot in $\Sphere^3$, and suppose $S_1$, $S_2$, $S_3 \in \Surf(K)$.
If we endow $\N \times \bar{\Z}^{\le 0}$ with the lexicographic ordering,
then the map $\tau^+$ satisfies the ultrametric inequality
\[
\tau^+(S_1, S_3) \le \max\{ \tau^+(S_1, S_2), \tau^+(S_2, S_3) \}.
\]
\end{lem}

\begin{proof}
For $i$, $j \in \{1,2,3\}$, let us write $\tau_{ij} = \tau(S_i, S_j)$ and $\tau'_{ij} = \tau'(S_i, S_j)$,
and let
\[
\ve{t}_i := V^{-g(S_i)} \cdot \ve{t}_{S_i, \zs}^-(1).
\]
Without loss of generality, we can suppose that $\tau_{12} \le \tau_{23}$.
By Proposition~\ref{prop:ultra}, we have $\tau_{13} \le \max \{\tau_{12}, \tau_{23}\} = \tau_{23}$.
As we have endowed $\N \times \bar{\Z}^{\le 0}$ with the lexicographic ordering,
it suffices to consider the case when $\tau_{13} = \tau_{23}$, which we denote by $\tau$,
as otherwise $\tau^+(S_1, S_3) < \tau^+(S_2, S_3)$.

Choose basepoints $w$ and $z$ on $K$, and write $\bK = (K, w, z)$.
First, assume that $\tau_{12} < \tau$. By the definition of $\tau_{12}$, there
exists $x \in C(\bK, \I{\tau_{12}})$ such that $\d x = \ve{t}_1 - \ve{t}_2$.
If we write $m := -\tau'_{23}$, by definition, there is an $y \in C(\bK, \L{m}{\tau})$ such that
$\d y = \ve{t}_2 - \ve{t}_3$. Since $\tau_{12} < \tau$ and $\d$ respects the Alexander filtration,
there is an inclusion of complexes
\[
\iota \colon C(\bK, \I{\tau_{12}}) \to C(\bK, \L{m}{\tau}).
\]
It follows that $\iota(x) + y \in C(\bK, \L{m}{\tau})$ satisfies
$\d(\iota(x) + y) = \ve{t}_1 - \ve{t}_3$. As $\tau_{13} = \tau$,
it follows that $-\tau'_{13} \ge m$, hence $\tau'_{13} \le \tau'_{23}$
and $\tau^+(S_1, S_3) \le \tau^+(S_2, S_3)$.

Now suppose that $\tau_{12} = \tau$. If we write $m' := -\tau'_{12}$ and $m := -\tau'_{23}$,
then there exist $x \in C(\bK, \L{m'}{\tau})$ and $y \in C(\bK, \L{m}{\tau})$ such that
$\d x = \ve{t}_1 - \ve{t}_2$ and $\d y = \ve{t}_2 - \ve{t}_3$.
Without loss of generality, we can assume that $m' \ge m$. Then there is a natural projection
\[
\pi \colon C(\bK, \L{m'}{\tau}) \to C(\bK, \L{m}{\tau})
\]
that is a chain map, and preserves the elements $\ve{t}_i$ for $i \in \{1,2,3\}$.
Since $x + \pi(y) \in C(\L{m}{\tau})$
and $\d(x + \pi(y)) = \ve{t}_1 - \ve{t}_3$, we have $-\tau'_{13} \ge m$,
hence $\tau'_{13} \le \tau'_{23}$ and $\tau^+(S_1, S_3) \le \tau^+(S_2, S_3)$.
\end{proof}

\section{Link Floer homology and the stabilization distance}\label{sec:stab}

In this section we prove our main technical results about stabilizations and the link Floer TQFT,
and show that our invariants $\tau$ and $V_k$ give lower bounds on $\mu_{\st}$,
while $\kappa_0$ and $\cI$ give lower bounds on $g_{\dest}$.

\subsection{Algebraic reduction}

In this section, we consider the relation between the Heegaard Floer homology of multi-pointed 3-manifolds and the link Floer homology of unlinks.

There are two natural ways to reduce $\cCFL^-$ to $\CF^-$ via a tensor product. Let $M_{V=1}$ denote the $(\bF[U,V],\bF[\hat U])$-bimodule with underlying vector space $\bF[\hat{U}]$, where $U$ acts on the left by $\hat{U}$, and $V$ acts by 1. We have $\hat{U}$ act on the right by ordinary multiplication. There is also an $(\bF[U,V],\bF[\hat U])$-bimodule $M_{U=1}$  with underlying vector space $\bF[\hat{U}]$, defined similarly, except that we have $V$ act on the left by $\hat{U}$ and we have $U$ act by $1$.

There are canonical isomorphisms
\[
\begin{split}
\cCFL^-(Y,\bL,\frs)\otimes_{\bF[U,V]} M_{V=1}&\iso \CF^-(Y,\ws,\frs)\\
\cCFL^-(Y,\bL,\frs)\otimes_{\bF[U,V]} M_{U=1}&\iso \CF^-(Y,\zs,\frs-\PD[L]).
\end{split}
\]
These isomorphisms are obtained by taking any Heegaard diagram for $(Y,\bL)$, and ignoring the $\zs$ basepoints, or ignoring the $\ws$ basepoints.

In particular, for any $\bF[U,V]$-equivariant map $F$ from $\cCFL^-(Y_1,\bL_1,\frs_1)$ to $\cCFL^-(Y_2,\bL_2,\frs_2)$, we obtain a map $F|_{V=1}$ from $\CF^-(Y_1,\ws_1,\frs_1)$ to $\CF^-(Y_2,\ws_2,\frs_2)$. There is also a map $F|_{U=1}$ from $\CF^-(Y_1,\zs_1,\frs_1-\PD[L_1])$ to $\CF^-(Y_2,\zs_2,\frs_2-\PD[L_2])$. 

An important special case is when $\bL$ is an unlink, and each link component has exactly two basepoints. We say that a diagram for $(Y,\bL)$ is a \emph{minimal unlink diagram} if each $\ws$ basepoint occurs in the same component of $\Sigma\setminus (\as\cup \bs)$ as a $\zs$ basepoint. In this case, a Seifert disk is canonically specified by picking a collection of arcs in $\Sigma\setminus (\as\cup \bs)$ which connects each $\ws$ basepoint to a $\zs$ basepoint. By pushing the interiors of these arcs off of $\Sigma$, in both directions, a collection of Seifert disks for $\bL$ is spanned. In particular, there is a canonical Seifert surface $S$ of $\bL$ which is determined by the diagram, and $A_S(U^i V^j\xs)=j-i$ for all intersection points $\xs$.

Additionally, in the case of a minimal unlink diagram $\cH=(\Sigma,\as,\bs,\ws,\zs)$, there is a canonical isomorphism
\begin{equation}
\cCFL^-(\Sigma,\as,\bs,\ws,\zs,\frs)\iso \CF^-(\Sigma,\as,\bs,\ws,\frs)\otimes_{\bF[\hat{U}]} \bF[U,V],
\label{eq:minimal-unlink-diagram-isomorphism}
\end{equation}
where we view $\hat{U}$ as acting on $\bF[U,V]$ via the product $UV$. 

In particular, if we are given minimal unlink diagrams for $(Y_1,\bL_1)$ and $(Y_2,\bL_2)$ as well as an $\bF[U,V]$-equivariant map $F$ from  $\cCFL^-(Y_1,\bL_1,\frs_1)$ to $\cCFL^-(Y_2,\bL_2,\frs_2)$, we may view $F|_{U=1}\otimes \id_{\bF[U,V]}$ and $F|_{V=1}\otimes \id_{\bF[U,V]}$ as also being maps from $\cCFL^-(Y_1,\bL_1,\frs_1)$ to $\cCFL^-(Y_2,\bL_2,\frs_2)$. For our purposes, it is useful to compare these maps to the original map $F$:

\begin{lem}
\label{lem:compute-using-reductions} Suppose that $\bL_1\subset Y_1$ and $\bL_2\subset Y_2$ are unlinks, and pick minimal unlink diagrams for $(Y_1,\bL_1)$ and $(Y_2,\bL_2)$, respectively. Suppose that $F\colon \cCFL^-(Y_1,\bL_1,\frs_1)\to \cCFL^-(Y_2,\bL_2,\frs_2)$ is an $\bF[U,V]$-equivariant map, which is homogeneously graded with respect to the Alexander grading, and shifts the Alexander grading by $\Delta$.
\begin{enumerate}
\item If $\Delta\ge 0$, then
\[
F=V^\Delta\cdot (F|_{V=1}\otimes \id_{\bF[U,V]})\quad \text{and} \quad U^\Delta \cdot  F=(F|_{U=1}\otimes \id_{\bF[U,V]}).
\]
\item If $\Delta\le 0$, then
\[
V^{-\Delta}\cdot F=(F|_{V=1}\otimes \id_{\bF[U,V]})\quad \text{and} \quad F=U^{-\Delta} \cdot (F|_{U=1}\otimes \id_{\bF[U,V]}).
\]
\end{enumerate}
\end{lem}
\begin{proof} Consider the claim for $\Delta\ge 0$. In this case, the maps $F$ and $F|_{V=1}\otimes \id_{\bF[U,V]}$ agree up to an overall power of $V$. Since $F|_{V=1}\otimes \id_{\bF[U,V]}$ preserves the Alexander grading, the overall power is $V^\Delta$. The same argument works for the other claims.
\end{proof}

\begin{rem} Lemma~\ref{lem:compute-using-reductions} is stated using two fixed minimal unlink diagrams for $(Y_1,\bL_1)$ and $(Y_2,\bL_2)$, so we do not claim that the map $F|_{V=1}\otimes \id_{\bF[U,V]}$ and $F|_{U=1}\otimes \id_{\bF[U,V]}$ are natural maps. We may view these maps as being natural if we fix a set of Seifert disks for $\bL_1$ and $\bL_2$. 
\end{rem}

\subsection{Stabilizations and link Floer homology}\label{sec:stabandHFL}

In this section, we prove our main computational results about stabilizations
and link Floer homology.
Before we state our computational results, we recall that the link cobordism
maps admit extensions
\[
F_{W,\cF,\frs} \colon \Lambda^*(H_1(W)/\Tors) \otimes \cCFL^-(Y_1,\bL_1,\frs|_{Y_1}) \to
\cCFL^-(Y_2,\bL_2,\frs|_{Y_2})
\]
that incorporate the action of $\Lambda^* (H_1(W)/\Tors)$, similar to the
cobordism maps of Ozsv\'ath and Szab\'o~\cite{OSTriangles}; see
\cite{ZemAbsoluteGradings}*{Section~12.2} for a description.

If $F$ and $G$ are two maps from $\bF_2[U,V]$ to
$\cCFL^-(Y,\bL,\frs)$, we say that $F \simeq G$ \emph{modulo the action of $H_1(Y)$}, and write
\[
F \simeq G \mod H_1(Y),
\]
if there are classes $[\gamma_1], \dots, [\gamma_k] \in H_1(Y)$, as well as $\F_2[U,V]$-equivariant maps $J_1, \dots, J_k$ from $\bF_2[U,V]$ to
$\cCFL^-(Y,\bL,\frs)$, such that
\[
F + G \simeq \sum_{i=1}^k A_{\gamma_i} \circ J_i.
\]
(Note that if $\gamma$ and $\gamma'$ are homologous 1-cycles in $Y$, then $A_\gamma \simeq A_{\gamma'}$.)

\begin{lem}\label{lem:closedsurfacecomputation}
Suppose that $(W,\cF) \colon \emptyset \to (Y,\bU)$ is a decorated link
cobordism from the empty set to a doubly-based unknot $\bU$ in $Y$, equipped
with a Seifert disk $D$, and let $\frs \in \Spin^c(W)$. Pick a Heegaard diagram for $(Y,\bU)$ where the $w$ and $z$ basepoints are immediately adjacent. 
 Write $\cF = (S,\cA)$, and suppose that $H_1(Y) \to H_1(W)$ is a surjection.
\begin{enumerate}
\item\label{it:claim1} Suppose $\cA$ has a single component (necessarily a non-closed arc),
and write
\[
h(S\cup D,\frs) := \frac{\langle\, c_1(\frs), [S \cup D] \,\rangle - [S \cup D] \cdot [S \cup D]}{2}.
\]
If $h(S \cup D,\frs) + g(S_{\zs}) - g(S_{\ws}) \ge 0$, then
\[
F_{W,\cF,\frs} \simeq  U^{g(S_{\ws})} V^{g(S_{\zs}) + h(S \cup D, \frs)} \cdot F_{W,\frs}\otimes \id_{\bF[U,V]} \mod H_1(Y),
\]
with respect to the isomorphism from  Equation~\eqref{eq:minimal-unlink-diagram-isomorphism}.
If $h(S \cup D,\frs) + g(S_{\zs}) - g(S_{\ws}) \le 0$, then
\[
F_{W,\cF,\frs} \simeq U^{g(S_{\ws}) - h(S \cup D, \frs)} V^{g(S_{\zs})} \cdot F_{W,\frs - \PD[S]}\otimes \id_{\bF[U,V]} \mod H_1(Y).
\]
\item\label{it:claim2} If $\cA$ has a closed component $\gamma$, then
\[
F_{W,\cF,\frs} \simeq 0 \mod H_1(Y).
\]
\end{enumerate}
\end{lem}

\begin{proof}
The proof is a modification of \cite{ZemAbsoluteGradings}*{Proposition~9.7}.
Consider first Claim~\eqref{it:claim1}, in  the case when $h(S \cup
D,\frs) + g(S_{\zs}) - g(S_{\ws}) \ge 0$. Note that the latter quantity is the
Alexander grading change of the map $F_{W,\cF,\frs}$.  The reduction
$F_{W,\cF,\frs}|_{V=1}$ is computed explicitly in
\cite{ZemCFLTQFT}*{Theorem~C}, and depends only on $W$, $\frs$, and the
embedding of the subsurface $S_{\ws}$ in $W$. When $S_{\ws}$ is a connected
surface with a single boundary component, according to
\cite{ZemAbsoluteGradings}*{Lemma~9.6}, the $V = 1$ reduction satisfies
\begin{equation}
F_{W,\cF,\frs}|_{V=1} \simeq F_{W,\frs}(\xi_{\ws} \otimes -), \label{eq:V=1reductioncomputation1}
\end{equation}
where $\xi_{\ws} \in \Lambda^*(H_1(W)/\Tors) \otimes \bF_2[\hat{U}]$ is an element
equal to $\hat{U}^{g(S_{\ws})}$ modulo $H_1(W)$. Since $H_1(Y)$
surjects onto $H_1(W)$, we can commute $\xi_{\ws}$ with $F_{W,\frs}$ to obtain the relation
\[
F_{W,\cF,\frs}|_{V=1} \simeq \hat{U}^{g(S_{\ws})} \cdot F_{W,\frs} \mod H_1(Y).
\]
Applying Lemma~\ref{lem:compute-using-reductions}, we conclude that
\begin{equation}
F_{W,\cF,\frs} \simeq  V^{h(S\cup D,\frs) + g(S_{\zs}) - g(S_{\ws})} \hat{U}^{g(S_{\ws})} \cdot F_{W,\frs}\otimes \id_{\bF[U,V]} \mod H_1(Y). \label{eq:almostclaim1}
\end{equation}
Claim~\eqref{it:claim1} in this case follows by rearranging equation~\eqref{eq:almostclaim1}
using the fact that $\hat{U} = UV$.

The argument fails when the Alexander grading $h(S \cup
D,\frs) + g(S_{\zs}) - g(S_{\ws})$ is negative, since multiplication by
$V^{h(S \cup D,\frs) + g(S_{\zs}) - g(S_{\ws})}$ is not a filtered map. Instead we
must consider the $U = 1$ reduction of $F_{W,\cF,\frs}$. According to
\cite{ZemAbsoluteGradings}*{Lemma~9.6}, the $U = 1$ reduction satisfies
\[
F_{W,\cF,\frs}|_{U=1} \simeq F_{W,\frs - \PD[S]}(\xi_{\zs}\otimes -),
\]
for an element $\xi_{\zs} \in \Lambda^*(H_1(W)/\Tors) \otimes \bF_2[\hat{U}]$
equal to $\hat{U}^{g(S_{\zs})}$ modulo $H_1(W)$. Using this fact, the formula
\[
F_{W,\cF,\frs} \simeq U^{g(S_{\ws}) - h(S\cup D,\frs)} V^{g(S_{\zs})} \cdot (F_{W,\frs-\PD[S]}\otimes \id_{\bF[U,V]})\mod H_1(Y)
\]
can be established by the same strategy as before.

Next, we consider Claim~\eqref{it:claim2}, where $\cA$ contains a closed component. As in
the proof of Claim~\eqref{it:claim1}, the key will be to consider the maps
$F_{W,\cF,\frs}|_{V=1}$ and $F_{W,\cF,\frs}|_{U=1}$.
Let $\Delta$ denote the quantity
\[
\Delta := h(S \cup D,\frs) + \frac{\chi(S_{\ws}) - \chi(S_{\zs})}{2},
\]
which we note is the Alexander grading of the map $F_{W,\cF,\frs}$.

We first consider the case when $\Delta\ge 0$.
Let us write $C_{\ws,0}$ and $C_{\zs,0}$ for the components of $S_{\ws}$ and
$S_{\zs}$ that intersect $\d S$. We will reduce to the case when
$\d C_{\ws,0}$ contains a closed component disjoint from $\d S$. If $C_{\ws,0} = S_{\ws}$,
then $\d C_{\ws,0}$ trivially contains a closed component
disjoint from $\d S$. If $C_{\ws,0}$ is not the only component of $S_{\ws}$,
then, since $S$ is connected, we can find a properly embedded path
$\gamma_{\zs} \colon I \to C_{\zs,0}$, with both endpoints on $\cA$, such that
$\gamma_{\zs}(0) \in C_{\ws,0}$, and $\gamma_{\zs}(1)$ is a point in the
boundary of another component, $C_{\ws,1}$ of $S_{\ws}$. There are four cases we consider:
\begin{enumerate}
\item\label{it:1} $C_{\ws,1}$ is planar, and $|\d C_{\ws,1}| = 1$.
\item\label{it:2} $C_{\ws,1}$ is planar, and $|\d C_{\ws,1}| = 2$.
\item\label{it:3} $C_{\ws,1}$ is planar, and $|\d C_{\ws,1}| > 2$.
\item\label{it:4}  $g(C_{\ws,1})>0$.
\end{enumerate}

In Case~\eqref{it:1}, the surface $C_{\ws,1}$ is topologically a disk, which is
necessarily disjoint from $\d S$, since $|\cA \cap \d S| = 2$. We claim that the
map $F_{W,\cF,\frs} \simeq 0$. Indeed, the cobordism map
$F_{W,\cF,\frs}$ can be factored through the composition of a
quasi-stabilization, followed by a quasi-destabilization, and such a
composition clearly vanishes.

We next consider Case~\eqref{it:2}, when $C_{\ws,1}$ is an annulus, which is
disjoint from $\d S$. In this case, we also have $F_{W,\cF,\frs} \simeq 0$. To
see this, pick a properly embedded path $\gamma_{\ws} \colon I \to C_{\ws,1}$
that connects the two boundary components of $C_{\ws,1}$, and such that
$\gamma_{\ws}(0) = \gamma_{\zs}(1)$; see the top of Figure~\ref{fig::12}.
We concatenate $\gamma_{\zs}$ and $\gamma_{\ws}$ to get a path $\gamma$. A neighborhood of $\gamma$ is the
domain of a bypass. Let $\cA'$ and $\cA''$ denote the other two dividing sets
in the bypass triple; see the bottom row of Figure~\ref{fig::12}.
The bypass relation (relation~\eqref{eq:R6} above and its interpretation in terms of decorated cobordisms from Figure~\ref{fig::18}) implies that
\begin{equation}
F_{W,(S,\cA),\frs} \simeq F_{W,(S,\cA'),\frs} + F_{W,(S,\cA''),\frs}.
\label{eq:bypassforannulus}
\end{equation}
The key observation is that $\cA'$ and $\cA''$ are actually isotopic, so
equation~\eqref{eq:bypassforannulus} implies that $F_{W,(S,\cA),\frs} \simeq 0$.
The isotopy between $\cA'$ and $\cA''$ is shown in Figure~\ref{fig::12}.

\begin{figure}[ht!]
	\centering
	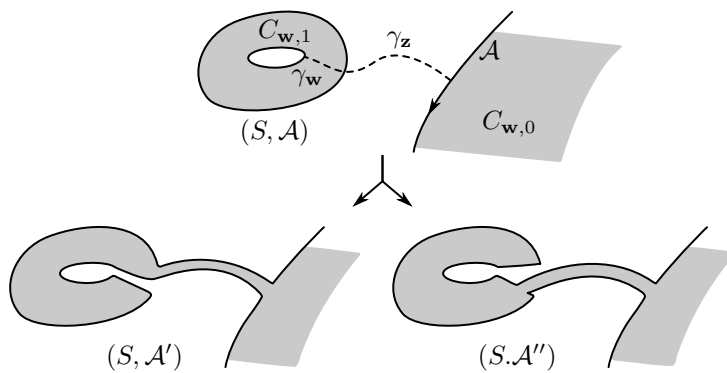
	\caption{When $C_{\ws,1}$ is an annulus, the two other dividing sets in a bypass are isotopic.}
\label{fig::12}
\end{figure}

We now consider Cases~\eqref{it:3} and~\eqref{it:4}, when $C_{\ws,1}$ is
planar and $|\d C_{\ws,1}| > 2$, or when $g(C_{\ws,1}) > 0$, respectively. In
both cases, we let $\gamma_{\ws} \colon I \to C_{\ws,1}$ be a properly
embedded curve which is non-separating and satisfies
$\gamma_{\zs}(1) = \gamma_{\ws}(0)$; see the top of Figure~\ref{fig::11}.
If $g(C_{\ws,1}) > 0$, we require that both
ends of $\gamma_{\ws}$ are on the same component of $\d C_{\ws,1}$. We let
$\gamma$ denote the concatenation of $\gamma_{\ws}$ and $\gamma_{\zs}$. As in
Case~\eqref{it:2}, we consider the bypass triple obtained by taking a regular
neighborhood of the image of $\gamma$.  Let $\cA'$ and $\cA''$ denote the
other two dividing sets in the bypass triple, shown on the bottom of Figure~\ref{fig::11}.

Let $C_{\ws,0}'$ and $C_{\ws,0}''$ denote the type-$\ws$ subregions of
$S \setminus \cA'$ and $S \setminus \cA''$ that intersect $\d S$.
In Cases~\eqref{it:3} and~\eqref{it:4}, it is easy to check that
$\d C_{\ws,0}'$ and $\d C_{\ws,0}''$ both contain a closed curve disjoint from $\d S$,
so it is sufficient to show the main claim for each of $F_{W,(S,\cA'),\frs}$ and
$F_{W,(S,\cA''),\frs}$ separately. Case~\eqref{it:4} is
illustrated in Figure~\ref{fig::11}.

\begin{figure}[ht!]
	\centering
	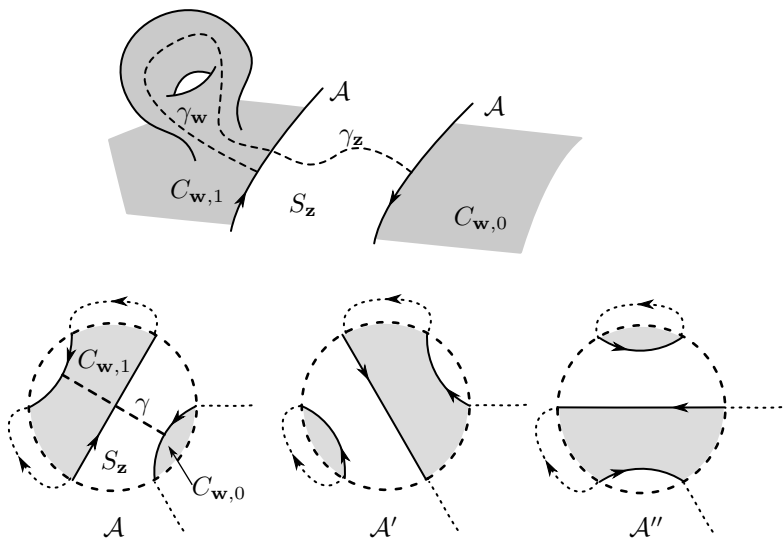
	\caption{The bypass relation in Case~\eqref{it:4}, when $g(C_{\ws,1}) > 0$. On the
    bottom row, the domain of the bypass is shown. The dotted lines outside of
    the domain of the bypass indicate the configuration of the dividing set
    outside the bypass region.}
\label{fig::11}
\end{figure}

We now proceed to show that if $\cA$ is a dividing set on $S$ such that
$\d C_{\ws,0}$ contains a closed component $\gamma$ disjoint from $\d S$, then
\begin{equation}
F_{W,\cF,\frs}|_{V=1} = A_\gamma \circ G
\label{eq:closedloopimplieshomologyaction}
\end{equation}
for some map $G$, where $A_\gamma$ denotes the homology action of the curve
$\gamma$. To establish equation~\eqref{eq:closedloopimplieshomologyaction},
we must recall some additional facts about the functor
$F_{W,\cF,\frs}|_{V=1}$. According to \cite{ZemCFLTQFT}*{Theorem~C}, the
chain homotopy type of the map $F_{W,\cF,\frs}|_{V=1}$ depends only on
$W$, the embedded surface $S_{\ws}$ (which is not properly embedded) and
$\frs\in \Spin^c(W)$. To describe the reduction in more detail, we recall
that a \emph{ribbon 1-skeleton} of $S_{\ws}$ is a choice of embedded graph
$\Gamma_{\ws} \subset S_{\ws}$ such that $\Gamma_{\ws}\cap \d S_{\ws} = \{w\}$,
and $S_{\ws}$ is a regular neighborhood of $\Gamma_{\ws}$ in $S$; see
\cite{ZemCFLTQFT}*{Definition~14.5}.

There is a simple way to construct the ribbon 1-skeleton of the subsurface
$S_{\ws}$. One starts with a collection of arcs $a \subset S_{\ws}$ such that
$a \cap \d S_{\ws} = \{w\}$, and such that each component of $S_{\ws}$ contains
exactly one arc. One then takes inward translates $C_1, \dots, C_n$ of the boundary components of $\d S_{\ws}$ which do not contain a basepoint of $\ws$, which one connects to
$a$ by adjoining an embedded arc (disjoint from the other arcs). The complement of this graph in $S_{\ws}$
consists of a collection of $|\d S_{\ws}|$ connected surfaces, each with a
single boundary component. The total genus of these surfaces is $g(S_{\ws})$.
We then pick a geometric symplectic basis of $H_1$ of the complement of this
graph (i.e., a collection of simple closed curves $A_1, \dots, A_g$,
$B_1, \dots, B_g$ that form a basis of $H_1$ and satisfy $|A_i \cap
B_j| = \delta_{ij}$). By connecting $a$ with one of the curves
in each pair in the symplectic basis by an arc, we obtain a ribbon 1-skeleton of
$S_{\ws}$. An example is shown in Figure~\ref{fig::13}.

\begin{figure}[ht!]
	\centering
\begingroup%
  \makeatletter%
  \providecommand\color[2][]{%
    \errmessage{(Inkscape) Color is used for the text in Inkscape, but the package 'color.sty' is not loaded}%
    \renewcommand\color[2][]{}%
  }%
  \providecommand\transparent[1]{%
    \errmessage{(Inkscape) Transparency is used (non-zero) for the text in Inkscape, but the package 'transparent.sty' is not loaded}%
    \renewcommand\transparent[1]{}%
  }%
  \providecommand\rotatebox[2]{#2}%
  \newcommand*\fsize{\dimexpr\f@size pt\relax}%
  \newcommand*\lineheight[1]{\fontsize{\fsize}{#1\fsize}\selectfont}%
  \ifx\svgwidth\undefined%
    \setlength{\unitlength}{192.82673827bp}%
    \ifx\svgscale\undefined%
      \relax%
    \else%
      \setlength{\unitlength}{\unitlength * \real{\svgscale}}%
    \fi%
  \else%
    \setlength{\unitlength}{\svgwidth}%
  \fi%
  \global\let\svgwidth\undefined%
  \global\let\svgscale\undefined%
  \makeatother%
  \begin{picture}(1,0.68761133)%
    \lineheight{1}%
    \setlength\tabcolsep{0pt}%
    \put(0,0){\includegraphics[width=\unitlength,page=1]{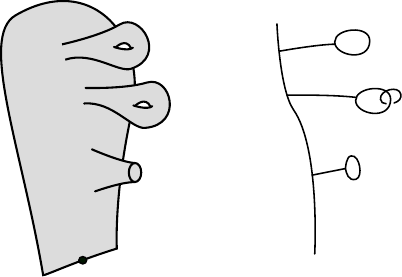}}%
    \put(0.15997085,0.2756731){\color[rgb]{0,0,0}\makebox(0,0)[t]{\lineheight{1.25}\smash{\begin{tabular}[t]{c}$S_{\ws}$\end{tabular}}}}%
    \put(0.74868433,0.28026361){\color[rgb]{0,0,0}\makebox(0,0)[rt]{\lineheight{1.25}\smash{\begin{tabular}[t]{r}$\Gamma_{\ws}$\end{tabular}}}}%
    \put(0.23165934,0.01683659){\color[rgb]{0,0,0}\makebox(0,0)[lt]{\lineheight{1.25}\smash{\begin{tabular}[t]{l}$w$\end{tabular}}}}%
    \put(0,0){\includegraphics[width=\unitlength,page=2]{fig13.pdf}}%
    \put(0.80248328,0.027391){\color[rgb]{0,0,0}\makebox(0,0)[lt]{\lineheight{1.25}\smash{\begin{tabular}[t]{l}$w$\end{tabular}}}}%
    \put(0,0){\includegraphics[width=\unitlength,page=3]{fig13.pdf}}%
  \end{picture}%
\endgroup%

	\caption{A ribbon 1-skeleton $\Gamma_{\ws}$ for a genus 2 component of $S_{\ws}$
    with 2 boundary components.}
    \label{fig::13}
\end{figure}

The second author~\cite{ZemGraphTQFT} constructed maps on $\CF^-$
induced by cobordisms with embedded ribbon graphs.
By \cite{ZemCFLTQFT}*{Theorem~C}, the reduction
$F_{W,\cF,\frs}|_{V=1}$ is chain homotopic to the graph cobordism map
$F_{W,\Gamma_{\ws},\frs}$ for a ribbon 1-skeleton $\Gamma_{\ws}$ of $S_{\ws}$.
Note that $\Gamma_{\ws}$ inherits a ribbon structure; i.e., a cyclic
ordering of the edges adjacent to each vertex, from
the orientation of $S_{\ws}$.

By picking an appropriate ribbon 1-skeleton $\Gamma_{\ws}$ of $\S_{\ws}$ (see Figure~\ref{fig::14}), we
can decompose the graph cobordism $(W,\Gamma_{\ws})$ such that it
is a sequence of graph cobordisms $(W_3,\gamma_3) \circ (W_2,\Gamma_2) \circ (W_1,\Gamma_1)$,
satisfying the following conditions:
\begin{enumerate}
\item $W_1$ is a 4-dimensional 1-handlebody, and the graph $\Gamma_1$ intersects $\d W_1$ in a single point.
\item $W_2$ is a cylinder $I \times Y$, and $\Gamma_2$ is a graph of the form
$(I \times \{p\}) \cup \gamma$, as shown in Figure~\ref{fig::14}, where
$\gamma$ is a loop induced by one of the boundary components of
$S_{\ws}$ which is disjoint from $\d S$.
\item $W_3$ is a cobordism between two connected 3-manifolds,
and $\gamma_3$ is a path connecting the two components.
\end{enumerate}
This can be achieved as follows. We pick $\Gamma_{\ws}$ by first taking an arc $a$ in $\Sigma_{\ws}$ such that $a\cap \d \Sigma_{\ws}=\{w\}$. We then join closed loops (as described above) for components of $\d S_{\ws}$ not containing $w$, and also for a symplectic basis $A_1,\dots, A_g, B_1,\dots, B_g$ as above. We assume that of these loops, $\g$ is joined the closest to $w$ along the arc $a$. We pick an ordered handle decomposition for $W$ into 0-, 1-, 2- and 3-handles. We let $W_1$ be the the union of the 0- and 1-handles. We let $W_2$ be a regular neighborhood of $\d W_1$, and we let $W_3$ be the 2-handles and 3-handles.  By flowing using a gradient like Morse function for this handle decomposition, we may isotope all of the closed loops of $\Gamma_{\ws}$ so that they lie below $W_3$. Therefore we may assume that $\Gamma_{\ws}\cap W_2$ consists of a subarc of $a$ with $\gamma$ spliced in, and that $ W_3\cap \Gamma_{\ws}$ consists only of a single arc, as claimed above.

\begin{figure}[ht!]
	\centering
\begingroup%
  \makeatletter%
  \providecommand\color[2][]{%
    \errmessage{(Inkscape) Color is used for the text in Inkscape, but the package 'color.sty' is not loaded}%
    \renewcommand\color[2][]{}%
  }%
  \providecommand\transparent[1]{%
    \errmessage{(Inkscape) Transparency is used (non-zero) for the text in Inkscape, but the package 'transparent.sty' is not loaded}%
    \renewcommand\transparent[1]{}%
  }%
  \providecommand\rotatebox[2]{#2}%
  \newcommand*\fsize{\dimexpr\f@size pt\relax}%
  \newcommand*\lineheight[1]{\fontsize{\fsize}{#1\fsize}\selectfont}%
  \ifx\svgwidth\undefined%
    \setlength{\unitlength}{140.4423405bp}%
    \ifx\svgscale\undefined%
      \relax%
    \else%
      \setlength{\unitlength}{\unitlength * \real{\svgscale}}%
    \fi%
  \else%
    \setlength{\unitlength}{\svgwidth}%
  \fi%
  \global\let\svgwidth\undefined%
  \global\let\svgscale\undefined%
  \makeatother%
  \begin{picture}(1,1.01884534)%
    \lineheight{1}%
    \setlength\tabcolsep{0pt}%
    \put(0,0){\includegraphics[width=\unitlength,page=1]{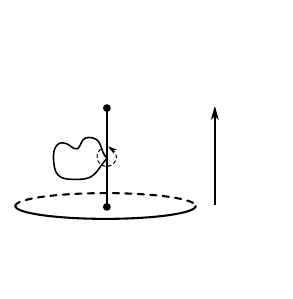}}%
    \put(0.76139992,0.45880973){\color[rgb]{0,0,0}\makebox(0,0)[lt]{\lineheight{1.25}\smash{\begin{tabular}[t]{l}$(W_2, \Gamma_2)$\end{tabular}}}}%
    \put(0.16525594,0.50269232){\color[rgb]{0,0,0}\makebox(0,0)[rt]{\lineheight{1.25}\smash{\begin{tabular}[t]{r}$\gamma$\end{tabular}}}}%
    \put(0,0){\includegraphics[width=\unitlength,page=2]{fig14.pdf}}%
    \put(0.76139992,0.17604575){\color[rgb]{0,0,0}\makebox(0,0)[lt]{\lineheight{1.25}\smash{\begin{tabular}[t]{l}$(W_1, \Gamma_1)$\end{tabular}}}}%
    \put(0,0){\includegraphics[width=\unitlength,page=3]{fig14.pdf}}%
    \put(0.76139992,0.78697142){\color[rgb]{0,0,0}\makebox(0,0)[lt]{\lineheight{1.25}\smash{\begin{tabular}[t]{l}$(W_3, \gamma_3)$\end{tabular}}}}%
    \put(0,0){\includegraphics[width=\unitlength,page=4]{fig14.pdf}}%
  \end{picture}%
\endgroup%

	\caption{A decomposition of the graph cobordism $(W,\Gamma_{\ws})$.  The
    loop $\gamma$ in $\Gamma_2$ corresponds to a closed curve in $\d C_{\ws,0}$
    disjoint from $\d S$.}\label{fig::14}
\end{figure}

The composition law for graph cobordism maps implies that
\begin{equation}
F_{W,\Gamma_{\ws},\frs} \simeq F_{W_3,\gamma_3,\frs|_{W_3}} \circ F_{W_2, \Gamma_2,\frs|_{W_2}}
\circ F_{W_1,\Gamma_1,\frs|_{W_1}}.
\label{eq:graphcobordismdemanip1}
\end{equation}
Since $\gamma_3$ is a path, the map $F_{W_3,\gamma_3,\frs|_{W_3}}$
agrees with Ozsv\'{a}th and Szab\'{o}'s cobordism map. By \cite{ZemDualityMappingTori}*{Proposition~4.6}, we have
\begin{equation}
F_{W_2,\Gamma_2,\frs|_{W_2}} \simeq A_\gamma.
\label{eq:graphcobordismmaniup2}
\end{equation}
Since $H_1(Y)$ surjects onto $H_1(W)$, we have
\begin{equation}
F_{W_3,\gamma_3,\frs|_{W_3}} \circ A_\gamma \simeq A_\gamma \circ  F_{W_3,\gamma_3,\frs|_{W_3}}.
\label{eq:graphcobordismmanip3}
\end{equation}
Combining equations~\eqref{eq:graphcobordismdemanip1},
\eqref{eq:graphcobordismmaniup2}, and~\eqref{eq:graphcobordismmanip3}, we
obtain the relation
\begin{equation}
F_{W,\cF,\frs}|_{V=1}\simeq F_{W,\Gamma_{\ws},\frs} \simeq
A_\gamma\circ F_{W_3,\gamma_3,\frs|_{W_3}} \circ F_{W_1,\Gamma_1,\frs|_{W_1}}.
\label{eq:graphcobordismmanip4}
\end{equation}

Since we assumed that the Alexander grading shift $\Delta$ was nonnegative, by using Lemma~\ref{lem:compute-using-reductions} we obtain
\[
\begin{split}
F_{W,\cF,\frs} &\simeq V^{\Delta} \cdot  (F_{W,\cF,\frs}|_{V=1})\otimes \id_{\bF[U,V]}\\
& \simeq V^{\Delta} \cdot (A_\gamma\circ G)\otimes \id_{\bF[U,V]}\\
&\simeq (A_{\g}\otimes \id_{\bF[U,V]}) \circ (V^{\Delta} \cdot G\otimes \id_{\bF[U,V]})\\
&\simeq A_\gamma \circ (V^{\Delta}\cdot G\otimes \id_{\bF[U,V]}),
\end{split}
\]
where $G \simeq F_{W_3,\gamma_3,\frs|_{W_3}} \circ F_{W_1,\Gamma_1,\frs|_{W_1}}$. In the last line, we are using the fact that $A_\g$ preserves the Alexander grading, so $(A_\g|_{V=1})\otimes \id_{\bF[U,V]}$ coincides with the ordinary action of $A_\g$ on $\cCFL^-(Y,\bU)$ by Lemma~\ref{lem:compute-using-reductions}.  This proves the claim.

The case when the Alexander grading change $\Delta$ is negative is handled similarly,
using the $U = 1$ reductions instead.
\end{proof}

Next, we compute the effect of a stabilization, for a simple dividing set:

\begin{lem}\label{lem:stabilizationlemma}
Suppose that $(W,\cF) \colon (Y_1,\bL_1) \to (Y_2,\bL_2)$ is a decorated
link cobordism with $b_1(W)=0$. Write $\cF = (S,\cA)$, and suppose that $S$ is connected.
Let $S'$ be a $(n, g)$-stabilization of $S$ along $(B^4, S_0)$.  Let $D_1, \dots, D_n$ denote
the components of $S \cap B^4$, and let $\hat{D} \subset S$ be a disk
that contains $D_1,\dots, D_n$ and intersects $\cA$ in a single arc. Consider the subsurface
\[
S_0' := (\hat{D} \setminus (D_1 \cup \cdots \cup D_n)) \cup S_0 \subset S'.
\]
Let $\cA'$ be a dividing set on $S'$ that agrees with $\cA$ outside $\hat{D}$,
and write $\cF' = (S', \cA')$.
\begin{enumerate}
\item\label{claim:1} Suppose that $\cA'$ intersects $S_0'$ in a single arc that
divides $S_0'$ into two connected components.
Let $S_{\ws}$ and $S_{\ws}'$ denote the type-$\ws$ regions,
and $S_{\zs}$ and $S_{\zs}'$ the type-$\zs$ regions
of $\cF$ and $\cF'$, respectively. Then
\[
F_{W,\cF',\frs} \simeq U^{g(S_{\ws}') - g(S_{\ws})} V^{g(S_{\zs}') - g(S_{\zs})} \cdot F_{W,\cF,\frs}.
\]
\item~\label{claim:2} If $\cA \cap S_0'$ contains a closed component, then
\[
F_{W,\cF',\frs} \simeq 0.
\]
\end{enumerate}
\end{lem}

\begin{figure}[ht!]
	\centering
\begingroup%
  \makeatletter%
  \providecommand\color[2][]{%
    \errmessage{(Inkscape) Color is used for the text in Inkscape, but the package 'color.sty' is not loaded}%
    \renewcommand\color[2][]{}%
  }%
  \providecommand\transparent[1]{%
    \errmessage{(Inkscape) Transparency is used (non-zero) for the text in Inkscape, but the package 'transparent.sty' is not loaded}%
    \renewcommand\transparent[1]{}%
  }%
  \providecommand\rotatebox[2]{#2}%
  \newcommand*\fsize{\dimexpr\f@size pt\relax}%
  \newcommand*\lineheight[1]{\fontsize{\fsize}{#1\fsize}\selectfont}%
  \ifx\svgwidth\undefined%
    \setlength{\unitlength}{332.8387729bp}%
    \ifx\svgscale\undefined%
      \relax%
    \else%
      \setlength{\unitlength}{\unitlength * \real{\svgscale}}%
    \fi%
  \else%
    \setlength{\unitlength}{\svgwidth}%
  \fi%
  \global\let\svgwidth\undefined%
  \global\let\svgscale\undefined%
  \makeatother%
  \begin{picture}(1,0.32655939)%
    \lineheight{1}%
    \setlength\tabcolsep{0pt}%
    \put(0,0){\includegraphics[width=\unitlength,page=1]{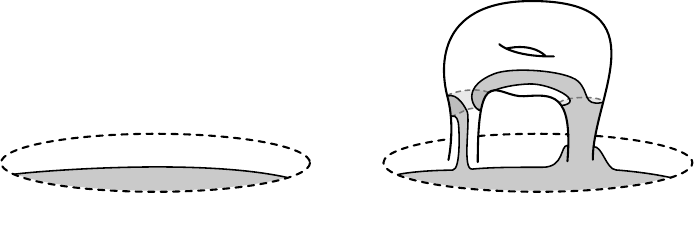}}%
    \put(0.22418129,0.00644324){\color[rgb]{0,0,0}\makebox(0,0)[t]{\lineheight{1.25}\smash{\begin{tabular}[t]{c}$(S,\cA)$\end{tabular}}}}%
    \put(0.77564778,0.00644324){\color[rgb]{0,0,0}\makebox(0,0)[t]{\lineheight{1.25}\smash{\begin{tabular}[t]{c}$(S',\cA')$\end{tabular}}}}%
    \put(0,0){\includegraphics[width=\unitlength,page=2]{fig8.pdf}}%
  \end{picture}%
\endgroup%

	\caption{An example of a stabilization considered in Lemma~\ref{lem:stabilizationlemma}. The dividing set $\cA$ in the region $\hat{D}\subset S$ is shown on the left, and the dividing set $\cA'$ on the stabilization $S'$ is shown on the right.}\label{fig::8}
\end{figure}

\begin{proof} We first show Claim~\eqref{claim:1}. Consider the punctured disk
\[
\hat{D}_0 := \hat{D}\setminus (D_1 \cup \cdots \cup D_n).
\]
Let $N$ denote the total space of the unit normal disk bundle of $\hat{D}_0$
in $W \setminus B^4$. Note that $N$ is diffeomorphic to $\hat{D}_0 \times B^2$. Define
\[
W_0 := B^4 \cup N,
\]
which, after rounding corners, we can view as a codimension 0 submanifold of
$W$ with smooth boundary. In fact, $W_0$ is a 4-dimensional genus $n-1$
handlebody $(\Sphere^1 \times B^3)^{\natural (n-1)}$. 
Let $Y$ denote $\d W_0$. We observe that $S_0'$, as defined above, is equal to $W_0\cap S'$. We can view 
$(W_0,S_0')$ as a link cobordism from the empty link to the pair $(Y,K)$
where $K=\d \hat{D}\times \{0\}$.

Let us write $C_1, \dots, C_n$ for the components of $\d \hat{D}_0 \setminus \d \hat{D}$,
and $U_1, \dots, U_n$ for the components of the unlink
$S_0 \cap \d B^4$. We can view
\[
\d N = (\hat{D}_0 \times \d B^2) \cup (\d\hat{D} \times B^2) \cup \bigcup_{i=1}^n (C_i \times B^2).
\]
Hence, we can write
\[
Y = \left(\d N \setminus \bigcup_{i=1}^n (C_i \times B^2) \right) \cup
\left(\d B^4 \setminus \bigcup_{i=1}^n N(U_i) \right),
\]
where the two manifolds are glued along their $n$ torus boundary components.

We now claim that $K$ is an unknot in $Y$. It is at this step that we use the
fact that $S \cap \d B^4$ is an unlink. To see that $K$ is an unknot, we will
construct a Seifert disk $D_K$ for $K$ in $Y$. Let $r$ denote a radial arc
from $0 \in B^2$ to a point $p \in \d B^2$. Let $A$ denote the annulus
\[
A:= \d \hat{D} \times r \subset Y.
\]
We then attach the punctured disk $\hat{D}_0 \times \{p\}$ to the annulus $A$. The resulting surface has boundary
\[
\d(A \cup (\hat{D}_0 \times \{p\})) = K \cup \bigcup_{i=1}^n (C_i \times \{p\}).
\]
Next, we note that the image of $C_i \times \{p\}$ in $\d N(U_i) \subset \d B^4$
is the Seifert longitude, since the disks $D_i \subset S \cap B^4$ can be
pushed into $\d B^4$ to give Seifert disks of $U_i$ that intersect $\d N(U_i)$
along $C_i \times \{p\}$. By capping $C_i \times \{p\}$ with Seifert
disks of the $U_i$, we obtain the Seifert disk $D_K$ of $K$ in $Y$.

Let us write $\cF_0$ for the decorated surface $(\hat{D}, \cA \cap \hat{D})$,
$\cF_0'$ for the decorated surface $(S_0', \cA' \cap S_0')$, and $\frs_0$ for
$\frs|_{W_0}$. Since $H_2(W_0) = 0$, by applying Lemma~\ref{lem:closedsurfacecomputation}
to both $F_{W_0,\cF_0,\frs_0}$ and $F_{W_0,\cF_0',\frs_0}$, we compute that
\begin{equation}
F_{W_0,\cF_0',\frs_0} \simeq U^{g(S_{0,\ws}')} V^{g(S_{0,\zs}') }\cdot  F_{W_0,\cF_0,\frs_0}
\mod H_1(Y).
\label{eq:computationofstabilzationpiece}
\end{equation}

Write $W_1 := W \setminus \Int(W_0)$,  $\cF_1 := \cF'\cap W_1$, and
$\frs_1 := \frs|_{W_1}$. Since $W_0$ is a 4-dimensional handlebody, we conclude
that $b_1(W_1) = 0$. Noting that the map $\delta \colon H^1(Y) \to  H^2(W)$
is trivial, and using the $\Spin^c$ composition law, we conclude from
equation~\eqref{eq:computationofstabilzationpiece} that
\begin{equation}
F_{W,\cF',\frs} \simeq U^{g(S_{0,\ws}')} V^{g(S_{0,\zs}')} \cdot F_{W,\cF,\frs}.
\label{eq:postcompose}
\end{equation}
Noting that $g(S'_{\ws}) - g(S_{\ws}) = g(S_{0,\ws}')$ and
$g(S'_{\zs}) - g(S_{\zs}) = g(S_{0,\zs}')$, the proof of Claim~\eqref{claim:1} is complete.

We now consider Claim~\eqref{claim:2}. In this case, Lemma~\ref{lem:closedsurfacecomputation} implies that
$F_{W_0,\cF_0',\frs} \simeq \sum_{i=1}^k A_{\gamma_i} \circ J_i$ for some
filtered, equivariant maps $J_1, \dots , J_k$.  The map $H_1(Y) \to H_1(W_1)/\Tors$
is trivial since $b_1(W) = 0$ and $W_0$ is a 4-dimensional 1-handlebody.
Hence, using the composition law,
\[
F_{W,\cF',\frs} \simeq F_{W_1,\cF_1,\frs_1} \circ F_{W_0,\cF_0',\frs_0} \simeq
F_{W_1,\cF_1,\frs_1} \circ \left(\sum_{i=1}^k A_{\gamma_i} \circ J_i \right) \simeq 0,
\]
concluding the proof of Claim~\eqref{claim:2}.
\end{proof}

Lemma~\ref{lem:stabilizationlemma} computes the result of a stabilization of
a link cobordism, when the dividing set is nicely arranged on the
stabilization. However, to prove geometric bounds on the secondary versions of $V_k$,
we will need to consider more general dividing sets on stabilizations.

\begin{define}\label{def:stabilizationtypedepthd}
Suppose that $(B^4,S)$ is an undecorated knot cobordism from $\emptyset$ to
an arbitrary knot $K$ in $\Sphere^3$. Let $\bK$ denote $K$ decorated with two
basepoints, and let $\frs_0$ be the unique $\Spin^c$ structure on $B^4$.
We say that $S$ satisfies the \emph{decoration-independence condition (DI)} if the following holds:
\begin{enumerate}[leftmargin=1.5cm,label=(\textit{DI}), ref=\textit{DI}]
\item\label{def:DI} For any decoration
$\cF = (S, \cA)$ whose dividing set intersects $K$ in exactly two points,
\begin{enumerate}[label=(\arabic*)]
\item  the filtered, equivariant chain homotopy type of the map
\[
F_{B^4,\cF,\frs_0} \colon \bF_2[U,V] \to \cCFL^-(\bK)
\]
depends only on $g(S_{\ws})$ and $g(S_{\zs})$ when $|\cA| = 1$, and
\item $F_{B^4,\cF,\frs_0} \simeq 0$ when $|\cA| > 1$.
\end{enumerate}
\end{enumerate}
\end{define}

Note that, if $S$ is a stabilization of a slice disk $(B^4, D)$, then, by
Lemma~\ref{lem:stabilizationlemma}, the link cobordism $(B^4, S)$ satisfies the decoration-independence condition \eqref{def:DI}.

\begin{define}\label{def:depth}
Let $S$ and $\bK$ be as in Definition~\ref{def:stabilizationtypedepthd}.
Suppose $d \ge g(S)$ is an integer. We say that $S$ satisfies
the decoration-independence condition \eqref{def:DI} \emph{above degree~$d$} if for any decoration $\cF = (S, \cA)$ compatible with $\bK$,
and for any $i$, $j \in \N$ satisfying $i + j + g(S) \ge d$,
\begin{enumerate}
\item the chain homotopy type of the map $U^i V^j \cdot F_{B^4,\cF,\frs_0}$
depends only on $i + g(S_{\ws})$ and $j + g(S_{\zs})$ when $|\cA| = 1$, and
\item $U^i V^j \cdot F_{B^4,\cF,\frs} \simeq 0$ when $|\cA| > 1$.
\end{enumerate}
We define the invariant $\cI(S) \in \N$
to be the minimal $d \ge g(S)$ such that $S$ satisfies condition~\eqref{def:DI}  above degree~$d$.
\end{define}

\begin{rem}
 The quantity $\cI(S)$ is finite for every surface $S$. This can be seen as follows. Two $\bF[U,V]$-equivariant chain maps $f$, $g\colon \bF[U,V]\to \cCFL^-(S^3,\bK)$ are $\bF[U,V]$-equivariantly chain homotopic if and only if $[f(1)]=[g(1)]$, as elements of $\cHFL^-(S^3,\bK)$. However, the rank of $\cHFL^-(S^3,\bK)$ in $(\gr_{\ws}, \gr_{\zs}$)-bigrading $(-2n,-2m)$ is 1 whenever $n$, $m \ge 0$ and $n+m$ is sufficiently large.
\end{rem}

Note that, to compute $\cI(S)$, one would need to determine the cobordism maps
for infinitely many dividing sets on $S$, which is a formidable task.
However, to obtain a lower bound on $\cI(S)$, it suffices to find
two dividing sets $\cA_1$ and $\cA_2$ on $S$, both consisting of a single arc,
and integers $i_1$, $i_2$, $j_1$, and $j_2$, such that
\[
i_1 + g(S_{1,\ws}) = i_2 + g(S_{2,\ws}) \text{ and }
j_1 + g(S_{1,\zs}) = j_2 + g(S_{2,\zs}) \text{, but}
\]
\[
U^{i_1} V^{j_1} \cdot F_{B^4, (S,\cA_1), \frs_0} \not\simeq
U^{i_2} V^{j_2} \cdot F_{B^4, (S,\cA_2),\frs_0},
\]
where $S_{k, \ws}$ and $S_{k, \zs}$ denote the type-$\ws$ and type-$\zs$
subsurfaces of $S$ with respect to the decoration $\cA_k$ for $k \in \{1, 2\}$.
In this case, $\cI(S) > i_1 + i_2 + g(S)$.

\begin{prop}\label{prop:2foldstabisstabtype}
Suppose that $(B^4,S)$ satisfies the decoration-independence condition \eqref{def:DI} above degree~$d$,
and let $S'$ be a stabilization of $S$.
Then $(B^4,S')$ satisfies condition~\eqref{def:DI} above degree~$\max\{d, g(S')\}$.
\end{prop}

Our proof of Proposition~\ref{prop:2foldstabisstabtype} uses the following
combinatorial lemma about dividing sets on surfaces:

\begin{lem}\label{lem:closedcomponentsbypass}
Suppose that $\cA_1$, $\cA_2$, and $\cA_3$ are three dividing sets that fit
into a bypass triple on a surface $S$ with $|\d S| = 1$, and $|\cA_i \cap \d
S| = 2$ for $i \in \{1,2,3\}$. Then the number of $\cA_i$ that have no
closed loops is even.
\end{lem}

\begin{proof}
If $\cA_1$, $\cA_2$, and $\cA_3$ all have a closed loop, then the
statement is true since 0 is even, so instead assume that $\cA_1$ has
no closed loops. Note that this implies that $|\cA_1| = 1$, since $|\d S \cap \cA_1| = 2$.
The dividing sets $\cA_i$ can be consistently oriented, by declaring their
orientation to be the boundary orientation of $S_{\ws}$. Let $D$ denote the
bypass region. The set $\cA_1 \cap D$ consists of three arcs, which we label
as $a_1$, $a_2$, and $a_3$; see Figure~\ref{fig::10}. The main claim can be
proven by considering separately six cases, corresponding to the possible
relative orderings of the arcs $a_1$, $a_2$, and $a_3$, as they appear on
$\cA_1$. Let us first consider the case when the arcs appear ordered
$(a_1,a_2,a_3)$, read left-to-right. In this case, $\cA_1$ has no closed
loops by assumption, and by inspecting Figure~\ref{fig::10}, we see that
exactly one of $\cA_2$ and $\cA_3$ also has no closed loops. The arguments
when the arcs appear along $\cA_1$ with ordering $(a_1,a_3,a_2)$,
$(a_2,a_1,a_3)$, $(a_2,a_3,a_1)$,  $(a_3,a_1,a_2)$, or $(a_3,a_2,a_1)$ are
easy modifications of the above argument.
\end{proof}

\begin{figure}[ht!]
	\centering
\begingroup%
  \makeatletter%
  \providecommand\color[2][]{%
    \errmessage{(Inkscape) Color is used for the text in Inkscape, but the package 'color.sty' is not loaded}%
    \renewcommand\color[2][]{}%
  }%
  \providecommand\transparent[1]{%
    \errmessage{(Inkscape) Transparency is used (non-zero) for the text in Inkscape, but the package 'transparent.sty' is not loaded}%
    \renewcommand\transparent[1]{}%
  }%
  \providecommand\rotatebox[2]{#2}%
  \newcommand*\fsize{\dimexpr\f@size pt\relax}%
  \newcommand*\lineheight[1]{\fontsize{\fsize}{#1\fsize}\selectfont}%
  \ifx\svgwidth\undefined%
    \setlength{\unitlength}{250.09380687bp}%
    \ifx\svgscale\undefined%
      \relax%
    \else%
      \setlength{\unitlength}{\unitlength * \real{\svgscale}}%
    \fi%
  \else%
    \setlength{\unitlength}{\svgwidth}%
  \fi%
  \global\let\svgwidth\undefined%
  \global\let\svgscale\undefined%
  \makeatother%
  \begin{picture}(1,0.36485426)%
    \lineheight{1}%
    \setlength\tabcolsep{0pt}%
    \put(0.2442293,0.00636979){\color[rgb]{0,0,0}\makebox(0,0)[rt]{\lineheight{0}\smash{\begin{tabular}[t]{r}$\cA_1$\end{tabular}}}}%
    \put(0.57835685,0.00636979){\color[rgb]{0,0,0}\makebox(0,0)[rt]{\lineheight{0}\smash{\begin{tabular}[t]{r}$\cA_2$\end{tabular}}}}%
    \put(0.94308843,0.00636979){\color[rgb]{0,0,0}\makebox(0,0)[rt]{\lineheight{0}\smash{\begin{tabular}[t]{r}$\cA_3$\end{tabular}}}}%
    \put(0,0){\includegraphics[width=\unitlength,page=1]{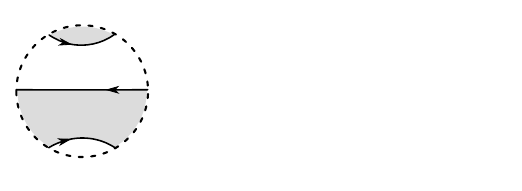}}%
    \put(0.15404029,0.25186831){\color[rgb]{0,0,0}\makebox(0,0)[t]{\lineheight{1.25}\smash{\begin{tabular}[t]{c}$a_1$\end{tabular}}}}%
    \put(0.15404029,0.20339908){\color[rgb]{0,0,0}\makebox(0,0)[t]{\lineheight{1.25}\smash{\begin{tabular}[t]{c}$a_2$\end{tabular}}}}%
    \put(0.15404029,0.11479143){\color[rgb]{0,0,0}\makebox(0,0)[t]{\lineheight{1.25}\smash{\begin{tabular}[t]{c}$a_3$\end{tabular}}}}%
    \put(0,0){\includegraphics[width=\unitlength,page=2]{fig10.pdf}}%
  \end{picture}%
\endgroup%

	\caption{The proof of Lemma~\ref{lem:closedcomponentsbypass}, when the
    arcs $a_1$, $a_2$, and $a_3$ appear on $\cA_1$ with order $(a_1,a_2,a_3)$,
    read left-to-right. The bypass region is the disk shown. The dashed lines
    outside the bypass regions represent the configuration of dividing arcs
    outside the bypass region. In the case at hand, $\cA_1$ and $\cA_3$ have
    no closed components, while $\cA_2$ has two.}\label{fig::10}
\end{figure}

\begin{proof}[Proof of Proposition~\ref{prop:2foldstabisstabtype}]
Fix integers $i$, $j \ge 0$ such that $i + j + g(S') \ge d$. Analyzing the proof of
Lemma~\ref{lem:stabilizationlemma}, we can find a 4-dimensional 1-handlebody
$W_0$ whose boundary we denote $Y$, such that
\begin{enumerate}
\item $S' \cap Y = S \cap Y$ is an unknot in $Y$;
\item $S \cap W_0$ is a disk, and $S' \cap W_0$ is a connected,
genus $g(S') - g(S)$ surface with only one boundary component.
\end{enumerate}
Let $J$ denote $S'\cap Y$. Note that Lemma~\ref{lem:stabilizationlemma}
immediately implies the statement for any dividing set $\cA' \subset S'$
(connected or disconnected) that intersects $J$ in exactly two points.

We now show the main claim by induction on $|\cA' \cap J|$.
We have established the base case, $|\cA' \cap J| = 2$.
If $\cA'$ is a dividing set on $S'$ with $|\cA' \cap J| \ge 4$, then, using the
bypass relation as shown in Figure~\ref{fig::9}, we can write
\begin{equation}
F_{B^4,(S',\cA'),\frs_0} \simeq F_{B^4,(S',\cA''),\frs_0} + F_{B^4,(S',\cA'''),\frs_0},
\label{eq:bypasstoreduceintersections}
\end{equation}
where $\cA''$ and $\cA''$ are dividing sets satisfying
\[
|\cA'' \cap J| = |\cA''' \cap J| = |\cA' \cap J| - 2.
\]

\begin{figure}[ht!]
	\centering
\begingroup%
  \makeatletter%
  \providecommand\color[2][]{%
    \errmessage{(Inkscape) Color is used for the text in Inkscape, but the package 'color.sty' is not loaded}%
    \renewcommand\color[2][]{}%
  }%
  \providecommand\transparent[1]{%
    \errmessage{(Inkscape) Transparency is used (non-zero) for the text in Inkscape, but the package 'transparent.sty' is not loaded}%
    \renewcommand\transparent[1]{}%
  }%
  \providecommand\rotatebox[2]{#2}%
  \newcommand*\fsize{\dimexpr\f@size pt\relax}%
  \newcommand*\lineheight[1]{\fontsize{\fsize}{#1\fsize}\selectfont}%
  \ifx\svgwidth\undefined%
    \setlength{\unitlength}{246.94983479bp}%
    \ifx\svgscale\undefined%
      \relax%
    \else%
      \setlength{\unitlength}{\unitlength * \real{\svgscale}}%
    \fi%
  \else%
    \setlength{\unitlength}{\svgwidth}%
  \fi%
  \global\let\svgwidth\undefined%
  \global\let\svgscale\undefined%
  \makeatother%
  \begin{picture}(1,0.52162396)%
    \lineheight{1}%
    \setlength\tabcolsep{0pt}%
    \put(0.50295009,0.14636718){\color[rgb]{0,0,0}\makebox(0,0)[t]{\lineheight{1.25}\smash{\begin{tabular}[t]{c}$+$\end{tabular}}}}%
    \put(0.50295009,0.25732811){\color[rgb]{0,0,0}\makebox(0,0)[t]{\lineheight{1.25}\smash{\begin{tabular}[t]{c}$\simeq$\end{tabular}}}}%
    \put(0,0){\includegraphics[width=\unitlength,page=1]{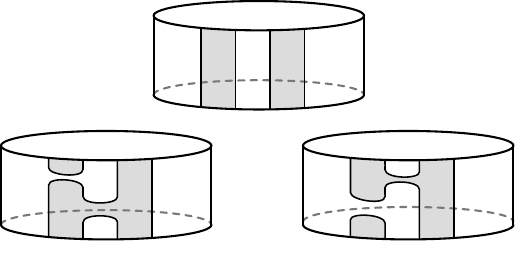}}%
    \put(0.7282273,0.39604033){\color[rgb]{0,0,0}\makebox(0,0)[lt]{\lineheight{1.25}\smash{\begin{tabular}[t]{l}$(S',\cA')$\end{tabular}}}}%
    \put(0.20639675,0.00476121){\color[rgb]{0,0,0}\makebox(0,0)[t]{\lineheight{1.25}\smash{\begin{tabular}[t]{c}$(S',\cA'')$\end{tabular}}}}%
    \put(0.79342934,0.00476121){\color[rgb]{0,0,0}\makebox(0,0)[t]{\lineheight{1.25}\smash{\begin{tabular}[t]{c}$(S',\cA''')$\end{tabular}}}}%
  \end{picture}%
\endgroup%

	\caption{Reducing $|J \cap \cA'|$ by 2, using the bypass relation. The
    annulus shown is a neighborhood of $J$ in the surface $S'$. Using a small
    isotopy, we may push the two bigons in each of the bottom two annuli out of
    the neighborhood of $J$ which is shown.}\label{fig::9}
\end{figure}
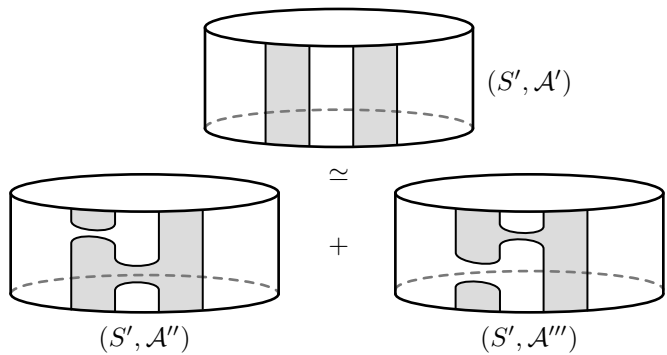

Let us write $S_{\ws}''$, $S_{\zs}''$, $S_{\ws}'''$, and $S_{\zs}'''$ for the
type-$\ws$ and type-$\zs$ subregions of $S' \setminus \cA''$ and $S' \setminus \cA'''$.
There are two cases to consider: when $\cA'$ has no closed components,
or when $\cA'$ has at least one  closed component.

Let us consider the case when $\cA'$ has no closed components. In this case,
by Lemma~\ref{lem:closedcomponentsbypass}, we know that exactly one of
$\cA''$ and $\cA'''$ has no closed components, while the other has a closed
component. For definiteness, let us say that $\cA''$ has no closed
components. Note that, in this case, $g(S_{\ws}') = g(S_{\ws}'')$ and
$g(S_{\zs}') = g(S_{\zs}'')$.

By our inductive hypothesis, we know that $U^i V^j \cdot F_{B^4,
(S',\cA'''),\frs_0} \simeq 0$. Combining this with
equation~\eqref{eq:bypasstoreduceintersections}, we conclude that
\[
U^i V^j \cdot F_{B^4,(S', \cA'),\frs_0} \simeq U^i V^j \cdot  F_{B^4,(S',\cA''),\frs_0}.
\]
By the inductive hypothesis, $U^i V^j \cdot F_{B^4,(S',\cA''),\frs_0}$
depends only on the integers $i + g(S_{\ws}'')$
and $j + g(S_{\zs}'')$, and hence the same holds for $U^i V^j \cdot
F_{B^4,(S',\cA'),\frs_0}$.

Next, we consider the case when $\cA'$ has a closed component. We
wish to show that
\[
U^i V^j \cdot F_{B^4,(S',\cA'),\frs_0} \simeq 0.
\]
By Lemma~\ref{lem:closedcomponentsbypass}, one of the following two cases holds:
Either $\cA''$ and $\cA'''$ both have a closed component, or neither $\cA''$ nor
$\cA'''$ has a closed component. If $\cA''$ and $\cA'''$ both have a closed
component, then $U^i V^j \cdot F_{B^4,(S',\cA''),\frs_0}$ and $U^i V^j \cdot
F_{B^4,(S',\cA'''),\frs_0}$ are both chain homotopic to zero, by induction. If neither $\cA''$
and $\cA'''$ have a closed component, we note that
$g(S_{\ws}'') = g(S_{\ws}''')$ and $g(S_{\zs}'') = g(S_{\zs}''')$, so
\[
U^i V^j \cdot F_{B^4,(S',\cA''),\frs_0} \simeq U^i V^j \cdot F_{B^4,(S',\cA'''),\frs_0}
\]
by induction. In both cases, the sum
\[
U^i V^j \cdot F_{B^4,(S',\cA''),\frs_0} + U^i V^j \cdot F_{B^4,(S',\cA'''),\frs_0}\simeq 0.
\]
 Hence, by equation~\eqref{eq:bypasstoreduceintersections},
 $U^i V^j \cdot F_{B^4,(S',\cA'),\frs_0} \simeq 0$, completing the proof.
\end{proof}

\subsection{Destabilizing genus bounds from \texorpdfstring{$\cI$}{I} and \texorpdfstring{$\kappa_0$}{kappa-nought}}

In this section, we show that the invariants $\kappa_0(S)$ and $\cI(S)$ give lower bounds on the quantity $g_{\dest}(S)$,
introduced in Definition~\ref{def:dest}. We begin with the invariant $\cI(S)$ from
Definition~\ref{def:depth}.

\begin{thm}\label{thm:I}
If $K$ is slice a knot in $\Sphere^3$ and $S \in \Surf(K)$, then
\[
\cI(S) \le g_{\dest}(S).
\]
\end{thm}

\begin{proof}
Let $S_1, \dots, S_n \in \Surf(K)$ be a sequence of surfaces as in Definition~\ref{def:dest}
connecting $S_1 = S$ with a slice disk $S_n = D$, such that
$g_{\dest}(S) = \max \{g(S_1),\dots, g(S_n)\}$.

The result follows immediately from Proposition~\ref{prop:2foldstabisstabtype},
with the following explanation.
The disk $S_n$ trivially satisfies the decoration-independence condition~\eqref{def:DI} above degree~$0$.
By Proposition~\ref{prop:2foldstabisstabtype}, if $S_k$ for $k \in \{2, \dots, n\}$ satisfies condition~\eqref{def:DI}
above degree~$d$ and $S_{k-1}$ is a stabilization of $S_k$, then $S_{k-1}$
satisfies condition~\eqref{def:DI} above degree~$\max\{d, g(S_{k-1})\}$. Using the
stabilization formula, Lemma~\ref{lem:stabilizationlemma}, the converse is
also true: If $S_k$ satisfies condition~\eqref{def:DI} above degree~$d$ and $S_{k-1}$ is a
destabilization of $S_k$, then $S_{k-1}$ also satisfies condition~\eqref{def:DI} above degree~$d$.
Hence, by induction, we see that $S = S_1$ satisfies condition~\eqref{def:DI}
above degree~$d = \max\{ g(S_1), \dots, g(S_n)\}$, and hence $\cI(S) \le g_{\dest}(S)$.
\end{proof}

We note that the invariant $\cI(S)$ is not easy to determine,
since it involves computing the cobordism maps for infinitely
many decorations on $S$. The invariant $\kappa_0(S)$
defined in Section~\ref{sec:kappa} is easier to compute
because it involves calculating just a single cobordism map
on $\HFK^-_{U=0}$, as opposed to infinitely many on $\cCFL^-$.
We now prove that $\kappa_0(S)$ also bounds $g_{\dest}(S)$:

\begin{thm}\label{thm:kappa0gdest}
If $K$ is a slice knot in $\Sphere^3$ and $S \in \Surf(K)$, then
\[
\kappa_0(S) \le g_{\dest}(S).
\]
\end{thm}

\begin{proof}
Suppose that $g(S) > 0$.
Recall that $\ve{t}_{S,\ws}^-$ is defined by
decorating $S$ with a dividing set consisting of a single
arc such that $g(S_{\ws})=g(S)$ and $g(S_{\zs})=0$. Suppose
that $S_1,\dots, S_n$ is a stabilization sequence of surfaces
in $\Surf (K)$ such that $S_1=S$ and $S_n$ is a slice disk.
Let
\[
d := \max\{g(S_1),\dots g(S_n)\}.
\]
By Theorem~\ref{thm:I}, the surface $S$ satisfies the decoration-independence condition \eqref{def:DI}
above degree~$d$. There are two cases:
$d=g(S)$ or $d>g(S)$. If $d=g(S)$, then the stabilization formula implies that
\[
\ve{t}_{S,\ws}^-\simeq U^{g(S)}\cdot \ve{t}_{S_n}^-,
\]
so $\ve{t}_{S,\ws}^-$ vanishes on $\HFK^-_{U=0}$, implying that
\[
\kappa_0(S) = g(S) = g_{\dest}(S).
\]

We now consider the second case, where $d>g(S)$.
We note that
\[
V^{d-g(S)}\cdot \ve{t}_{S,\ws}^- \simeq F_{B^4, (S',\cA_{\ws}')},
\]
where $(S',\cA_{\ws}')$ is obtained from $(S,\cA_{\ws})$
by performing $d - g(S)$ trivial 1-handle stabilizations along~$S_{\zs}$.
Since $(S',\cA_{\ws}')$ satisfies condition~\eqref{def:DI},
by definition the map $F_{B^4,(S',\cA_{\ws}')}$
depends only on the dividing set through the genera of the
type-$\ws$ and type-$\zs$ subregions. Hence, if $\cA'\subset S'$
is any other dividing set on $S'$ consisting of a single
arc, such that the genera of the type-$\ws$ and type-$\zs$
subregions are the same as those of $(S',\cA_{\ws}')$, then
\[
F_{B^4, (S',\cA_{\ws}')}\simeq F_{B^4,(S',\cA')}.
\]
We pick a dividing set $\cA'\subset S'$ such that one
of the trivial stabilizations of $S'$ occurs in the
type-$\ws$ subregion. See Figure~\ref{fig::41}.

\begin{figure}[ht!]
	\centering
\begingroup%
  \makeatletter%
  \providecommand\color[2][]{%
    \errmessage{(Inkscape) Color is used for the text in Inkscape, but the package 'color.sty' is not loaded}%
    \renewcommand\color[2][]{}%
  }%
  \providecommand\transparent[1]{%
    \errmessage{(Inkscape) Transparency is used (non-zero) for the text in Inkscape, but the package 'transparent.sty' is not loaded}%
    \renewcommand\transparent[1]{}%
  }%
  \providecommand\rotatebox[2]{#2}%
  \newcommand*\fsize{\dimexpr\f@size pt\relax}%
  \newcommand*\lineheight[1]{\fontsize{\fsize}{#1\fsize}\selectfont}%
  \ifx\svgwidth\undefined%
    \setlength{\unitlength}{425.29463801bp}%
    \ifx\svgscale\undefined%
      \relax%
    \else%
      \setlength{\unitlength}{\unitlength * \real{\svgscale}}%
    \fi%
  \else%
    \setlength{\unitlength}{\svgwidth}%
  \fi%
  \global\let\svgwidth\undefined%
  \global\let\svgscale\undefined%
  \makeatother%
  \begin{picture}(1,0.28167361)%
    \lineheight{1}%
    \setlength\tabcolsep{0pt}%
    \put(0.13805917,0.00189207){\color[rgb]{0,0,0}\makebox(0,0)[t]{\lineheight{1.25}\smash{\begin{tabular}[t]{c}$(S,\cA_{\ws})$\end{tabular}}}}%
    \put(0.49779082,0.00189207){\color[rgb]{0,0,0}\makebox(0,0)[t]{\lineheight{1.25}\smash{\begin{tabular}[t]{c}$(S',\cA_{\ws}')$\end{tabular}}}}%
    \put(0.86457673,0.00189207){\color[rgb]{0,0,0}\makebox(0,0)[t]{\lineheight{1.25}\smash{\begin{tabular}[t]{c}$(S',\cA')$\end{tabular}}}}%
    \put(0,0){\includegraphics[width=\unitlength,page=1]{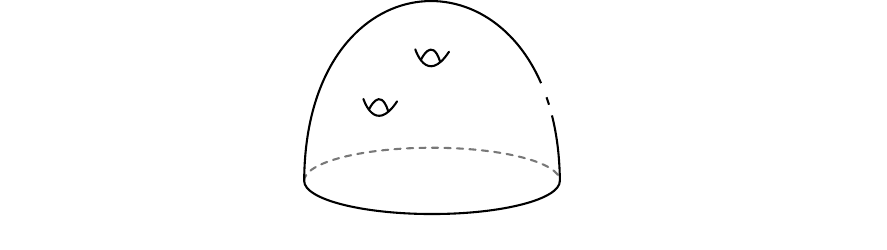}}%
    \put(0.61746113,0.0397123){\color[rgb]{0,0,0}\makebox(0,0)[lt]{\lineheight{1.25}\smash{\begin{tabular}[t]{l}$K$\end{tabular}}}}%
    \put(0,0){\includegraphics[width=\unitlength,page=2]{fig41.pdf}}%
    \put(0.27533377,0.0397123){\color[rgb]{0,0,0}\makebox(0,0)[lt]{\lineheight{1.25}\smash{\begin{tabular}[t]{l}$K$\end{tabular}}}}%
    \put(0,0){\includegraphics[width=\unitlength,page=3]{fig41.pdf}}%
    \put(0.96711797,0.0397123){\color[rgb]{0,0,0}\makebox(0,0)[lt]{\lineheight{1.25}\smash{\begin{tabular}[t]{l}$K$\end{tabular}}}}%
    \put(0,0){\includegraphics[width=\unitlength,page=4]{fig41.pdf}}%
  \end{picture}%
\endgroup%

	\caption{The surfaces $(S,\cA_{\ws})$, $(S',\cA_{\ws}')$, and $(S',\cA')$ from the proof of Theorem~\ref{thm:kappa0gdest}.}
    \label{fig::41}
\end{figure}

Using the stabilization formula,
we conclude that there is a decorated surface $\cF$ such that
\[
V^{d-g(S)}\cdot F_{B^4,(S,\cA_{\ws})}\simeq F_{B^4,(S',\cA')}\simeq U\cdot F_{B^4,\cF},
\]
from which we conclude that $V^{d-g(S)}\cdot t^-_{S,\ws} \simeq  0$
since the action of $U$ is trivial on $\CFK_{U=0}$.
\end{proof}

\begin{rem}
If $S$ is a genus $g > 0$ stabilization of a surface,
then  $\kappa_0(S) = g(S)$, since $\ve{t}_{S,\ws}^-
\simeq U^g \cdot G$ for some map $G$, so $[t^-_{S,\ws}(1)]=0$ in $\HFK^-_{U=0}(\bK)$.
Also, we note that if $S$ satisfies the decoration independence condition
\eqref{def:DI} at degree $d>g(S)$, then the map $V^{d-g(S)}\cdot \ve{t}_{S,\ws}^-$ vanishes on $\HFK^-_{U=0}$.
This follows by adapting the argument from the proof of Theorem~\ref{thm:kappa0gdest}.
Hence
\[
\kappa_0(S)\le \max\{g(S)+1,\cI(S)\}.
\]
\end{rem}

\subsection{Stabilization distance bounds from \texorpdfstring{$\tau$}{tau} and \texorpdfstring{$V_k$}{Vk}}
\label{sec:stabilizationdistanceproofs}
\begin{thm} \label{thm:tauD1D2bound}
Let $K$ be a knot in $\Sphere^3$, and let $S$, $S' \in \Surf(K)$. Then
\[
\tau(S, S') \le \mu_{\st}(S, S').
\]
\end{thm}

\begin{proof}
Let us write $m = \mu_{\st}(S, S')$.
Suppose that $S_1, \dots, S_k$ is a stabilization sequence of surfaces in $B^4$ connecting $S$ and
$S'$, as in Definition~\ref{def:stabgenus}, such that
\[
\max \{\, g(S_1), \dots, g(S_k) \,\} = m.
\]
Let $\bK$ denote $K$ decorated with two basepoints. By Lemma~\ref{lem:stabilizationlemma},
if $S_{i+1}$ is obtained from $S_i$ by a stabilization, then
the map $\ve{t}_{S_{i+1},\zs}^-$ is filtered chain homotopic to $V^{g(S_{i+1}) - g(S_i)} \cdot  \ve{t}_{S_i,\zs}^-$.
Similarly, if $S_{i+1}$ is obtained from $S_i$ by a destabilization, then
the map $\ve{t}_{S_{i+1},\zs}^-$ is filtered chain homotopic to $V^{g(S_i) - g(S_{i+1})} \cdot \ve{t}_{S_{i+1},\zs}^-$.
It follows that all of the maps $V^{m-g(S_i)}\cdot \ve{t}_{S_i,\zs}^-$ coincide
for $i\in \{1,\dots, n\}$. In particular,
\begin{equation}\label{eq:homotopyequivalencestabseqfullcomp}
V^{m-g(S)} \cdot \ve{t}_{S, \zs}^- \simeq V^{m-g(S')} \cdot \ve{t}_{S', \zs}^-.
\end{equation}
The map $\ve{t}_{S_i,\zs}^-$ on $\cCFL^-$ increases the Alexander grading by $g(S_i)$,
so $V^{-g(S_i)} \cdot \ve{t}_{S_i,\zs}^-$ determines a well-defined map from $\bF_2[\hat{U}]$ into $C(\bK, R_{0,-g(S_i)})\subset \CFK^\infty(\bK)$. Hence, from equation~\eqref{eq:homotopyequivalencestabseqfullcomp}, we conclude that the induced elements $[V^{-g(S)}\cdot \ve{t}_{S,\zs}^-(1)]$ and $ [V^{-g(S')}\cdot \ve{t}_{S',\zs}^-(1)]$ coincide in $H_*(C(\bK,\I{m}))$. By Lemma~\ref{lem:tausubquotientshapes}, this implies that $\tau(S,S')\le m$, completing the proof.
\end{proof}

A different algebraic perspective on the previous proof can be given using the formulation
of $\tau(S, S')$ in terms of $\HFK^-_{U=0}(\bK)$ described in Lemma~\ref{lem:slightreformulationoftau},
and the computation of the effect of stabilizations from Lemma~\ref{lem:stabilizationlemma}.

We now show that the local $h$-invariants give a lower bound on the stabilization
distance between two slice disks:

\begin{thm}\label{thm:Vkbound}
If $D$ and $D'$ are slice disks of $K$ and $k \le \mu_{\st}(D,D')$, then
\[
V_k(D,D') \le \left\lceil \frac{\mu_{\st}(D,D')-k}{2} \right \rceil.
\]
If $k\ge \mu_{\st}(D,D')$, then $V_k(D,D')=0$.
\end{thm}

\begin{proof}[Proof of Theorem~\ref{thm:Vkbound}]
Suppose first that $k\le \mu_{\st}(D,D')$,
and that $S_1, \dots, S_n$ is a sequence of embedded surfaces in
$\Surf(K)$ such that $S_{i+1}$ is either obtained from $S_i$
by a stabilization or destabilization. Further, we assume that $S_1 = D$ and
$S_n = D'$. Let $d$ denote $\max \{g(S_1),\dots, g(S_n)\}$.
Since $S_n$ is a slice disk, $g_{\dest}(S_i) \le d$ for $i \in \{1, \dots, n\}$.
Furthermore, $\cI(S_i) \le g_{\dest}(S_i)$ by Theorem~\ref{thm:I}, hence
$S_i$ satisfies the decoration-independence condition \eqref{def:DI} above degree~$d$.

Next, we fix an integer $k$ such that $0 \le k \le d$. By increasing $d$ by 1,
if necessary, we may assume that $(d-k)/2$ is an integer.
Let $\bK$ denote $K$ decorated with two basepoints. We decorate
each surface $S_i$ with a single dividing arc $\cA_i$, and we pick
nonnegative integers $n_i$ and $m_i$ such that $g(S_i) + n_i + m_i = d$ and
\[
g(S_{i,\ws}) + n_i = \frac{d-k}{2} \text{ and } g(S_{i,\zs}) + m_i = \frac{d+k}{2}.
\]
We note that, since each $S_i$ satisfies condition~\eqref{def:DI} above degree~$d$, it
follows that the map $U^{n_i} V^{m_i} \cdot F_{B^4, (S_i,\cA_i), \frs_0}$ depends on
the dividing set $\cA_i$ only up to the quantities $n_i + g(S_{i,\ws})$ and
$m_i + g(S_{i,\zs})$, and is independent of the choice of $\cA_i$.
Using the stabilization formula, Lemma~\ref{lem:stabilizationlemma}, it thus
follows that all of the maps $U^{n_i} V^{m_i} \cdot  F_{W,(S_i,\cA_i),\frs_0}$
are chain homotopic. In particular,
\begin{equation}
U^{(d-k)/2} V^{(d+k)/2} \cdot \ts_{D}^{\infty} \simeq
U^{(d-k)/2} V^{(d+k)/2} \cdot \ts_{D'}^{\infty}.
\label{eq:tdsequal}
\end{equation}

Note that $U^{(d-k)/2} V^{(d+k)/2} \cdot \ts_{D_i}^{\infty}(1)$ is not an
element of $\CFK^\infty(\bK)$, since it lives in Alexander grading $k$.
In fact, $U^{(d-k)/2} V^{(d+k)/2} \cdot \ts_{D_i}^{\infty}(1)$ is an element of
the subcomplex of $\cCFL^-(\bK)$ of Alexander grading $k$. Multiplication by $V^{-k}$ gives a chain
isomorphism between the subset of
$\cCFL^-(\bK)$ in Alexander grading $k$ and the
subcomplex of $\cCFL^\infty( \bK)$ generated over $\bF_2$ by elements
$U^i V^j\cdot \xs$ with $A(\xs) + (i - j) = 0$, $i \ge 0$, and $j \ge - k$. The latter is
$A_k^-(\bK)$, by definition. Hence, from equation~\eqref{eq:tdsequal}, it follows that
\[
\hat{U}^{(d-k)/2} \cdot [\ts_{D}^{\infty}(1)] = \hat{U}^{(d-k)/2} \cdot [\ts_{D'}^{\infty}(1)] \in
H_*(A_k^-(\bK)),
\]
where $\hat{U} = U V$. It follows that
\[
V_k(D,D') \le \frac{d-k}{2},
\]
completing the proof when $k\le \mu_{\st}(D,D')$.

The statement for $k \ge \mu_{\st}(D,D')$ follows from the statement for
$k = \mu_{\st}(D,D')$, together with the monotonicity result from Lemma~\ref{lem:monotonicity}.
\end{proof}

\section{Regular homotopies and the double point distance}\label{sec:doublepoints}

\subsection{The double point distance}
If $K$ is a knot in $\Sphere^3$, we denote by $\Imm(K)$ the set of \emph{immersed}
connected surfaces in $B^4$ with boundary $K$. Furthermore, for $g \in \N$, we write $\Imm_g(K)$ for
the subset of $\Imm(K)$ consisting of genus $g$ surfaces.
If $S$, $S' \in \Imm(K)$, then a \emph{regular homotopy} from $S$ to $S'$ is a 1-parameter family
$\{\, S_t : t \in I \,\}$ in $\Imm(K)$ that is continuous in the $C^\infty$-topology,
and such that $S_0 = S$ and $S_1 = S'$. If $S$ and $S'$ are regularly homotopic,
then $g(S) = g(S')$. For a \emph{generic} regular homotopy, at all but finitely many $t$,
the surface $S_t$ is embedded away from finitely many transverse double points.
At finitely many $t$, the immersion $S_t$ has a single non-transverse double point,
where a pair of double points of opposite signs is created or canceled.
In particular, the algebraic number of double points is constant along a generic regular homotopy.

\begin{lem}
  Let $K$ be a knot in $\Sphere^3$, and let $g \in \N$. Then any two surfaces $S$, $S' \in \Surf_g(K)$
  are regularly homotopic relative to $K$.
\end{lem}

\begin{proof}
  By extending the proof of Hirsch~\cite[Theorem~8.2]{HirschImmersions} using the relative version
  of his $h$-principle~\cite[Theorem~5.9]{HirschImmersions},
  we obtain that the regular homotopy class of $S$ relative to $K$
  is determined by the relative normal Euler class of $S$. Since $S$ is embedded and $[S] = 0$
  in $H_2(B^4, \d B^4)$, there is a 3-manifold-with-boundary $M$ embedded in $B^4$ such
  that $S \subset \d M$ and $\d M \setminus S \subset \Sphere^3$ is a Seifert surface of~$K$.
  In particular, $M$ induces a normal framing of $S$ that restricts to the Seifert framing along $K$.
  Hence, the normal Euler class of $S$ relative to the Seifert framing vanishes.
  Since the same holds for $S'$, we obtain that $S$ and $S'$ are regularly homotopic relative to $K$.
\end{proof}

The regular homotopy class of a generic immersed surface $S \in \Imm(K)$ is determined by the algebraic
number of its double points. If $\{\, S_t : t \in I \,\}$ is a regular homotopy
such that $S_0$ is embedded, then the algebraic number of double points of $S_t$ is zero
for every $t \in I$ where $S_t$ is generic.

\begin{define}\label{def:sing}
Given an immersed surface $S \in \Imm(K)$, let $\Sing(S)$
be the set of its double points (this might be infinite when $S$ is not generic).
If $S$, $S' \in \Surf_g(K)$, then we define
\[
\tilde{\mu}_{\Sing}(S,S') := \frac12 \min_{\{\, S_t : t \in I \,\}} \max \{\, |\Sing(S_t)| : t \in I \,\},
\]
where the minimum is taken over all generic regular homotopies $\{\, S_t : t \in I \,\}$
such that $S_0 = S$ and $S_1 = S'$. Furthermore, we set \
\[
\mu_{\Sing}(S,S') = \tilde{\mu}_{\Sing}(S,S') + g.
\]
When $S$, $S' \in \Surf(K)$ and $g(S) \neq g(S')$, we set $\mu_{\Sing}(S, S') = \infty$.
We call $\mu_{\Sing}(S,S')$ the \emph{double point distance} between $S$ and $S'$.
\end{define}

Since $\tilde{\mu}_{\Sing}(S, S') = 0$ if and only if $S$ and $S'$ are isotopic, the
function $\tilde{\mu}_{\Sing}$ is an ultrametric on $\Surf_g(K)$ for every $g \in \N$.
Furthermore, $\mu_{\Sing}$ is a metric filtration whose normalization is $\tilde{\mu}_{\Sing}$.
The goal of this section is to prove that, if $S$, $S' \in \Surf_g(K)$, then
\begin{equation}\label{eq:tauboundsmusing}
\tau(S,S') \le \mu_{\Sing}(S,S').
\end{equation}
 If $g>0$, we will also show that
\begin{equation}\label{eq:kappaboundsmusing}
\kappa(S,S') \le \mu_{\Sing}(S,S').
\end{equation}
Equations~\eqref{eq:tauboundsmusing} and \eqref{eq:kappaboundsmusing}
are proven in Theorems~\ref{thm:doublepointbound} and~\ref{thm:kappaandmuSing}.
Finally, in Theorem~\ref{thm:hinvdoublepointhighergenus}, we will show that
the local $h$-invariants also give lower bounds on $\mu_{\Sing}(S,S')$.

\subsection{Movies of immersions and regular homotopies}

Suppose that $S \in \Imm(K)$ is the image of a proper immersion
\[
f \colon \bar{S} \to B^4.
\]
Let $B' \subset \Int(B^4)$ be a ball disjoint from $S$.
After a suitable identification between $B^4 \setminus \Int(B')$ and $I \times S^3$,
we can view $S$ as an immersed surface in $I \times \Sphere^3$,
satisfying $S \cap (\{0\} \times \Sphere^3) = \emptyset$
and $S \cap (\{1\} \times \Sphere^3) = K$. We can visualize
$S$ by considering the movie $\{\, S^s : s \in I\,\}$,
where $S^s$ is obtained by projecting $(\{s\} \times \Sphere^3) \cap S$ into $\Sphere^3$.
We orient $S^s$ as the boundary of $([0, s] \times \Sphere^3) \cap S$
using the outward-normal-first convention.

If $S$ is generic, then $\pi_I \circ f$ is a Morse function on $\bar{S}$
and the double points are on regular level sets,
where $\pi_I \colon I \times \Sphere^3 \to I$ is the projection onto the $I$-factor.
Hence $S^s$ is an immersed link whenever $s$ is a regular value. If $s$ is a critical value of
index zero, then an unknotted component is born. If $s$ has index one,
the link undergoes a saddle move, and if it has index two,
an unknotted component dies. Generically, passing a double point of $S$ locally corresponds
to a crossing change of $S^s$; see Gompf--Stipsicz \cite[Figure~6.25]{GompfStipsicz}.
We now explain why this is true, and how to read off the intersection sign.

\begin{lem}\label{lem:double}
Generically, as we pass a positive (negative) double point $p$ of $S$,
a negative (positive) crossing of $S^s$ changes to a positive (negative) crossing; see Figure~\ref{fig::2}.
\end{lem}

\begin{proof} 
Suppose that $S$ is the image of an immersion $f \colon \bar{S} \looparrowright B^4$.
The set of points $x \in \bar{S}$ such that $S$ is not
transverse to the sets $\{s\} \times \Sphere^3$ at $f(x)$
is generically 0-dimensional, and hence disjoint from the
two preimages of $p$. Hence, generically, passing the double
point $p$ corresponds to two strands of $S^s$ passing through each other.

Write  the double point $p \in S \subseteq I \times \Sphere^3$ as $p = (s_0,p_0)$, where $s_0 \in I$ and $p_0 \in \Sphere^3$. Suppose that, at $s = s_0$, a negative crossing of $S^s$ turns into a positive crossing. Let $v_+$, $v_- \in T_{p_0} \Sphere^3$
denote oriented tangent vectors for the upper and lower strands of the crossing, respectively.
Let $\gamma \colon (s_0-\epsilon, s_0+\epsilon) \to \Sphere^3$ denote
the trajectory of a point on the upper strand, chosen to pass through
a point on the lower strand at $s_0$. By inspection of the crossing
change, the triple $(v_+,v_-,\gamma'(s_0))$ is a positive basis of
$T_{p_0} \Sphere^3$. Using the product orientation on $I \times \Sphere^3$,
the 4-tuple $(\d/\d s, v_+, v_-, \gamma'(s_0))$ is an oriented basis
for $I \times \Sphere^3$. It is easy to see that oriented bases
for the tangent spaces of the two sheets of $S$ at $p$ are
\[
(\d/\d s+\gamma'(s_0),v_+)\quad \text{and} \quad (\d/\d s-\gamma'(s_0), v_-),
\]
respectively, which concatenate to form a positive basis of $I \times \Sphere^3$.

A similar argument applies when a positive crossing turns into a negative one at $p$.
\end{proof}

Now suppose that $\{\, S_t : t \in [-1,1] \,\}$ is a generic regular homotopy in $\Imm(K)$,
and that a pair of double points $p_+$ and $p_-$ appear as $t$ passes $0\in [-1,1]$.
The immersed surface $S_{0}$ has a non-transverse double point $p\in B^4$.
Write $p=(s_0,p_0)$, where $s_0 \in I$ and $p_0 \in \Sphere^3$.

A local model for a double point creation
can be visualized via a 2-parameter family
\[
\{\, S_t^s : (s,t) \in [s_0-\epsilon,s_0+ \epsilon] \times [-\epsilon,\epsilon] \,\}
\]
of immersed links in $\Sphere^3$ that is constant outside a neighborhood $N(p_0)$ containing the crossing.
The families  $\{\, S_t^{s_0-\epsilon} : t \in [-\epsilon, \epsilon] \,\}$
and $\{\, S_t^{s_0 + \epsilon} : t \in [-\epsilon, \epsilon] \,\}$
are constant and have a positive crossing in $N(p_0)$; we denote this link by $L_+$.
For $t < 0$, the intersection $S_t^s \cap N(p_0)$ is a positive crossing
and the family of links $S_t^s$ is embedded (and hence isotopic to $L_+$) for all $s \in [s_0-\epsilon, s_0+\epsilon]$.
For $t > 0$, the positive crossing $L_+ \cap N(p_0)$ changes
to a negative crossing, and then back to a positive crossing. Let $L_-$ be the link
obtained by changing $L_+ \cap N(p_0)$ to a negative crossing.
If we fix $t$, self-intersections in the 1-parameter family
$\{\, S_t^s \colon s \in [s_0 - \epsilon, s_0 + \epsilon] \,\}$ correspond to double points of the surface
the family traces out in $[s_0 - \epsilon, s_0 + \epsilon] \times \Sphere^3$.
The movie $\{\, S_t^s : s \in [s_0 - \epsilon, s_0 + \epsilon] \,\}$ for $t > 0$ is shown in the top of Figure~\ref{fig::2}.
We prove the above in the following lemma.

\begin{figure}[ht!]
	\centering
	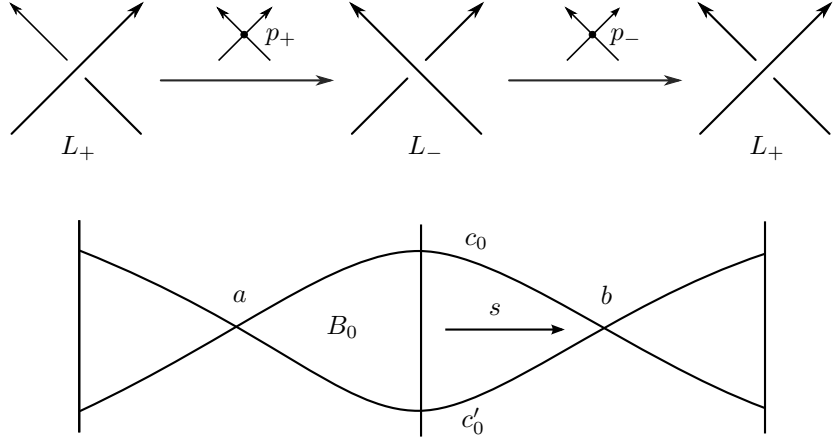
	\caption{The top row shows a movie of a pair of double points after they have been born during
    a regular homotopy of an immersed surface. The bottom row shows the standard model of a Whitney disk
    used in the proof, which gives a canonical neighborhood of the pair of canceling double points.}
    \label{fig::2}
\end{figure}

\begin{lem}\label{lem:birth}
Let $\{\, S_t : t \in [-1,1] \,\}$ be a generic regular homotopy of immersed surfaces such
that a pair of double points is born at time $0$ at a point $p \in B^4$.
Furthermore, let $B' \subset \Int(B^4)$ be a ball disjoint from $S_t$ for every $t \in [-1,1]$.
Then there is an identification of $B^4 \setminus \Int(B')$ with $I \times \Sphere^3$ and an $\epsilon > 0$
such that, if $p$ corresponds to $(s_0, p_0)$, the 1-parameter family of immersed links
\[
\{\, S_t^s : s \in [s_0 - \epsilon, s_0 + \epsilon] \,\}
\]
is diffeomorphic to the constant 1-parameter family $L_+ := S_{-\epsilon}^{s_0 - \epsilon}$ for $t = -\epsilon$,
where $L_+$ has a positive crossing in $N(p_0)$,
and to the 1-parameter family shown on the top row of Figure~\ref{fig::2} for $t = \epsilon$,
where the positive crossing of $L_+$ in $N(p_0)$ changes to negative, and then back to positive.
\end{lem}

\begin{proof}
  Choose an identification between $B^4 \setminus \Int(B')$ and $I \times S^3$ such that
  $s(p_-) < s(p_+)$, where $s$ is the $I$-coordinate. We write $S_t$ as the image of a 1-parameter
  family of immersions $f_t \colon \bar{S} \looparrowright I \times \Sphere^3$. Up to
  isotopy, we can express any regular homotopy of a surface in a 4-manifold
  as a composition of finger moves and Whitney moves; see Gabai~\cite[Proposition~4.3]{LightBulb}
  and Freedman--Quinn~\cite[Section~1.5]{FQ}. In particular, $p_+$ and $p_-$ admit a Whitney disk $B$.
  A neighborhood of the Whitney disk can be put in the standard form of Milnor~\cite[Lemma~6.7]{Milnor};
  see the bottom row of Figure~\ref{fig::2}.
  This is given by an embedding $\varphi \colon U \times \R \times \R \to B^4$,
  where $U$ is a neighborhood of the disk $B_0$ in $\R^2$ enclosed by arcs $c_0$ and $c_0'$
  that transversely intersect at points $a$ and $b$. We have $\varphi(B_0) = B$,
  $\varphi(a) = p_+$ and $\varphi(b) = p_-$, and let us write $c = \varphi(c_0)$ and $c' = \varphi(c_0')$.
  The preimages of the two branches of $S_1$ meeting at $p_+$ and $p_-$ are $(U \cap c_0) \times \R \times \{0\}$
  and $(U \cap c_0') \times \{0\} \times \R$, respectively. The isotopy $S_t$ at the finger move
  is modeled on the isotopy of $c_0'$ shown in \cite[Figure~6.3]{Milnor}
  that creates the intersection points $a$ and $b$ with $c_0$.
  This isotopy is constant in the normal $\R \times \R$ direction.

The immersed surface $S_0$ has a non-transverse double point.
If $x$, $y \in \bar{S}$ are the two preimages of the double point,
let $v$ denote a generator of the 1-dimensional vector space $(f_0)_*(T_x\bar{S})\cap (f_0)_*(T_y \bar{S})$.
For $t>0$, the movie $S_t^s$ has an extra pair of double points.
By the choice of the coordinate function~$s$, we have $ds(v) \neq 0$ and
$s(p_-) < s(p_+)$. Both $p_+$ and $p_-$ correspond to a crossing change in the movie
$\{\, S^s_t : s \in [s_0-\epsilon, s_0+\epsilon] \,\}$ for $t > 0$ by
Lemma~\ref{lem:double}. By arranging for the Whitney disk to be symmetric
about $s = s_0$ in a small neighborhood of $s_0$, the movie for the second
double point is obtained by reversing the movie for the first double point.
When $t < 0$, the curves $c_0$ and $c_0'$ become disjoint,
and so the movie $\{\, S_t^s : s \in [s_0-\epsilon,s_0+ \epsilon] \,\}$ is just an isotopy of the link $L_+$,
completing the proof.
\end{proof}

\subsection{The desingularization of an immersed surface}

\begin{define}\label{def:desingularization}
Suppose $S \in \Imm(K)$ is a generic immersed surface in $B^4$; i.e., an immersion
with only transverse double points. The \emph{desingularization} of $(B^4, S)$
is the link cobordism $(B^4(S), \hat{S})$ obtained as follows:
\begin{enumerate}
\item The 4-manifold $B^4(S)$ is constructed by blowing up the 4-manifold at each \emph{negative} double point of $S$.
  Topologically, this corresponds to connected summing with $\overline{\CP}^2$.
\item The surface $\hat{S}$ is constructed from the proper transform of $S$ in $B^4(S)$
  by resolving each \emph{positive} double point (increasing the genus of $S$ by 1 at each point).
\end{enumerate}
\end{define}

Definition~\ref{def:desingularization} makes sense for any immersed oriented link cobordism as well.
For a movie presentation of the resolution of a positive double point, see Figure~\ref{fig::3},
taken from the book of Gompf and Stipsicz \cite[Figure~6.30]{GompfStipsicz}.
For a movie of the blowup of a negative double point, see Figure~\ref{fig::4}. The 4-dimensional
2-handle of $\overline{\CP}^2$ is attached along a $(-1)$-framed unknot that links the negative crossing of $L_+$.

\begin{figure}[ht!]
	\centering
\begingroup%
  \makeatletter%
  \providecommand\color[2][]{%
    \errmessage{(Inkscape) Color is used for the text in Inkscape, but the package 'color.sty' is not loaded}%
    \renewcommand\color[2][]{}%
  }%
  \providecommand\transparent[1]{%
    \errmessage{(Inkscape) Transparency is used (non-zero) for the text in Inkscape, but the package 'transparent.sty' is not loaded}%
    \renewcommand\transparent[1]{}%
  }%
  \providecommand\rotatebox[2]{#2}%
  \newcommand*\fsize{\dimexpr\f@size pt\relax}%
  \newcommand*\lineheight[1]{\fontsize{\fsize}{#1\fsize}\selectfont}%
  \ifx\svgwidth\undefined%
    \setlength{\unitlength}{238.39565023bp}%
    \ifx\svgscale\undefined%
      \relax%
    \else%
      \setlength{\unitlength}{\unitlength * \real{\svgscale}}%
    \fi%
  \else%
    \setlength{\unitlength}{\svgwidth}%
  \fi%
  \global\let\svgwidth\undefined%
  \global\let\svgscale\undefined%
  \makeatother%
  \begin{picture}(1,0.56032218)%
    \lineheight{1}%
    \setlength\tabcolsep{0pt}%
    \put(0,0){\includegraphics[width=\unitlength,page=1]{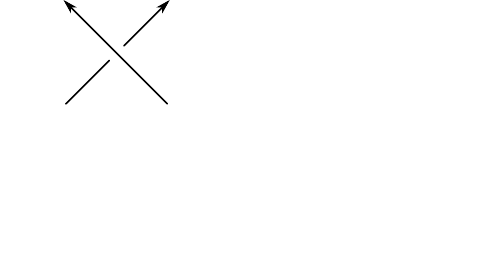}}%
    \put(0.22110793,0.31979004){\color[rgb]{0,0,0}\makebox(0,0)[lt]{\lineheight{0}\smash{\begin{tabular}[t]{l}$L_-$\end{tabular}}}}%
    \put(0,0){\includegraphics[width=\unitlength,page=2]{fig3.pdf}}%
    \put(0.74477925,0.31979004){\color[rgb]{0,0,0}\makebox(0,0)[lt]{\lineheight{0}\smash{\begin{tabular}[t]{l}$L_+$\end{tabular}}}}%
    \put(0,0){\includegraphics[width=\unitlength,page=3]{fig3.pdf}}%
    \put(0.53271185,0.49996121){\color[rgb]{0,0,0}\makebox(0,0)[lt]{\lineheight{1.25}\smash{\begin{tabular}[t]{l}$p_+$\end{tabular}}}}%
    \put(0,0){\includegraphics[width=\unitlength,page=4]{fig3.pdf}}%
    \put(0.0932442,0.01350111){\color[rgb]{0,0,0}\makebox(0,0)[lt]{\lineheight{0}\smash{\begin{tabular}[t]{l}$L_-$\end{tabular}}}}%
    \put(0,0){\includegraphics[width=\unitlength,page=5]{fig3.pdf}}%
    \put(0.86859805,0.00720913){\color[rgb]{0,0,0}\makebox(0,0)[lt]{\lineheight{0}\smash{\begin{tabular}[t]{l}$L_+$\end{tabular}}}}%
    \put(0,0){\includegraphics[width=\unitlength,page=6]{fig3.pdf}}%
  \end{picture}%
\endgroup%

	\caption{Resolving a positive double point. The top row is the singular knot cobordism.
    The bottom is our choice of resolution.}\label{fig::3}
\end{figure}

\begin{figure}[ht!]
	\centering
	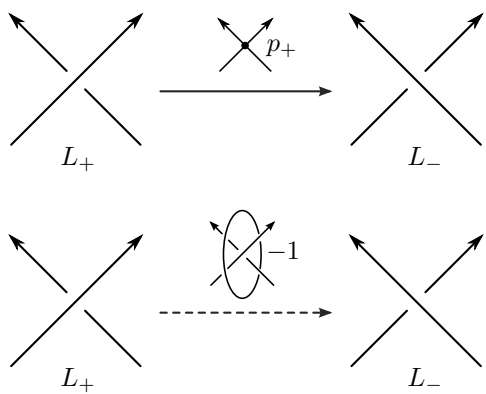
	\caption{Resolving a negative double point. The top is the singular knot cobordism.
    The bottom is the resolution, obtained by blowing up the surface at the double point.}\label{fig::4}
\end{figure}

Let $(W,\cF)$ with $\cF = (S, \cA)$ be an immersed, decorated link cobordism,
such that the double points are disjoint from $\cA$. Furthermore, suppose that
the two branches meeting at a double point either both lie in $S_{\ws}$, or
they both lie in $S_{\zs}$. We write $(W(S), \hat{\cF})$ for the decorated
link cobordism with double points resolved as described above.

\begin{lem}\label{lem:computedesingularization}
Suppose $(W_0,\cF_0)$ is a non-singular link cobordism, and $(W,\cF)$ with $\cF = (S, \cA)$ is
obtained from $(W_0,\cF_0)$ by a double point birth, corresponding to a tangency
between either two branches of $S_{\ws}$, or two branches of $S_{\zs}$.
Let $(W(S),\hat{\cF})$ denote the resolved link cobordism, as described above.
Then $W(S) = W_0 \# \overline{\CP}^2$, and
$\hat{\cF}$ is obtained from $\cF_0$ by a 1-handle stabilization
along $S_{\ws}$ or $S_{\zs}$, and disjoint from $\overline{\CP}^2$.

Let $\hat{\frs}$ be a $\Spin^c$ structure on $\hat{W}$ such that
$\langle\, c_1(\hat{\frs}), E \,\rangle = \pm 1$, where $E$ denotes the exceptional divisor
in $\hat{W}$, and agrees with $\frs$ on $W$. Then
\[
F_{\hat{W}, \hat{\cF}, \hat{\frs}}\simeq  U\cdot F_{W,\cF,\frs}
\]
if the double points both occur in  $S_{\ws}$, and
\[
F_{\hat{W}, \hat{\cF}, \hat{\frs}} \simeq V \cdot F_{W,\cF,\frs},
\]
if the double points both occur in $S_{\zs}$.
\end{lem}

\begin{proof}
By Lemma~\ref{lem:birth}, there is a movie presentation of $\cF$ as in the top of Figure~\ref{fig::2},
where a positive crossing changes to negative, and then back to positive.
Furthermore, the movie of the decorated surface $\cF_0$ only differs from that of $\cF$
by locally changing the above movie to one where the positive crossing stays positive.

We consider the composition of the resolutions of the positive and negative
double points shown in Figures~\ref{fig::3} and \ref{fig::4}; see the top
row of Figure~\ref{fig::movie}. We can arrange
that the resolution of the negative double point occurs immediately before
the resolution of the positive double point. The resolution of the positive
double point is a pair of saddles, corresponding to attaching bands~$B_1$ and~$B_2$.
The composition of the two resolutions can be rearranged such that we first attach the band $B_1$,
then attach a 4-dimensional 2-handle along the $-1$ framed unknot $\cU$, and
finally attach the band $B_2$; see the second row of Figure~\ref{fig::movie}.

\begin{figure}[ht!]
	\centering
\begingroup%
  \makeatletter%
  \providecommand\color[2][]{%
    \errmessage{(Inkscape) Color is used for the text in Inkscape, but the package 'color.sty' is not loaded}%
    \renewcommand\color[2][]{}%
  }%
  \providecommand\transparent[1]{%
    \errmessage{(Inkscape) Transparency is used (non-zero) for the text in Inkscape, but the package 'transparent.sty' is not loaded}%
    \renewcommand\transparent[1]{}%
  }%
  \providecommand\rotatebox[2]{#2}%
  \newcommand*\fsize{\dimexpr\f@size pt\relax}%
  \newcommand*\lineheight[1]{\fontsize{\fsize}{#1\fsize}\selectfont}%
  \ifx\svgwidth\undefined%
    \setlength{\unitlength}{371.82357294bp}%
    \ifx\svgscale\undefined%
      \relax%
    \else%
      \setlength{\unitlength}{\unitlength * \real{\svgscale}}%
    \fi%
  \else%
    \setlength{\unitlength}{\svgwidth}%
  \fi%
  \global\let\svgwidth\undefined%
  \global\let\svgscale\undefined%
  \makeatother%
  \begin{picture}(1,0.63513662)%
    \lineheight{1}%
    \setlength\tabcolsep{0pt}%
    \put(0,0){\includegraphics[width=\unitlength,page=1]{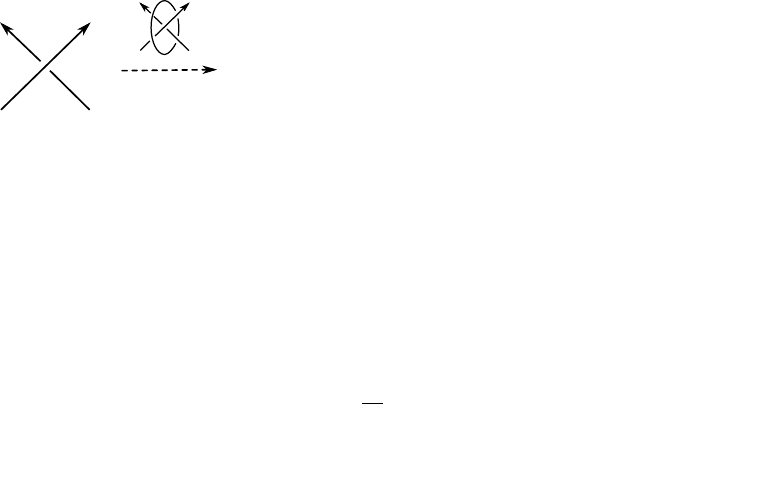}}%
    \put(0.23612834,0.59530439){\color[rgb]{0,0,0}\makebox(0,0)[lt]{\lineheight{0}\smash{\begin{tabular}[t]{l}$-1$\end{tabular}}}}%
    \put(0,0){\includegraphics[width=\unitlength,page=2]{movie.pdf}}%
    \put(0.53493506,0.35498992){\color[rgb]{0,0,0}\makebox(0,0)[lt]{\lineheight{0}\smash{\begin{tabular}[t]{l}$-1$\end{tabular}}}}%
    \put(0,0){\includegraphics[width=\unitlength,page=3]{movie.pdf}}%
    \put(0.81532514,0.15625882){\color[rgb]{0,0,0}\makebox(0,0)[lt]{\lineheight{0}\smash{\begin{tabular}[t]{l}$-1$\end{tabular}}}}%
    \put(0,0){\includegraphics[width=\unitlength,page=4]{movie.pdf}}%
  \end{picture}%
\endgroup%

	\caption{In the top row, we show the movie of the resolution of a canceling
    pair of double points, which consists of a blowup followed by two band moves.
    We then commute the blowup and the first band move, giving rise to the movie in
    the second row. Finally, we slide the second band over the $(-1)$-framed 2-handle,
    giving the third row. The two bands form a tube, and the blowup now happens away
    from the surface.}\label{fig::movie}
\end{figure}

We can slide the band $B_2$ over $\cU$, though it gains a full right-handed
twist when we do this; see the third row of Figure~\ref{fig::movie}.
The unknot $\cU$ is now totally unlinked from the knot
and bands, and $B_1$ and $B_2$ are simply dual bands,
corresponding to a 1-handle stabilization of the surface $\cF_0$
disjoint from the $-1$ framed unknot giving the $\overline{\CP}^2$ summand of $W(S)$.

Applying Lemma~\ref{lem:stabilizationlemma} for the effect of the 1-handle stabilization,
and using the standard blow-up formula for $-1$ surgery on an unknot
contained in a ball disjoint from the link, we see that the composition is multiplication by either $U$ or
$V$, depending on whether the 1-handle is added to
$\Sigma_{\ws}$ or $\Sigma_{\zs}$. Using the composition law, the proof is complete.
\end{proof}

\subsection{Tau, nu, and the double point distance}
We now prove that $\tau$ gives a lower bound on~$\mu_{\Sing}$:

\begin{thm}\label{thm:doublepointbound}
If $S$, $S' \in \Surf_g(K)$, then
\[
 \tau(S, S') \le \mu_{\Sing}(S, S').
\]
\end{thm}

\begin{proof}
Suppose that $\{\, S_t : t \in I \,\}$ is a generic regular homotopy between
embedded surfaces $S$, $S' \in \Surf(K)$. The immersion $S_t$ fails to
be self-transverse at times $s_1, \dots, s_{n-1} \in (0,1)$.
Pick a point $t_i \in (s_{i-1}, s_i)$ for every $i \in \{2, \dots, n-1\}$,
and let $S_i = S_{t_i}$. We write $S_1 = S$ and $S_n = S'$.
Let $(B^4(S_i), \hat{\cF}_i)$ denote the desingularization of $(B^4, S_i)$, as described in
Definition~\ref{def:desingularization}. Let $\hat{\frs}_i$ denote any
maximal $\Spin^c$ structure on $B^4(S_i)$ (i.e.,
$c_1(\frs)^2 + b_2(B^4(S_i)) = 0$), such that $\hat{\frs}_{i+1}$ is obtained by
blowing up or blowing down $\hat{\frs}_i$.

We decorate each $\hat{\cF}_i$ such that the type-$\ws$ region is a bigon along
$K$, and the rest of $\hat{\cF}_i$ is of type-$\zs$. In particular, all double
points occur in regions of type-$\zs$. If $S_i$ is obtained from $S_{i-1}$
via a double point birth, then Lemma~\ref{lem:computedesingularization} implies that
\[
F_{B^4(S_{i+1}), \hat{\cF}_{i+1}, \hat{\frs}_{i+1}} \simeq
V \cdot F_{B^4(S_i), \hat{\cF}_{i}, \hat{\frs}_i}.
\]
Similarly, if $S_{i+1}$ is obtained from $S_i$ via a double point cancellation, then
\[
V \cdot F_{B^4(S_{i+1}), \hat{\cF}_{i+1}, \hat{\frs}_{i+1}} \simeq
F_{B^4(S_i), \hat{\cF}_{i}, \hat{\frs}_i}.
\]
It follows that
\[
V^n \cdot \ve{t}_{S,\zs}^\infty \simeq V^n \cdot \ve{t}_{S',\zs}^\infty,
\]
where $n$ is the maximal number of positive double points of any $S_i$. Since the algebraic count of double points
of each $S_i$ is zero, we have $n = \frac{1}{2} \max \{\, |\Sing(S_t)| : t \in I \,\}$.
It now follows from Lemma~\ref{lem:slightreformulationoftau} that
\[
\tau(S, S') \le \frac{1}{2} \max \{\, |\Sing(S_t)| : t \in I \,\} + g.
\]
Hence $\tau(S, S') \le \mu_{\Sing}(S, S')$, as claimed.
\end{proof}

We now show that $\nu$, introduced in Section~\ref{sec:tau+},
gives a slightly better lower bound on the stabilization distance and the double point distance
than $\tau$.

\begin{prop}\label{prop:tau+}
If $S$, $S' \in \Surf(K)$, then
\begin{equation}
\nu(S, S') \le \min\{\mu_{\st}(S, S'), \mu_{\Sing}(S, S')\}.
\label{eq:nuinequality}
\end{equation}
\end{prop}

\begin{proof}
Let $\bK$ be $K$ decorated with two basepoints, and write $g = g(S)$ and $g' = g(S')$.
By Theorems~\ref{thm:tauD1D2bound} and~\ref{thm:doublepointbound},
if $n$ is either $\mu_{\st}(S, S')$ or $\mu_{\Sing}(S, S')$, then
$\tau(S, S') \le n$ (recall that $\mu_{\Sing}(S, S') = \infty$ when $g \neq g'$,
so the inequality obviously holds in this case).
Furthermore, their proofs imply that
\begin{equation}
V^{n-g} \cdot \ve{t}_{S, \zs}^\infty \simeq V^{n-g'} \cdot \ve{t}_{S', \zs}^\infty.
\label{eq:generalchainhomotopy}
\end{equation}
If $\tau(S, S') < n$, then equation~\eqref{eq:nuinequality} automatically holds,
since $\nu(S, S') \le \tau(S, S') + 1$,
so it is sufficient to consider the case when $\tau(S, S') = n$.
It follows from equation~\eqref{eq:generalchainhomotopy} that
\begin{equation}
V^{-g} \cdot \ve{t}^-_{S, \zs}(1) - V^{-g'} \cdot \ve{t}^-_{S', \zs}(1) = \d x
\label{eq:tDsequaltoboundaryinsubcomplex}
\end{equation}
for some $x \in C(\bK, i \ge 0, j \ge -n)$. It follows that the elements
$\ve{t}^-_{S, \zs}(1)$ and $\ve{t}^-_{S', \zs}(1)$ agree in the homology
of any quotient of $C(\bK, i \ge 0, j \ge -n)$ by a filtered subcomplex.
Since $C(\bK, \L{m}{n})$ is the quotient of $C(\bK, i \ge 0, j \ge -n)$ by a
filtered subcomplex for any $m \in \N$, it follows that
\[
[V^{-g} \cdot \ve{t}_{S, \zs}^-(1) - V^{-g'} \cdot \ve{t}_{S', \zs}^-(1)] = 0 \in
H_*(C(\bK, \L{m}{n}))
\]
for all $m \in \N$. Hence $\tau'(S, S') = -\infty$ and $\nu(S, S') = \tau(S, S') = n$.
\end{proof}

\subsection{Kappa and the double point distance}
In this section, we show that the kappa invariant also gives a lower bound on the double point distance.

\begin{thm}\label{thm:kappaandmuSing}
Let $K$ be a knot in $\Sphere^3$.
If $S$, $S' \in \Surf_g(K)$ and $g>0$, then
\[
\kappa(S,S') \le \mu_{\Sing}(S,S').
\]
\end{thm}

The proof requires several steps.
Suppose that $S\in \Imm(K)$ is  a generic,
properly immersed surface in $B^4$ (by generic, we mean $S$ has a discrete collection
of transverse double points, disjoint from $\d S$).
The surface $S$ is the image of an immersion $f \colon \bar{S} \looparrowright B^4$, and
write $\hat{S}\subset B^4(S)$ for the desingularization of $S$.
Let
\[
\ve{p}^+ \subset \bar{S}
\]
 denote the preimages of the positive double points of $S$.
Note that each positive double point of $S$ contributes two
points to  $\ve{p}^+$. Let $P$ be a subset of the positive double points of $S$.

\begin{define}\label{def:A(T)}
Suppose that $S\in \Imm(K)$ is a generic immersed surface.
Let $T \subset \bar{S}$ denote an embedded tree such that the following hold:
\begin{enumerate}[label=(\textit{T\arabic*}), ref=\textit{T\arabic*}]
\item\label{num:tree1} $T \cap \d \bar{S}$ consists of a single point.
\item\label{num:tree2} Each point of $f^{-1}(P)$ is a leaf of $T$,
  and $T$ is disjoint from $\ve{p}^+ \setminus f^{-1}(P)$.
\end{enumerate}
Given such a tree $T$, we define an induced decoration $\cA_{\ws}(T)$ on $\bar{S}$, as follows.
The underlying dividing set of $\cA_{\ws}(T)$  is $\d N(T) \setminus \d \bar{S}$. We declare $N(T)$
to be the type-$\zs$ subregion, and the complement of $N(T)$ to be the type-$\ws$ subregion.
We note that the decoration $\cA_{\ws}(T)$ on $\bar{S}$ induces a decoration on the desingularized surface
$\hat{S} \subset B^4(S)$, for which we also write $\cA_{\ws}(T)$. There is an analogous decoration $\cA_{\zs}(T)$, obtained by reversing the roles of $\ws$ and $\zs$. 
\end{define}

Note that $g(\hat{S}_{\ws}) = g(\bar{S}) + |\ve{p}^+|/2 - |P|$ and $g(\hat{S}_{\zs}) = |P|$.
We now prove the following, somewhat surprising fact:

\begin{prop}\label{prop:independentofT}
Suppose that $S \in \Imm(K)$ is a generic immersed surface in $B^4$ with boundary~$K$,
and that $P$ is a subset of the positive double points of $S$. Let $T \subset \bar{S}$ be a tree satisfying
conditions~\eqref{num:tree1} and~\eqref{num:tree2}. If $\frs \in \Spin^c( B^4(S) )$, the chain homotopy type of the map
\[
F_{B^4(S), (\hat{S}, \cA_{\ws}(T)), \frs}\colon \cR^\infty\to \cCFL^\infty(\bK)
\]
is independent of the choice of tree $T$.
\end{prop}

To prove Proposition~\ref{prop:independentofT}, we need a set of moves that can be used
to connect two trees satisfying conditions~\eqref{num:tree1} and~\eqref{num:tree2}.
We introduce the following \emph{tree-moves}:

\begin{enumerate}[label=(\textit{TM\arabic*}), ref=\textit{TM\arabic*}]
\item\label{num:treemove1}  $T$ is replaced by another tree $T'$ satisfying \eqref{num:tree1} and \eqref{num:tree2} such that $\d N(T)$ and $\d N(T')$ are isotopic through dividing sets which are fixed on $\d \bar{S}$ and never intersect $\ve{p}^+$.
\item\label{num:treemove2} Suppose that $e$ is an edge of $T$ which contains a point $p$ of $\ve{p}^+$, and that $e'$ is an embedded path in $\bar{S}$ such that $e'\cap T = \d e'$, and such that $e \cap e'$ consists of a single point $t$. The tree $T$ is replaced with the tree $T'$ formed by adding $e'$, and removing a segment of $e$ that is not between $p$ and $t$; see the top row of Figure~\ref{fig::31}.
\end{enumerate}

\begin{lem}\label{lem:treemovessufficient}
If $\bar{S}$ is a surface, $\ve{p}^+ \subset \bar{S}$ is a collection of points,
and $f^{-1}(P) \subset \ve{p}^+$ is a chosen subset, then any two trees satisfying \eqref{num:tree1}
and \eqref{num:tree2} can be connected by moves \eqref{num:treemove1} and \eqref{num:treemove2}.
\end{lem}

\begin{proof}
We pick a subset $A \subset \bar{S}$ consisting of $g(\bar{S})$ compressing
curves on $\bar{S}$ such that we are left with a disk after surgering
$\bar{S}$ on $A$. If $T$ is a tree satisfying \eqref{num:tree1} and
\eqref{num:tree2}, we give $T$ the partial order determined by setting $T\cap
\d S$ to be the maximal point. If $t\in T$, define
\[
L(t) = \{x \in T : x\le t\}.
\]

As a first step, we show that any tree $T$ satisfying conditions
\eqref{num:tree1} and \eqref{num:tree2} can be connected by moves
\eqref{num:treemove1} and \eqref{num:treemove2} to a tree which is disjoint
from $A$. To establish this, we show that if $|A \cap T| > 0$, we can always
reduce $|A \cap T|$ by 1, using moves~\eqref{num:treemove1} and~\eqref{num:treemove2}.
To do this, we pick any point $t \in A \cap T$ such
$L(t) \cap A = \emptyset$. There are two cases: either $L(t) \cap
\ve{p}^+ = \emptyset$ or $L(t) \cap \ve{p}^+ \neq \emptyset$.

If $L(t)\cap \ve{p}^+ = \emptyset$, then we can just isotope $L(t)$ (an
instance of move~\eqref{num:treemove1}) so that it no longer intersects $A$,
thus reducing $|A\cap T|$.

Next, we consider the case that $L(t) \cap \ve{p}^+ \neq \emptyset$. In this
case, we pick a point $t' \in L(t)$ such that $L(t')$ is a subset of a single
edge of $T$, and $L(t')$ contains a point $p \in \ve{p}^+$. We let $e'$ be any
embedded path in $S\setminus A$ such that $e' \cap T = \d e'$, and $\d e'$
consists of $t'$ and another point of $T$ which is not contained in $L(t)$.
Let $e$ denote a subinterval of the edge of $T$ containing $t'$, such that
$t'$ is the smaller endpoint of $e$. We can then
use move~\eqref{num:treemove2} to replace $e$ with $e'$. This reduces
$L(t)\cap \ve{p}^+$ by 1, and does not increase $|A \cap T|$. Repeating this
procedure, we may reduce to the case that $L(t) \cap \ve{p}^+ = \emptyset$.
Arguing as before, an isotopy of $L(t)$ can then be used to reduce $|A \cap T|$ by 1.

Hence, if $T$ and $T'$ are two trees satisfying conditions~\eqref{num:tree1}
and \eqref{num:tree2}, by applying moves~\eqref{num:treemove1} and
\eqref{num:treemove2}, we may assume that $T$ and $T'$ are both disjoint from
$A$. We may compress $\bar{S}$ along $A$ to get a disk $D$, containing $T$ and
$T'$, as well as a collection of $2g(\bar{S})$ points $\ve{p}$ corresponding to the
curves in $A$. We note that $T$ and $T'$ are disjoint from $\ve{p}$.
Furthermore, isotoping an edge of $T$ or $T'$ across a point in $\ve{p}$ may
be achieved by move~\eqref{num:treemove2}. In this manner, we can reduce the
claim to the case that $S$ is a disk, and it is straightforward to see that
in this situation that $T$ and $T'$ can be related by applying move~\eqref{num:treemove1}.
\end{proof}

\begin{proof}[Proof of Proposition~\ref{prop:independentofT}]
By Lemma~\ref{lem:treemovessufficient}, it is sufficient to show invariance
of $F_{B^4(S), (\hat{S}, \cA_{\ws}(T)),\frs}$
under moves~\eqref{num:treemove1} and \eqref{num:treemove2}.
First note that, up to isotopy, the decoration $\cA_{\ws}(T)$
depends only on a regular neighborhood of $T \subset \bar{S}$,
so move~\eqref{num:treemove1} does not change the cobordism map
$F_{B^4(S),(\hat{S},\cA_{\ws}(T)),\frs}$.

We now address move~\eqref{num:treemove2}.
Suppose that $e$ is an edge of $T$ which has exactly
one endpoint at a point $p \in \ve{p}^+$ and another
at a vertex $v \in T \setminus \ve{p}^+$. Suppose that
$e'$ is an embedded path on $\bar{S}$ such that $e' \cap T = \d e'$ and $e \cap e' = \{t\}$.
Let $T'$ denote the tree obtained by removing a segment of $e$ not between $t$ and $p$,
and inserting $e'$. There is a bypass relation
\begin{equation}\label{eq:bypassrelationfortrees}
F_{B^4(S), (\hat{S}, \cA_{\ws}(T)), \frs} +
F_{B^4(S), (\hat{S}, \cA_{\ws}(T')),\frs} +
F_{B^4(S), (\hat{S}, \cA''), \frs} \simeq 0
\end{equation}
for a third decoration $\cA''\subset \hat{S}$, which is shown in Figure~\ref{fig::31}.

\begin{figure}[ht!]
	\centering
	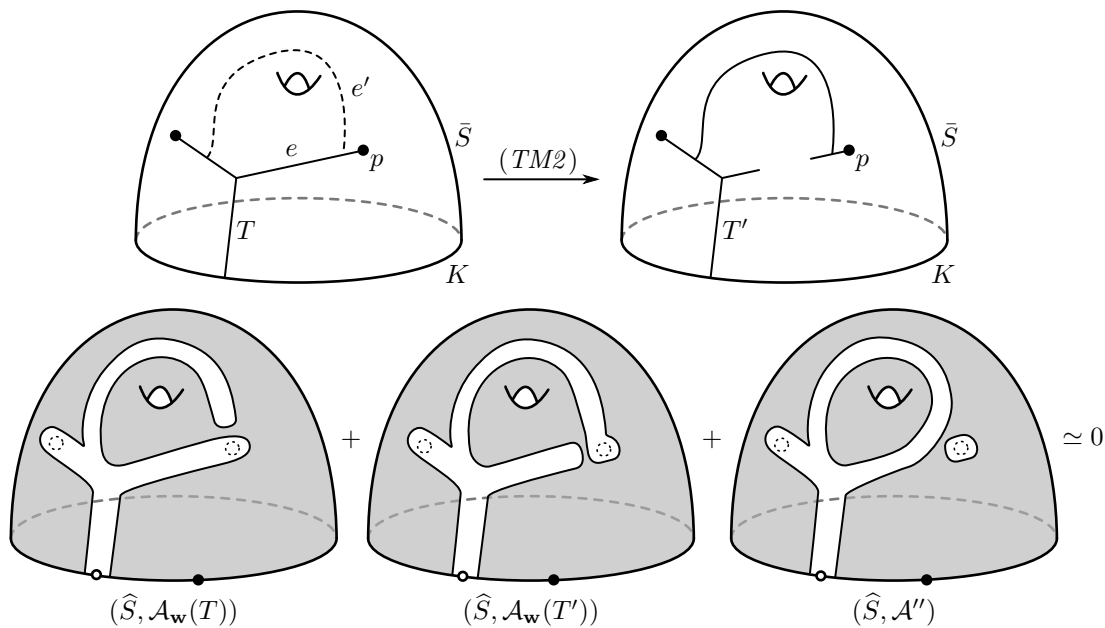
	\caption{The graphs $T$ and $T'$ on the top row are related by move~\eqref{num:treemove2}.
	The point $p$ is in $\ve{p}^+$. On the bottom row, the associated
	decorations $\cA_{\ws}(T)$ and $\cA_{\ws}(T')$ on the desingularized surface $\hat{S}$ are shown,
	as well as a third decoration $\cA''$, which fits
	into a bypass triple with $\cA_{\ws}(T)$ and $\cA_{\ws}(T')$.
	The dashed circles denote where a tube is added on the desingularized surface.}
\label{fig::31}
\end{figure}

We now claim that
\begin{equation}
F_{B^4(S),(\hat{S},\cA''),\frs} \simeq 0
\label{eq:dividingsetwithclosedloopiszero}
\end{equation}
for any $\frs \in \Spin^c(B^4(S))$, which will complete the proof when
combined with equation~\eqref{eq:bypassrelationfortrees}.
The key observation is that the dividing set $\cA''$ has a closed loop,
which is homologically essential in the tube that
is added to form the desingularized surface $\hat{S}$. Let
$B$ be a 4-ball in $B^4(S)$ containing the point $f(p)$.
We note that $\hat{S}$ intersects $\d B$ in a negative Hopf link $H$.
Let $A_{0}$ denote the annulus that is inserted into $B$
to form the desingularized surface $\hat{S}$. By isotoping
$\cA''$ in $\hat{S}$, we  may assume that $\cA''$
intersects $A_{0}$ in the dividing set $\cA_0''$ shown in Figure~\ref{fig::33}.

\begin{figure}[ht!]
	\centering
\begingroup%
  \makeatletter%
  \providecommand\color[2][]{%
    \errmessage{(Inkscape) Color is used for the text in Inkscape, but the package 'color.sty' is not loaded}%
    \renewcommand\color[2][]{}%
  }%
  \providecommand\transparent[1]{%
    \errmessage{(Inkscape) Transparency is used (non-zero) for the text in Inkscape, but the package 'transparent.sty' is not loaded}%
    \renewcommand\transparent[1]{}%
  }%
  \providecommand\rotatebox[2]{#2}%
  \newcommand*\fsize{\dimexpr\f@size pt\relax}%
  \newcommand*\lineheight[1]{\fontsize{\fsize}{#1\fsize}\selectfont}%
  \ifx\svgwidth\undefined%
    \setlength{\unitlength}{380.52066902bp}%
    \ifx\svgscale\undefined%
      \relax%
    \else%
      \setlength{\unitlength}{\unitlength * \real{\svgscale}}%
    \fi%
  \else%
    \setlength{\unitlength}{\svgwidth}%
  \fi%
  \global\let\svgwidth\undefined%
  \global\let\svgscale\undefined%
  \makeatother%
  \begin{picture}(1,0.23987467)%
    \lineheight{1}%
    \setlength\tabcolsep{0pt}%
    \put(0,0){\includegraphics[width=\unitlength,page=1]{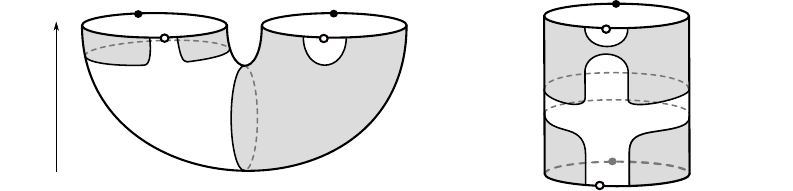}}%
    \put(0.05903712,0.11356894){\color[rgb]{0,0,0}\makebox(0,0)[rt]{\lineheight{1.25}\smash{\begin{tabular}[t]{r}$\cF_{0}''$\end{tabular}}}}%
    \put(0,0){\includegraphics[width=\unitlength,page=2]{fig33.pdf}}%
    \put(0.89976214,0.10887198){\color[rgb]{0,0,0}\makebox(0,0)[lt]{\lineheight{1.25}\smash{\begin{tabular}[t]{l}$\Psi_z\circ \Phi_w$\end{tabular}}}}%
  \end{picture}%
\endgroup%

	\caption{On the left is the annulus $A_{0}\subset \hat{S}$,
	which is added to desingularize the immersed surface $S$ at the
	negative double point $f(p)$. The decoration
	$\cA''_0$ on  $A_0$ is shown. On the right is
	the dividing set corresponding to the map $\Phi_{w} \circ \Psi_{z}$,
	on a cylindrical link cobordism.}
\label{fig::33}
\end{figure}

Let us write $\cF''$ for $(\hat{S},\cA'')$ and $\cF_{0}'' = (A_{0}, \cA_0'')$.
Using the composition law, it is sufficient to show that
\[
F_{B, \cF_{0}'', \frs_0} \simeq 0,
\]
where $\frs_0$ is the unique torsion $\Spin^c$ structure on $B$.

Let $\bH^+$ denote the positive Hopf link, decorated with basepoints
$w_1$ and $z_1$ on one component, and $w_2$ and $z_2$
on the other. By isotoping the dividing set on $\cF_{0}''$,
we may factor the map $F_{B,\cF_{0}'',\frs_0}$ through
\[
\Phi_{w_1}\Psi_{z_1}\colon \cCFL^-(\bH^+)\to \cCFL^-(\bH^+).
\]

\begin{figure}[ht!]
	\centering
	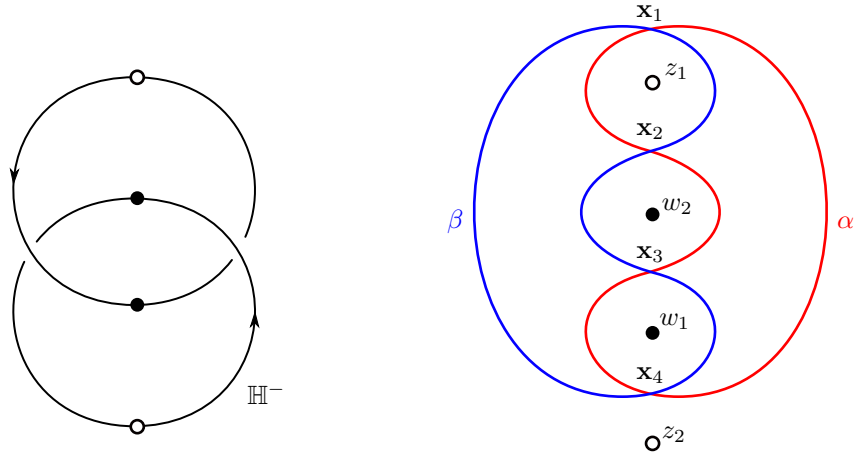
	\caption{The positive Hopf link (left), and a genus zero Heegaard diagram (right).}
\label{fig::32}
\end{figure}

\begin{figure}[ht!]
\centering
\[\cCFL^-(\bH^+)=\left(
\begin{tikzcd}
\xs_1\arrow{r}{V}\arrow{d}{U} &\xs_2\\
\xs_4& \xs_3\arrow{u}{V}\arrow{l}{U}
\end{tikzcd}
\right), \quad \Phi_{w_1}=\left(\begin{tikzcd}
\xs_1 &\xs_2\\
\xs_4& \xs_3\arrow{l}
\end{tikzcd}\right), \quad  \Psi_{z_1}=\left(\begin{tikzcd}
\xs_1\arrow{r} &\xs_2\\
\xs_4& \xs_3
\end{tikzcd}\right)
\]
\caption{The chain complex $\cCFL^-(\bH^+)$, as well as
the maps $\Phi_{w_1}$ and $\Psi_{z_1}$.}
\label{fig::Hopfcomplex}
\end{figure}

A diagram for $\bH^+$ is shown in Figure~\ref{fig::32}. 
The differential can be computed by simply counting bigons.
It is not hard to see that the complex $\cCFL^-(\bH^+)$
is chain homotopy equivalent to the complex shown in
Figure~\ref{fig::Hopfcomplex}. The chain homotopy type
of the maps $\Phi_{w_1}$ and $\Psi_{z_1}$ are also
shown in Figure~\ref{fig::Hopfcomplex}.
Examining the maps $\Phi_{w_1}$ and $\Psi_{z_1}$ shown
in Figure~\ref{fig::Hopfcomplex}, we see that $\Phi_{w_1}\Psi_{z_1}$ vanishes, completing the proof. 
\end{proof}

\begin{define}
Given a generic immersed surface $S \in \Imm(K)$ that is the image of an immersion $f \colon \bar{S} \looparrowright B^4$,
as well as a subset $P$ of the positive double points of $S$,
pick a tree $T \subset \bar{S}$ satisfying conditions~\eqref{num:tree1} and \eqref{num:tree2}.
We form the decoration $\cA_{\ws}(T)$ of the
desingularization $(B^4(S), \hat{S})$ as in Definition~\ref{def:A(T)}.
For $\frs \in \Spin^c(B^4(S))$, we define the map
\[
\ve{t}_{S, P, \ws, \frs}^- \colon \cR^- \to \cCFL^-(\bK)
\]
to be the decorated link cobordism map $F_{B^4(S), (\hat{S}, \cA_{\ws}(T)),\frs}$,
which is independent of the choice of $T$ by Proposition~\ref{prop:independentofT}.
Analogously, we can define the map
\[
\ve{t}_{S, P, \zs, \frs}^- \colon \cR^- \to \cCFL^-(\bK).
\]
\end{define}

We can now prove that $\kappa(S,S')$ is a lower bound for $\mu_{\Sing}(S,S')$:

\begin{proof}[Proof of Theorem~\ref{thm:kappaandmuSing}]
Suppose that $S_1,\dots, S_n$ is a sequence of
immersed surfaces, each of which is obtained from the
previous via creating or canceling a pair of double points,
up to diffeomorphism fixing $\d B^4$ pointwise, and $S_1 = S$ and $S_n = S'$.
Furthermore, let $P_k$ be the set of all positive double points of $S_k$ for $k \in \{1, \dots, n\}$.

Using Lemma~\ref{lem:computedesingularization}, we conclude that, if
\[
m := \frac{1}{2} \max \left\{|\Sing(S_1)|,\dots, |\Sing(S_n)|\right\},
\]
and $\frs_1,\dots,\frs_n$ is a stabilization sequence of $\Spin^c$
structures on $B^4(S_1),\dots, B^4(S_n)$ that are all maximal,
then the filtered chain homotopy types of the maps
\[
V^{m - \tfrac{1}{2} |\Sing(S_i)|} \cdot \ve{t}_{S_k, P_k, \ws, \frs}^-
\]
coincide for $k \in \{1, \dots, n\}$. In particular, the maps on $\CFK^-_{U=0}$ coincide, so
\[
V^m \cdot [t_{S, \ws}^-(1)] = V^m \cdot [t_{S',\ws}^-(1)]  \text{ in }  \HFK^-_{U=0}(\bK),
\]
and hence $\kappa(S, S') \le m + g$.
\end{proof}

\subsection{The local $h$-invariants and the double point distance}

In this section, we show the following:

\begin{thm}\label{thm:hinvdoublepointhighergenus}
Suppose that $S$, $S' \in \Surf_g(K)$ and $g \le k \le \mu_{\Sing}(S,S')$. Then
\[
V_k(S,S') \le \left\lceil\frac{\mu_{\Sing}(S,S')-k}{2} \right\rceil.
\]
If $k\ge \mu_{\Sing}(S,S')$, then $V_k(S,S')=0$.
\end{thm}

\begin{define}
Suppose that $S \in \Imm(K)$ is a generic immersed surface. We denote by $P^+$ the set of positive
double points of $S$. Furthermore, let $d$ be an integer satisfying $d \ge g(S) + |P^+|$. We say that $S$
satisfies the \emph{singular, decoration-independence condition (SDI) above degree~$d$} if the following holds:
\begin{enumerate}[leftmargin=1.5cm, label=(\textit{SDI}), ref=\textit{SDI}]
\item\label{def:SDI} For all
maximal $\frs \in \Spin^c(B^4(S))$, for every $i$, $j \in \N$ satisfying
\[
g(S) + |P^+| + i + j \ge d,
\]
and for all $P \subset P^+$, the chain homotopy type of the map
\[
U^i V^j\cdot  \ve{t}^-_{S, P, \zs, \frs} \colon \cR^- \to \cCFL^-(\bK)
\]
depends only on $\frs$ and the quantities
\[
i + |P| \quad \text{and} \quad j + |P^+| - |P|.
\]
\end{enumerate}

\end{define}

\begin{lem}\label{lem:stabilizationtypedoublepoints}
Let $S \in \Imm(K)$ be a generic immersion with positive double points $P^+$, and
suppose that $S'$ is obtained from $S$ by the birth of a pair of double points.
Write $p^+$ and  $p^-$ for the new positive and negative double points of $S'$, respectively.
 If $S$ satisfies the singular decoration-independence condition~\eqref{def:SDI} above degree~$d$, then $S'$ satisfies condition~\eqref{def:SDI} above degree
\[
\max\{ d, g(S') + |P^+\cup \{p^+\}| \}.
\]
Similarly, if $S'$ satisfies~\eqref{def:SDI} above degree~$d$, then so does $S$.
\end{lem}

\begin{proof}
Consider first the case that $S$ satisfies condition~\eqref{def:SDI} above degree~$d$.
 Let $\frs \in \Spin^c(B^4(S))$ and $\frs' \in
\Spin^c(B^4(S'))$ be such that $\frs'$ agrees with $\frs$ on
$B^4(S) \setminus N(p^-)$ and $\langle\, c_1(\frs'), E \,\rangle = \pm 1$,
where $E$ is the exceptional divisor that appears after blowing up $B^4(S)$ at $p^-$.

Let $P \subset P^+$ and $P' \subset P^+ \cup \{p^+\}$.
By Lemma~\ref{lem:computedesingularization}, we have that
\begin{equation}
U^{i} V^j \cdot \ve{t}^-_{S', P', \zs, \frs'} \simeq
\begin{cases}
U^{i+1}V^j \cdot \ve{t}^-_{S, P' \setminus \{p^+\}, \zs, \frs} & \text{ if } p^+ \in P'\\
U^i V^{j+1} \cdot \ve{t}^-_{S, P', \zs, \frs} & \text{ if } p^+ \not\in P'.
\label{eq:formulaafterdoublepoint}
\end{cases}
\end{equation}
Since $S$ satisfies condition~\eqref{def:SDI} above degree~$d$, the
expression on the right side of equation~\eqref{eq:formulaafterdoublepoint}
depends only on the quantities $i + |P'|$ and $j + 1 + |P^+| - |P'|$, as long as
\[
i + j + 1 + |P^+| + g(S) \ge d.
\]
Hence $S'$ satisfies condition~\eqref{def:SDI} above degree~$d$ if $d > g(S) + |P^+|$,
or above degree~$d + 1$ if $d = g(S) + |P^+|$.

Next, we suppose that $S'$ satisfies condition~\eqref{def:SDI}
above degree
\[
d \ge g(S') + |P^+\cup \{p^+\}|=g(S')+|P^+|+1.
\]
The stabilization formula from equation~\eqref{eq:formulaafterdoublepoint}
 shows that $S$ satisfies condition~\eqref{def:SDI} above degree~$d$, as well.
\end{proof}

We can now prove that $V_k(S,S')$ gives a lower bound on $\mu_{\Sing}(S,S')$:

\begin{proof}[Proof of Theorem~\ref{thm:hinvdoublepointhighergenus}]
Suppose first that $g\le k \le \mu_{\Sing}(S,S')$.
Pick a sequence of generic immersed surfaces $S_1,\dots, S_n \in \Imm(K)$ such that
consecutive surfaces differ by the birth or death of a pair of double points.
Furthermore, assume that $S_1 = S$ and $S_n = S'$. Note that, trivially, $S$ and $S'$ satisfy the singular decoration-independence condition~\eqref{def:SDI} above degree~$g := g(S) = g(S')$.
We pick a stabilization sequence of maximal $\Spin^c$ structures $\frs_1,\dots,\frs_n$
on $B^4(S_1),\dots, B^4(S_n)$, respectively.

Write $m = \tfrac{1}{2} \max \{|\Sing(S_1)|, \dots, |\Sing(S_n)|\}$. By
Lemma~\ref{lem:stabilizationtypedoublepoints}, each of the immersed surfaces
$S_i$ satisfies condition~\eqref{def:SDI} above degree~$g + m$.
By adding one additional birth-death pair of double points, if necessary,
we may assume that $\frac{m+g-k}{2}$ and $\frac{m-g+k}{2}$ are both integers.
Note also that both expressions are nonnegative, since $g \le k \le m + g$ by assumption.
Let $P_l^+$ be the set of positive double points of $S_l$ for $l \in \{1,\dots,n\}$.
Since
\[
\frac{m+g-k}{2}+\frac{m-g+k}{2}=m
\]
and $|P_l^+| \le m$ for all $l$, we can pick subsets $P_l \subset P_l^+$ such that
\[
|P_l| \le \frac{m+g-k}{2} \qquad \text{and} \qquad |P_l^+| - |P_l| \le \frac{m-g+k}{2}.
\]
We then pick sequences of nonnegative integers $i_l$ and $j_l$ such that
\[
i_l + |P_l| = \frac{m+g-k}{2} \qquad \text{ and } \qquad  j_l + |P_l^+| - |P_l| = \frac{m-g+k}{2}
\]
for all $1 \le l \le n$.

Using our computation of the effect of a double point birth from
Lemma~\ref{lem:computedesingularization}, as well as
Lemma~\ref{lem:stabilizationtypedoublepoints} to change the decorations as
needed, we conclude that the maps
\[
U^{i_l} V^{j_l} \cdot \ve{t}^-_{S_l, P_l, \zs, \frs_l} \colon \cR^- \to \cCFL^-(\bK)
\]
are all filtered chain homotopic to one another. In particular,
\begin{equation}\label{eq:doublepointboundfromVk}
U^{\frac{g+m-k}{2}} V^{\frac{m-g+k}{2}} \cdot \ve{t}_{S,\zs }^-\simeq
U^{\frac{g+m-k}{2}} V^{\frac{m-g+k}{2}} \cdot \ve{t}_{S',\zs }^-.
\end{equation}
The maps in equation~\eqref{eq:doublepointboundfromVk} increase the Alexander grading by $k$.
Multiplying equation~\eqref{eq:doublepointboundfromVk} by $V^{-k}$ and rearranging, we conclude that
\[
\hat{U}^{\frac{g+m-k}{2}}\cdot [V^{-g} \cdot  \ve{t}_{S,\zs}^-(1)] =
\hat{U}^{\frac{g+m-k}{2}}\cdot [V^{-g} \cdot  \ve{t}_{S',\zs}^-(1)] \text{ in } H_*(A_k^-(\bK)),
\]
completing the proof for $k \le \mu_{\Sing}(S,S')$.

To verify that $V_k(S,S') = 0$ when $k \ge \mu_{\Sing}(S,S')$,
we note that the above result implies that $V_{\mu_{\Sing}(S,S')}(S,S') = 0$.
The monotonicity result from Lemma~\ref{lem:monotonicity} then implies the claim for $k > \mu_{\Sing}(S,S')$.
\end{proof}

\section{Upsilon and an infinite family of topological metric filtrations}
\label{sec:M-distance}

Let $K$ be a knot in $\Sphere^3$, and let $S$, $S' \in \Surf(K)$.
The invariant $\Upsilon_{(S,S')}$ gives a family of algebraically defined functions
\[
\Upsilon_{(S,S')}(t) \colon \Surf(K) \times \Surf(K) \to \R^{\ge 0}
\]
parametrized by $t \in [0,2]$. In this section, we describe
a topologically defined family $M_{(S,S')}(t)$
of functions that are bounded below by $\Upsilon_{(S,S')}(t)$.

\subsection{The topological $M$-metric on \texorpdfstring{$\Surf(K)$}{Surf(K)}}

The topological $M$-metric will be defined using the following
generalized stabilization operation:

\begin{define}\label{def:generalizedstabilization}
Suppose that $(W, S)$ bounds $(\Sphere^3, K)$, and that $B^4 \subset \Int(W)$ is an
embedded 4-ball such that $\d B^4 \cap S$ is an $n$-component unlink $U_n$. Further,
suppose that $B^4 \cap S$ consists of disks $D_1, \dots, D_n$ that can be
smoothly isotoped into $\d B^4$ relative to $U_n$. We say that $(W',S')$ is a \emph{generalized
stabilization} of $(W, S)$ if it is formed by removing $(B^4, S \cap B^4)$ from
$(W, S)$, and gluing in a link cobordism $(X_0, S_0)$ such that the following hold:
\begin{enumerate}
\item $(X_0, S_0)$ is a cobordism from $\emptyset$ to $(\d B^4, U_n)$,
\item $b_2^+(X_0) = b_1(X_0) = 0$,
\item $S_0$ is connected and oriented.
\end{enumerate}
\end{define}

We remark that, although Definition~\ref{def:generalizedstabilization}
clearly generalizes the stabilization operation from
Section~\ref{sec:stabilizationgenusdef}, it may still seem somewhat
unmotivated.  We note that, after a double point creation, the desingularization of an
immersed surface changes by a generalized stabilization:

\begin{example}
Suppose that $S$ is an immersed surface in $B^4$ which bounds $K$, and $S'$ is
obtained from $S$ by a creation of a pair of canceling double points. If
$(W, \hat{S})$ and $(W', \hat{S}')$ denote their desingularizations, as defined
in Definition~\ref{def:desingularization}, then $(W', \hat{S}')$ can be
obtained from $(W, \hat{S})$ by a generalized stabilization. Indeed, since a
double point creation can be achieved by a finger move supported in a
neighborhood of a path $\lambda$ connecting two points on $S$, we can take
$B^4$ to be a neighborhood of $\lambda$. Clearly, $\d B^4$ intersects $S$
along two disks, and $S \cap \d B^4$ is a 2-component unlink. Since $S'$
differs from $S$ only inside $B^4$, the desingularization $(W, \hat{S})$ can
be obtained by cutting out $B^4$ and gluing in $(\overline{\CP}^2 \setminus
B^4, S_0)$ for an annulus $S_0$ in $\overline{\CP}^2 \setminus B^4$.
\end{example}

If $W$ is a compact, oriented 4-manifold with boundary $\Sphere^3$, we let
$\Char(W)$ denote the set of \emph{characteristic vectors} of the
intersection form $Q_W$; i.e., the set of elements $C \in H^2(W) \cong H^2(W, \d W)$ such that
\[
\langle\, x \cup x, [W,\d W] \,\rangle \equiv \langle\, C \cup x, [W,\d W] \,\rangle \mod 2
\]
for every $x \in H^2(W,\d W)$.
It is well known that $\Char(W) = \{\, c_1(\frs) : \frs \in \Spin^c(W) \,\}$.

Suppose $(W,S)$ is a link cobordism from $\emptyset$ to $(\Sphere^3, K)$, with
$b_1(W) = b_2^+(W) = 0$. For $C \in \Char(W)$, let
\[
H_t(W, [S], C) : =\frac{C^2 + b_2(W) - 2t \langle\, C,[S] \,\rangle + 2t[S] \cdot [S]}{4},
\]
where $[S] \in H_2(W) \cong H_2(W, \d W)$. Furthermore, for $\frs \in \Spin^c(W)$, we write
\[
H_t(W, [S], \frs) := H_t(W, [S], c_1(\frs)).
\]
If $t \in [0,2]$, we define the \emph{$M$-degree} of $(W,S)$ to be the function
\[
M_{(W,S)}(t) := \min_{C \in \Char(W)} -H_t(W, [S], C) +  t \cdot g(S).
\]

\begin{define}
Suppose that $K$ is a knot in $\Sphere^3$, and $S$, $S' \in \Surf(K)$.
We say that $\vec{\cS} = \{\cS_1,\dots, \cS_n\}$ is a
\emph{generalized stabilization sequence} from $S$ to $S'$ if the
following hold:
\begin{enumerate}
\item Each $\cS_i = (W_i, S_i)$ is a link cobordism bounding $(\Sphere^3, K)$,
    such that $S_i$ is connected and $b_1(W_i) = b_2^+(W_i) = 0$.
\item $\cS_1 = (B^4, S)$ and $\cS_n = (B^4, S')$.
\item Up to diffeomorphism fixing $\Sphere^3$ pointwise,
    $\cS_{i+1}$ is obtained from $\cS_i$ via a generalized
    stabilization or destabilization.
\end{enumerate}
We write $P_{\st}(S, S')$ for the set of stabilization sequences
connecting $S$ and $S'$.
\end{define}

\begin{define}
Let $K$ be a knot in $\Sphere^3$ and suppose $S$, $S' \in \Surf(K)$.
Given a stabilization sequence
$\vec{\cS} = \{\cS_1, \dots, \cS_n\}$ from $S$ to $S'$,
we define the \emph{$M$-degree} of the sequence $\vec{\cS}$ to be the function
\[
M_{\vec{\cS}}(t) := \max_{1 \le i \le n} M_{(W_i, S_i)}(t).
\]
Furthermore, the \emph{$M$-distance} of $S$ and $S'$ is the function
$M_{(S,S')} \colon [0,2] \to \R^{\ge 0}$ defined by
\[
M_{(S,S')}(t) := \min_{\vec{\cS} \in P_{\st}(S, S')} M_{\vec{\cS}}(t).
\]
\end{define}

For each $t \in [0,2]$, the quantity $M_{(S,S')}(t)$ is a metric filtration on $\Surf(K)$.

\subsection{The $M$-metric and \texorpdfstring{$\Upsilon$}{upsilon}}

In this section, we prove the following:

\begin{thm}\label{thm:Upsilonbound}
If  $K$ is a knot in $\Sphere^3$ and  $S$, $S' \in \Surf(K)$, then
\[
\Upsilon_{(S,S')}(t) \le M_{(S,S')}(t).
\]
\end{thm}

The proof of Theorem~\ref{thm:Upsilonbound} is similar to the proof of
Theorem~\ref{thm:tauD1D2bound}, our bound on $\tau$.
It is convenient to introduce the following notation.
Suppose $(W, \cF) \colon (Y_1, \bL_1)\to (Y_2, \bL_2)$  is a decorated link cobordism,
 and $\frs \in \Spin^c(W)$. Write $\cF = (S, \cA)$.
If $\frs|_{Y_i}$ is torsion and $\bL_i = (L_i, \ws_i, \zs_i)$
is null-homologous for $i \in \{1, 2\}$, we define the quantities
\[
G_{\ws}(W,\cF,\frs) := \frac{c_1(\frs)^2 - 2\chi(W) - 3\sigma(W)}{4} +
\chi(S_{\ws}) - \frac{1}{2}(|\ws_1| + |\ws_2|)
\]
and
\[
G_{\zs}(W,\cF,\frs) := \frac{c_1(\frs - \PD[S])^2 - 2\chi(W) - 3\sigma(W)}{4} +
\chi(S_{\zs}) - \frac{1}{2}(|\zs_1| + |\zs_2|).
\]
For $t\in [0,2]$, we define
\[
G_t(W,\cF,\frs):=\left(1-\frac{t}{2}\right)\cdot G_{\ws}(W,\cF,\frs)+\frac{t}{2}\cdot  G_{\zs}(W,\cF,\frs).
\]

In the case when $(W,\cF)$ is a cobordism from $\emptyset$ to $(\Sphere^3, \bK)$
with $b_2^+(W) = b_1(W) = 0$, and the dividing set $\cA$ consists of a single arc
that divides $S$ into two components, we have
\[
G_{\ws}(W,\cF,\frs) = \frac{c_1(\frs)^2 + b_2(W)}{4} - 2g(S_{\ws}) \text{ and }
G_{\zs}(W,\cF,\frs) = \frac{c_1(\frs-\PD[S])^2 + b_2(W)}{4} - 2g(S_{\zs}).
\]
In this situation, we also have
\begin{equation}\label{eqn:Gt}
G_t(W,\cF,\frs) = H_t(W, [S], \frs) - (2-t) \cdot g(S_{\ws}) - t \cdot g(S_{\zs}).
\end{equation}

We now compute the effect of a generalized stabilization;
cf.~Lemma~\ref{lem:stabilizationlemma}:

\begin{lem}\label{lem:stabilizationlemma2}
Suppose that $(W,\cF) \colon (Y_1, \bL_1) \to (Y_2, \bL_2)$ is a decorated link
cobordism with $b_1(W) = 0$. Write $\cF = (S,\cA)$. Suppose that $(W',S')$ is a
generalized stabilization of $(W,S)$, obtained by cutting out $B^4 \subset W$
and gluing in a link cobordism $(X_0, S_0)$ with $b_1(X_0) = b_2^+(X_0) = 0$,
as in Definition~\ref{def:generalizedstabilization}. Let $D_1, \dots, D_n$ denote
the components of $S \cap B^4$, and suppose that $\hat{D} \subset S$ is a disk
that contains each of $D_1, \dots, D_n$, and intersects $\cA$ in a single arc. Define
\[
S_0' := (\hat{D} \setminus (D_1 \cup \cdots \cup D_n)) \cup S_0,
\]
and suppose that $\cA'$ is a dividing set on $S'$ that intersects $S_0'$ in
a single arc, and agrees with $\cA$ outside $\hat{D}$. Write
$\cF' = (S',\cA')$. If $\frs' \in \Spin^c(W')$ agrees with $\frs \in \Spin^c(W)$
on $W \setminus B^4$, then
\[
F_{W',\cF',\frs'} \simeq U^{-d_1/2} V^{-d_2/2} \cdot F_{W,\cF,\frs},
\]
where $d_1 := G_{\ws}(W',\cF',\frs') - G_{\ws}(W,\cF,\frs)$ and
$d_2 := G_{\zs}(W',\cF',\frs') - G_{\zs}(W,\cF,\frs)$.
\end{lem}

\begin{proof}
Let $\hat{D}_0$ denote the punctured disk $\hat{D} \setminus B^4$.
We write $N(\hat{D}_0)$ for the total space of the unit normal disk bundle of
$\hat{D}_0$ in $W \setminus B^4$, and let
\[
W_0 := N(\hat{D}_0) \cup B^4 \text{ and } W_0' := N(\hat{D}_0) \cup X_0.
\]
Note that $W_0$ and $W_0'$ are topologically obtained from $B^4$ and $X_0$, respectively,
by attaching a collection of 4-dimensional 1-handles. Write $Y := \d W_0 = \d W_0'$.
We can view  $(W_0, \hat{D})$ and $(W_0', S_0')$ as undecorated link cobordisms
from the empty set to the knot
\[
K := \d \hat{D} \times \{0\} = S' \cap Y \subset N(\hat{D}_0)
\]
in $Y$. As in Lemma~\ref{lem:stabilizationlemma}, the knot $K$ is an unknot in $Y$,
since we can explicitly construct a Seifert disk $D_K$.
Let us write $\cF_0 = (\hat{D}, \cA)$ and $\cF_0' = (S'_0, \cA')$.

Suppose $\frs' = \frs \# \frt$, for $\frt \in \Spin^c(X_0).$
Consider the quantity
\[
h(S_0' \cup D_K, \frs') := \frac{\langle\, c_1(\frs'), [S_0'\cup D_K] \,\rangle -
[S_0' \cup D_K] \cdot [S_0' \cup D_K]}{2}.
\]
If $h(S_0'\cup D_K, \frs') + g(S_{0,\zs}') - g(S_{0,\ws}') \ge 0$,
then, according to Lemma~\ref{lem:closedsurfacecomputation},
\begin{equation}
F_{W_0',\cF_0',\frs'|_{W_0'}} \simeq U^{g(S_{0,\ws}')} V^{g(S_{0,\zs}') +
h(S_0'\cup D_K,\frs')} \cdot (F_{W_0',\frs'|_{W_0'}}\otimes \bF[U,V]) \mod H_1(Y),
\label{eq:genstab1}
\end{equation}
while
\begin{equation}
F_{W_0,\cF_0,\frs|_{W_0}} \simeq (F_{W_0,\frs|_{W_0}}\otimes \bF[U,V]) \mod H_1(Y).
\label{eq:genstab2}
\end{equation}
Up to diffeomorphism, we can write $W_0$ and $W_0'$ as the compositions
\[
W_0 \iso W_1 \circ B^4 \text{ and } W_0' \iso W_1 \circ X_0,
\]
where $W_1 \iso (I \times \Sphere^3) \cup N(\hat{D}_0)$.
We note $W_1$ is a 1-handle cobordism.

The map
\[
F_{X_0,\frt} \colon \HF^-(\Sphere^3) \to \HF^-(\Sphere^3)
\]
is an injection since $b_1(X_0) = b_2^+(X_0) = 0$, by the proof of \cite{OSIntersectionForms}*{Theorem~9.1}.
The map $F_{X_0,\frt}$ has grading
\[
d := \frac{c_1(\frt)^2 + b_2(X_0)}{4},
\]
and hence must be chain homotopic to multiplication by $\hat{U}^{-d/2}$. Hence
\begin{equation}
F_{W_0',\frs'|_{W_0'}} \simeq  F_{W_1,\frs|_{W_1}} \circ F_{X_0,\frt} \simeq
\hat{U}^{-d/2} \cdot F_{W_1,\frs|_{W_1}} \circ F_{B^4,\frs|_{B^4}} \simeq
\hat{U}^{-d/2} \cdot F_{W_0,\frs|_{W_0}}.
\label{eq:genstab3}
\end{equation}
Write $W_2 := W \setminus W_1$. The inclusion $H_1(Y) \to H_1(W_2)/\Tors$ is trivial,
since $Y$ is the boundary of a 4-dimensional 1-handlebody in $W$.
Similarly, the coboundary maps $ H^1(Y) \to H^2(W)$ and $H^1(Y)\to H^2(W')$ are both trivial.
Hence, combining equations~\eqref{eq:genstab1}, \eqref{eq:genstab2}, and~\eqref{eq:genstab3},
and using the $\Spin^c$ composition law, we conclude that
\[
F_{W',\cF',\frs'} \simeq U^{-d/2+g(S_{0,\ws}')} V^{-d/2+h(S_0' \cup D_K,\frs') + g(S_{0,\zs}')}
\cdot F_{W,\cF,\frs}.
\]
It is easy to see that
\[
-\frac{d}{2} + g(S_{0,\ws}') = -\frac{1}{2} \left( G_{\ws}(W',\cF',\frs') - G_{\ws}(W,\cF,\frs)\right)
\]
and
\[-\frac{d}{2} + h(S_0'\cup D_K,\frs') + g(S_{0,\zs}') =
-\frac{1}{2} \left(G_{\zs}(W',\cF',\frs') - G_{\zs}(W,\cF,\frs)\right),
\]
which completes the proof in the case when $h(S_0'\cup D_K, \frs') + g(S_{0,\zs}') - g(S_{0,\ws}') \ge 0$.

The proof when $h(S_0'\cup D_K, \frs') + g(S_{0,\zs}') - g(S_{0,\ws}') \le 0$ is an easy modification,
using the corresponding subcase of Lemma~\ref{lem:closedsurfacecomputation}.
\end{proof}

Lemma~\ref{lem:stabilizationlemma2} also immediately computes
the effect of a generalized stabilization on the $t$-modified
versions of the link cobordism maps:

\begin{cor}\label{cor:tmodifiedstabilizationformula} If $(W',\cF')$
 is a generalized stabilization of $(W,\cF)$ and $\frs'\in \Spin^c(W')$
  restricts to $\frs$ on $W\setminus B^4$, then
\[
\tF_{W',\cF',\frs'}\simeq v^{-G_t(W',\cF',\frs')+G_t(W,\cF,\frs)}\cdot \tF_{W,\cF,\frs}.
\]
\end{cor}

We can now prove Theorem~\ref{thm:Upsilonbound}:

\begin{proof}[Proof of Theorem~\ref{thm:Upsilonbound}]
Fix $t \in [0,2]$. Suppose $\vec{\cS} = \{\cS_1, \dots, \cS_n\}$ is a
stabilization sequence connecting $(B^4, S)$ and $(B^4, S')$, and write
$\cS_i = (W_i, S_i)$. Decorate each $S_i$ with a dividing set $\cA_i$ consisting
of a single arc, such that the type-$\ws$ subregion has genus 0, and the
type-$\zs$ subregion has genus $g(S_i)$. We can assume the dividing arc is
very near to the knot $K$, and the the type-$\ws$
 subregion is unaffected by any of the stabilizations.
Write $\cF_i = (S_i, \cA_i)$.

Let us call a sequence $\vec{\frs} = \{\frs_1,\dots, \frs_n\}$ of $\Spin^c$ structures on
$W_1, \dots, W_n$, respectively, a \emph{stabilization sequence} if,
whenever $(W_{i+1}, S_{i+1})$ is obtained by stabilizing
$(W_i, S_i)$ with the negative definite link cobordism $(X_0, S_0)$, the
$\Spin^c$ structure $\frs_{i+1}$ can be written as $\frs_i \# \frt_i$ for
$\frt_i \in \Spin^c(X_0)$. We require an analogous condition whenever $W_{i+1}$ is
a generalized destabilization of $W_i$. We define
\[
M_{\vec{\cS},\vec{\frs}}(t) := \max_{1\le i\le n} -G_t(W_i,\cF_i,\frs_i) =
\max_{1\le i\le n} -H_t(W_i, [S_i], \frs_i) + t \cdot g(S_i),
\]
where the second equality follows from equation~\eqref{eqn:Gt} and the fact that $g(S_{i,\ws}) = 0$
and $g(S_{i,\zs}) = g(S_i)$.

By Corollary~\ref{cor:tmodifiedstabilizationformula},
if $\vec{\frs}=\{\frs_1, \dots, \frs_n\}$ is a
stabilization sequence of $\Spin^c$ structures on $W_1, \dots, W_n$,
then all of the elements
\[
v^{ M_{\vec{\cS},\vec{\frs}}(t) + G_t(W_i,\cF_i,\frs_i)} \cdot [\tF_{W_i,\cF_i,\frs_i}(1)]
\]
coincide in $\tHFK^-(\bK)$. In particular,
\begin{equation}
v^{ M_{\vec{\cS},\vec{\frs}}(t) - t \cdot g(S)} \cdot [\tF_{S, \zs}(1)] =
v^{ M_{\vec{\cS},\vec{\frs}}(t) - t \cdot g(S')} \cdot [\tF_{S', \zs}(1)],
\label{eq:stabilizationsequencewithrandomspinc}
\end{equation}
as $G_t(W_1,\cF_1,\frs_1) = -t \cdot g(S)$ and $G_t(W_n,\cF_n,\frs_n) = -t \cdot g(S')$.

Suppose that $(W_{i \pm 1},S_{i \pm 1})$ is obtained by stabilizing $(W_i, S_i)$ using
the negative definite link cobordism $(X_0, S_0)$, and
$\frs_{i} \in \Spin^c(W_i)$ and $\frt_i \in \Spin^c(X_0)$. Then
\[
H_t(W_i \# X_0, [S_i] + [S_0], \frs_i \# \frt_i) = H_t(W_i, [S_i], \frs_i) + H_t(X_0, [S_0], \frt_i).
\]
Hence, the $\Spin^c$ structure $\frs_{i} \# \frt_i$ minimizes
$-H_t(W_{i \pm 1}, [S_{i \pm 1}], \frs_i \# \frt_i)$ for a fixed $t$ if and
only if $\frs_i$ minimizes $-H_t(W_i, [S_i], \frs_i)$ and $\frt_i$ minimizes $-H_t(X_0, [S_0], \frt_i)$.
It follows that we can always construct a stabilization sequence of $\Spin^c$ structures
$\vec{\frs} = \{\frs_1, \dots, \frs_n\}$ such that
\[
-H_t(W_i,[S_i],\frs_i) = \min_{\frs \in \Spin^c(W_i)} - H_t(W_i,[S_i],\frs) = M_{(W_i, S_i)}(t) - t \cdot g(S_i).
\]
Then
\begin{equation}
M_{\vec{\cS},\vec{\frs}}(t) = \max_{1\le i\le n} -H_t(W_i, [S_i], \frs_i) + t \cdot g(S_i) =
\max_{1\le i\le n} M_{(W_i, S_i)}(t) = M_{\vec{\cS}}(t).
\label{eq:maximalstabseqspinc}
\end{equation}
Combining equations~\eqref{eq:stabilizationsequencewithrandomspinc}
 and~\eqref{eq:maximalstabseqspinc}, we conclude that
\begin{equation}
\Upsilon_{(S,S')}(t) \le M_{\vec{\cS}}(t)
\label{eq:UpsilonleMPon01}
\end{equation}
for any $t \in [0,2]$.
Minimizing equation~\eqref{eq:UpsilonleMPon01} over all
$\vec{\cS} \in P_{\st}(S, S')$ yields the result.
\end{proof}

\section{The summand-swapping diffeomorphism}
\label{sec:rigidmotiondeformation}

If $K$ is a knot in $\Sphere^3$, one can construct an order $n$ automorphism of the knot
$K^{\# n}$, corresponding to cyclically permuting the summands. In this
section, we investigate the case when $n = 2$, and compute the induced map on
knot Floer homology. We will use this in Section~\ref{sec:examples} to construct
pairs of slice disks for which we can explicitly compute the secondary invariants
defined in Section~\ref{sec:invariants}.

\subsection{Construction of the diffeomorphism map}\label{sec:construct-rigidmotion}
If $K \subset \Sphere^3$ is a knot, there is a diffeomorphism
\[
R^{\pi} \colon (\Sphere^3, K\# K) \to (\Sphere^3, K\# K)
\]
that switches the two summands of $K \# K$. In fact, for an appropriate embedding of
$K \# K$ into $\Sphere^3$, the diffeomorphism $R^{\pi}$ can be realized as
an order 2 rigid motion of $\Sphere^3$: Isotope $K$ into the $y \ge 1$ half-space of $\R^3 \subset \Sphere^3 = \R^3 \cup \{\infty\}$
such that the line segment $I := [-1,1] \times \{(1,0)\} \subset K$. For $\varphi \in \R$,
let $R^\varphi$ be $\varphi$-rotation about the $z$-axis. If we let $K \# K$ be the equivariant smoothing of
\[
(K \setminus I) \cup R^{\pi}(K \setminus I) \cup R^{\pi/2}(I) \cup R^{-\pi/2}(I),
\]
then $R^{\pi}$ is an orientation-preserving automorphism of $(\Sphere^3, K \# K)$.

In particular, the knot $K \# K$ is 2-periodic, and has no fixed points.
We pick two basepoints, $w$, $z \in K \setminus I$, such that $w$ follows $z$
with respect to the orientation of $K$. We let $w'$ and $z'$ denote their images
on $R^{\pi}(K \setminus I)$ under the map $R^{\pi}$. We write $\bK = (K, w, z)$
and $\bK \# \bK = (K \# K, w, z)$.

We define the element
\[
\ve{R}^{\pi} := \rho \circ R^\pi \in \MCG(\Sphere^3, K\# K, w, z),
\]
where
\[
\rho \colon (\Sphere^3, K \# K, w', z') \to (\Sphere^3, K \# K, w, z)
\]
is a half-twist diffeomorphism in the direction of the knot's orientation
that swaps the pairs $(w, z)$ and $(w', z')$.
We note that the diffeomorphism $(\mathbf{R}^{\pi})^2$ is isotopic to a full Dehn twist along $K \# K$.
Hence, by \cite{ZemQuasi}*{Theorem~B},
\begin{equation}
(\ve{R}^{\pi}_*)^2 \simeq \id + \Phi_w^{\bK\# \bK} \circ \Psi_z^{\bK\# \bK},
\label{eq:rigidmotionsquarestoSarkarmap}
\end{equation}
where $\Phi_w^{\bK\# \bK}$ and $\Psi_z^{\bK\# \bK}$ are the basepoint actions on $\cCFL^\infty(\bK \# \bK)$
described in Section~\ref{sec:basepointactions}; see also
the work of Sarkar~\cite{SarkarMovingBasepoints}.

\begin{rem}
 Note that any $\bF[U,V]$-equivariant homotopy equivalence
 \[
 \cCFL^\infty(\bK\# \bK)\simeq \cCFL^\infty(\bK)\otimes \cCFL^\infty(\bK)
 \]
will intertwine $\Phi_w^{\bK\# \bK}$  (resp. $\Psi_z^{\bK\# \bK}$) with
$\Phi_w\otimes \id+\id\otimes \Phi_w$ (resp. $\Psi_z \otimes \id+\id\otimes \Psi_z$), up to chain homotopy. This is because if $C$ and $C'$ are free chain complexes over $\bF[U,V]$ and $F\colon C\to C'$ is a chain map, one easily shows that $F\circ \Phi\simeq \Phi'\circ F$, where $\Phi$ and $\Phi'$ are the algebraic analogs of the map $\Phi_w$ on $C$ and $C'$. 
\end{rem}

We now wish to compute a formula for the chain homotopy type of the induced
map $\ve{R}^{\pi}_*$. Note that there is a filtered chain map
\[
\Sw \colon \cCFL^\infty(\bK) \otimes\cCFL^\infty(\bK) \to
\cCFL^\infty(\bK) \otimes \cCFL^\infty(\bK),
\]
obtained by switching the two factors. Note that $\Sw$ cannot be chain
homotopic to $\ve{R}^{\pi}$, since $\Sw \circ \Sw = \id$, which would violate
equation~\eqref{eq:rigidmotionsquarestoSarkarmap}. In this section, we prove the following:

\begin{thm}\label{thm:rigidmotionformula}
Let $\bK = (K, w, z)$ be a doubly-based knot in $\Sphere^3$, and
consider the doubly-based knot $\bK \# \bK = (K \# K, w, z)$ defined above. Then there is a
filtered chain homotopy equivalence between $\cCFL^\infty(\bK \# \bK)$ and
$\cCFL^\infty(\bK) \otimes_{\cR^\infty} \cCFL^\infty(\bK)$ that intertwines $\ve{R}^{\pi}_*$ with
\[
\Sw \circ \left(\id \otimes \id + \id \otimes (\Phi_w \circ \Psi_z) +
\Psi_z \otimes \Phi_w \right).
\]
\end{thm}

\begin{rem}
We say the endomorphisms $F$ and $G$ of a chain complex $C$ are
\emph{homotopy conjugate} if there is a homotopy automorphism
$A \colon C \to C$ such that $F \circ A \simeq A \circ G$.
It is not hard to see that the four maps
\[
\begin{split}
&\Sw \circ \left(\id \otimes \id + \id \otimes (\Phi_w \circ \Psi_z) + \Psi_z \otimes \Phi_w \right), \\
&\Sw \circ \left(\id \otimes \id + \id \otimes (\Phi_w \circ \Psi_z) + \Phi_w \otimes \Psi_z \right), \\
&\Sw \circ \left(\id \otimes \id  + (\Phi_w \circ \Psi_z) \otimes \id + \Psi_z \otimes \Phi_w \right), \\
&\Sw \circ \left(\id \otimes \id  + (\Phi_w \circ \Psi_z) \otimes \id + \Phi_w \otimes \Psi_z \right)
\end{split}
\]
are all homotopy conjugate endomorphisms of  $\cCFL^\infty(\bK)\otimes \cCFL^\infty(\bK)$.
Indeed, the map $A$ can be taken to be one of the maps $\Sw$, $\id \otimes \id + \Phi_w \otimes \Psi_z$,
or $\id \otimes \id + \Psi_z \otimes \Phi_w$,
since $\Phi^2_w \simeq 0$, $\Psi^2_z \simeq 0$, and $\Phi_w \circ \Psi_z \simeq \Psi_z \circ \Phi_w$.
\end{rem}

\subsection{Proof of the formula for the summand-swapping diffeomorphism map}
\label{sec:proofrigidmotiondiffeo}

\begin{proof}[Proof of Theorem~\ref{thm:rigidmotionformula}]
Let us write $\bK \# \bK = (K \# K, w, z)$, where $w$ and $z$ appear on the right copy of $K$.
Let $w'$ and $z'$ denote their images under $R^{\pi}$, on the left copy of $K$.
We define our connected sum map
\[
E \colon \cCFL^\infty(\bK\# \bK) \to \cCFL^\infty(K, w', z') \otimes_{\cR^\infty} \cCFL^\infty(\bK)
\]
as the composition
\[
E := F_{\bS_2} F_B^{\zs} T_{w',z'}^+,
\]
where $F_{\bS_2}$ denotes the 3-handle map induced by the framed 2-sphere $\bS_2$
that separates the two copies of $K$ after the band surgery along $B$.
A schematic of the link cobordism corresponding to $E$
is shown in Figure~\ref{fig::19}.

\begin{figure}[ht!]
	\centering
	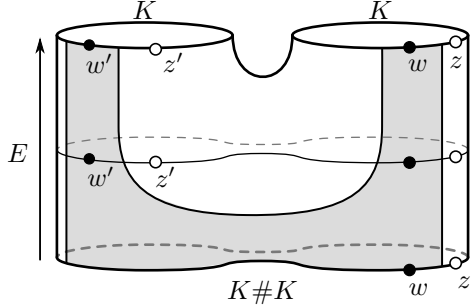
	\caption{A schematic of the map $E := F_{\bS_2} F_{B}^{\zs} T_{w',z'}^+$.
    A decomposition of the surface is shown, corresponding to the factors of $F_B^{\zs}$ and $T_{w',z'}^+$ in $E$.
    The 3-handle map $F_{\bS_2}$ is not shown.}
    \label{fig::19}
\end{figure}

The map $E$ is a chain homotopy equivalence. Indeed, $E$ is the map
induced by a pair-of-pants link cobordism in a 3-handle cobordism
that is diffeomorphic to the reverse of one of the two connected sum
cobordisms constructed in \cite{ZemConnectedSums}*{Section~5}
(in fact, it is diffeomorphic to the link cobordism inducing the map $E_1$, described therein).
According to \cite{ZemCFLTQFT}*{Proposition~5.1}, the map $E$ is a chain homotopy equivalence,
and a homotopy inverse is given by turning around and reversing the orientation of the cobordism.

Expanding the definitions of $E$ and $\ve{R}^{\pi}$, and observing that $\rho \circ R^\pi$ and
$R^\pi \circ \rho$ are equal as automorphisms of $(S^3, \bK \# \bK)$, 
we have
\begin{equation}
E \ve{R}^\pi_* \simeq F_{\bS_2} F_{B}^{\zs} T_{w',z'}^+ R^{\pi}_* \rho_*.
\label{eq:expandERpi}
\end{equation}
We note that, for the doubly-based knot $\rho(\bK \# \bK) = (K \# K, w', z')$, we have 
\begin{equation}
T_{w',z'}^+ R^{\pi}_* \simeq R^{\pi}_* T_{w,z}^+
\label{eq:quasiandR}
\end{equation}
by the functoriality of the quasi-stabilization operation. Similarly,
\begin{equation}
F_{B}^{\zs} R^{\pi}_* \simeq R^{\pi}_* F_{B}^{\zs},
\label{eq:bandmapsandR}
\end{equation}
since the diffeomorphism $R^{\pi}$ preserves the connected sum band $B$ setwise.
Finally, we note that
\begin{equation}
F_{\bS_2} R^{\pi}_* \simeq \Sw F_{\bS_2},
\label{eq:commuteRand3-handlemap}
\end{equation}
since the framed sphere $\bS_2$ is fixed setwise by $R^{\pi}$.
We remark that, in equation~\eqref{eq:commuteRand3-handlemap},
$R^{\pi}$ reverses the orientation of the framed 2-sphere $\bS_2$,
 though this has no effect on the cobordism map.

Applying the relations from equations~\eqref{eq:quasiandR}, \eqref{eq:bandmapsandR},
and~\eqref{eq:commuteRand3-handlemap} to equation~\eqref{eq:expandERpi} yields
\begin{equation}
E\ve{R}^{\pi}_* \simeq \Sw F_{\bS_2} F_B^{\zs} T_{w,z}^+ \rho_*.
\label{eq:movedRpitoleft}
\end{equation}

Next, we examine the expression $F_{\bS_2} F_B^{\zs} T_{w,z}^+ \rho_*$
appearing in equation~\eqref{eq:movedRpitoleft}. We perform the following manipulations:
\begin{equation}
\begin{alignedat}{3}
F_{\bS_2} F_B^{\zs} T_{w,z}^+ \rho_* &\simeq
F_{\bS_2} F_B^{\zs} T_{w,z}^+ S_{w,z}^- T_{w',z'}^+
&&\quad (\text{equation}~\eqref{eq:R5}) \\
&\simeq  F_{\bS_2} F_{B}^{\zs} T_{w',z'}^+ + F_{\bS_2} F_B^{\zs} S_{w,z}^+ T_{w,z}^- T_{w',z'}^+
&&\quad (\text{equation}~\eqref{eq:R6}).
\end{alignedat}
\label{eq:bypassrel1}
\end{equation}
Next, we compute
\begin{equation}
\begin{alignedat}{3}
F_{\bS_2} F_B^{\zs} S_{w,z}^+ T_{w,z}^- T_{w',z'}^+ &\simeq
F_{\bS_2} F_B^{\zs} S_{w,z}^+ S_{w,z}^- \Psi_z  T_{w',z'}^+
&&\quad (\text{equation}~\eqref{eq:R2}) \\
&\simeq F_{\bS_2} F_B^{\zs} \Phi_w \Psi_z T_{w',z'}^+
&&\quad (\text{equation}~\eqref{eq:R3}).
\end{alignedat}
\end{equation}

We note that $\Psi_{z'} T_{w',z'}^+ \simeq 0$,
since $T_{w',z'}^+ \simeq \Psi_{z'} S_{w',z'}^+$ and $\Psi_{z'}^2 \simeq 0$
by equations~\eqref{eq:R1} and~\eqref{eqn:R7}. Hence
\begin{equation}
F_{\bS_2} F_B^{\zs} \Phi_w \Psi_z  T_{w',z'}^+ \simeq
F_{\bS_2} F_B^{\zs} \Phi_w (\Psi_z + \Psi_{z'}) T_{w',z'}^+.
\end{equation}

From equations~\eqref{eq:R7} and~\eqref{eq:R8}, we conclude that
\begin{equation}
F_{\bS_2} F_{B}^{\zs} \Phi_w (\Psi_z + \Psi_{z'}) T_{w',z'}^+ \simeq
F_{\bS_2} \Phi_w (\Psi_z + \Psi_{z'}) F_{B}^{\zs} T_{w',z'}^+.
\end{equation}
Finally, we note that $\Phi_{w}$, $\Psi_{z}$, and $\Psi_{z'}$
commute with $F_{\bS_2}$ by \cite{ZemCFLTQFT}*{Lemma~8.3}, hence
\begin{equation}
F_{\bS_2} \Phi_w (\Psi_z + \Psi_{z'}) F_{B}^{\zs} T_{w',z'}^+ \simeq
(\id \otimes (\Phi_w \Psi_{z}) +  \Psi_{z'} \otimes \Phi_w)
F_{\bS_2} F_{B}^{\zs} T_{w',z'}^+.
\label{eq:commutePhiPsiwith3-handle}
\end{equation}

Combining, equations~\eqref{eq:bypassrel1}--\eqref{eq:commutePhiPsiwith3-handle},
and using the fact that $E := F_{\bS_2} F_B^{\zs} T_{w',z'}^+$, we see that
\begin{equation}
F_{\bS_2} F_{B}^{\zs}T_{w,z}^+ \rho_* \simeq
(\id \otimes \id + \id \otimes \Phi_w \Psi_z + \Psi_{z'} \otimes \Phi_w) E.
\label{eq:manipulateexpressionrecoverE}
\end{equation}
By applying $\Sw$ to equation~\eqref{eq:manipulateexpressionrecoverE}, and
combining it with equation~\eqref{eq:movedRpitoleft}, we obtain that
\[
E\ve{R}^{\pi}_* \simeq \Sw (\id \otimes \id + \id \otimes \Phi_w \Psi_z + \Psi_{z'} \otimes \Phi_w) E.
\]
Upon relabeling $w'$ and $z'$ as $w$ and $z$, we obtain the formula in the statement.
\end{proof}

\section{The trace formula}
\label{sec:trace}

If $(Y,\bL)$ is a multi-based link, the identity decorated link cobordism
\[
(W_{\id}, \cF_{\id})\colon (Y,\bL)\to (Y,\bL)
\]
is constructed by decorating $(I\times Y,I\times L)$ with a dividing set
$\cA=I\times \ve{p}$, where $\ve{p}\subset L$ consists of one point in each
component of $L\setminus (\ws\cup \zs)$.

By changing which ends of $(W_{\id}, \cF_{\id})$ are designated as incoming
or outgoing, we get two other decorated link cobordisms, which we denote by
\[
(W_{\tr}, \cF_{\tr}) \colon (-Y \sqcup Y, -\bL \sqcup \bL) \to \emptyset \quad\text{and}\quad
(W_{\cotr}, \cF_{\cotr}) \colon \emptyset \to (Y \sqcup -Y, \bL \sqcup -\bL).
\]

The $\cR$-module $\cCFL^\infty(-Y,-\bL,\frs)$ is canonically isomorphic to
$\Hom_{\cR^\infty}(\cCFL^\infty(Y,\bL,\frs),\cR^\infty)$, and hence there is
a \emph{canonical trace pairing}
\[
\tr \colon \cCFL^\infty(-Y,-\bL,\frs) \otimes_{\cR^\infty} \cCFL^\infty(Y,\bL,\frs) \to \cR^\infty.
\]
Similarly there is a canonical cotrace map
\[
\cotr\colon \cR^\infty\to \cCFL^\infty(Y,\bL,\frs)\otimes_{\cR^\infty} \cCFL^\infty(-Y,-\bL,\frs),
\]
obtained by dualizing the trace pairing.
In this section, we prove the following:

\begin{thm}\label{thm:traceforuma}
The trace and cotrace cobordisms induce the canonical trace and cotrace maps:
\[
F_{W_{\tr}, \cF_{\tr},\frs}\simeq \tr \quad \text{and} \quad F_{W_{\cotr}, \cF_{\cotr},\frs}\simeq \cotr.
\]
\end{thm}

Our proof of Theorem~\ref{thm:traceforuma} is similar to the proofs of
\cite{ZemDualityMappingTori}*{Theorem~1.6} and
\cite{JZContactHandles}*{Theorem~1.2}.

\subsection{Heegaard triples and link cobordisms}

\begin{define}
We say that $(\Sigma,\as,\bs,\gs,\ws,\zs)$ is a \emph{Heegaard link triple}
if $(\Sigma,\as,\bs,\gs)$ is a Heegaard triple diagram decorated with two disjoint
collections of basepoints, $\ws$ and $\zs$. Furthermore, for each $\sigmas\in
\{\as,\bs,\gs\}$,  each component of $\Sigma\setminus \sigmas$ is planar, and
contains exactly one $\ws$ basepoint, and exactly one $\zs$ basepoint.
\end{define}

We remark that such a Heegaard triple is called a \emph{doubly multi-pointed Heegaard triple} in~\cite{ZemCFLTQFT}.
If $(\Sigma,\as,\bs,\gs,\ws,\zs)$ is a Heegaard link triple
and $\sigmas, \etas \in \{\as,\bs,\gs\}$, then we write $(Y_{\sigma, \eta}, \bL_{\sigma, \eta})$
for the multi-based link defined by the diagram $(\Sigma, \sigmas, \etas, \ws, \zs)$.
There is a decorated link cobordism
\[
(X_{\a,\b,\g},\cF_{\a,\b,\g}) \colon (Y_{\a,\b}\sqcup Y_{\b,\g}, \bL_{\a,\b} \sqcup \bL_{\b,\g}) \to
(Y_{\a,\g}, \bL_{\a,\g}),
\]
described in \cite{JZContactHandles}*{Section~9.4}, which is a
refinement of the construction from \cite{OSDisks}*{Section~8.1}. The
4-manifold $X_{\a,\b,\g}$ is constructed as the union
\[
X_{\a,\b,\g} := (\Delta \times \Sigma) \cup (e_\a \times U_\a) \cup (e_{\b} \times U_{\b})\cup (e_{\g} \times U_{\g}),
\]
where $U_{\a},$ $U_{\b}$, and $U_{\g}$ denote handlebodies with boundary
$\Sigma$, with compressing curves $\as$, $\bs$, and~$\gs$, respectively.

The decorated surface $\cF_{\a,\b,\g} = (S_{\a,\b,\g}, \cA_{\a,\b,\g})$ is constructed as follows.
We pick embedded paths in $\Sigma \setminus \as$, $\Sigma \setminus \bs$, and
$\Sigma \setminus \gs$ connecting the $\zs$-basepoints to the
$\ws$-basepoints, and then push the interiors of these arcs into the
interior of $U_{\a}$, $U_{\b}$, or $U_{\g}$, respectively.
We obtain three sets of properly embedded arcs
$\ell_{\a} \subset U_{\a}$, $\ell_{\b} \subset U_{\b}$, and $\ell_{\g}\subset U_{\g}$.
The surface $S_{\a,\b,\g}$ is defined as
\[
S_{\a,\b,\g} := (\Delta \times (\ws \cup \zs)) \cup (e_{\a} \times \ell_{\a}) \cup
(e_{\b} \times \ell_{\b}) \cup (e_{\g} \times \ell_{\g}).
\]
To obtain $\cA_{\a,\b,\g}$, choose subsets $\ve{p}_{\a} \subset \ell_{\a}$,
$\ve{p}_{\b} \subset \ell_{\b}$, and $\ve{p}_{\g} \subset \ell_{\g}$ consisting of
one point in the interior of each component of $\ell_{\a}$, $\ell_{\b}$, and
$\ell_{\g}$, and set
\[
\cA_{\a,\b,\g} := (e_{\a} \times \ve{p}_{\a}) \cup (e_{\b} \times \ve{p}_{\b}) \cup (e_{\g} \times \ve{p}_{\g}).
\]

\begin{thm}\label{thm:heegaardtriplemap}
If $(\Sigma,\as,\bs,\gs,\ws,\zs)$ is a Heegaard link triple, then the
cobordism map $F_{W_{\a,\b,\g},\cF_{\a,\b,\g},\frs}$ is chain homotopic to
the holomorphic triangle map
\[
F_{\a,\b,\g,\frs} \colon \cCFL^\infty(\Sigma,\as,\bs,\ws,\zs,\frs|_{Y_{\a,\b}}) \otimes \cCFL^\infty(\Sigma,\bs,\gs,\ws,\zs,\frs|_{Y_{\b,\g}}) \to \cCFL^\infty(\Sigma,\as,\gs,\ws,\zs,\frs|_{Y_{\a,\g}}),
\]
defined by the formula
\[
F_{\a,\b,\g,\frs}(\xs\otimes \ys) := \sum_{\zs\in \bT_{\a}\cap \bT_{\g}} \sum_{\substack{\psi\in \pi_2(\xs,\ys,\zs)\\
\mu(\psi)=0\\
\frs_{\ws}(\psi)=\frs}} \# \cM(\psi) \cdot U^{n_{\ws}(\psi)} V^{n_{\zs}(\psi)} \cdot \zs
\]
for $\x \in \T_{\a} \cap \T_{\b}$ and $\y \in \T_{\b} \cap \T_{\g}$.
\end{thm}

We now demonstrate that the trace formula follows quickly from Theorem~\ref{thm:heegaardtriplemap}:

\begin{proof}[Proof of Theorem~\ref{thm:traceforuma}]
We will focus on the claim that $F_{W_{\tr},\cF_{\tr},\frs}\simeq \tr$. The
claim about the cotrace cobordism follows from the formula for the trace
cobordism. Indeed, if $(W,\cF)$ is a decorated link cobordism from
$(Y_1,\bL_1)$ to $(Y_2,\bL_2)$, and
\[
(W^\vee,\cF^\vee)\colon (-Y_2,-\bL_2)\to (-Y_1,-\bL_1)
\]
is the cobordism obtained by turning around $(W,\cF)$, then it is
straightforward to adapt the proof of \cite{OSTriangles}*{Theorem~3.5} to see
that $F_{W^\vee,\cF^\vee,\frs}$ is equal to the dual map
\[
(F_{W,\cF,\frs})^\vee \colon \Hom_{\cR^\infty}(\cCFL^\infty(Y_2,\bL_2,\frs|_{Y_2}), \cR^\infty) \to \Hom_{\cR^\infty}(\cCFL^\infty(Y_1,\bL_1,\frs|_{Y_1}), \cR^\infty).
\]

To establish the formula for $F_{W_{\tr}, \cF_{\tr},\frs}$, we pick a diagram
$(\Sigma,\as,\bs,\ws,\zs)$ for $(Y,\bL)$, as well as a small Hamiltonian
translate $\bs'$ of $\bs$, and we consider the Heegaard triple
$(\Sigma,\bs',\as,\bs,\ws,\zs)$.  The decorated link cobordism
$(X_{\b',\a,\b},\cF_{\b',\a,\b})$ is in fact equal to $(I\times Y,I\times L)$
with a neighborhood of $\{\tfrac{1}{2}\}\times U_{\b}$ removed. Hence, using
Theorem~\ref{thm:heegaardtriplemap} and the composition law,  we can write
\begin{equation}
F_{W_{\tr},\cF_{\tr},\frs}(\xs\otimes \ys) = (G\circ F_{\b',\a,\b,\frs})(\xs\otimes \ys),
\label{eq:trace=3-handlefollowedbytrace}
\end{equation}
where $\x \in \T_{\b'} \cap \T_{\a}$ and $\y \in \T_{\a} \cap \T_{\b}$,
and $G$ is a composition of 3-handle and 4-handle maps. The map $G$ takes the form
\[
G(\Theta) = \begin{cases}1& \text{ if } \Theta = \Theta_{\b',\b}^-, \\
0& \text{ otherwise}
\end{cases}
\]
for $\Theta \in \T_{\b'} \cap \T_{\b}$, and extended $\cR^\infty$-equivariantly. On the other hand
equation~\eqref{eq:trace=3-handlefollowedbytrace} says that
$F_{W_{\tr}, \cF_{\tr},\frs}(\xs \otimes \ys)$ is exactly equal to the count of Maslov index~0
holomorphic triangles with vertices $\xs$, $\ys$, and
$\Theta_{\b',\b}^-$. Note that $\Theta_{\b',\b}^- = \Theta_{\b,\b'}^+$, and that
the transition map
\[
\Phi_{\b\to \b'}^{\a} \colon \cCFL^\infty(\Sigma,\as,\bs,\ws,\zs,\frs) \to
\cCFL^\infty(\Sigma,\as,\bs',\ws,\zs,\frs)
\]
can be computed via the triangle
map $F_{\a,\b,\b',\frs}(- \otimes \Theta_{\b,\b'}^+)$. Observing that
$F_{\a,\b,\b',\frs}$ and $F_{\b',\a,\b,\frs}$ count the exact same
holomorphic triangles, we conclude that
\[
F_{W_{\tr},\cF_{\tr},\frs}(\xs \otimes \ys) = \tr( \xs\otimes \Phi_{\b \to \b'}^{\a}(\ys)),
\]
completing the proof.
\end{proof}

\subsection{Compound 1- and 3-handle maps and some related counts of holomorphic curves}

Suppose that $(\Sigma,\as,\bs,\ws,\zs)$ and
$(\Sigma_0,\xis,\zetas,\ws_0,\zs_0)$ are two multi-pointed Heegaard diagrams,
and that we have a choice of injection
\[
i \colon \ws_0 \to \zs.
\]
Suppose further that $(\Sigma_0,\xis,\zetas,\ws_0,\zs_0)$ satisfies the following:
\begin{enumerate}[label=(\textit{D\arabic*}), ref=\textit{D\arabic*}]
\item\label{num:compdiagram1} The curves $\xis$ can be related to the curves $\zetas$ by a sequence of
    isotopies and handleslides
in the complement of $\ws_0$ and $\zs_0$.
\item\label{num:compdiagram2} Each $\ws_0$-basepoint is contained in the same region of
    $\Sigma_0\setminus (\xis\cup \zetas)$ as a $\zs_0$-basepoint.
\item\label{num:compdiagram3} $|\xi_i\cap \zeta_j|=2\delta_{ij}$, and $\xi_i\cap \zeta_i$ consists of
    two points that have relative Maslov grading~1.
\end{enumerate}

Note that condition~\eqref{num:compdiagram1} implies that
$(\Sigma_0,\xis,\zetas,\ws_0,\zs_0)$ is a diagram for an unlink $\bU$ in
$(\Sphere^1\times \Sphere^2)^{\# g(\Sigma_0)}$, each of whose components
contains exactly two basepoints. Condition~\eqref{num:compdiagram2} implies
that $\gr_{\ws}(\xs)=\gr_{\zs}(\xs)$ for any intersection point $\xs\in
\bT_{\xi}\cap \bT_{\zeta}$. Finally, condition~\eqref{num:compdiagram3}
implies that there is a top-graded intersection point
$\Theta^+_{\xi,\zeta}\in \bT_{\xi} \cap \bT_{\zeta}$, and a bottom-graded
intersection point $\Theta^-_{\xi,\zeta}\in \bT_{\xi}\cap \bT_{\zeta}$.

We form the surface $\Sigma\#_i \Sigma_0$ by joining $\Sigma$ and $\Sigma_0$
together with a connected sum tube for each point $w_0\in \ws_0$, which is
attached near the points $w_0$ and $i(w_0)$. Let us write
\[
\zs' := (\zs\setminus i(\ws_0))\cup \zs_0.
\]
There is an induced Heegaard diagram $(\Sigma \#_i \Sigma_0, \as\cup \xis,\bs\cup \zetas,\ws,\zs')$.

We define the \emph{compound 1-handle map}
\[
F_1^{\xi,\zeta}\colon \cCFL^\infty_{J_s}(\Sigma,\as,\bs,\ws,\zs) \to
\cCFL^\infty_{J_s(\ve{T})}(\Sigma \#_i \Sigma_0,\as\cup \xis,\bs\cup \zetas,\ws,\zs')
\]
via the formula
\[
F_1^{\xi,\zeta}(\xs) := \xs \times \Theta^+_{\xi,\zeta}
\]
for $\x \in \T_\a \cap \T_\b$, and extended $\cR^\infty$-equivariantly.
Here $J_s$ and $J_s(\ve{T})$ are 1-parameter families of almost complex
structures on $\S \times [0,1] \times \R$ and $(\Sigma \#_i \Sigma_0) \times [0,1] \times \R$,
respectively, that we will describe shortly.

Similarly there is a \emph{compound 3-handle map}
\[
F_3^{\xi,\zeta} \colon \cCFL^\infty_{J_s(\ve{T})}(\Sigma \#_i \Sigma_0, \as \cup \xis,\bs \cup \zetas,\ws,\zs')
\to \cCFL^\infty_{J_s}(\Sigma,\as,\bs,\ws,\zs),
\]
defined via the formula
\[
F_3^{\xi,\zeta}(\xs\times \Theta) := \begin{cases} \xs& \text{ if } \Theta = \Theta^-_{\xi,\zeta}, \\
0& \text{ if } \Theta\neq \Theta^-_{\xi,\zeta}
\end{cases}
\]
for $\x \in \T_\a \cap \T_\b$ and $\Theta \in \T_\xi \cap \T_\zeta$, and extended $\cR^\infty$-equivariantly.

We now wish to show that the compound 1-handle and 3-handle maps are chain
maps. This involves an argument involving analyzing how holomorphic curves
behave as one degenerates the almost complex structure.  Lipshitz's
cylindrical reformation of Heegaard Floer homology \cite{LipshitzCylindrical}
provides the technical framework necessary to perform the analysis. Let us
write $n=|\ws_0|$, the number of connected sum tubes we add.    Given almost
complex structures $J_s$ and $J_s'$ on $\Sigma\times [0,1]\times \R$ and $\Sigma_0\times
[0,1]\times \R$ that are split in a neighborhood of the connected sum
points, as well as an $n$-tuple of positive numbers $\ve{T}=(T_1,\dots,
T_n)$, we can form an almost complex structure $J_s(\ve{T})$ on $(\Sigma\#_i
\Sigma_0)\times [0,1]\times \R$ by inserting necks of length $T_1,\dots, T_n$
along the connected sum tubes.

\begin{prop}\label{prop:compound1-handleischainmap}
Suppose that $(\Sigma_0,\xis,\zetas,\ws_0,\zs_0)$ is a Heegaard diagram
satisfying conditions \eqref{num:compdiagram1}, \eqref{num:compdiagram2}, and
\eqref{num:compdiagram3}. If $\ve{T}$ is an $n$-tuple of neck lengths, all of
whose components are sufficiently large, the compound 1-handle map
$F_1^{\xi,\zeta}$ and the compound 3-handle map $F_3^{\xi,\zeta}$ are chain
maps.
\end{prop}

Proposition~\ref{prop:compound1-handleischainmap} follows from a careful
analysis of how holomorphic curves in $(\Sigma\#_i \Sigma_0)\times [0,1]\times
\R$ degenerate as one sends all  components of $\ve{T}$ to $+\infty$,
simultaneously.  The technical details of the proof can be found in
\cite{ZemDualityMappingTori}*{Proposition~6.1}.

There is an analogue of Proposition~\ref{prop:compound1-handleischainmap} for
the holomorphic triangle maps, which we will need for our proof of the trace
formula.  Suppose that $(\Sigma_0,\xis,\zetas,\taus,\ws_0,\zs_0)$ is a
Heegaard link triple satisfying the following:
\begin{enumerate}[label=(\textit{T}), ref=\textit{T}]
\item \label{num:propertytriple} All three sub-diagrams of $(\Sigma_0,\xis,\zetas,\taus,\ws_0,\zs_0)$ satisfy conditions~\eqref{num:compdiagram1}, \eqref{num:compdiagram2}, and~\eqref{num:compdiagram3}.
\end{enumerate}
Note that condition~\eqref{num:propertytriple} implies that there are
top-graded intersection points $\Theta_{\xi,\zeta}^+$, $\Theta_{\xi,\tau}^+$,
and $\Theta_{\zeta,\tau}^+$, as well as bottom-graded intersection
points $\Theta_{\xi,\zeta}^-$, $\Theta_{\xi,\tau}^-$, and $\Theta_{\zeta,\tau}^-$.

Suppose $\cT = (\Sigma,\as,\bs,\gs,\ws,\zs)$ and
$\cT_0 = (\Sigma_0,\xis,\zetas,\taus,\ws_0,\zs_0)$ are Heegaard link triples
such that the latter satisfies condition~\eqref{num:propertytriple} above,
and that we have a fixed injection
\[
i \colon \ws_0 \to \zs.
\]
We can construct a  surface $\Sigma\#_i \Sigma_0$  as we did before, as well as a Heegaard link triple
\[
\cT \# \cT_0 := (\Sigma \#_i \Sigma_0, \as \cup \xis, \bs \cup \zetas, \gs \cup \taus, \ws, \zs'),
\]
where $\zs' := (\zs \setminus i(\ws_0)) \cup \zs_0$.

Using a Mayer--Vietoris argument, it is not hard to see that there is an isomorphism
\[
\Spin^c(X_{\a\cup \xi,\b\cup \zeta,\g\cup \tau}) \iso \Spin^c(X_{\a,\b,\g}) \times \Spin^c(X_{\xi,\zeta,\tau}),
\]
under which $\frs_{\ws\cup \ws_0}(\psi\# \psi_0)$ is identified with
$(\frs_{\ws}(\psi), \frs_{\ws_0}(\psi_0))$. For triples
$(\Sigma_0,\xis,\zetas,\taus,\ws_0,\zs_0)$ satisfying
condition~\eqref{num:propertytriple}, the 4-manifold $X_{\xi,\zeta,\tau}$
becomes $\#^{g(\Sigma_0)} (\Sphere^1\times \Sphere^3)$ once we glue in 3- and
4-handles along the boundary. In particular, there is a unique $\Spin^c$
structure $\frs_0\in \Spin^c(X_{\xi,\zeta,\tau})$ which restricts to the
torsion $\Spin^c$ structure on all three boundary components. If $\frs\in
\Spin^c(X_{\a,\b,\g})$, there is thus a well-defined $\Spin^c$ structure
$\frs\# \frs_0\in \Spin^c(X_{\a\cup \xi,\b\cup \zeta,\g\cup \tau})$.

The holomorphic triangle counts from
\cite{ZemDualityMappingTori}*{Proposition~6.3} carry over to our present
situation without change to imply the following:

\begin{prop}\label{prop:holomorphictrianglecounts}
Suppose that $\cT=(\Sigma,\as,\bs,\gs,\ws,\zs)$ and
$\cT_0 = (\Sigma_0,\xis,\zetas,\taus,\ws_0,\zs_0)$ are Heegaard link triples
such that the latter satisfies condition~\eqref{num:propertytriple}, and $i\colon \ws_0\to \zs$
is a chosen injection. Let $\cT \# \cT_0$ denote the Heegaard link triple described
above. Then, for a tuple of sufficiently large neck-lengths $\ve{T}$, the
following hold:
\begin{align*}
F_{ \cT\# \cT_0,J(\ve{T}),\frs\# \frs_0} (F_1^{\xi,\zeta}(-)\otimes F_1^{\zeta,\tau}(-)) &= F_1^{\xi,\tau} F_{\cT,J,\frs}(-\otimes -),\\
F_3^{\xi,\tau}F_{ \cT \# \cT_0,J(\ve{T}),\frs \# \frs_0} (F_1^{\xi,\zeta}(-)\otimes -) &= F_{\cT,J,\frs}(-\otimes F_3^{\zeta,\tau}(-)),\\
F_3^{\xi,\tau}F_{\cT\# \cT_0,J(\ve{T}),\frs\# \frs_0} (-\otimes F_1^{\zeta,\tau}(-)) &= F_{\cT,J,\frs} (F_3^{\xi,\zeta}(-)\otimes -).
\end{align*}
\end{prop}

\begin{rem}
Consider the special case when
\[
(\Sigma_0,\xis,\zetas,\taus,\ws_0,\zs_0)=(\Sphere^2,\xi,\zeta,\tau,\{w_0,w_0'\},\{z_0,z_0'\})
\]
is a Heegaard triple where $w_0$ and $z_0$, and also $w_0'$ and $z_0'$
are adjacent, and where any two of $\xi$, $\zeta$, and $\tau$
intersect in two points. Then Propositions~\ref{prop:compound1-handleischainmap}
and~\ref{prop:holomorphictrianglecounts} imply more standard relations between the
holomorphic disk and triangle counts, and the 1-handle and 3-handle
maps, for 1-handles with feet attached near the $\zs$ basepoints on the
original Heegaard diagram. Compare \cite{OSTriangles}*{Theorem~2.14} and
\cite{ZemGraphTQFT}*{Lemma~8.5 and Theorem~8.8}.
\end{rem}

Finally, we need an additional holomorphic triangle count, due to Manolescu
and Ozsv\'{a}th \cite{MoIntSurg}*{Proposition~6.2}, 
which is useful in the
proof that the quasi-stabilization maps are well defined. Suppose that
$\cT=(\Sigma,\as,\bs,\gs,\ws,\zs)$ is a Heegaard triple, and write $A$ for a
distinguished component of $\Sigma\setminus \as$. Let $w$ and $z$ be the two
basepoints in $A$. Suppose $\alpha_s$ is a simple closed curve in $A$ that
divides $A$ into two components, one of which contains $w$, and the other
contains $z$. Let $p\in \alpha_s\setminus (\bs\cup \gs)$ be an arbitrary
choice of point. We form the quasi-stabilized Heegaard triple $\cT^+$ via the
formula
\[
\cT^+ := (\Sigma,\as\cup \{\alpha_s\}, \bs\cup \{\beta_0\}, \gs \cup \{\gamma_0\}, \ws\cup \{w_0\}, \zs\cup \{z_0\}),
\]
where $\a_s,$ $\b_0$, $\g_0,$ $w_0$, and $z_0$ are as shown in
Figure~\ref{fig::28}. (Compare Figure~\ref{fig::43}). The curves $\beta_0$ and $\gamma_0$  are contained in a
small disk centered at the point $p$. The basepoints $w_0$ and $z_0$ are both
contained in this disk, and are in the regions bounded by $\beta_0$ and $\gamma_0$.

\begin{figure}[ht!]
	\centering
	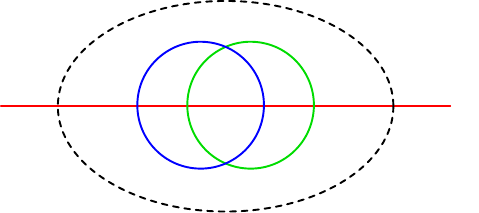
	\caption{The quasi-stabilized Heegaard triple
    $\cT^+=(\Sigma, \as\cup \{\alpha_s\}, \bs\cup \{\beta_0\}, \gs\cup \{\gamma_0\},\ws\cup \{w_0\}, \zs\cup \{z_0\})$
    considered in Proposition~\ref{prop:holomorphictrianglesandquasistabs}.}
    \label{fig::28}
\end{figure}

Abusing notation slightly, write
\[
\a_s\cap \b_0 = \{\theta^{\ws}, \theta^{\zs}\}, \qquad \a_s\cap \g_0 =
\{\theta^{\ws}, \theta^{\zs}\},\qquad \text{and} \qquad \beta_0\cap \gamma_0 = \{\theta^+, \theta^-\},
\]
where $\theta^{\ws}$ denotes the top $\gr_{\ws}$-graded intersection point,
and $\theta^{\zs}$ denotes the top $\gr_{\zs}$-graded intersection point.
When the relative gradings coincide, we write $\theta^+$ for the top-graded
intersection point, and $\theta^-$ for the bottom.

We note that, if $(\Sigma,\as,\bs,\gs,\ws,\zs)$ is a Heegaard triple and
$(\Sigma, \as\cup \{\a_s\}, \bs\cup \{\b_0\}, \gs\cup \{\g_0\},\ws,\zs)$ is a
quasi-stabilization, then the 4-manifolds $X_{\a,\b,\g}$ and $X_{\a\cup
\{\a_0\}, \b\cup \{\b_0\}, \g\cup \{\g_0\}}$ are canonically diffeomorphic,
since the handlebodies $U_{\a}$, $U_{\b}$, and $U_{\g}$ are unchanged after we
quasi-stabilize the triple.
In particular,
\[
\Spin^c(X_{\a,\b,\g})\iso \Spin^c(X_{\a\cup \{\a_0\}, \b\cup \{\b_0\}, \g\cup \{\g_0\}}).
\]
We will need the following holomorphic curve count:

\begin{prop}\label{prop:holomorphictrianglesandquasistabs}
Suppose that $\cT=(\Sigma,\as,\bs,\gs,\ws,\zs)$ is a Heegaard link triple,
and $\cT^+$ is its quasi-stabilization. Write $\bU$ for a doubly-based unknot
containing the basepoints $w_0$ and $z_0$, with Seifert disk $D$,
intersecting $\Sigma$ in an arc connecting $w_0$ and $z_0$ disjoint
from $\beta_0$ and $\gamma_0$. Then
\[
T_{w_0,z_0}^+ F_{\cT,\frs}(\x \otimes \y) =
F_{\cT^+}(T_{w_0,z_0}^+(\x) \otimes \cB^+_{\bU,D}(\y)) \text{ and }
S_{w_0,z_0}^+ F_{\cT}(\x \otimes \y) =
F_{\cT^+,\frs}(S_{w_0,z_0}^+(\x) \otimes \cB^+_{\bU,D}(\y)).
\]
\end{prop}

\begin{proof}
From the definitions of the maps $S_{w_0,z_0}^+$, $T_{w_0,z_0}^+$, and
$\cB_{\bU,D}^+$ (see Sections~\ref{sec:backgroundquasibasepoint} and
\ref{sec:birthsandquasis}), the  main claim is equivalent to the claim that
\[
F_{\cT,\frs}(\xs \otimes \ys) \times \theta^{\ws} =
F_{\cT^+,\frs}\left((\xs \times \theta^{\ws}) \otimes (\ys \times \theta^+)\right) \text{ and }
F_{\cT,\frs}(\xs \otimes \ys) \times \theta^{\zs} =
F_{\cT^+,\frs}\left((\xs \times \theta^{\zs}) \otimes (\ys\times \theta^+)\right),
\]
which is exactly the statement of \cite{MoIntSurg}*{Proposition~6.2}.
\end{proof}

\subsection{Twisted conjugate Heegaard diagrams for links}
\label{sec:twistedconjugateddiagrams}
Analogous to the proofs of the trace formulas in
\cite{ZemDualityMappingTori} and \cite{JZContactHandles},
to prove Theorem~\ref{thm:heegaardtriplemap}, we define a special kind
of operation on a Heegaard diagram for a link, whose result we will call the
\emph{twisted conjugate} of the original. We describe the construction presently.

If $\bL = (L,\ws,\zs)$ is a multi-based link, we write $\overline{\bL}$ for the
multi-based link $(L,\zs,\ws)$, obtained by switching the roles of the basepoints.
Given a Heegaard diagram $\cH=(\Sigma,\as,\bs,\ws,\zs)$ for $(Y,\bL)$, we can
obtain a diagram for $(Y,\overline{\bL})$ by reversing the orientation of
$\Sigma$, and switching the roles of $\as$ and $\bs$. The resulting diagram
$\bar{\cH}:=(\bar{\Sigma},\bs,\as,\zs,\ws)$ is referred to as the
\emph{conjugate} of $\cH$; see \cite{OSProperties}*{Section~2.2} and
\cite{HMInvolutive}*{Section~6.1}.

To obtain a diagram for $(Y,\bL)$, we can modify the embedding of
$\bar{\Sigma}$  in a neighborhood of $L$. By isotoping $\bar{\Sigma}$ along
$L$ in the positive direction according to the orientation of $L$,  we
obtain the \emph{positive twisted conjugate} diagram $\Tw^+(\bar{\cH})$.
Analogously, if we twist in the negative direction, we obtain the
\emph{negative twisted conjugate} $\Tw^-(\bar{\cH})$. The diagrams
$\Tw^+(\bar{\cH})$ and $\Tw^-(\bar{\cH})$ are illustrated in
Figure~\ref{fig::21}. We write $\Tw^+(\bar{\Sigma})$ and
$\Tw^-(\bar{\Sigma})$ for the underlying Heegaard surfaces of the twisted
conjugate diagrams.

\begin{figure}[ht!]
	\centering
	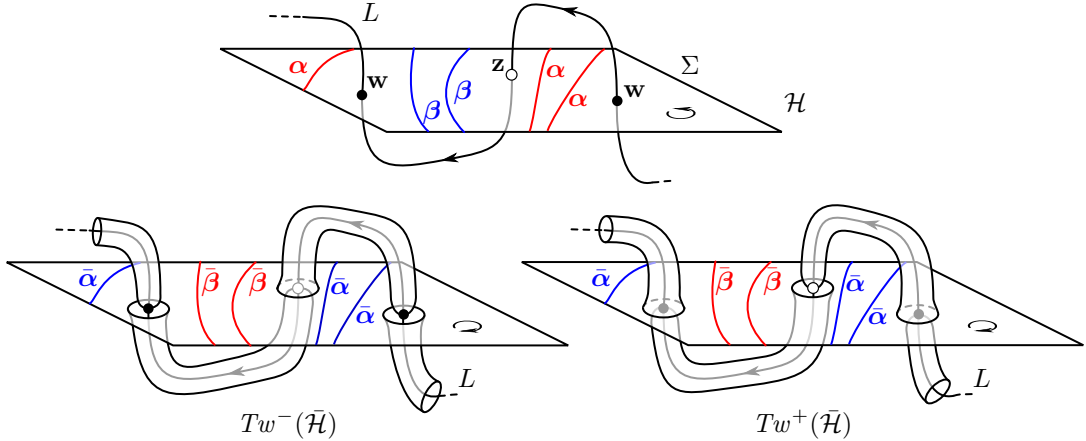
	\caption{On top, we show a link diagram $\cH$, and its twisted conjugates
    $\Tw^-(\bar{\cH})$ bottom left and $\Tw^+(\bar{\cH})$ bottom right.}
\label{fig::21}
\end{figure}

\subsection{Doubling Heegaard diagrams for links}\label{sec:doublingdiagrams}

An additional type of operation on Heegaard diagrams we will encounter in the proof
of Theorem~\ref{thm:heegaardtriplemap} is \emph{doubling}.
In the case of links, if $\cH=(\Sigma,\as,\bs,\ws,\zs)$ is a
diagram for $(Y,\bL)$, then there are four natural variations of the doubling
procedure, producing four diagrams
\begin{equation}
D^{\zs}_\a(\cH),\quad D^{\zs}_\b(\cH), \quad D^{\ws}_{\a}(\cH), \quad\text{and}\quad
D^{\ws}_{\b}(\cH).\label{eq:allfourdoubles}
\end{equation}
For our purposes, it will be sufficient to consider any one of the four
diagrams in equation~\eqref{eq:allfourdoubles}. We focus on
$D^{\zs}_{\a}(\cH)$, whose construction we describe presently.

We first construct the underlying Heegaard surface $D^{\zs}_{\a}(\Sigma)$.
Let $N(\Sigma) = [-1,1] \times \Sigma$ denote a regular
neighborhood of $\Sigma$ in $Y$. Let $N(\zs)$ denote a regular neighborhood
of the basepoints in $\Sigma$, which is a collection of disks, and let us write
\[
\Sigma_0 := \Sigma \setminus N(\zs).
\]
We define
\[
D^{\zs}_{\a,0}(\Sigma) := \d ([-1,0] \times \Sigma_0).
\]
We can view $D^{\zs}_{\a,0}(\Sigma)$ as being formed by gluing a copy of
$\Sigma \setminus N(\zs)$ to $\bar{\Sigma} \setminus N(\zs)$ along their
boundaries. Note that $D^{\zs}_{\a,0}(\Sigma) \cap L$ consists of the original
basepoints $\ws \subset \Sigma$, as well as another collection of basepoints
$\zs'$, which are the images of $\ws$ on $\bar{\Sigma}$.

We now describe the attaching curves on $D_{\a,0}^{\zs}(\Sigma)$.  Let
$m = |\ws| = |\zs|$.  We pick embedded and pairwise disjoint arcs $\lambda_1, \dots, \lambda_m$
on $\Sigma$, each traveling from a $\zs$-basepoint to a $\ws$-basepoint. We
assume further that each basepoint in $\ws\cup \zs$ is an endpoint of exactly
one $\lambda_i$. We assume that the interiors of the $\lambda_i$ are disjoint from $\ws\cup \zs$.

Next, we pick a collection $A$ of subarcs of $\d \Sigma_0$, such that each
component of $\d \Sigma_0$ contains exactly one subarc. We further require
that $A$ be disjoint from each $\lambda_i$. Pick a collection $d_1,\dots,
d_n$ of properly embedded and pairwise disjoint arcs on $\Sigma_0$ that have both boundary
components on $A$, are disjoint from the $\lambda_i$, and such that
they form a basis of $H_1(\Sigma_0, A)$.

By doubling the arcs $d_1,\dots, d_n$ across the connected sum tubes onto all
of $D_{\a,0}^{\zs}(\Sigma)$, we obtain $n$ pairwise disjoint simple closed
curves $\delta_1, \dots, \delta_n$ on $D_{\a,0}^{\zs}(\Sigma)$
that do not intersect the arcs $\lambda_i$. Let us write
\[
\Ds := \{\delta_1, \dots, \delta_n\}.
\]
We can now define an initial version of the doubled diagram as
\[
D_{\a,0}^{\zs}(\cH) := (D_{\a,0}^{\zs}(\Sigma), \as \cup \bar{\bs}, \Ds, \ws, \zs'),
\]
where $\bar{\bs}$ is the copy of $\bs$ on $\bar{\Sigma}$.

Via an isotopy of $Y$ supported in a neighborhood of $L$ that fixes
$\ws$, we can move $\zs'$ to $\zs$. We let $D^{\zs}_{\a}(\Sigma)$ denote the
Heegaard surface obtained by isotoping $D^{\zs}_{\a,0}(\Sigma)$ in such a
manner. The diagram $D^{\zs}_{\a}(\cH)$ is similarly obtained by pushing
forward the attaching curves on $D^{\zs}_{\a}(\cH)$ under such an isotopy;
see Figure~\ref{fig::20}.

\begin{figure}[ht!]
	\centering
	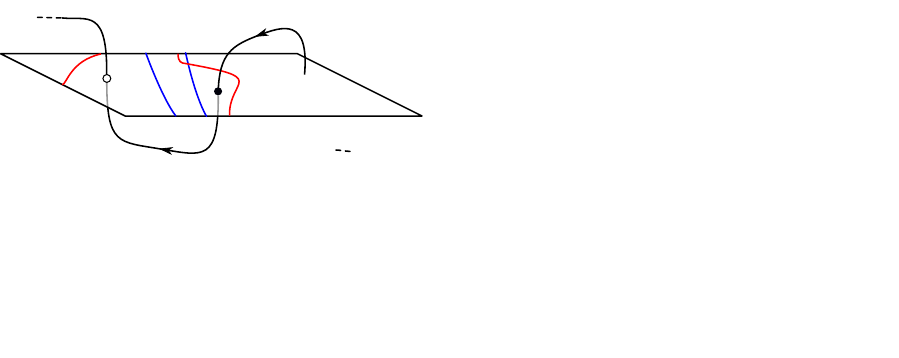
	\caption{The link diagram $\cH = (\Sigma,\as,\bs,\ws,\zs)$ is shown top left,
    the preliminary double $D_{\a,0}^{\zs}(\cH)$ top right,
    and the double $D_{\a}^{\zs}(\cH)$ at the bottom.
    The curves $\Ds$ bounding disks in the $\bs$-handlebody of $D_{\a}^{\zs}(\cH)$
    (i.e., the region between the two copies of $\Sigma$) are not shown.}
\label{fig::20}
\end{figure}

A diagram $D_{\b}^{\zs}(\cH)$ can be constructed using a variation of the
above construction, by having $[0,1] \times (\Sigma \setminus N(\zs))$ play the
role of the $\bs$-handlebody. Diagrams $D_{\a}^{\ws}(\cH)$ and
$D_{\b}^{\ws}(\cH)$ can be defined similarly, by instead adding tubes  near
the $\ws$-basepoints.

We now proceed to show that $D_{\a}^{\zs}(\cH)$ is a valid Heegaard diagram
for $(L,\ws,\zs)$. Note that it is clearly sufficient to show that
$D_{\a,0}^{\zs}(\cH)$ is a valid Heegaard diagram for $(L,\ws,\zs')$. To this
end, we prove the following fact about the $\Ds$ curves:

\begin{lem}\label{lem:homologicallyindependent}
The curves $\delta_1,\dots, \delta_n$ are homologically independent in both
$H_1(D_{\a,0}^{\zs}(\Sigma) \setminus N(\ws))$ and $H_1(D_{\a,0}^{\zs}(\Sigma) \setminus N(\zs'))$.
\end{lem}

\begin{proof}
Let $p_i$ denote the point $\lambda_i\cap \d \Sigma_0$, and write
$\ve{p}=\{p_1,\dots, p_m\}$. Since the $\lambda_i$, as well as their images
on $\bar{\Sigma}$, are disjoint from the $\delta_j$, if follows that
$\delta_1,\dots, \delta_n$ are homologically independent in
$H_1(D_{\a,0}^{\zs}(\Sigma)\setminus N(\ws))$ if and only if they are
independent in $H_1(D_{\a,0}^{\zs}(\Sigma)\setminus N(\zs'))$, which in turn
occurs if an only if they are homologically independent in
$H_1(D_{\a,0}^{\zs}(\Sigma)\setminus N(\ve{p}))$. Noting that
$D_{\a,0}^{\zs}(\Sigma)\setminus N(\ve{p})$ can be viewed as $\Sigma_0$ glued
to $\bar{\Sigma}_0$ along the arcs $A$, we consider the sequence
\begin{equation}
H_1(D_{\a,0}^{\zs}(\Sigma) \setminus N(\ve{p})) \to  H_1(D_{\a,0}^{\zs}(\Sigma) \setminus N(\ve{p}), \bar{\Sigma}_0) \to
H_1(\Sigma_0, A).
\label{eq:somemapsonhomology}
\end{equation}
Here, the first map is induced by inclusion, and the second is the inverse of the
excision isomorphism. Since the curves $\delta_1,\dots, \delta_n$ are mapped to $d_1,\dots,d_n$ by
the composition of the two maps in equation~\eqref{eq:somemapsonhomology},
which are homologically independent in $H_1(\Sigma_0, A)$, we
conclude that the curves $\delta_1,\dots, \delta_n$ are also homologically independent.
\end{proof}

Lemma~\ref{lem:homologicallyindependent} implies that $\Ds$ is a valid set of
attaching curves on $D_{\a,0}^{\zs}(\Sigma)$, as the following basic lemma
demonstrates:

\begin{lem}
Suppose that $\Sigma_0$ is a connected surface-with-boundary, and
$\delta_1,\dots, \delta_n \subset \Sigma_0$ is a collection of pairwise
disjoint simple closed curves, with $n=g(\Sigma_0)+|\d \Sigma_0|-1$.  Then
each component of $\Sigma_0 \setminus (\delta_1\cup \cdots \cup \delta_n)$ is
planar and contains exactly one component of $\d \Sigma_0$ if and only if
$\delta_1,\dots, \delta_n$ are  homologically independent $H_1(\Sigma_0)$.
\end{lem}

\begin{proof}
Assume first that each component of $\Sigma_0\setminus (\delta_1\cup \cdots
\cup \delta_n)$ is planar and contains exactly one component of $\d
\Sigma_0$.  Each curve $\delta_i$ determines two boundary components of
$\Sigma_0\setminus (\delta_1\cup \cdots \cup  \delta_n)$. A simple
Mayer--Vietoris argument for gluing along  these two boundary components shows
that the curves $\delta_1,\dots, \delta_n$ are homologically independent in
$H_1(\Sigma_0)$.

Conversely, suppose that $\delta_1,\dots, \delta_n$ are homologically
independent in $H_1(\Sigma_0)$. We note that, if any component of
$\Sigma_0 \setminus (\delta_1 \cup \cdots \cup \delta_n)$ does not contain a
component of $\d \Sigma_0$, then we obtain a non-trivial relation in
$H_1(\Sigma_0)$ amongst the $\delta_i$. Hence each component of
$\Sigma_0 \setminus (\delta_1 \cup \cdots \cup \delta_n)$ contains at least one
component of $\d \Sigma_0$. If any component $C$ of $\Sigma_0 \setminus
(\delta_1 \cup \cdots \cup \delta_n)$ is non-planar or contains more than one
component of $\d \Sigma_0$, then we can pick a simple closed curve $\delta'$
in $C$ such that $[\delta']$ is not in the span of the classes $[\delta_i]$,
for $\delta_i \subset \d C$. Using a Mayer--Vietoris argument, such a class
$[\delta']$ remains homologically independent from the classes $[\delta_i]$
in $H_1(\Sigma_0)$, and is clearly disjoint from the curves
$\delta_1, \dots, \delta_n$. However, it is easily verified that the maximal
rank of a subspace of $H_1(\Sigma_0)$ on which the intersection form
$Q_{\Sigma_0}$ vanishes is $n = g(\Sigma_0) + |\d \Sigma_0| - 1$, so such a curve
$\delta'$ cannot exist, since we would obtain a subspace of rank $n+1$ on
which $Q_{\Sigma_0}$ vanished. Hence each component of $\Sigma_0 \setminus
(\delta_1 \cup \cdots \cup \delta_n)$ must be planar and contain exactly one
component of $\d \Sigma_0$.
\end{proof}

\subsection{Transition maps and doubled Heegaard diagrams}

Suppose that $\cH=(\Sigma,\as,\bs,\ws,\zs)$ is a Heegaard diagram for
$(Y,\bL)$. Let $D_{\a}^{\zs}(\cH) = (D_{\a}^{\zs}(\Sigma),\as\cup \bar{\bs},
\Ds,\ws,\zs)$ denote the doubled diagram from
Section~\ref{sec:doublingdiagrams}, and let
$\Tw^-(\bar{\cH}) = (\Tw^-(\bar{\Sigma}), \bar{\bs}, \bar{\as}, \ws, \zs)$
be the negative twisted conjugate from
Section~\ref{sec:twistedconjugateddiagrams}. In this section, we describe
compact formulas for the transition maps between the link Floer complexes for
$\cH$, $D^{\zs}_{\a}(\cH)$, and $\Tw^-(\bar{\cH})$.

As a first observation, we note that
$(D_{\a}^{\zs}(\Sigma),\bs \cup \bar{\bs}, \Ds,\ws,\zs)$
is a multi-pointed diagram for an unlink in
$(\Sphere^1 \times \Sphere^2)^{\# g(\Sigma)}$, where each component has
exactly two basepoints. Hence there is a well-defined top-graded generator
\[
[\Theta_{\b \cup \bar{\b},\Delta}^+] \in \cHFL^-(D_{\a}^{\zs}(\cH),\bs\cup \bar{\bs}, \Ds, \ws, \zs, \frs_0),
\]
where $\frs_0$ is the torsion $\Spin^c$ structure on $(\Sphere^1 \times \Sphere^2)^{\# g(\Sigma)}$.

\begin{lem}\label{lem:changeofdiagramstodouble}
If $\cH$ is a Heegaard diagram and $D_{\a}^{\zs}(\cH)$ is its double, then
\[
\Phi_{\cH\to D_{\a}^{\zs}(\cH)}(-) \simeq
F_{\a\cup \bar{\b}, \b\cup \bar{\b}, \Delta} (F_{1}^{\bar{\b},\bar{\b}}(-) \otimes \Theta_{\b\cup \bar{\b}, \Delta}^+).
\]
\end{lem}

\begin{proof}
The key observation is that the map $F_1^{\bar{\b}, \bar{\b}}$ is equal to a
composition of 1-handle maps, while
\[
F_{\a\cup \bar{\b}, \b \cup \bar{\b}, \Delta}(- \otimes \Theta^+_{\b \cup \bar{\b}, \Delta})
\]
is chain homotopic to the 2-handle map for a collection of 2-handles that cancel
the 1-handles which were added by $F_1^{\bar{\b},\bar{\b}}$.
See \cite{ZemDualityMappingTori}*{Proposition~7.2} for a detailed proof of a closely related result.
\end{proof}

Next, we need a simple formula for the transition map from $\Tw^-(\bar{\cH})$
to $D_{\a}^{\zs}(\cH)$. A handle cancellation argument yields the following:

\begin{lem}\label{lem:changeofdiagramdoubletoconj}
There is a chain homotopy
\[
\Phi_{\Tw^-(\bar{\cH}) \to D_{\a}^{\zs}(\cH)}(-) \simeq
F_{\a\cup \bar{\b}, \a \cup \bar{\a}, \Delta}(F_1^{\a,\a}(-) \otimes
\Theta^+_{\a \cup \bar{\a}, \Delta}).
\]
\end{lem}

Finally,  we describe a formula for the transition map from
$D_{\a}^{\zs}(\cH)$ to $\cH$, which is essentially just the dual of
Lemma~\ref{lem:changeofdiagramstodouble}. Suppose $D_1,\dots, D_n$ are
compressing disks attached to the $\bar{\Sigma}$ portion of
$\Sigma \#_i \bar{\Sigma}$ that bound the curves in $\bar{\bs}$. If we surger
$\Sigma \#_i \bar{\Sigma}$ along the $\bar{\bs}$ curves using the compressing disks
$D_1, \dots, D_n$, we simply obtain the original Heegaard surface $\Sigma$ (up
to isotopy, relative to $L \cap \Sigma$). With this in mind, the handle
cancellation argument used to prove Lemma~\ref{lem:changeofdiagramstodouble}
implies the following:

\begin{lem}\label{lem:changeofdiagramsfromdouble}
There is a chain homotopy
\[
\Phi_{D_{\a}^{\zs}(\cH) \to \cH}(-) \simeq F_{3}^{\bar{\b}, \bar{\b}}
F_{\a \cup \bar{\b}, \Delta, \b \cup \bar{\b}}(- \otimes \Theta_{\Delta, \b \cup \bar{\b}}^+).
\]
\end{lem}

\subsection{Intertwining maps and connected sums}

Suppose that $(Y_j, \bL_j)$ for $j \in \{1, 2\}$ is a 3-manifolds with a
multi-based link, and  $\cH_j = (\Sigma_j, \as_j, \bs_j, \ws_j, \zs_j)$
is a Heegaard diagram for $(Y_j, \bL_j)$. Suppose also we have chosen a bijection $i$ from
$\zs_1$ to $\ws_2$. We can form the \emph{generalized connected sum} of
$(Y_1,\bL_1)$ and $(Y_2,\bL_2)$, for which we write
$(Y_1 \#_i Y_2,\bL_1 \#_i \bL_2)$, by deleting 3-balls centered at each point in $\zs_1$ and $\ws_2$,
and gluing the boundary components according to our chosen bijection
between $\zs_1$ and $\ws_2$. The link $\bL_1 \#_i \bL_2$ is decorated with the
basepoints $\ws_1$ and $\zs_2$.

We can construct a Heegaard surface $\Sigma_1\#_i \Sigma_2$ for $(Y_1\#_i Y_2,
\bL_1\#_i \bL_2)$ by adding a connected sum tube between $\Sigma_1$ and
$\Sigma_2$ near each basepoint in $\zs_1$ and the corresponding basepoint in
$\ws_2$. We define
\[
\cH_1 \#_i\cH_2 := (\Sigma_1 \#_i \Sigma_2, \as_1 \cup \as_2, \bs_1 \cup \bs_2, \ws_1, \zs_2)
\]
for the resulting diagram.

Adapting the construction from \cite{OSProperties}*{Section~6.2}, we can define an \emph{intertwining map}
\[
\cG \colon \cCFL^\infty(\cH_1, \frs_1) \otimes_{\cR^\infty} \cCFL^\infty(\cH_2,\frs_2) \to
\cCFL^\infty(\cH_1\#_i \cH_2, \frs_1\# \frs_2),
\]
via the formula
\begin{equation}
\cG(-,-) := F_{\a_1 \cup \a_2,\b_1 \cup \a_2, \b_1 \cup \b_2} (F_{1}^{\a_2, \a_2}(-) \otimes F_1^{\b_1,\b_1}(-)).
\end{equation}

We now show that the map $\cG$ is chain homotopic to a link cobordism map. We
define the decorated link cobordism $(W,\cF)$, as follows. Write
$\zs_1=\{z_{1,1},\dots, z_{1,n}\}$ and $\ws_2=\{w_{2,1},\dots, w_{2,n}\}$,
ordered such that $z_{1,i}$ and $w_{2,i}$ are paired. The 4-manifold $W$ is
obtained by attaching $n$ 1-handles, such that the $i^{\text{th}}$ 1-handle
has one foot at $z_{1,i}$, and the other foot at $w_{2,i}$. We construct a
surface $S$ inside the 1-handle cobordism by attaching a band inside each
1-handle. We construct a dividing set $\cA\subset S$ as follows. For each
pair $(z_{1,i}, w_{2,i})$ we add a dividing arc $a_i$ which has one end on
$Y_1$, and the other on $Y_2$, and travels through the 1-handle connecting
$z_{1,i}$ to $w_{2,i}$. One end of $a_i$ occurs immediately after $z_{1,i}$,
and the other end occurs immediately after $w_{2,i}$. The remaining arcs of
$\cA$ are of the form $I\times \{p\}$, for points $p\in L_1\cup L_2$. We
define $\cF:=(S,\cA)$; see Figure~\ref{fig::29}.

\begin{figure}[ht!]
	\centering
\begingroup%
  \makeatletter%
  \providecommand\color[2][]{%
    \errmessage{(Inkscape) Color is used for the text in Inkscape, but the package 'color.sty' is not loaded}%
    \renewcommand\color[2][]{}%
  }%
  \providecommand\transparent[1]{%
    \errmessage{(Inkscape) Transparency is used (non-zero) for the text in Inkscape, but the package 'transparent.sty' is not loaded}%
    \renewcommand\transparent[1]{}%
  }%
  \providecommand\rotatebox[2]{#2}%
  \newcommand*\fsize{\dimexpr\f@size pt\relax}%
  \newcommand*\lineheight[1]{\fontsize{\fsize}{#1\fsize}\selectfont}%
  \ifx\svgwidth\undefined%
    \setlength{\unitlength}{175.68904576bp}%
    \ifx\svgscale\undefined%
      \relax%
    \else%
      \setlength{\unitlength}{\unitlength * \real{\svgscale}}%
    \fi%
  \else%
    \setlength{\unitlength}{\svgwidth}%
  \fi%
  \global\let\svgwidth\undefined%
  \global\let\svgscale\undefined%
  \makeatother%
  \begin{picture}(1,0.58769986)%
    \lineheight{1}%
    \setlength\tabcolsep{0pt}%
    \put(0,0){\includegraphics[width=\unitlength,page=1]{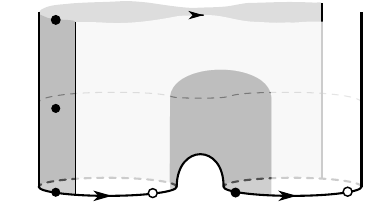}}%
    \put(0.13366504,0.01500126){\color[rgb]{0,0,0}\makebox(0,0)[t]{\lineheight{1.25}\smash{\begin{tabular}[t]{c}$w_1$\end{tabular}}}}%
    \put(0.41164574,0.00536948){\color[rgb]{0,0,0}\makebox(0,0)[t]{\lineheight{1.25}\smash{\begin{tabular}[t]{c}$z_1$\end{tabular}}}}%
    \put(0,0){\includegraphics[width=\unitlength,page=2]{fig29.pdf}}%
    \put(0.04475681,0.29386087){\color[rgb]{0,0,0}\makebox(0,0)[rt]{\lineheight{1.25}\smash{\begin{tabular}[t]{r}$G$\end{tabular}}}}%
    \put(0,0){\includegraphics[width=\unitlength,page=3]{fig29.pdf}}%
    \put(0.66297073,0.00536948){\color[rgb]{0,0,0}\makebox(0,0)[t]{\lineheight{1.25}\smash{\begin{tabular}[t]{c}$w_2$\end{tabular}}}}%
    \put(0.96848347,0.01413565){\color[rgb]{0,0,0}\makebox(0,0)[t]{\lineheight{1.25}\smash{\begin{tabular}[t]{c}$z_2$\end{tabular}}}}%
    \put(0,0){\includegraphics[width=\unitlength,page=4]{fig29.pdf}}%
  \end{picture}%
\endgroup%

	\caption{The decorated link cobordism used to define the map $G$,
    when $\bL_1$ and $\bL_2$ are both doubly-based knots. The orientation of $\bL_1$ and $\bL_2$ is shown.}
\label{fig::29}
\end{figure}

We define the cobordism map
\[
G:=F_{W,\cF}.
\]
In the case when $\bL_1$ and $\bL_2$ are doubly-based knots,
using the decomposition shown in Figure~\ref{fig::29},  we see
\begin{equation}
G \simeq S_{w_2,z_1}^- F_B^{\ws} F_1,
\label{eq:connectedsummap}
\end{equation}
where $F_1$ denotes the 1-handle map. More generally, if $\bL_1$ and $\bL_2$
have many basepoints, the cobordism map $G$ is a composition of $n$ terms
which each have the form shown in  equation~\eqref{eq:connectedsummap}.

Note that there is an asymmetry between $Y_1$ and $Y_2$ in the definition of
$G$. At each pair of basepoints we delete, we could instead do a type-$\zs$
band map, followed by the $T_{w_{2},z_1}^-$ quasi-stabilization map. The
corresponding decorated link cobordism is not diffeomorphic to $(W,\cF)$
(they can be distinguished by looking at the order in which the boundary
components appear on the subsurface $S_{\ws}$, with respect to the boundary
orientation). There is a similar asymmetry also in the definition of $\cG$
since the formula defining $\cG$  is not invariant under switching the roles
of $Y_1$ and $Y_2$.

\begin{prop}\label{prop:intertwiningmap=cobordismmap}
The intertwining map $\cG$ is chain homotopic to the link cobordism map $G$.
\end{prop}

\begin{proof}
The proof is similar to the proof of
\cite{ZemDualityMappingTori}*{Proposition~8.1}. The idea is that we
exhibit a chain homotopy inverse of $G$, which we denote $E$, and show that
\begin{equation}
E\circ \cG \simeq \id.\label{eq:EcG=id}
\end{equation}
For notational simplicity, we restrict to the case when $\bL_1$ and $\bL_2$
are both doubly-based knots. The proof we present extends to the more
general case by an elaboration of notation.

We define the map $E$ via the formula
\begin{equation}
E:=F_3 F_{B}^{\ws}S_{w_2,z_1}^+.\label{eq:formulaforG}
\end{equation}
We note $E$ is the cobordism map for the decorated link cobordism obtained by
turning around and reversing the orientation of the link cobordism used to
define $G$. The fact that $E$ and $G$ are chain homotopy inverses of each
other follows from \cite{ZemCFLTQFT}*{Proposition~5.1}.
Using equation~\eqref{eq:formulaforG}, we see equation~\eqref{eq:EcG=id} is equivalent to
\begin{equation}
F_3F_B^{\ws} S_{w_2,z_1}^+ F_{\a_1\cup \a_2,\b_1 \cup \a_2, \b_1 \cup \b_2}(F_{1}^{\a_2, \a_2} \otimes
F_1^{\b_1,\b_1}) \simeq \id.
\label{eq:intertwined=cobordism1}
\end{equation}

We pick two curves in the connected sum region of $\Sigma_1\# \Sigma_2$,
which we label as $\xi_s$ and $\zeta_0$.   The curve $\zeta_0$ bounds a small
disk containing the basepoints $w_2$ and $z_1$, while $\xi_s$ wraps all the
way around the connected sum tube; see Figure~\ref{fig::23}. We write
\[
\xi_s\cap \zeta_0=\{ \theta^{\ws}_{\xi_s,\zeta_0},\theta^{\zs}_{\xi_s,\zeta_0}\},
\]
where $\theta^{\ws}_{\xi_s,\zeta_0}$ is the top $\gr_{\ws}$-graded
intersection point and $\theta^{\zs}_{\xi_s,\zeta_0}$ denotes the top
$\gr_{\zs}$-graded intersection point.

\begin{figure}[ht!]
	\centering
	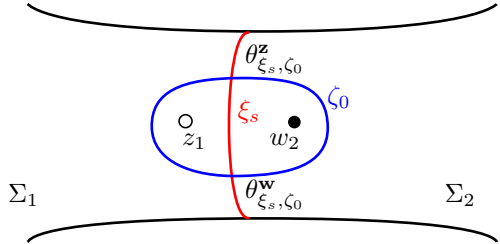
	\caption{The connected sum region of $\Sigma_1\# \Sigma_2$.}
\label{fig::23}
\end{figure}

The map $S_{w_2,z_1}^+$ appearing in equation~\eqref{eq:intertwined=cobordism1} is defined by the equation
\[
S_{w_2,z_1}^+(\xs) := \xs\times \theta^{\ws}_{\xi_s,\zeta_0},
\]
extended $\cR^\infty$-equivariantly. Similarly, there is a birth map
(corresponding to the cobordism map for a doubly-based unknot being born),
given by the formula
\[
\cB^+(\xs) := \xs \times \theta^+_{\zeta_0,\zeta_0},
\]
where $\theta^+_{\zeta_0,\zeta_0}$ denotes the top-graded intersection point
of~$\zeta_0$ and a small Hamiltonian translate of~$\zeta_0$. Note that, to
ease the notational burden, we will henceforth not distinguish between a
curve and its Hamiltonian translate (though, implicitly, when we are counting
holomorphic triangles, we must translate some of the curves using
Hamiltonians).

We introduce the following shorthand notation for sets of attaching curves on $\Sigma_1\# \Sigma_2$:
\[
L := \as_1 \cup \as_2, \qquad M := \bs_1 \cup \as_2, \qquad \text{and} \qquad R := \bs_1\cup \bs_2.
\]
We define
\[
L_0 := \as_1 \cup \{\zeta_0\}\cup \as_2 \qquad \text{and} \qquad L_s := \as_1\cup \{\xi_s\}\cup \as_2,
\]
and we define $M_0$, $M_s$, $R_0$, and $R_s$ similarly.

By Proposition~\ref{prop:holomorphictrianglesandquasistabs},
\begin{equation}
\begin{split}
&S_{w_2,z_1}^+F_{L, M, R} (F_{1}^{\a_2, \a_2}(-) \otimes F_1^{\b_1,\b_1}(-)) \simeq \\
&F_{L_s,M_0, R_0}(S_{w_2,z_1}^+F_{1}^{\a_2, \a_2}(-) \otimes \cB^+ F_1^{\b_1,\b_1}(-)).
\label{eq:intertwine=linkcob1}
\end{split}
\end{equation}
The band map $F_B^{\ws}$ is defined via the triangle count
\begin{equation}
F_{B}^{\ws}(-) := F_{L_s, R_0, R_s}(-\otimes (\Theta^+_{R,R} \times \theta^{\zs}_{\zeta_0,\xi_s})).
\label{eq:intertwine=linkcob2}
\end{equation}
We note that
\[
\Theta^+_{R,R} \times \theta^{\zs}_{\zeta_0,\xi_s} = T_{w_2,z_1}^+(\Theta^+_{R,R}),
\]
by definition, so equation~\eqref{eq:intertwine=linkcob2} reads
\begin{equation}
F_{B}^{\ws}(-) = F_{L_s, R_0, R_s}(-\otimes T_{w_2,z_1}^+(\Theta^+_{R,R})).
\label{eq:intertwine=linkcob3}
\end{equation}
Combining equations~\eqref{eq:intertwine=linkcob1} and~\eqref{eq:intertwine=linkcob3},
and using associativity, we see that
\begin{equation}
\begin{split}
&F_B^{\ws} S_{w_2,z_1}^+ \cG(-,-) \simeq \\
&F_{L_s, R_0, R_s}\left(F_{L_s,M_0, R_0}
\left(S_{w_2,z_1}^+F_{1}^{\a_2, \a_2}(-) \otimes \cB^+ F_1^{\b_1,\b_1}(-)\right)
\otimes T_{w_2,z_1}^+(\Theta^+_{R,R})\right) \simeq\\
&F_{L_s,M_0,R_s}\left(S_{w_2,z_1}^+ F_{1}^{\a_2, \a_2}(-) \otimes  F_{M_0,R_0,R_s}\left(\cB^+ F_1^{\b_1,\b_1}(-) \otimes T_{w_2,z_1}^+(\Theta^+_{R,R})\right)\right).
\label{eq:intertwine=linkcob4}
\end{split}
\end{equation}
Using Proposition~\ref{prop:holomorphictrianglesandquasistabs}, equation~\eqref{eq:intertwine=linkcob4} becomes
\begin{equation}
F_{L_s,M_0,R_s}\left(S_{w_2,z_1}^+ F_{1}^{\a_2, \a_2}(-) \otimes
T_{w_2,z_1}^+ F_{M,R,R}\left( F_1^{\b_1,\b_1}(-) \otimes \Theta^+_{R,R}\right)\right).
\label{eq:intertwine=linkcob5}
\end{equation}
The expression  $F_{M,R,R}(- \otimes \Theta^+_{R,R})$ is the
change of diagrams map for shifting the curves $R$ slightly, which we can
safely delete, since we are already precomposing with a change of diagrams
map on $\cCFL^\infty(Y_1,\bL_1,\frs_1)\otimes
\cCFL^\infty(Y_2,\bL_2,\frs_2)$.
Hence equation~\eqref{eq:intertwine=linkcob5} becomes
\begin{equation}
F_{L_s,M_0,R_s}(S_{w_2,z_1}^+ F_{1}^{\a_2, \a_2}(-)\otimes T_{w_2,z_1}^+ F_1^{\b_1,\b_1}(-)).
\label{eq:intertwine=linkcob6}
\end{equation}

Define the map
\[
\Top_{(\Sigma_1,\beta_1,\beta_1)}^+ \colon \cCFL^\infty(\Sigma_2,\as_2,\bs_2) \to
\cCFL^\infty(\Sigma_1,\bs_1,\bs_1) \otimes \cCFL^\infty(\Sigma_2,\as_2,\bs_2)
\]
via the formula
\[
\Top^+_{(\Sigma_1,\beta_1,\beta_1)}(\xs) := \Theta^+_{\b_1,\b_1} \otimes \xs,
\]
extended $\cR^\infty$-equivariantly. Next, we claim
\begin{equation}
T^+_{z_1,w_2}F_1^{\b_1,\b_1}(-) \simeq F_{B'}^{\ws} F_1^{\xi_s,\xi_s} \Top^+_{(\Sigma_1,\b_1,\b_1)}(-),
\label{eq:intertwine=linkcob7}
\end{equation}
where
\[
F_{B'}^{\ws}\colon \cCFL^\infty(\Sigma_1 \# \Sigma_2, M_s,R_s) \to \cCFL^\infty(\Sigma_1 \# \Sigma_2, M_0,R_s)
\]
is the band map
\begin{equation}
F_{B'}^{\ws}(-) := F_{M_0,M_s,R_s}(T_{z_1,w_2}^+(\Theta_{M,M}^+) \otimes -).
\label{eq:intertwine=linkcob11}
\end{equation}
Note that $F_{B'}^{\ws}$ is an $\alpha$-band map, because the handlebody
$U_{M_0}$ is playing the role of the $\alpha$-handlebody. Furthermore, the
quasi-stabilization map $T_{z_1,w_2}^+$ appears in
equation~\eqref{eq:intertwine=linkcob11} instead of $T_{w_2,z_1}^+$ because
$U_{M_0}$ is now playing the role of the $\alpha$-handlebody instead of the
$\beta$-handlebody, so the induced orientation of the strands it contains are
reversed, and hence, in this handlebody, $z_1$ now immediately follows $w_2$.

It is possible to establish equation~\eqref{eq:intertwine=linkcob7} via a
direct holomorphic triangle count. Indeed, by using
Proposition~\ref{prop:holomorphictrianglecounts}, one could delete the
portion added via the compound 1-handle map, and reduce the computation to a
holomorphic triangle count supported in a disk, involving three isotopic
attaching curves. A holomorphic triangle count could then be performed by
using a neck-stretching argument, as in \cite{ZemCFLTQFT}*{Lemma~8.6}.

A more conceptually enlightening approach for proving
equation~\eqref{eq:intertwine=linkcob7}, and the approach we take, is to
interpret the maps appearing as cobordism maps and use properties of the link Floer TQFT. We note that the map
$\Top_{(\Sigma_1,\b_1,\b_1)}^+$ can be written as the composition of a single 0-handle map, followed by $|\bs_1|$ 1-handle maps.
After rearranging handles and canceling the 0-handle added by $\Top^+_{(\Sigma_1,\b_1,\b_1)}$
with the 1-handle added by $F_1^{\xi_s,\xi_s}$,
the composition $F_1^{\xi_s,\xi_s}\Top^+_{(\Sigma_1,\b_1,\b_1)}$
can be rewritten as the composition of the cobordism map $F_1$ induced by attaching
$|\bs_1|$ 1-handles to $Y_2$, followed by a birth cobordism map, which adds the
doubly-based knot $\bU=(U,w_1,z_1)$. Hence, we can write
\begin{equation}
F_1^{\xi_s,\xi_s} \Top^+_{(\Sigma_1,\b_1,\b_1)} \simeq \cB_{\bU}^+ F_1.
\label{eq:intertwine=linkcob16}
\end{equation}

Similarly,
\begin{equation}
F_{1}^{\b_1,\b_1} \simeq \phi_* F_1,
\label{eqn:intertwine=linkcob17}
\end{equation}
where $\phi$ is an isotopy of $Y_2\# (\Sphere^1\times \Sphere^2)^{\#
|\bs_1|}$ that moves the knot in $Y_2$ into the 1-handle region formed when
we attach $(\Sigma_1,\bs_1,\bs_1)$.

We can decompose the isotopy $\phi$ as the composition of an isotopy
$\phi_0$, which fixes $w_2$ and $z_2$ but moves the link, followed by an
isotopy $\tau^{ w_1\from w_2}$ that fixes the link setwise, is supported in
a neighborhood of the link, fixes $z_2$, but moves $w_2$ to $w_1$. Using
equations~\eqref{eq:intertwine=linkcob16} and \eqref{eqn:intertwine=linkcob17},
we see that equation~\eqref{eq:intertwine=linkcob7} is equivalent to
\begin{equation}
T_{z_1,w_2}^+ \tau^{w_1\from w_2}_*(\phi_0)_* F_1 \simeq F_{B'}^{\ws} \cB^+_{\bU} F_1.
\label{eqn:intertwine=linkcob12}
\end{equation}
We now simply note that equation~\eqref{eq:movebasepointandquasistabilization} implies
$T_{z_1,w_2}^+ \tau^{w_1 \from w_2}_* \simeq T_{w_1,z_1}^+$, while
equation~\eqref{eq:band-birth=quasi} implies $F_{B'}^{\ws} \cB_{\bU}^+ \simeq
T_{w_1,z_1}^+(\phi_0)_*$. Together, these establish
equation~\eqref{eqn:intertwine=linkcob12}, and hence also equation~\eqref{eq:intertwine=linkcob7}.

If we substitute the formula~\eqref{eq:intertwine=linkcob11} for $F_{B'}^{\ws}$
into equation~\eqref{eq:intertwine=linkcob7}, equation~\eqref{eq:intertwine=linkcob6} becomes
\begin{equation}
F_{L_s,M_0,R_s}\left(S_{w_2,z_1}^+F_{1}^{\a_2, \a_2}(-) \otimes  F_{M_0,M_s,R_s} \left(T_{z_1,w_2}^+
(\Theta_{M,M}^+) \otimes F_1^{\xi_s,\xi_s} \Top^+_{(\Sigma_1,\b_1,\b_1)}(-)\right)\right).
\label{eq:intertwine=linkcob9}
\end{equation}
By associativity, we see that equation~\eqref{eq:intertwine=linkcob9} is chain homotopic to
\begin{equation}
F_{L_s,M_s,R_s}\left(F_{L_s,M_0,M_s} \left(S_{w_2,z_1}^+F_{1}^{\a_2, \a_2}(-) \otimes T_{z_1,w_2}^+(\Theta_{M,M}^+)\right)
\otimes F_1^{\xi_s,\xi_s}\Top^+_{(\Sigma_1,\b_1,\b_1)}(-) \right)
\label{eq:intertwine=linkcob10}
\end{equation}
After post-composing equation~\eqref{eq:intertwine=linkcob10} with the
3-handle map $F_3^{\xi_s,\xi_s}$, and pulling the 3-handle map inside
the outer triangle map using Proposition~\ref{prop:holomorphictrianglecounts}, we obtain that the
composition $E\circ \cG$ is chain homotopic to
\begin{equation}
F_{\cT_1 \sqcup \cT_2}\left(F_3^{\xi_s,\xi_s} F_{L_s,M_0,M_s} \left(S_{w_2,z_1}^+ F_{1}^{\a_2, \a_2}(-) \otimes
T_{z_1,w_2}^+(\Theta_{M,M}^+)\right) \otimes \Top^+_{(\Sigma_1,\b_1,\b_1)}(-)\right).
\label{eq:intertwine=linkcob14}
\end{equation}
where $\cT_1\sqcup \cT_2$ denotes the disjoint union of the Heegaard
triples $(\Sigma_1,\as_1,\bs_1,\bs_1)$ and $(\Sigma_2,\as_2,\as_2,\bs_2)$.
Since the outer triangle map is on the disjoint union of $\Sigma_1$ and
$\Sigma_2$, we direct our attention to the inner triangle map. We claim that
\begin{equation}
F_3^{\xi_s,\xi_s} F_{L_s,M_0,M_s} \left(S_{w_2,z_1}^+ F_{1}^{\a_2, \a_2}(-) \otimes T_{z_1,w_2}^+(\Theta_{M,M}^+)\right) \simeq \Top_{(\Sigma_2,\a_2,\a_2)}^+.
\label{eq:intertwine=linkcob12}
\end{equation}
Note that, by the definition of $F_B^{\ws}$, the left-hand side of equation~\eqref{eq:intertwine=linkcob12} is
\begin{equation}
F_3^{\xi_s,\xi_s} F_B^{\ws} S_{w_2,z_1}^+ F_1^{\a_2,\a_2}.
\label{eq:intertwine=linkcob13}
\end{equation}
Our strategy for proving equation~\eqref{eq:intertwine=linkcob12} will be to
manipulate equation~\eqref{eq:intertwine=linkcob13} using algebraic
properties of the TQFT until it becomes the cobordism map for the disjoint
union of the identity cobordism $I\times Y_1$, and a 4-dimensional handlebody
bounding $Y_{\a_2,\a_2}$.

We can write $F_1^{\a_2,\a_2}$ as $(\phi_*)F_1$, where $F_1$ is a 1-handle
cobordism, and $\phi$ is a diffeomorphism that moves a small portion of the
link near $z_1$ into $Y_{\a_2,\a_2}$, and sends $z_1$ to $z_2$. Note that we
can write $\phi$ as a composition $\rho^{z_1\to z_2}\circ \phi_0$, where
$\phi_0$ moves a small portion of $K_1$ near $z_1$, but fixes $z_1$, and
$\rho^{z_1\to z_2}$ is a diffeomorphism that is fixed outside a neighborhood
of the subarc of $\phi_0(K_1)$ containing $z_1$ and $z_2$, but sends $z_1$ to
$z_2$. We note that the map $\rho^{z_1\to z_2}_*$ satisfies the relation
\begin{equation}
\rho^{z_1\to z_2}_* \simeq S_{w_2,z_1}^- T_{z_2,w_2}^+,\label{eq:basepointmovingjustz}
\end{equation}
by \cite{ZemCFLTQFT}*{Lemma~4.25}; cf.~equation~\eqref{eq:R5} and Figure~\ref{fig::17}.

We perform the following manipulation:
\begin{equation}
\begin{alignedat}{3}
\phi_*&\simeq \rho^{z_1\to z_2}_*(\phi_0)_*&&\\
&\simeq S_{w_2,z_1}^-T_{z_2,w_2}^+(\phi_0)_*&&\qquad (\text{equation~\eqref{eq:basepointmovingjustz}})\\
&\simeq  S_{w_2,z_1}^- F_{B}^{\ws} \cB^+_{\bU}&&\qquad (\text{equation~\eqref{eq:band-birth=quasi2}}),
\end{alignedat}
\end{equation}
where $\bU$ denotes $(U,w_2,z_2)$. Hence equation~\eqref{eq:intertwine=linkcob13}  becomes
\begin{equation}
F_3^{\xi_s,\xi_s} F_B^{\ws} S_{w_2,z_1}^+ S_{w_2,z_1}^-F_{B}^{\ws} \cB^+_{\bU} F_1.\label{eq:intertwine=linkcob15}
\end{equation}

We note that the 1-handles of $F_1$ can be moved to the left of all the other
maps. After moving $F_1$ to the left, the birth cobordism map $\cB^+_{\bU}$
becomes the composition of a 0-handle map $F_0$ (which adds a 4-ball
containing $\bU$), and the 1-handle map $F_1^{\xi_s,\xi_s}$. We also note
that $S_{w_2,z_1}^+S_{w_2,z_1}^-\simeq \Phi_{w_2}$ by equation~\eqref{eq:R3}.
Hence equation~\eqref{eq:intertwine=linkcob15} becomes
\begin{equation}
F_1' F_3^{\xi_s,\xi_s}F_B^{\ws} \Phi_{w_2} F_{B}^{\ws} F_1^{\xi_{s},\xi_s} F_0,
\label{eq:intertwine=linkcob17}
\end{equation}
where $F_1'$ denotes the cobordism map for attaching $|\as_2|$ 1-handles to the 3-sphere added by $F_0$. Using equation~\eqref{eq:connectedsum=tensorprod2}, we can reduce equation~\eqref{eq:intertwine=linkcob17} to the expression
\[
F_1' F_0,
\]
which is clearly just $\Top^+_{(\Sigma_2,\a_2,\a_2)}$, establishing equation~\eqref{eq:intertwine=linkcob12}.

Applying the relation from equation~\eqref{eq:intertwine=linkcob12} to equation~\eqref{eq:intertwine=linkcob14},
it follows that $E \circ \cG$ is chain homotopic to
\[
F_{\cT_1\sqcup \cT_2} \left(\Top_{(\Sigma_2,\a_2,\a_2)}^+(-) \otimes \Top_{(\Sigma_1,\b_1,\b_1)}^+(-)\right).
\]
This holomorphic triangle count appears on the disjoint union of $\Sigma_1$
and $\Sigma_2$, and is clearly just the tensor product $\Phi_{\b_1\to
\b_1}^{\a_1}\otimes \Phi_{\b_2}^{\a_2\to \a_2}$, completing the proof.
\end{proof}

\begin{rem}
There is another chain homotopy equivalence $E'$ from the connected sum to
the disjoint union, defined via the formula $E' := F_3F_B^{\zs} T_{w_2,z_1}^+$.
The map $E'$ corresponds to a pair-of-pants link cobordism where the
type-$\ws$ and type-$\zs$ regions have been switched from the cobordism
corresponding to $E$. One might expect the above argument to also go through
using $E'$ to try to cancel $\cG$, by just replacing each
type-$T$ quasi-stabilization map with a type-$S$ quasi-stabilization map, and
replacing each type-$\ws$ band map with a type-$\zs$ band map. However, the
careful reader will discover that such a strategy fails at equation~\eqref{eq:intertwine=linkcob7}.
\end{rem}

\subsection{Proof of the triangle cobordism formula}

We now prove that the cobordism map induced by
$(X_{\a,\b,\g},\cF_{\a,\b,\g})$ is chain homotopic to the holomorphic
triangle map:

\begin{proof}[Proof of Theorem~\ref{thm:heegaardtriplemap}]
A handlebody description of the 4-manifold $X_{\a,\b,\g}$ is given in the
proof of \cite{ZemDualityMappingTori}*{Theorem~9.1}. Let $f_{\b}$ denote a
Morse function on the handlebody $U_\b$ compatible with the attaching curves
$\bs\subset \Sigma$, which has $\Sigma$ as a maximal level set.
The 4-manifold $X_{\a,\b,\g}$ has the following handlebody description:
\begin{itemize}
\item A 1-handle for each index 0 critical point of $f_{\b}$, with one foot at the critical point in
    $\bar{U}_{\b} \subset Y_{\a,\b}$, and the other foot at its image in $U_{\b}\subset Y_{\b,\g}$.
\item A 2-handle for each index 1 critical point of $f_{\b}$. The attaching
    sphere is equal to the union of the corresponding descending manifold in
    $U_{\b}$, concatenated across the connected sum tubes with its  mirror image
    in $\bar{U}_{\b}$. There is a canonical framing specified by taking an
    arbitrary framing in the portion in $U_{\b}$, and mirroring it in the portion
    in $\bar{U}_{\b}$.
\end{itemize}

Let $W_1$ denote the 1-handle cobordism, and let $W_2$ denote the 2-handle
cobordism. Let $\cF_1$ and $\cF_2$ denote the intersection of the decorated surface
$\cF_{\a,\b,\g}$ with $W_1$ and $W_2$, respectively. Note that $\cF_1$ is obtained by
attaching a collection of bands, one for each 1-handle, each containing a
single dividing arc that meets both $Y_{\a,\b}$ and $Y_{\b,\g}$. Hence, we
can write
\begin{equation}
F_{X_{\a,\b,\g},\cF_{\a,\b,\g}} \simeq F_{W_2,\cF_2} \circ F_{W_1,\cF_1}.
\end{equation}

Let $\cH_{\a,\b} = (\Sigma,\as,\bs,\ws,\zs)$,
and let  $\Tw^-(\bar{\cH}_{\b,\g}) = (\Tw^-(\bar{\Sigma}), \bar{\gs}, \bar{\bs}, \ws, \zs)$
denote the negative twisted conjugate of the diagram
$\cH_{\b, \g} = (\Sigma,\bs,\gs,\ws,\zs)$ described in Section~\ref{sec:twistedconjugateddiagrams}.
To show the main claim, we need to compute the cobordism map starting at the
diagram $\cH_{\a,\b} \sqcup \cH_{\b,\g}$. However, for the computation, it is
more convenient to start at the diagram $\cH_{\a,\b} \sqcup
\Tw^-(\bar{\cH}_{\b,\g})$. Hence, we precompose with the change of diagrams
map $\id \otimes \Phi_{\cH_{\b,\g} \to \Tw^-(\bar{\cH}_{\b,\g})}$. To simplify
notation, we will omit writing this change of diagrams map for most of the
proof, though it will reappear at the end.

We note that, by Proposition~\ref{prop:intertwiningmap=cobordismmap},
\begin{equation}
F_{W_1,\cF_1}\simeq \cG,
\end{equation} where
\[
\cG \colon \cCFL^\infty(\cH_{\a,\b}) \otimes \cCFL^\infty(\Tw^-(\bar{\cH}_{\b,\g})) \to
\cCFL^\infty(\Sigma \#_i \Tw^-(\bar{\Sigma}), \as \cup \bar{\gs}, \bs \cup \bar{\bs}, \ws, \zs)
\]
is the intertwining map defined by the equation
\[
\cG(-,-) := F_{\a \cup \bar{\g}, \b \cup \bar{\g}, \b \cup \bar{\b}} (F_1^{\bar{\g},\bar{\g}}(-) \otimes F_1^{\b,\b}(-)).
\]
Hence
\begin{equation}
F_{X_{\a,\b,\g}, \cF_{\a,\b,\g}} \simeq F_{W_2, \cF_2} \circ \cG.
\end{equation}

Next, it is not hard to see that the Heegaard triple $(\Sigma\#_i
\Tw^-(\bar{\Sigma}), \as\cup \bar{\gs}, \bs\cup \bar{\bs}, \Ds,\ws,\zs)$ can
be used to compute the 2-handle cobordism map $F_{W_2,\cF_2}$,
where $\Ds$ was defined in Section~\ref{sec:doublingdiagrams}; see
\cite{ZemDualityMappingTori}*{Lemma~7.7} for a detailed argument in a very
closely related context. It follows that
\[
F_{W_2,\cF_2} \simeq F_{\a \cup \bar{\g}, \b \cup \bar{\b}, \Delta}(- \otimes \Theta_{\b \cup \bar{\b}, \Delta}^+).
\]
Hence, omitting the initial factor of $\id \otimes \Phi_{\cH_{\b,\g} \to \Tw^-(\bar{\cH}_{\b,\g})}$,
we have, by associativity,
\begin{equation}
\begin{split}
F_{X_{\a,\b,\g}, \cF_{\a,\b,\g}}(-,-)& \simeq (F_{W_2,\cF_2}\circ  \cG)(-,-)\\
&\simeq F_{\a \cup \bar{\g}, \b \cup \bar{\b}, \Delta} \left(F_{\a \cup \bar{\g},\b \cup \bar{\g}, \b \cup \bar{\b}}
\left(F_1^{\bar{\g},\bar{\g}}(-) \otimes F_1^{\b,\b}(-)\right) \otimes \Theta_{\b \cup \bar{\b}, \Delta}^+ \right)\\
&\simeq  F_{\a \cup \bar{\g}, \b \cup \bar{\g},\Delta} \left(F_1^{\bar{\g},\bar{\g}}(-) \otimes
F_{\b \cup \bar{\g}, \b\cup \bar{\b}, \Delta} \left(F_1^{\b,\b}(-) \otimes \Theta_{\b \cup \bar{\b}, \Delta}^+ \right)\right).
\end{split}
\label{eq:triangle=cobordism1}
\end{equation}

It is not hard to see that the Heegaard diagram corresponding to the codomain
of the map in equation~\eqref{eq:triangle=cobordism1} is the double
$D_{\a}^{\zs}(\cH_{\a,\g})$ of the diagram $\cH_{\a,\g} = (\S, \as, \gs, \ws, \zs)$,
so we must postcompose with the transition map
$\Phi_{D_{\a}^{\zs}(\cH_{\a,\g})\to \cH_{\a,\g}}$, which we computed in
Lemma~\ref{lem:changeofdiagramsfromdouble}. Accordingly, our expression from
equation~\eqref{eq:triangle=cobordism1} for $F_{X_{\a,\b,\g},\cF_{\a,\b,\g}}$ becomes
\begin{equation}
F_3^{\bar{\g},\bar{\g}} F_{\a\cup \bar{\g}, \Delta, \g \cup \bar{\g}}
\left(F_{\a \cup \bar{\g}, \b \cup \bar{\g}, \Delta} \left(F_1^{\bar{\g},\bar{\g}}(-) \otimes
F_{\b \cup \bar{\g}, \b \cup \bar{\b}, \Delta} \left(F_1^{\b,\b}(-) \otimes \Theta_{\b \cup \bar{\b}, \Delta}^+\right)\right)
\otimes \Theta^+_{\Delta,\g\cup \bar{\g}}\right).
\label{eq:triangle=cobordism2}
\end{equation}
Associativity implies that equation~\eqref{eq:triangle=cobordism2} is chain homotopic to
\begin{equation}
F_3^{\bar{\g},\bar{\g}} F_{\a\cup \bar{\g}, \b\cup \bar{\g}, \g\cup \bar{\g}}
\left(F_1^{\bar{\g},\bar{\g}}(-) \otimes F_{\b\cup \bar{\g}, \Delta, \g\cup \bar{\g}}
\left(F_{\b\cup \bar{\g}, \b\cup \bar{\b}, \Delta} \left(F_1^{\b,\b}(-)\otimes \Theta_{\b\cup \bar{\b},\Delta}^+\right) \otimes
\Theta^+_{\Delta, \g\cup \bar{\g}}\right)\right).
\label{eq:triangle=cobordism3}
\end{equation}
Using Proposition~\ref{prop:holomorphictrianglecounts},
we see equation~\eqref{eq:triangle=cobordism3} is chain homotopic to
\begin{equation}
F_{\a,\b,\g}\left(-\otimes F_3^{\bar{\g},\bar{\g}} \left(F_{\b\cup \bar{\g}, \Delta, \g\cup \bar{\g}}
\left(F_{\b \cup \bar{\g}, \b \cup \bar{\b}, \Delta} \left(F_1^{\b,\b}(-) \otimes \Theta_{\b \cup \bar{\b},\Delta}^+\right)
\otimes \Theta_{\Delta, \g \cup \bar{\g}}^+\right)\right)\right).
\label{eq:triangle=cobordism4}
\end{equation}
Lemmas~\ref{lem:changeofdiagramdoubletoconj} and~\ref{lem:changeofdiagramsfromdouble}
imply that equation~\eqref{eq:triangle=cobordism4} is chain homotopic to
\[
F_{\a,\b,\g}(- \otimes \Phi_{\Tw^-(\bar{\cH}_{\b,\g}) \to \cH_{\b,\g}}).
\]
The transition map inside the triangle map cancels the initial factor of
$\id\otimes \Phi_{\cH_{\b,\g}\to \Tw^-(\cH_{\b,\g})}$, which we have been
omitting writing up until now. Hence $F_{X_{\a,\b,\g},\cF_{\a,\b,\g}}\simeq
F_{\a,\b,\g}$, concluding the proof.
\end{proof}

\section{Examples}\label{sec:examples}

In this section, we perform some model computations to illustrate our
invariants defined in Section~\ref{sec:invariants} for pairs of slice disks.
Our two main examples will be slice disks constructed by roll-spinning,
and deform-spinning using the rigid motion deformation from Section~\ref{sec:rigidmotiondeformation}.
See Section~\ref{sec:spinning} for the definitions of roll-spinning and deform-spinning.

\subsection{Invariants of deform-spun slice disks}
As stated in the introduction, our main computational
results rely on a formula for the fundamental principal invariants of
deform-spun slice disks, generalizing \cite[Theorem~5.1]{SliceDisks}
from the hat to the full infinity version of knot Floer homology:

\begin{customthm}{\ref{thm:spinning}}
  Let $D_{K, \varphi}$ be a slice disk of the knot $-K \# K$, obtained by deform-spinning
  a knot $K$ in $\Sphere^3$ using an automorphism $\varphi$ of $(\Sphere^3, K)$.
  Let $w$ and $z$ be basepoints on $K$, and write $C := \cCFL^\infty(K, w, z)$. Then
  \[
  E \circ \ve{t}_{D_{K, \varphi}}^\infty \simeq (\id \otimes \varphi_*) \circ \cotr \in
  \Hom_{\cR^\infty}(\cR^\infty, C^\vee \otimes C),
  \]
  where the chain homotopy equivalence $E \colon \cCFL^\infty(-K \# K, w, z) \to C^\vee \otimes C$ is described
  in \cite{ZemConnectedSums}*{Section~5}.
\end{customthm}

\begin{proof}
  This follows from the trace formula in Theorem~\ref{thm:traceforuma}, using the same argument as the proof of
\cite[Theorem~5.1]{SliceDisks}. 
\end{proof}

We now prove the following:

\begin{prop}\label{prop:roll}
Let $D_{K,\id}$ and $D_{K,r}$ be the canonical and the 1-roll-spun slice disks of $-K \# K$, respectively.
Then
\[
\tau(D_{K,\id}, D_{K,r}) \le 1.
\]
\end{prop}

\begin{proof}
Let $w$ and $z$ be basepoints on $K$, and write $\bK = (K, w, z)$ and $-\bK \# \bK = (-K \# K, w, z)$.
By Lemma~\ref{lem:slightreformulationoftau}, we can calculate
$\tau(D_{K,\id}, D_{K,r})$ using $\HFK^-_{V=0}(-\bK \# \bK)$. By Theorem~\ref{thm:spinning},
\begin{equation}\label{eqn:cotr0}
E \circ t_{D_{K,\id}}^- \simeq  \cotr \in \Hom \left(\bF_2[U], \CFK^-_{V=0}(\bK)^\vee \otimes_{\bF_2[U]} \CFK^-_{V=0}(\bK)\right).
\end{equation}
Furthermore, $r_* \simeq \id + \Phi_w \circ \Psi_z$ by \cite{ZemQuasi}*{Theorem~B},  so
\begin{equation}\label{eqn:cotr1}
E \circ t_{D_{K,r}}^- \simeq (\id \otimes (\id + \Phi_w \circ \Psi_z)) \circ \cotr.
\end{equation}
Hence, if we can show that $U \cdot \Phi_w \circ \Psi_z$ is $U$-equivariantly chain homotopic
to zero, then
\[
U \cdot t_{D_{K,\id}}^- \simeq U \cdot t_{D_{K,r}}^-,
\]
so $\tau(D_{K,\id}, D_{K,r}) \le 1$. We note that $\Phi_w$ has a simple algebraic
interpretation on $\CFK^-_{V=0}(\bK)$. It is the map obtained by writing the
differential as a matrix with entries in $\bF_2[U]$, and then differentiating
each entry; cf.~equation~\eqref{eq:defPhi}.
According to \cite{HMZConnectedSum}*{Proposition~6.3}, since
$\CFK^-_{V=0}(\bK)$ is a finitely generated, free, $\Z$-graded chain complex over $\bF_2[U]$,
the map $U \cdot \Phi_w$ is $U$-equivariantly chain homotopic to
zero, and hence so is $U \cdot \Phi_w \circ \Psi_z$. The claim follows.
\end{proof}

\begin{question}
  In light of Proposition~\ref{prop:roll}, it is natural to ask whether
  $\mu_{\st}(D_{K,\id}, D_{K,r}) \le 1$ for any roll-spun slice disks $D_{K,\id}$ and $D_{K,r}$;
  cf.~Conjecture~\ref{conj:rt}.
  This would give a topological proof of Proposition~\ref{prop:roll} by Theorem~\ref{thm:tauD1D2bound}.
\end{question}

\subsection{Computational examples}

In this section, we compute the invariants $\tau$, $V_k$,
and $\Upsilon$ for several pairs of deform-spun slice disks. We begin by considering the
complex $\cCFL^\infty(4_1)$ for the figure-eight knot $4_1$, which
is shown in Figure~\ref{fig::4_1}.

\begin{figure}[ht!]
\centering
\[
\begin{tikzcd}
\cCFL^\infty(4_1)&&\xs_0\arrow[bend left =20]{r}{V}& \xs_1& \xs_2& \xs_3\arrow[bend left = 20]{r}{V}
\arrow[bend left=20]{lll}{U}& \xs_4 \arrow[bend left=20]{lll}{U}\\
(\gr_{\ws}, A)&= & (1,1)&(0,0)&(0,0)&(0,0)&(-1,-1)
\end{tikzcd}
\]
\caption{The complex $\cCFL^\infty(\Sphere^3,4_1)$.}\label{fig::4_1}
\end{figure}

\begin{lem}\label{lem:computationtauandVbyhand41}
Let $K$ be the figure-eight knot,
and let $D_{K,\id}$ and $D_{K,r}$ denote the canonical and the 1-roll-spun slice disk
slice disks of $-4_1 \# 4_1$. Then
\[
\tau(D_{K,\id}, D_{K,r}) = 1, \text{ } V_0(D_{K,\id},D_{K,r}) = 1 \text{, and }
V_1(D_{K,\id}, D_{K,r}) = 0.
\]
\end{lem}

\begin{proof}
The Alexander filtered chain complex $\hCFKf(4_1)$, obtained by setting $U = 0$ and $V = 1$,
has the form
\[
\hCFKf(4_1) = \big( (\xs_0)_1 \longrightarrow (\xs_1)_0 \qquad (\xs_2)_0 \qquad
(\xs_3)_0 \longrightarrow (\xs_4)_{-1}\big).
\]
The notation $(\xs_i)_j$ means the intersection points $\xs_i$,
which has Alexander grading $j$. Using equations~\eqref{eq:defPhi} and~\eqref{eq:defPsi},
\[
(\Phi_w \circ \Psi_z)(\xs_3) = \xs_1,
\]
and $\Phi_w \circ \Psi_z$ vanishes on all other generators. The complex
$\hCFKf(-4_1)$ is obtained by dualizing $\hCFKf(4_1)$ (note that, although $4_1$ is
amphichiral, to compute the trace formula, it is better to ignore this fact). Hence
\[
\hCFKf(-4_1) = \big( (\xs_0^\vee)_{-1} \longleftarrow (\xs_1^\vee)_0 \qquad
(\xs_2^\vee)_0 \qquad (\xs_3^\vee)_0 \longleftarrow (\xs_4^\vee)_1 \big).
\]
Using equations~\eqref{eqn:cotr0} and~\eqref{eqn:cotr1},
\[
E(\hat{t}_{D_{K,r}}(1) - \hat{t}_{D_{K,\id}}(1)) = (\id \otimes (\Phi_w \circ \Psi_z)) \circ \cotr(1) =
\xs_3^\vee \otimes (\Phi_w \circ \Psi_z)(\xs_3) = \xs_3^\vee \otimes \xs_1.
\]

We now observe that $\xs_3^\vee \otimes \xs_1$ is nonzero in the homology
of the 0-filtration level
\[
\cG_0 \left(\hCFKf(-4_1) \otimes_{\bF_2} \hCFKf(4_1)\right),
\]
where $\cG_i$ denotes Alexander filtration level $i$.
Indeed, the only elements mapped to $\xs_3^\vee \otimes \xs_1$
by the differential are $\xs_4^\vee \otimes \xs_1$
and $\xs_3^\vee \otimes \xs_0$. However, neither of these are in $\cG_0$;
instead, they are in $\cG_1$. Hence $\tau(D_{K,\id}, D_{K,r}) = 1$.

We now consider the invariants $V_0$ and $V_1$. The complex $A_0^-(-4_1 \# 4_1)$
has 25 generators over $\bF_2[\hat{U}]$. The generators are the
monomials $U^n V^m \cdot \xs_i^\vee \otimes \xs_j$, where $n$, $m \ge 0$, and
\begin{equation}
A(\xs_i^\vee) + A(\xs_j) + m - n = 0.
\label{eq:41Alexandergradingiszero}
\end{equation}
Similarly, $A_{1}^-(-4_1 \# 4_1)$ is generated by the monomials
$U^n V^m \cdot \xs_i^\vee \otimes \xs_j$, where $n \ge 0$, $m \ge -1$,
and satisfy equation~\eqref{eq:41Alexandergradingiszero}.

As above, we can identify $t^-_{D_{K,r}}(1) - t^-_{D_{K,\id}}(1)$
with $U^0 V^0 \cdot \xs_3^\vee \otimes \xs_1$. It is straightforward to see that
$U^0 V^0 \cdot \xs_3^\vee \otimes \xs_1$ is not a boundary in $A_0^-(-4_1 \# 4_1)$,
so $V_0 \ge 1$. However,
\[
\begin{split}
\d(U^1 V^0 \cdot \xs_4^\vee \otimes \xs_1) &= U^1 V^1 \cdot \xs_3^\vee \otimes \xs_1 :=
\hat{U}  \cdot \xs_3^\vee \otimes \xs_1 \text{ and } \\
\d(U^0 V^{-1} \cdot \xs_4^\vee \otimes \xs_1) &=  \xs_3^\vee \otimes \xs_1,
\end{split}
\]
implying that $V_0 \le 1$ and $V_1 = 0$.
\end{proof}

\begin{lem}
Let $K$ denote the figure-eight knot, and let $D_{K,\id}$ and $D_{K,r}$ be
as in Lemma~\ref{lem:computationtauandVbyhand41}. Then
$\Upsilon_{(D_{K,\id}, D_{K,r})}(t)$ takes the form shown in
Figure~\ref{fig::4_1upsilon}.
\end{lem}
\begin{proof} The proof is similar to Lemma~\ref{lem:computationtauandVbyhand41}. Therein, we computed $t_{D_{K,r}}^-(1)-t_{D_K,\id}^-(1)$ to be $\xs_3^\vee\otimes \xs_1\in \CFK^\infty(4_1)^\vee\otimes \CFK^\infty(4_1)$. The two elements $y_1=V^{-1}\xs_4^\vee\otimes \xs_1$ and $y_2=U^{-1} \xs_3^\vee \otimes \xs_0$ lie in homogeneous $(\gr_{\ws},A)$-grading $(-1,0)$. It is straightforward to check that for $t\in [0,1]$, $y_1$ satisfies $\d y_1=\xs_3^\vee\otimes \xs_1$ and $y_1\in \cG_{t}^t(\bar 4_1\# 4_1)$. Furthermore, if $s<t$, then there are no elements $z\in \cG_s^t(\bar 4_1\# 4_1)$ such that $\d z=\xs_3^\vee\otimes \xs_1$. Similarly, if $t\in [1,2]$, then $\d y_2=\xs_3^\vee\otimes \xs_1$ and $y_2\in \cG_{2-t}^t(\bar 4_1\# 4_1)$, and there are no elements $z\in \cG_{s}^t(\bar 4_1\# 4_1)$ such that $\d z=\xs_3^\vee \otimes \xs_1$ for $s<2-t$. 
\end{proof}

\begin{figure}[ht!]
	\centering
	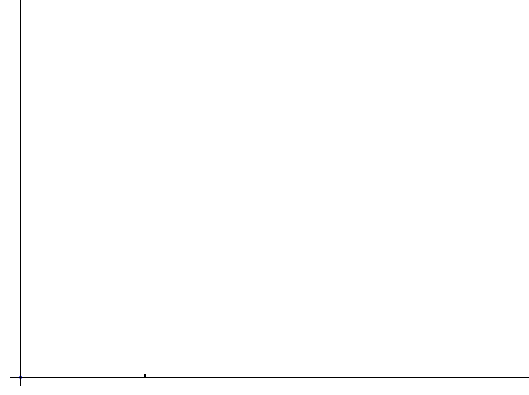
	\caption{$\Upsilon_{D_{4_1,\id}, D_{4_1,r}}(t)$.}
\label{fig::4_1upsilon}
\end{figure}

Our remaining examples were computed with the help of SageMath~\cite{sagemath}.
The program that computed these invariants can be found at \cite{GitHubCode}.

The next examples we consider are built from the knots $T_{3,4}$ and $T_{4,5}$.
Their associated full infinity complexes are shown in Figure~\ref{fig::T34-T45}.
If $K$ is a knot, we will write $D_{K \# K, \ve{R}^{\pi}}$ for the deform-spun slice disk
induced by the summand-swapping diffeomorphism $\ve{R}_\pi$ of $(\Sphere^3, K \# K)$
described in Section~\ref{sec:rigidmotiondeformation}.

\begin{figure}[ht!]
	\centering
\[
\begin{tikzcd}[column sep=.2cm,row sep=tiny,labels=description]
\cCFL^\infty(T_{3,4})&&\xs_0& \xs_1\arrow{l}{V}\arrow{r}{U^2}& \xs_2& \xs_3\arrow{l}{V^2}\arrow{r}{U} & \xs_4 \\
(\gr_{\ws}, A)&= & (-6,-3)&(-5,-2)&(-2,0)&(-1,2)&(0,3)
\\
\cCFL^\infty(T_{4,5})&&\xs_0& \xs_1\arrow{l}{V}\arrow{r}{U^3}& \xs_2& \xs_3\arrow{l}{V^2}\arrow{r}{U^2} & \xs_4 &\xs_5\arrow{l}{V^3}\arrow{r}{U}& \xs_6\\
(\gr_{\ws}, A)&= & (-12,-6)&(-11,-5)&(-6,-2)&(-5,0)&(-2,2)&(-1,5) & (0,6)
\end{tikzcd}
\]
	\caption{The complexes $\cCFL^\infty(T_{3,4})$ and $\cCFL^\infty(T_{4,5})$.}
\label{fig::T34-T45}
\end{figure}

The following has been computed using SageMath:

\begin{lem}
\label{lem:Sage}
\begin{enumerate}
\item For the pair $(D_{T_{3,4}\# T_{3,4}, \ve{R}^{\pi}}, D_{T_{3,4}\# T_{3,4},\id})$,
    we have $\tau = 2$, $V_0 = 1$, $V_1 = 1$, and $V_2 = 0$.
    A plot of $\Upsilon$ is shown on the left-hand side of Figure~\ref{fig:T34T45Upsilon}.
\item For the pair $(D_{T_{4,5} \# T_{4,5}, \ve{R}^{\pi}}, D_{T_{4,5}\# T_{4,5}, \id})$,
    we have $\tau = 3$, $V_0 = 2$, $V_1 = 1$, $V_2 = 1$, and $V_3 = 0$.
    A plot of $\Upsilon$ is shown on the right-hand side of Figure~\ref{fig:T34T45Upsilon}.
\end{enumerate}
\end{lem}

\begin{figure}[ht!]
	\centering
	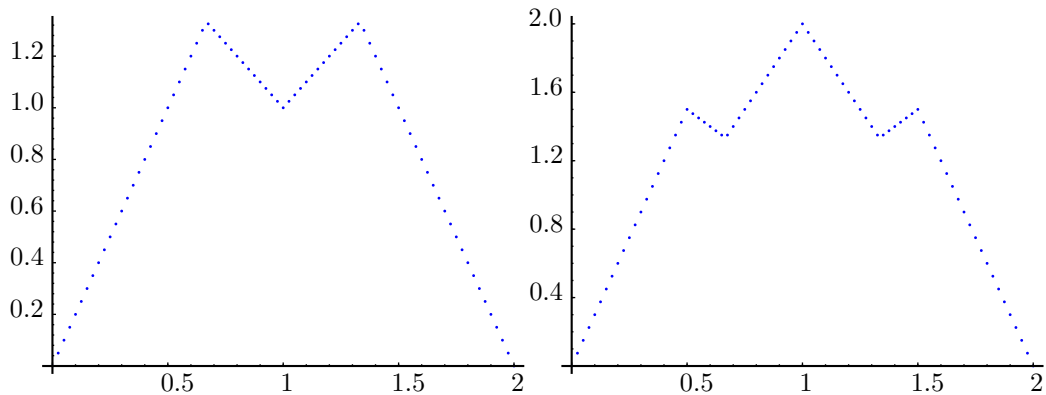
	\caption{The $\Upsilon(t)$ functions for the pairs
    $(D_{T_{3,4} \# T_{3,4}, \ve{R}^{\pi}}, D_{T_{3,4}\# T_{3,4},0}, \id)$ (left)
    and $(D_{T_{4,5} \# T_{4,5}, \ve{R}^{\pi}}, D_{T_{4,5} \# T_{4,5}, \id})$ (right),
    computed using SageMath.}\label{fig:T34T45Upsilon}
\end{figure}

An immediate corollary of Lemma~\ref{lem:Sage} and Theorems~\ref{thm:tauD1D2bound}
and~\ref{thm:doublepointbound} is the following:

\begin{cor}
Let $\omega \in \{\mu_{\st},\mu_{\Sing}\}$. Then
  \[
  \begin{split}
  \omega(D_{T_{3,4}\# T_{3,4}, \ve{R}^{\pi}}, D_{T_{3,4}\# T_{3,4}, \id}) &\ge 2 \text{, and} \\
  \omega(D_{T_{4,5} \# T_{4,5}, \ve{R}^{\pi}}, D_{T_{4,5} \# T_{4,5}, \id}) &\ge 3.
  \end{split}
  \]
\end{cor}

\subsection{Slice disks with large stabilization distance}

We now prove Theorem~\ref{thm:arbitrarily-large-torsion-order-intro} of the introduction.

\begin{thm}\label{thm:arbitrarily-large-torsion-order}
 Given $n\ge 0$, there is a knot $K_n$ and a pair of slice disks $D_1$ and $D_2$ for $K_n$ such that $\tau(D_1,D_2)\ge n$. 
\end{thm}
If $K$ is a knot in $S^3$, consider the \emph{$V$-torsion order} of $\HFK^-_{U=0}(K)$, for which we write $\Tor_V(K)$. This is the minimal $n\in \N\cup \{\infty\}$ such that 
\[
V^n\cdot \Tor(\HFK^-_{U=0}(K))=\{0\}.
\]
See \cite{AEUnknotting} and \cite{JMZTorsion} for examples of applications of the torsion order in knot Floer homology.

Let $K$ be a knot in $S^3$, and consider the slice knot $J=K_1\# K_2\# \bar K_3\# \bar K_4$, where each $K_i$ denotes a copy of $K$. We define two slice disks for $J$, which are boundary connected sums of slice disks for pairs of summands, as follows. Let $D_1$ be the spun slice disk obtained by viewing $J$ as $(K_1\# \bar K_3)\# (K_2\# \bar  K_4)$ and taking the boundary connected sum of the Artin spun slice disks for $K_1\# \bar K_3$ and $K_2\# \bar K_4$, and let $D_2$ be the spun slice disks obtained by viewing $J$ as $(K_1\# \bar K_4)\#(K_2\# \bar K_3)$, and taking a similar boundary connected sum.

\begin{lem}\label{lem:tau=torsion-order}
If $K$ is a knot and $D_1$ and $D_2$ are the slice disks for $J=K\# K\# \bar K \# \bar K$ described above, then 
\[
\tau(D_1,D_2)=\Tor_V(K). 
\]
\end{lem}
\begin{proof}
Firstly, note that the connected sum formula and the duality formula for mirroring knots implies that $\Tor_V(L)=\Tor_V(-L)$, and also $\Tor_V(L\# M)=\max(\Tor_V(L),\Tor_V(M))$ for any knots $L$ and $M$. In particular $\Tor_V(J)=\Tor_V(K)$.

We claim firstly that 
\begin{equation}
\tau(D_1,D_2)\le \Tor_V(K).\label{eq:torsion-D1D2-inequality}
\end{equation} 
 This follows from algebraic considerations. Indeed, $[t_{D_1}^-(1)]=[t_{D_2}^-(1)]+\sigma$, where $\sigma\in \HFK^-_{U=0}(J)$ is $V$-torsion. Hence, if $n=\Tor_V(K)$, then $V^n\cdot \sigma=0$, so
\[
V^n\cdot  [t_{D_1}^-]=V^n\cdot [t_{D_2}].
\]
This establishes equation~\eqref{eq:torsion-D1D2-inequality}.

To establish the reverse inequality of equation~\eqref{eq:torsion-D1D2-inequality}, we argue as follows. Consider the connected sum  decomposition of $J$ as $(K_1\# \bar{K}_3)\#(K_2\# \bar{K}_4)$. The corresponding 2-sphere gives a pair-of-pants cobordism from $(S^3,J)$ to $(S^3,K\# \bar K)\sqcup (S^3, K\# \bar K)$. Denote the cobordism map by $F$. By composing $t_{D_1}^-$ and $t_{D_2}^-$ with $F$, we may view the induced elements as chain maps
\[
T_1,T_2\in \Hom_{\bF[V]}(\CFK^-_{U=0}(K\# \bar K), \CFK^-_{U=0}(K\# \bar K)).
\]
Since $F$ is a homotopy equivalence,
\[
V^n\cdot [t^-_{D_1}]=V^n\cdot [t^-_{D_1}]
\]
if and only if $V^n\cdot  [F(t^-_{D_1})]=V^n \cdot [F(t^-_{D_1})]$, which in turn occurs if and only if $V^n\cdot T_1\simeq V^n\cdot T_2$, where $\simeq$ denotes $\bF[V]$-equivariant chain homotopy.

The maps $T_1$ and $T_2$ may be identified with concordance maps for concordances from $K\# \bar K$ to $K\# \bar K$. The map $T_1$ is identified with the identity map $\id$. On the other hand, the map $T_2$ is the concordance map for a concordance which factors through the unknot. In particular, $T_2$ must annihilate all torsion, as $\Tor_V(\mathrm{Unknot})=0$. In particular, if $V^n\cdot \id\simeq V^n\cdot T_2$, then $n$ must be larger than the torsion order of $K\# \bar K$.
\end{proof}

We now prove Theorem~\ref{thm:arbitrarily-large-torsion-order}:
\begin{proof}[Proof of Theorem~\ref{thm:arbitrarily-large-torsion-order}]
It suffices to construct knots where $\Tor_V(K)\ge n$. This is straightforward. For example, $\Tor_V(T_{p,q})=\min(p,q)-1$. (This fact is well known, but a proof may be found in \cite{JMZTorsion}*{Lemma~5.3}).
\end{proof}

\section{The cobordism distance}

In this section, we consider the following notion of distance between two surfaces:

\begin{define} Suppose that $g\in \N$, and $S$, $S'$ are two slice surfaces of a knot $K \subset S^3$. We say that $S$ and $S'$ are \emph{strictly $g$-cobordant} if there is a smoothly embedded, orientable 3-manifold $Y \subset I \times B^4$, such that the following are satisfied:
\begin{enumerate}
	\item $\d Y = (I \times K) \cup -(\{0\} \times S) \cup (\{1\} \times S')$.
	\item Projection of $Y$ onto $I$ is Morse.
	\item The sum of the genera of the components of each regular level set of $Y$ at most $g$.
\end{enumerate}
\end{define}

We write $\mu_{\Cob}(S, S')$ for the minimal $g$ such that $S$ and $S'$ are strictly $g$-cobordant. We call this quantity the \emph{cobordism distance} of $S$ and $S'$. In the case of a 0-cobordism, 
this coincides with the notion of a \emph{$0$-cobordism} introduced by Melvin \cite{MelvinPHD} for 2-knots. He defined a \emph{$g$-cobordism} to be one where each component of every level set has genus at most $g$.

The main result of this section is that the invariant $\tau(D, D')$ gives a lower bound on $\mu_{\Cob}(D, D')$ for slice disks $D$ and $D'$.

\begin{thm}\label{thm:k-concordancestau}
Suppose $D$ and $D'$ are two slice disks of a knot $K$ in $S^3$. Then 
\[
\tau(D, D') \le \mu_{\Cob}(D, D').
\]
\end{thm}

The main additional subtlety in the proof of Theorem~\ref{thm:k-concordancestau} is that the level sets of a strict $g$-cobordism $Y$ need not be connected, whereas the link cobordism maps vanish when there is a closed component. Hence, some care is required in the proof.

\subsection{Tubing disconnected surfaces}

In this section, we describe a way of meaningfully assigning cobordism maps to disconnected surfaces by tubing the components together.

\begin{define}Suppose that $W$ is a compact 4-manifold with boundary $Y$, and $S$ is a properly embedded, orientable surface in $W$. Suppose further that $\d S$ is equal to a knot $K\subset Y$. A \emph{tubing} of $S$ is a properly embedded surface $\hat{S}\subset W$ with $\d \hat{S}=K$, obtained by attaching tubes to $S$ which are the boundaries of 3-dimensional 1-handles in $W$. Furthermore, we assume  $g(\hat{S})=g(S)$ and $\hat{S}$ is connected.
\end{define}

We now prove a local relation for the graph cobordism maps (cf.~\cite{ZemHatGraphTQFT}*{Lemma~6.2}):

\begin{lem}\label{lem:bypass-to-graphs}
The graph cobordism maps satisfy the relation shown on the bottom of Figure~\ref{fig::4concordance}.
\end{lem}
\begin{proof} We begin with the bypass relation for the knot Floer cobordism maps, which is shown on the top of Figure~\ref{fig::4concordance}. We may take the underlying link cobordism to be $I \times K$, where $K$ is an unknot. The bypass relation for the link cobordism maps is proven in \cite{ZemConnectedSums}*{Lemma~1.4}. We then obtain the graph relation by considering the $V=1$ reductions of the graph cobordism maps, following \cite{ZemCFLTQFT}*{Theorem~C}.
\end{proof}

\begin{figure}[ht!]
	\centering
\begingroup%
  \makeatletter%
  \providecommand\color[2][]{%
    \errmessage{(Inkscape) Color is used for the text in Inkscape, but the package 'color.sty' is not loaded}%
    \renewcommand\color[2][]{}%
  }%
  \providecommand\transparent[1]{%
    \errmessage{(Inkscape) Transparency is used (non-zero) for the text in Inkscape, but the package 'transparent.sty' is not loaded}%
    \renewcommand\transparent[1]{}%
  }%
  \providecommand\rotatebox[2]{#2}%
  \newcommand*\fsize{\dimexpr\f@size pt\relax}%
  \newcommand*\lineheight[1]{\fontsize{\fsize}{#1\fsize}\selectfont}%
  \ifx\svgwidth\undefined%
    \setlength{\unitlength}{383.8197953bp}%
    \ifx\svgscale\undefined%
      \relax%
    \else%
      \setlength{\unitlength}{\unitlength * \real{\svgscale}}%
    \fi%
  \else%
    \setlength{\unitlength}{\svgwidth}%
  \fi%
  \global\let\svgwidth\undefined%
  \global\let\svgscale\undefined%
  \makeatother%
  \begin{picture}(1,0.54390996)%
    \lineheight{1}%
    \setlength\tabcolsep{0pt}%
    \put(0,0){\includegraphics[width=\unitlength,page=1]{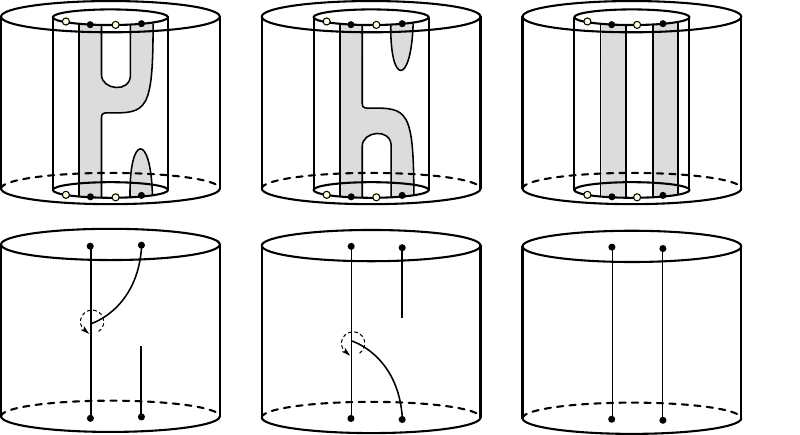}}%
    \put(0.30067364,0.39694121){\color[rgb]{0,0,0}\makebox(0,0)[t]{\lineheight{1.25}\smash{\begin{tabular}[t]{c}$+$\end{tabular}}}}%
    \put(0.62350009,0.39694121){\color[rgb]{0,0,0}\makebox(0,0)[t]{\lineheight{1.25}\smash{\begin{tabular}[t]{c}$+$\end{tabular}}}}%
    \put(0.30457119,0.12441527){\color[rgb]{0,0,0}\makebox(0,0)[t]{\lineheight{1.25}\smash{\begin{tabular}[t]{c}$+$\end{tabular}}}}%
    \put(0.62739764,0.12441527){\color[rgb]{0,0,0}\makebox(0,0)[t]{\lineheight{1.25}\smash{\begin{tabular}[t]{c}$+$\end{tabular}}}}%
    \put(0.94132281,0.39720074){\color[rgb]{0,0,0}\makebox(0,0)[lt]{\lineheight{1.25}\smash{\begin{tabular}[t]{l}$\simeq 0$\end{tabular}}}}%
    \put(0.93712877,0.12467474){\color[rgb]{0,0,0}\makebox(0,0)[lt]{\lineheight{1.25}\smash{\begin{tabular}[t]{l}$\simeq 0$\end{tabular}}}}%
  \end{picture}%
\endgroup%

	\caption{The bypass relation, as well as an induced relation obtained by setting $V=1$, in terms of graph cobordisms.}\label{fig::4concordance}
\end{figure}

\begin{prop}\label{prop:maps-independent-of-tubing} Suppose that $S$ is a properly embedded, oriented surface in $B^4$ with boundary equal to a knot $K$. Suppose that $\widehat{S}_1$ and $\widehat{S}_2$ are two tubings of $S$. Then
\[
\ve{t}_{\hat{S}_1,\zs}^-\simeq \ve{t}_{\hat{S}_2,\zs}^-.
\]
\end{prop}
\begin{proof} Any tubing of $S$ is isotopic to a tubing where each tube has one foot on the component of $S$ containing $K$, and one foot on a closed component of $S$. Since tubes are boundaries of 3-dimensional 1-handles, we may assume that, after an isotopy, any two such tubings have disjoint tubes. In particular, it suffices to change tubes one at a time. 

We assume that $T$ and $T'$ are tubes which have their feet on the same components of $S$. We assume the feet of the tubes are very close, and we pick an open neighborhood of the two tubes which is diffeomorphic to $S^1\times B^3$. We can factor the two cobordism maps through $(S^1\times S^2,O_2)$, where $O_2$ is a two-component unlink in $S^1\times S^2$.

We will prove the tube relation shown in  Figure~\ref{fig::1concordance}. This tube relation may be proven by considering the $V=1$ reduction of the link cobordism maps, and then applying the graph relation shown in Figure~\ref{fig::4concordance}. Since the link in $S^1\times S^2$ is an unlink with 2 basepoints per component, the link cobordism maps are determined by the graph cobordism maps, and so it is sufficient to prove the analogous formula for the graph cobordism maps. We do this in Figure~\ref{fig::5concordance}, using the local relation from Lemma~\ref{lem:bypass-to-graphs},
which is shown in Figure~\ref{fig::4concordance}.

We now claim that the cobordism map for the right-most surface in Figure~\ref{fig::1concordance} factors through the $H_1$-action. This is proven as follows. The $V=1$ reduction factors through the $H_1$-action by \cite{ZemDualityMappingTori}*{Proposition~4.6}. Since the Alexander grading change of the link cobordism map is zero, Lemma~\ref{lem:compute-using-reductions} implies that the link cobordism itself factors through the $H_1$-action. In particular, once we compose this cobordism map for $S^1\times B^3$ with the remainder of the cobordism map for the surface in $B^4$, we obtain the 0-map since $b_1(B^4)=0$.

\end{proof}

\begin{figure}[ht!]
	\centering
\begingroup%
  \makeatletter%
  \providecommand\color[2][]{%
    \errmessage{(Inkscape) Color is used for the text in Inkscape, but the package 'color.sty' is not loaded}%
    \renewcommand\color[2][]{}%
  }%
  \providecommand\transparent[1]{%
    \errmessage{(Inkscape) Transparency is used (non-zero) for the text in Inkscape, but the package 'transparent.sty' is not loaded}%
    \renewcommand\transparent[1]{}%
  }%
  \providecommand\rotatebox[2]{#2}%
  \newcommand*\fsize{\dimexpr\f@size pt\relax}%
  \newcommand*\lineheight[1]{\fontsize{\fsize}{#1\fsize}\selectfont}%
  \ifx\svgwidth\undefined%
    \setlength{\unitlength}{419.02800432bp}%
    \ifx\svgscale\undefined%
      \relax%
    \else%
      \setlength{\unitlength}{\unitlength * \real{\svgscale}}%
    \fi%
  \else%
    \setlength{\unitlength}{\svgwidth}%
  \fi%
  \global\let\svgwidth\undefined%
  \global\let\svgscale\undefined%
  \makeatother%
  \begin{picture}(1,0.27166358)%
    \lineheight{1}%
    \setlength\tabcolsep{0pt}%
    \put(0,0){\includegraphics[width=\unitlength,page=1]{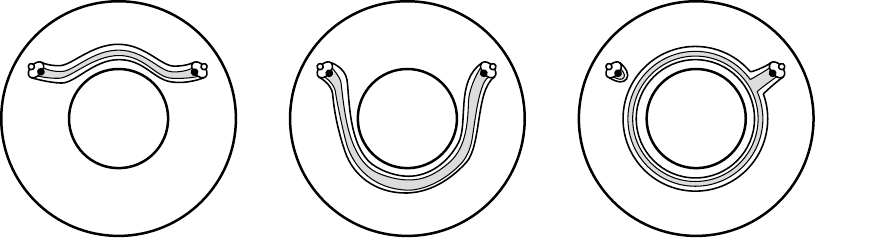}}%
    \put(0.30122676,0.12702028){\color[rgb]{0,0,0}\makebox(0,0)[t]{\lineheight{1.25}\smash{\begin{tabular}[t]{c}$+$\end{tabular}}}}%
    \put(0.63206789,0.12702028){\color[rgb]{0,0,0}\makebox(0,0)[t]{\lineheight{1.25}\smash{\begin{tabular}[t]{c}$+$\end{tabular}}}}%
    \put(0.94625307,0.127258){\color[rgb]{0,0,0}\makebox(0,0)[lt]{\lineheight{1.25}\smash{\begin{tabular}[t]{l}$\simeq 0$\end{tabular}}}}%
  \end{picture}%
\endgroup%

	\caption{A relation involving tubes.}\label{fig::1concordance}
\end{figure}

\begin{figure}[ht!]
	\centering
\begingroup%
  \makeatletter%
  \providecommand\color[2][]{%
    \errmessage{(Inkscape) Color is used for the text in Inkscape, but the package 'color.sty' is not loaded}%
    \renewcommand\color[2][]{}%
  }%
  \providecommand\transparent[1]{%
    \errmessage{(Inkscape) Transparency is used (non-zero) for the text in Inkscape, but the package 'transparent.sty' is not loaded}%
    \renewcommand\transparent[1]{}%
  }%
  \providecommand\rotatebox[2]{#2}%
  \newcommand*\fsize{\dimexpr\f@size pt\relax}%
  \newcommand*\lineheight[1]{\fontsize{\fsize}{#1\fsize}\selectfont}%
  \ifx\svgwidth\undefined%
    \setlength{\unitlength}{423.80272794bp}%
    \ifx\svgscale\undefined%
      \relax%
    \else%
      \setlength{\unitlength}{\unitlength * \real{\svgscale}}%
    \fi%
  \else%
    \setlength{\unitlength}{\svgwidth}%
  \fi%
  \global\let\svgwidth\undefined%
  \global\let\svgscale\undefined%
  \makeatother%
  \begin{picture}(1,0.26860291)%
    \lineheight{1}%
    \setlength\tabcolsep{0pt}%
    \put(0.29691601,0.13443715){\color[rgb]{0,0,0}\makebox(0,0)[t]{\lineheight{1.25}\smash{\begin{tabular}[t]{c}$+$\end{tabular}}}}%
    \put(0.62953695,0.13443715){\color[rgb]{0,0,0}\makebox(0,0)[t]{\lineheight{1.25}\smash{\begin{tabular}[t]{c}$+$\end{tabular}}}}%
    \put(0.94685861,0.13467219){\color[rgb]{0,0,0}\makebox(0,0)[lt]{\lineheight{1.25}\smash{\begin{tabular}[t]{l}$\simeq 0$\end{tabular}}}}%
    \put(0,0){\includegraphics[width=\unitlength,page=1]{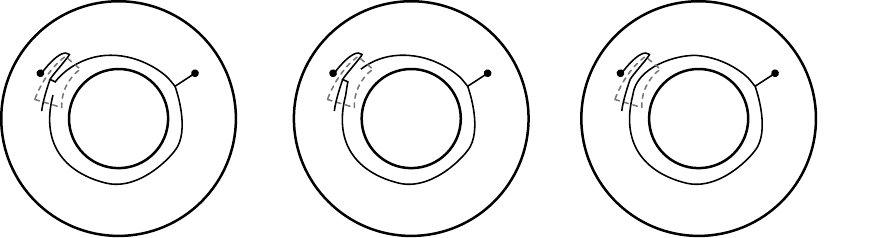}}%
  \end{picture}%
\endgroup%

	\caption{A relation involving graphs.}\label{fig::5concordance}
\end{figure}

\subsection{Proof of Theorem~\ref{thm:k-concordancestau}}

\begin{proof}[Proof of Theorem~\ref{thm:k-concordancestau}]
Suppose that $Y$ is a strict $g$-cobordism between slice disks $D$ and $D'$ of $K \subset S^3$. The projection from $Y$ onto the factor $I$ of $I \times B^4$ is a Morse function, by assumption. We may understand this Morse function as determining a sequence of 3-dimensional handles attached to $\{0\}\times D$, which build the 3-manifold $Y$. These handles may be of any index in $\{0,1,2, 3\}$. A 0-handle or 3-handle corresponds to adding or deleting an unknotted 2-sphere that is unlinked from the rest of the surface.  Attaching a 1-handle corresponds to a 1-handle stabilization, while a 2-handle corresponds to 1-handle destabilization. Let us write $S_1,\dots, S_n$ for a sequence of surfaces induced by a strict $g$-cobordism. By definition
\[
k=\max_{1\le i \le n} g(S_i),
\]
where $g(S_i)$ is the sum of the genera of the components of $S_i$.

The surfaces $S_1,\dots, S_n$ will in general not be connected. Let $\hat{S}_i$ be any tubing of $S_i$ for $i \in \{1, \dots, n\}$. Note that, by definition,
$g(\hat{S}_i) = g(S_i)$. We decorate each $\hat{S}_i$ with a dividing set such that $(\hat{S}_i)_{\ws}$ is a bigon. We write $\hat{\cS}_i$ for this decorated surface.
Proposition~\ref{prop:maps-independent-of-tubing} implies that the map $\ve{t}^-_{\hat{\cS}_i,\zs}$  is independent of the choice of tubing, and hence depends only on $S_i$.
 
 If $S_i$ is obtained from $S_{i-1}$ by a 0-handle, then we can pick tubes so that $\hat{S}_i$ is isotopic to $\hat{S}_{i-1}$.   If $S_i$ is obtained by a 3-handle, then, after changing tubes if necessary, the same is true. In particular, $\ve{t}^-_{\widehat{\cS}_i,\zs}=\ve{t}^-_{\widehat{\cS}_{i-1},\zs}$ if $S_i$ is obtained from $S_{i-1}$ by attaching a 0-handle or a 3-handle.
 
 If $S_{i}$ is obtained by attaching a 1-handle to $S_{i-1}$, then either $g(S_i)=g(S_{i-1})$ or $g(S_i)=g(S_{i-1})+1$. In the first case, the 1-handle connects two different components of $S_{i-1}$, and consequently, after changing tubes if necessary, $\hat{S}_i$ and $\hat{S}_{i-1}$ are isotopic. In the second case, $\hat{S}_i$ is obtained by stabilizing $\hat{S}_{i-1}$. We have the same conclusions for a 2-handle attachment, with the roles of $S_{i-1}$ and $S_i$ reversed. Consequently, using the formula in Lemma~\ref{lem:stabilizationlemma} for stabilization, the maps $V^k \cdot \ve{t}^-_{\hat{S}_i,\zs}$ coincide for all~$i$. So $\tau(D, D') \le k$ by Lemma~\ref{lem:slightreformulationoftau}, which completes the proof.
\end{proof}

\bibliographystyle{custom}
\bibliography{biblio}

\end{document}